\documentclass[a4paper,10pt,reqno]{amsart}
\numberwithin{equation}{section}
\usepackage{latexsym,amssymb}
\usepackage{hyperref}
\usepackage{amsmath,amsthm}
\usepackage{amsfonts}
\usepackage[american]{babel}
\usepackage{mathrsfs}
\usepackage{psfrag}
\usepackage{epsfig,inputenc}
\usepackage{graphpap,latexsym,epsf}
\usepackage{color}
\usepackage{amssymb,eucal,paralist,color,enumerate}
\usepackage{graphicx}

\newcommand{\bbM}{\mathbb{M}}

\newcommand{\R}{\mathbb{R}}
\newcommand{\N}{\mathbb{N}}

\mathchardef\emptyset="001F

\numberwithin{equation}{section}
\newtheorem{maintheorem}{Theorem}
\newtheorem{theorem}{Theorem}[section]

\newtheorem{lemma}[theorem]{Lemma}
\newtheorem{remark}[theorem]{Remark}

\newtheorem{definition}[theorem]{Definition}
\newtheorem{proposition}[theorem]{Proposition}
\newtheorem{problem}[theorem]{Problem}

\newtheorem{notation}[theorem]{Notation}

\newcommand{\eps}{\varepsilon}

\newcommand{\weakto}{\rightharpoonup} 

\newcommand{\forae}{\text{for a.a. }}
\newcommand{\aein}{\text{a.e.\ in }}

\newcommand{\down}{\downarrow}
\newcommand{\weaksto}{\overset{*}{\rightharpoonup}}

\newcommand{\AC}{\mathrm{AC}}

\newcommand{\dualoperator}

\def\calA{{\mathcal A}}  
 \def\calE{{\mathcal E}} \def\calF{{\mathcal F}}
 \def\calH{{\mathcal H}} 
 \def\calK{{\mathcal K}} \def\calL{{\mathcal L}}
  
 \def\calQ{{\mathcal Q}} \def\calR{{\mathcal R}}
 \def\calT{{\mathcal T}}

  \def\rmC{{\mathrm C}}

\def\rmM{{\mathrm M}}

\def\BS{\boldsymbol} 

\def\bflambda{{\BS\lambda}}    \def\bfmu{{\BS\mu}}

\def\dd{\;\!\mathrm{d}} 

\newcommand{\pairing}[4]{ \sideset{_{ #1 }}{_{ #2 }}  {\mathop{\langle #3 , #4
\rangle}}}

\newcommand{\teta}{\vartheta}

\newcommand{\nchi}{{\raise.2ex\hbox{$\chi$}}}

\definecolor{ddcyan}{rgb}{0,0.1,0.9}
\definecolor{ddmagenta}{rgb}{0.8,0,0.8}
\definecolor{orange}{rgb}{0.6,0.2,0}
\definecolor{vgreen}{rgb}{0.1,0.5,0.2}
\definecolor{dred}{rgb}{.8,0,0}
\definecolor{Turk}{rgb}{0,0.7,0.4}

\newcommand{\piecewiseConstant}[2]{\overline{#1}_{\kern-1pt#2}}
\newcommand{\pwc}{\piecewiseConstant}

\newcommand{\upiecewiseConstant}[2]{\underline{#1}_{\kern-1pt#2}}

\newcommand{\upwc}{\upiecewiseConstant}
\newcommand{\piecewiseLinear}[2]{{#1}_{\kern-1pt#2}}
\newcommand{\pwl}{\piecewiseLinear}
\newcommand{\pwwll}[2]{\widehat{#1}_{\kern-1pt#2}}

\newcommand{\piecewiseVariational}[2]{\tilde{#1}_{\kern-1pt#2}}

\newcommand{\DDDn}[2]{\begin{array}[t]{c}#1\vspace*{-1em}\\_{#2}\end{array}}

\newcommand{\dddn}[2]{\DDDn{\begin{array}[t]{c}\underbrace{#1}\vspace*{.6em}\end{array}}{\text{\footnotesize #2}}}

\newcommand{\foraa}{\text{for a.a. }}

\newcommand{\ue}{u_\varepsilon}

\newcommand{\sig}[1]{E(#1)}

\newcommand{\norm}[2]{\| #1\|_{#2}}

\newcommand{\BV}{\mathrm{BV}}
\newcommand{\BD}{\mathrm{BD}}
\newcommand{\Var}{\mathrm{Var}}

\newcommand{\Dir}{\mathrm{Dir}}
\newcommand{\Neu}{\mathrm{Neu}}

 \def\trait #1 #2 #3 {\vrule width #1pt height #2pt depth #3pt}

 \def\fin{\hfill
         \trait .3 5 0
         \trait 5 .3 0
         \kern-5pt
         \trait 5 5 -4.7
         \trait 0.3 5 0
 \medskip}
 
\newcommand{\QED}{\mbox{}\hfill\rule{5pt}{5pt}\medskip\par}

\newcommand{\bsL}{\boldsymbol{L}}
\newcommand{\bsM}{\boldsymbol{M}}

\newcommand{\bsU}{\boldsymbol{U}}
\newcommand{\bsE}{\boldsymbol{E}}

\newcommand{\bbC}{\mathbb{C}}
\newcommand{\bbD}{\mathbb{D}}
\newcommand{\bbE}{\mathbb{E}}
\newcommand{\bsX}{\mathbf{X}}
\newcommand{\bsV}{\mathbf{V}}
\newcommand{\bsY}{\mathbf{Y}}

\newcommand{\bbB}{\mathbb{B}}

\newcommand{\mt}{\bbM}

\newcommand{\sym}{\mathrm{sym}}
\newcommand{\dev}{\mathrm{D}}
\newcommand{\uu}{u}

\newcommand{\condu}{\kappa}

\newcommand{\Ftau}[1]{F_\tau^{#1}}
\newcommand{\Ltau}[1]{\mathcal{L}_\tau^{#1}}
\newcommand{\gtau}[1]{H_\tau^{#1}}
\newcommand{\htau}[1]{h_\tau^{#1}}
\newcommand{\wtau}[1]{w_\tau^{#1}}
\newcommand{\utau}[1]{u_\tau^{#1}}
\newcommand{\vtau}[1]{v_\tau^{#1}} 
\newcommand{\btau}[1]{\mathfrak{h}_\tau^{#1}}

\newcommand{\ptau}[1]{p_\tau^{#1}}
\newcommand{\etau}[1]{e_\tau^{#1}}
\newcommand{\tetau}[1]{\teta_\tau^{#1}}
\newcommand{\sitau}[1]{\sigma_{\tau}^{#1}}
\newcommand{\simtau}[1]{\sigma_{M,\tau}^{#1}}
\newcommand{\sidevtau}[1]{(\sigma_{\tau}^{#1})_\dev}
\newcommand{\simdevtau}[1]{(\sigma_{M,\tau}^{#1})_\dev}
\newcommand{\dtau}[2]{\mathrm{D}_{#1,\tau}(#2)}
\newcommand{\Dtau}[2]{\mathrm{D}_{#1,\tau}(#2)}
\newcommand{\Ddtau}[2]{\mathrm{D}_{#1,\tau}^2(#2)}
\newcommand{\tetaum}[1]{\teta_{M,\tau}^{#1}}
\newcommand{\utaum}[1]{u_{M,\tau}^{#1}}
\newcommand{\ptaum}[1]{p_{M,\tau}^{#1}}
\newcommand{\etaum}[1]{e_{M,\tau}^{#1}}
\newcommand{\zetau}[1]{\zeta_{\tau}^{#1}}
\newcommand{\zetaum}[1]{\zeta_{\tau,M}^{#1}}

\newcommand{\sie}{\sigma_\eps}
\newcommand{\siedev}{(\sigma_{\eps})_\dev}
\newcommand{\tetae}{\teta_\eps}
\newcommand{\pe}{p_\eps}

\newcommand{\uek}{u_{\eps_k}}
\newcommand{\tetaek}{\teta_{\eps_k}}
\newcommand{\pek}{p_{\eps_k}}

\newcommand{\limk}{\hat{\mathsf{K}}}
\newcommand{\limte}{\Theta}

\newcommand{\RRR}{\color{magenta}}

\begin{document}
\title[From visco to perfect plasticity]{From visco to perfect plasticity in thermoviscoelastic materials}

\author{Riccarda Rossi}
\address{R.\ Rossi, DIMI, Universit\`a degli studi di Brescia,
via Branze 38, 25133 Brescia - Italy}
\email{riccarda.rossi\,@\,unibs.it}

\thanks{The author has been partially supported by the  Gruppo Nazionale per  l'Analisi Matematica, la
  Probabilit\`a  e le loro Applicazioni (GNAMPA)
of the Istituto Nazionale di Alta Matematica (INdAM)}

\date{March 16, 2018} 
\maketitle

\begin{abstract}
We consider a thermodynamically consistent  model  for thermoviscoplasticity. For the related PDE system, 
coupling the heat equation for the absolute temperature, the momentum balance  with viscosity and inertia for the displacement variable, and the flow rule for the plastic strain, we
 propose two  weak sol\-va\-bi\-li\-ty concepts, `entropic' and `weak energy' solutions,
where the highly nonlinear heat equation  is 
 suitably formulated. Accordingly, we prove two  existence results by passing to the limit in a carefully devised time discretization scheme.
 \par
 Furthermore, we study the asymptotic behavior of weak energy solutions  as the rate of the external data becomes slower and slower, which amounts to taking the vanishing-viscosity and inertia limit of the system. We prove their convergence to a global energetic solution to the Prandtl-Reuss model for perfect plasticity, whose evolution is `energetically' coupled to that of the (spatially constant) limiting temperature.
\end{abstract}

\section{\bf Introduction}
\noindent
Over the last decade, the mathematical study of rate-independent  systems
has received strong impulse. This is undoubtedly due to their ubiquity in  several branches of continuum mechanics,  see \cite{Miel05ERIS, MieRouBOOK},
but also to the manifold challenges posed by the analysis of rate-independent evolution. 
In particular, its intrinsically nonsmooth (in time) character makes it necessary to resort to suitable weak solvability notions: First and foremost,  the concept of 
\emph{(global) energetic solution}, developed in    \cite{MieThe99MMRI,MieThe04RIHM, Miel05ERIS}, cf.\ also 
 the notion of \emph{quasistatic evolution}, first introduced  for  models of crack propagation, cf.\ e.g.\ \cite{DMToa02, DFT05}. 
\par
Alternative solution concepts for rate-independent models have been subsequently proposed, on the grounds that 
the 
\emph{global stability} condition prescribed by the energetic notion fails to accurately describe the behavior of the system at jump times, as soon as the driving energy is 
nonconvex. Among the various selection criteria of mechanically feasible concepts,  let us mention here the \emph{vanishing-viscosity} approach, 
pioneered in \cite{EfeMie06RILS} and subsequently developed in the realm of \emph{abstract} rate-independent systems in \cite{MRS09,MRS12,MRS13} and, in parallel, 
in the context of  specific models in crack propagation  and damage, cf.\ e.g.\  \cite{ToaZan06?AVAQ,KnMiZa07?ILMC,LazzaroniToader,KnRoZa13VVAR}, as well  as in plasticity, see e.g.\ \cite{DalDesSol11,DalDesSol12,BabFraMor12, FrSt2013}. In  all of these applications, 
the evolution of the displacement variable is governed by the elastic equilibrium equation (with no viscosity or inertial terms), which is coupled to the rate-independent flow rule
for the  internal parameter describing the mechanical phenomenon under consideration.  In the `standard' vanishing-viscosity approach, 
 the viscous term, regularizing the temporal evolution and then sent to zero, is added \emph{only} to the flow rule.
 \par
   Let us mention that, in turn, for certain models a \emph{rate-dependent} flow rule seems more mechanically feasible, cf.\ e.g.\
    \cite{ZRSRZ06TDIT} in the frame of plasticity. Nonetheless, here we are interested in the vanishing-viscosity approach to rate-independence. 
 More specifically, we focus on 
 the extension of this approach  to \emph{coupled} systems. 
Recent papers have started to  address this issue  for  systems coupling the evolution of the  displacement and  of the internal variable.  In the context of the rate-dependent model, 
 \emph{both} the displacements and the internal variable
are subject to viscous dissipation (and possibly to inertia in the momentum balance), and the vanishing-viscosity limit is taken  \emph{both} in the momentum balance, and in the flow rule. 
The very first paper initiating this analysis 
 is \cite{DMSca14QEPP}, obtaining a \emph{perfect plasticity} (rate-independent) system in the limit of dynamic processes. We also quote \cite{Scala14}, where 
 this kind of approach was developed in the realm of a model for delamination, as well as \cite{MRS14}, tackling the analysis of 
 \emph{abstract, finite-dimensional} systems where the viscous terms vanish with different rates.
 \par
The model for 
small-strain associative elastoplasticity with the Prandtl-Reuss flow rule (without hardening) for the plastic strain, chosen in 
\cite{DMSca14QEPP} to pioneer the `full vanishing-viscosity' approach,  has been extensively studied. In fact, the existence theory for perfect plasticity
is by now classical, dating back to \cite{Johnson76, Suquet81,Kohn-Temam83}, cf.\ also \cite{Temam83}. It was revisited in \cite{DMDSMo06QEPL} in the 
framework of the aforementioned concept of (global) energetic solution to rate-independent systems, with the existence result established by passing to the limit 
in time-incremental minimization problems; a fine study of the flow rule for the plastic strain was also carried out in  \cite{DMDSMo06QEPL}. 
This variational approach has apparently given new impulse to the analysis of perfect plasticity, extended to the case of heterogeneous materials in
 \cite{Sol09,FraGia2012, Sol14,Sol15}; we also quote \cite{BMR12}  on the vanishing-hardening approximation of the  Prandtl-Reuss model.
 \par
 In \cite{DMSca14QEPP}, first of all an existence result for  a dynamic viscoelastoplastic system approximating the perfectly plastic one, featuring viscosity and inertia in the momentum balance, and 
 viscosity in the flow rule for the plastic tensor, has been obtained. Secondly, the authors have analyzed its behavior as the rate of the external data 
 becomes slower and slower: with a suitable rescaling, this amounts to taking the vanishing-viscosity and inertia limit of the system.
  They have shown that the (unique) solutions to the viscoplastic system converge, up to a subsequence, to a (global) energetic solution of the perfectly plastic system.
  \par
  In this paper, we aim to
  \textbf{use the model for perfect plasticity as a \emph{case study}} for the vanishing-viscosity analysis of rate-dependent systems that also
  \emph{encompass thermal effects}.    To our knowledge, this is the first paper where the vanishing-viscosity analysis in 
   a fully rate-dependent, and temperature-dependent, system has been performed.  
  \par
    Indeed, the analysis of systems with a \emph{mixed}  rate-dependent/rate-independent character, coupling the 
     \emph{rate-dependent}  
    evolution of the
  (absolute)  temperature and of the  displacement/deformation variables with the \emph{rate-independent} flow rule 
  of an internal variable, has been initiated   in \cite{Roub10TRIP}, and subsequently particularized to various mechanical models. 
   	While referring to \cite[Chap.\ 5]{MieRouBOOK}
  for  a survey of  these type of systems, we mention here   the   perfect plasticity  and damage   models studied in \cite{Roub-PP} in \cite{LRTT}, respectively.
   In the latter paper, 
 a vanishing-viscosity analysis (as the rate of the external loads and heat sources tends to zero) for the \emph{mixed} rate-dependent/independent damage model, 
  has  been performed.
  \par Instead, here the (approximating) thermoviscoplastic system will feature a \emph{rate-dependent} flow rule for the  plastic strain, and thus will be entirely
  rate-dependent.    
  \begin{itemize}
  \item
  First  of all, we will focus on  the  analysis of the rate-dependent system.  Exploiting the techniques  from \cite{Rocca-Rossi}, we  will  obtain two existence results, 
  which might be interesting in their own right,
   for two notions of solutions of the thermoviscoplastic system, referred to as `entropic' and `weak energy'. 
  Our proofs will be carried out  by passing to the limit in a carefully tailored time discretization scheme.
  \item
    Secondly, in the case of `weak energy' solutions we will perform the vanishing-viscosity asymptotics, obtaining a system where the 
  evolution of the displacement and of the elastic and plastic strains, in the sense  of 
  (global) energetic solutions, is coupled to  that  of the  (spatially constant) temperature variable. 
    In fact, 
  we could address this singular limit also for 
   entropic solutions, but the resulting formulation of the limiting rate-independent system would  be less meaningful due to 
   the too weak character of  the entropic solution notion, cf.\ also Remark \ref{rmk:2weak} ahead. 
  \end{itemize}
  \par
  Let us now get further insight into our analysis, first in the visco-, and then in the perfectly plastic cases.
  \subsection{The thermoviscoplastic system}
\label{ss:1.1}
The reference configuration is a bounded, open, Lipschitz domain  $\Omega\subset \R^d$, $d\in \{2, 3\}$, and we consider
 the evolution of the system in a time interval $(0,T)$.
Within the small-strain approximation, the momentum balance features the linearized strain tensor $\sig{u}= \tfrac12 \left( \nabla u+ \nabla u^\bot \right)$, decomposed
as 
\begin{equation}
\label{decomp-intro}
\sig u = e+p \qquad \text{ in } \Omega \times (0,T), 
\end{equation}
with $e \in \mt_\sym^{d \times d}$ (the space of symmetric $(d{\times}d)$-matrices) and $p \in \mt_\dev^{d \times d}$ (the space of symmetric
 $(d{\times}d)$-matrices with null trace)
 the elastic and plastic strains, respectively. 
In accord with the  Kelvin-Voigt rheology for materials subject to thermal expansion,   the stress 
is given by 
\begin{align}
\label{stress}
\sigma = \mathbb{D} \dot{e} + \mathbb{C}(e - \mathbb{E}\teta),
\end{align}
with 
$\teta$ the absolute temperature, and 
the elasticity, viscosity, and thermal expansion tensors $\bbC,\, \bbD,\, \bbE$  depending on the space variable $x$ (which
shall be overlooked in this Introduction for simplicity of exposition), symmetric, $\bbC$ and $\bbD$ positive definite. 
Then, we consider the following PDE system:
\begin{subequations}
\label{plast-PDE}
\begin{align}
& 
\label{heat}
\dot{\teta} - \mathrm{div}(\condu(\teta)\nabla \teta) =H+ \mathrm{R}(\teta,\dot{p}) + \dot{p}: \dot{p}+ \mathbb{D} \dot{e} : 
\dot{e} -\teta \mathbb{C}\mathbb{E}  : \dot{e} && \text{ in } \Omega \times (0,T),
\\
\label{mom-balance}
 &\rho \ddot{u} - \mathrm{div}\sigma = F && \text{ in } \Omega \times (0,T),
\\
&
\label{flow-rule}
\partial_{\dot{p}}  \mathrm{R}(\teta,\dot{p}) + \dot{p} \ni \sigma_{\mathrm{D}} && \text{ in } \Omega \times (0,T). 
\end{align}
\end{subequations}
The heat equation \eqref{heat} features as heat conductivity coefficient a nonlinear function $\condu \in \mathrm{C}^0(\R^+)$, which shall be
 supposed with a suitable growth. In the momentum balance \eqref{mom-balance}, $\rho>0$ is the (constant, for simplicity) mass density. 
 The evolution of the plastic strain  $p$ is given by the flow rule \eqref{flow-rule}, where
 $\sigma_\dev$ is the deviatoric part of the stress $\sigma$, and 
   the dissipation potential
 $\mathrm{R} :\R^+\times \mt_\dev^{d\times d} \to [0,+\infty) $ is lower semicontinuous, 
 and associated with a multifunction
 $K :\R^+ \rightrightarrows \mt_\dev^{d\times d}$, with values 
 in the compact and convex subsets of $\mt_\dev^{d\times d}$, via the relation
 \[
 R(\teta,  \dot{p})   = \sup_{\pi \in K(\teta)} \pi{:} \dot{p} \qquad \text{for all } (\teta, \dot p) \in \R^+\times \mt_\dev^{d\times d}
 \]
 (the  dependence of $K$ and $R$  on $x \in \Omega$ is  overlooked within this section).
 Namely, for every $\teta \in \R^+$ the potential  $R(\teta, \cdot)$ is the support function of the convex and compact set $K(\teta)$, which can be interpreted as the domain of 
 viscoelasticity, allowed to depend on $x\in \Omega$ as well as on the  temperature variable. 
 In fact, $R(\teta, \cdot)$ is the Fenchel-Moreau conjugate of the indicator function $I_{K(\teta)}$, and thus \eqref{flow-rule} (where $\partial_{\dot p}$ denotes the subdifferential in the sense of convex analysis w.r.t.\ the variable $\dot p$) rephrases as 
 \begin{equation}
 \label{flow-rule-rewritten}
 \begin{aligned} 
 \dot{p} \in \partial I_{K(\teta)}(\sigma_\dev {-} \dot p)  
 \
 \Leftrightarrow \   \dot{p}   = \sigma_{\mathrm{D}}  - \mathrm{P}_{K(\teta)} (\sigma_{\mathrm{D}} )
 \quad
 \text{ in } \Omega \times (0,T),
 \end{aligned}
 \end{equation}
 with $ \mathrm{P}_{K(\teta)} $ the projection operator onto $K(\teta)$. 
  The PDE system \eqref{plast-PDE}  is supplemented by the boundary conditions
\begin{subequations}
\label{bc}
\begin{align}
& 
\label{bc-u-1}
\sigma \nu =  g && \text{ on } \Gamma_\Neu \times (0,T),
\\
&
\label{bc-u-2}
u= w && \text{ on  } \Gamma_\Dir \times (0,T),
\\
&
\label{bc-teta}
\condu(\teta)\nabla \teta \nu = h  && \text{ on } \partial\Omega \times (0,T),
\end{align}
\end{subequations}
where $\nu$ is  the external unit normal to $\partial\Omega$, with $ \Gamma_\Neu $ and $\Gamma_\Dir$ its Neumann and  Dirichlet parts.
The body is subject to the volume force $F$,  to the applied traction $g$ on $\Gamma_\Neu$, and solicited by a displacement field $w$ applied on $\Gamma_\Dir$, while $H$ and $h$ are  bulk and surface (positive) heat sources, respectively. 
\par
A PDE system with the same  structure as  (\ref{plast-PDE}, \ref{bc}) was proposed in \cite{Roub-PP} to model the thermodynamics of perfect plasticity: i.e., a heat equation akin to \eqref{heat} and the momentum balance \eqref{mom-balance} were coupled to the \emph{rate-independent version} of the flow rule \eqref{flow-rule}, cf.\ \eqref{flow-rule-RIP} below. While the mixed rate-dependent/independent system in \cite{Roub-PP}   calls for a completely different analysis from our own, the modeling discussion  developed in   \cite[Sec.\ 2]{Roub-PP}   can be easily adapted to  system
 (\ref{plast-PDE}, \ref{bc}) to show 
 its  compliance with the first and 
second principle of thermodynamics. In particular, let us stress that, due to the presence of the \emph{quadratic} terms 
$ \dot{p}: \dot{p}$,  $\mathbb{D} \dot{e} : 
\dot{e} $, and $\teta \mathbb{C}\mathbb{E}  : \dot{e} $ on the right-hand side of \eqref{heat}, system (\ref{plast-PDE}, \ref{bc})  is \underline{thermodynamically consistent}.
\par
The analysis of (the Cauchy problem associated with) system  (\ref{plast-PDE}, \ref{bc})  
poses some significant mathematical difficulties:
\begin{description}
\item[\textbf{(1)}]
 First and foremost, its nonlinear character, and in particular the quadratic terms on the r.h.s.\ of \eqref{heat}, which is thus only estimated in $L^1((0,T) \times \Omega)$ as soon as $\dot p$ and $\dot e$ are estimated in $L^2((0,T) \times \Omega;\mt_\dev^{d\times d})$ and 
 $L^2((0,T) \times \Omega;\mt_\sym^{d\times d})$, respectively. Because of this, on the one hand obtaining suitable  estimates of the temperature variable turns out to be challenging. On the other hand, 
 appropriate weak formulations of \eqref{heat} are called for. 
 \end{description}
 \par
 In the one-dimensional case, existence results have been obtained for thermodynamically consistent (visco)\-plasticity models with hysteresis in \cite{KS97, KSS02, KSS03}.  In higher dimensions, suitable adjustments of the toolbox by \textsc{Boccardo \& Gallou\"et}
 \cite{Boccardo-Gallouet89} to handle  the heat equation with $L^1$/measure data have been devised in a series of recent papers on thermoviscoelasticity with rate-dependent/independent plasticity. In particular, we quote  \cite{Roub-Bartels-1}, dealing with a (rate-dependent) thermoviscoplastic model, where  thermal expansion effects are neglected, as well as  \cite{Roub-Bartels-2}, addressing rate-independent  plasticity with hardening  coupled with thermal effects, with the stress tensor given by 
$
 \sigma = \bbD \sig{u_t} + \bbC e - \bbC \bbE \teta $, and finally \cite{Roub-PP}, handling the thermodynamics of perfect plasticity.
 Let us point out that, in the estimates developed in  \cite{Roub-Bartels-2, Roub-PP}, a  crucial role is played by a sort of `compatibility condition' between the growth exponents of the ($\teta$-dependent)  heat capacity coefficient multiplying $\teta_t$, and of the heat conduction coefficient $\condu(\teta)$. This allows for  Boccardo-Gallou\"et type estimates, drawn from \cite{Roub10TRIP}. In the recent \cite{HMS17}, the analysis of the heat equation with $L^1$ right-hand side has been handled without growth conditions on the abovementioned coefficients by resorting to maximal parabolic regularity arguments, made possible   by  the crucial ansatz that the viscous contribution to $\sigma$ features $\sig{\dot u}$, in place of $\dot e$ as in \eqref{stress}.
 \par
 Here we  will instead stay with \eqref{stress}, which is  more consistent with \emph{perfect plasticity}. While   supposing that  the heat capacity  coefficient   is  constant
 (cf.\ also Remark \ref{rmk:in-LRTT} ahead), 
  we will  develop different arguments to derive estimates on the temperature variable based
  on a growth condition
 for the heat conduction coefficient. In this, we will follow
 the footsteps of \cite{FPR09, Rocca-Rossi}, analyzing thermodynamically consistent models for phase transitions and with damage. Namely, we shall suppose that 
 \begin{equation}
 \label{heat-cond-intro}
 \condu(\teta) \sim \teta^\mu \qquad \text{with } \mu>1.
 \end{equation}
We shall exploit \eqref{heat-cond-intro}
  upon testing \eqref{heat} by a suitable negative power of $\teta$ (all calculations can be rendered rigorously on the level of a time discretization scheme). In this way, we will   deduce a crucial estimate for $\teta$ in $L^2(0,T;H^1(\Omega))$.  Under \eqref{heat-cond-intro} we will address the weak solvability of  (\ref{plast-PDE}, \ref{bc}) in terms of  the 
 `entropic' notion of solution,
 proposed in the framework of models for heat conduction in fluids, cf.\ e.g.\ \cite{Feireisl2007,BFM2009}, and later used to weakly formulate models  for phase change \cite{FPR09} and, among other applications,
 for damage in thermoviscoelastic materials \cite{Rocca-Rossi}. In the framework of our plasticity system,
  this solution concept features the weak formulation of the momentum balance \eqref{mom-balance} and  the flow rule \eqref{flow-rule}, stated a.e.\ in $\Omega \times (0,T)$,  coupled with  
\begin{itemize}
\item[-] the 
\emph{entropy inequality}
  \begin{equation}
\label{entropy-ineq-intro}
\begin{aligned}
  & \int_s^t \int_\Omega  \log(\teta) \dot{\varphi} \dd x \dd r -   \int_s^t \int_\Omega \left( \condu(\teta) \nabla \log(\teta) \nabla \varphi - \condu(\teta) \frac\varphi\teta \nabla \log(\teta) \nabla \teta\right)   \dd x \dd r  
  \\ 
& \leq
  \int_\Omega \log(\teta(t)) \varphi(t) \dd x -  \int_\Omega \log(\teta(s)) \varphi(s) \dd x  \\ & \quad
    - \int_s^t \int_\Omega \left( H+ \mathrm{R}(\teta,\dot{p}) + |\dot{p}|^2+ \mathbb{D} \dot{e} : 
\dot{e} -\teta \bbB : \dot{e} \right) \frac{\varphi}\teta \dd x \dd r  - \int_s^t \int_{\partial\Omega} h \frac\varphi\teta \dd x \dd r
\end{aligned}
\end{equation}
with  $\varphi$  a sufficiently regular,  \emph{positive}  test function, 
  \item[-] the 
   \emph{total energy  inequality}
\begin{equation}
\label{total-enid-intro}
\begin{aligned}
& 
\frac{\rho}2 \int_\Omega |\dot{u}(t)|^2 \dd x +\mathcal{E}(\teta(t), e(t)) 
\\
&  \leq \frac{\rho}2 \int_\Omega |\dot{u}(s)|^2 \dd x +\mathcal{E}(\teta(s), e(s))  + \int_s^t \pairing{}{H_\Dir^1 (\Omega;\R^d)}{\mathcal{L}}{\dot u{-} \dot w}   +\int_s^t \int_\Omega H \dd x \dd r + \int_s^t \int_{\partial\Omega} h \dd S \dd r
\\
&
\begin{aligned}  \quad   +\rho \left( \int_\Omega \dot{u}(t) \dot{w}(t) \dd x -   \int_\Omega \dot{u}(s) \dot{w}(s) \dd x  
- \int_s^t \int_\Omega \dot{u}\ddot w \dd x \dd r \right)   &    + 
\int_s^t \int_\Omega \sigma: \sig{\dot w} \dd x \dd r 
\end{aligned}
 \end{aligned}
\end{equation}
  involving the total load $\calL$ associated with the external forces $F$ and $g$, and the  energy functional
$
 \mathcal{E}(\teta, e): = \int_\Omega \teta \dd x + \int_\Omega \tfrac12 \bbC e{:} e \dd x\,.
  $
  \end{itemize}
Both \eqref{entropy-ineq-intro} and \eqref{total-enid-intro}
are required to hold  for almost all $t \in (0,T]$ and almost
all $s\in (0,t)$, and for $s=0$.
\par
While referring to \cite{FPR09,Rocca-Rossi} for more details
and to Sec.\ \ref{ss:2.2} for a formal derivation of \eqref{entropy-ineq-intro}--\eqref{total-enid-intro},
let us point out here that this solution concept reflects the thermodynamic consistency of the model, since it corresponds 
to the requirement that the system
should satisfy the second and first principle of Thermodynamics.
 From an analytical viewpoint, observe that the entropy
inequality \eqref{entropy-ineq-intro} has the advantage that all the
 quadratic terms on the right-hand side of \eqref{heat}
feature as  multiplied  by a negative test  function. This
  allows
for upper semicontinuity arguments in the limit passage in  a
suitable approximation of
\eqref{entropy-ineq-intro}--\eqref{total-enid-intro}. Furthermore,  despite its weak character,
\emph{weak-strong uniqueness} results can  be seemingly obtained for the entropic formulation, cf.\ e.g.\ \cite{Fei-Nov} in the context of the 
Navier-Stokes-Fourier system modeling heat conduction in fluids.
\begin{description}
\item[\textbf{(2)}]
An additional analytical challenge is related to handling  a non-zero applied traction $g$ on the Neumann part of the boundary $\Gamma_\Neu$. This results in the term $ \int_0^T\pairing{}{H^1 (\Omega;\R^d)}{\mathcal{L}}{\dot u} \dd t $ on the r.h.s.\ of \eqref{total-enid-intro}, whose time discrete version is, in fact, the starting point in  the derivation of all of the a priori estimates.  The estimate of this term is delicate, since it would in principle involve the $H^1 (\Omega;\R^d)$-norm of $\dot {u} $, which is not controlled by the left-hand side of \eqref{total-enid-intro}. A by-part integration in time shifts the problem to estimating the $H^1 (\Omega;\R^d)$-norm of $u$, but the l.h.s.\ of 
 \eqref{total-enid-intro} only controls the $L^2(\Omega;\mt_\sym^{d\times d})$-norm of $e$. Observe that  this is  ultimately due to the form \eqref{stress} of the stress $\sigma$.
 \end{description}
 To overcome this problem, we will impose  that the data $F$ and $g$ comply with a suitable \emph{safe load} condition,  see also Remark \ref{rmk:diffic-1-test}.
 \par
 Finally, 
 \begin{description}
\item[\textbf{(3)}]
the presence of adiabatic effects in the momentum balance, accounted for by the thermal expansion  term coupling it with the heat equation, leads to yet another technical problem.
In fact, the estimate of the term $
\int_0^T \int_\Omega \teta \bbC \bbE{:} \sig{\dot w} \dd x \dd t $ contributing to the integral $\int_0^T \int_\Omega \sigma: \sig{\dot w} \dd x  \dd t $ on the 
r.h.s.\ of  \eqref{total-enid-intro} calls for suitable assumptions on the Dirichlet loading $w$, since the l.h.s.\ of 
\eqref{total-enid-intro} only controls the $L^1(\Omega)$-norm of $\teta$, cf.\  again  Remark \ref{rmk:diffic-1-test}.
\end{description}
\par
As already mentioned, we will tackle the existence analysis for the entropic formulation of system  (\ref{plast-PDE}, \ref{bc})  
by approximation via time discretization. In particular, along the footsteps of \cite{Rocca-Rossi}, we will carefully devise our time-discretization scheme
in such a way that the approximate solutions obtained by interpolation of the discrete ones fulfill discrete versions of the entropy and total energy inequalities, in addition to the
discrete momentum balance and flow rule. We will then obtain a series of a priori estimates allowing us to deduce suitable compactness information on the approximate solutions, and thus to pass to the limit. 
\par
In this way, under the basic growth condition \eqref{heat-cond-intro} on $\condu$ and under appropriate assumptions on the data, also tailored to the technical problems
 \textbf{(2)}\&\textbf{(3)}, we will prove our first main result,  \textbf{\underline{Theorem \ref{mainth:1}}}, stating the existence of entropic
 solutions to the Cauchy problem for system  (\ref{plast-PDE}, \ref{bc}).
 \par
 Under a more stringent growth condition on $\condu$, 
 we will prove in \textbf{\underline{Theorem \ref{mainth:2}}}
 an existence result for an enhanced notion of solution. Instead of the entropy and total energy inequalities, this concept features 
 \begin{itemize}
\item[-] a `conventional' weak formulation of the heat equation \eqref{heat}, namely
 \begin{equation} \label{eq-teta-intro}
\begin{aligned}
\pairing{}{}{\dot\teta}{\varphi}
+ \int_\Omega \condu(\teta) \nabla \teta\nabla\varphi \dd
x
 = \int_\Omega \left(H+
\mathrm{R}(\teta, \dot p) + \dot p : \dot p  + \mathbb{D} \dot e : e - \teta \bbC \bbE  :  \dot{e} \right) \varphi  \dd x  + \int_{\partial\Omega} h \varphi   \dd S
\end{aligned}
\end{equation}
 for all test functions $\varphi \in W^{1,p}(\Omega)$, with $\teta \in W^{1,1}(0,T; W^{1,p}(\Omega)^*)$, 
 $\condu(\teta) \nabla \teta \in L^{p'}((0,T){\times}\Omega)$,
   and $p>1$ sufficiently big, and 
 \item[-] the \emph{total energy balance}, i.e.\ \eqref{total-enid-intro} as an equality.
 \end{itemize}
 In view of this, we will refer to these improved solutions as `weak energy'. 
\subsection{The perfectly plastic system}
In investigating the  vanishing-viscosity and inertia limit   of system
   (\ref{plast-PDE}, \ref{bc}),
we shall confine the discussion to the asymptotic behavior of a family of \emph{weak energy solutions}. In this setup, we will 
extend the analysis developed in \cite{DMSca14QEPP} to the \emph{temperature-dependent} and  
\emph{spatially heterogeneous} cases, i.e.\ with  the tensors $\bbC,\, \bbD,\, \bbE$, and the elastic domain $K$, depending on $x\in \Omega$. However, 
we will drop the dependence of  $K$ on the (spatially discontinuous) temperature variable $\teta$ due to technical difficulties in the handling of the plastic dissipation potential, see Remark \ref{rmk:added-sol} ahead.
\par
Mimicking \cite{DMSca14QEPP}, 
 we will supplement  the thermoviscoplastic  system  with  rescaled data  $F^\eps\, g^\eps, \, w^\eps, \, H^\eps, \, h^\eps, $ with $\mathfrak{f}^\eps (t) = \mathfrak{f}(\eps t)$,
  for $t \in [0,T/\eps]$ and  for $\mathfrak{f} \in \{ F,\, g, \, w,\, H,\, h\}$.  Correspondingly, we will  consider a family $(\teta^\eps, u^\eps, e^\eps, p^\eps)_\eps$  of weak energy solutions
  to (the Cauchy problem for) system  (\ref{plast-PDE}, \ref{bc}), defined on $ [0,T/\eps]$. 
  We will further rescale them in such a way that they are defined on $[0,T]$, by setting 
$ \teta_\eps (t) = \teta^\eps (t/\eps)$, and defining analogously $u_\eps$, $e_\eps$, $p_\eps$
and the data  $F_\eps, \, g_\eps, \, w_\eps, \, H_\eps, \, h_\eps$. 
 Hence, the functions $(\teta_\eps, u_\eps, e_\eps, p_\eps)$ are 
\emph{weak energy} solutions of the rescaled system
\begin{subequations}
\label{plast-PDE-rescal}
\begin{align}
& 
\label{heat-rescal}
\eps\dot{\teta} - \mathrm{div}(\condu(\teta)\nabla \teta) =H+\eps \mathrm{R}(\teta,\dot{p}) + \eps^2\dot{p}: \dot{p}+ \eps^2\mathbb{D} \dot{e} : 
\dot{e} -\teta \mathbb{C}\mathbb{E}_\eps  : \dot{e}  && \text{ in } \Omega \times (0,T), 
\\
\label{mom-balance-rescal}
 &\rho \eps^2 \ddot{u} - \mathrm{div}\left( \eps\mathbb{D} \dot{e} + \mathbb{C}(e - \mathbb{E}_\eps\teta) \right) = F && \text{ in } \Omega \times (0,T),
\\
&
\label{flow-rule-rescal}
\partial_{\dot{p}}  \mathrm{R}(\teta,\dot{p}) + \eps \dot{p} \ni \left( \eps\mathbb{D} \dot{e} + \mathbb{C}(e - \mathbb{E}_\eps\teta) \right)_{\mathrm{D}} && \text{ in } \Omega \times (0,T),
\end{align}
\end{subequations}
supplemented with the boundary conditions \eqref{bc} featuring the rescaled data  $g_\eps, \, w_\eps, \, h_\eps$.  Observe that we will let 
the thermal expansion tensors vary with $\eps$. 
\par
For technical reasons expounded at length in Section \ref{s:6}, we will address the asymptotic analysis of system (\ref{plast-PDE-rescal}, \ref{bc}) only under the assumption that 
the tensors $\bbE_\eps$ scale in a suitable way with $\eps$, namely\
\begin{equation}
\label{scaling-intro}
 \bbE_\eps = \eps^\beta \bbE \quad \text{ with a given } \bbE \in \mt_\sym^{d\times d} \text{ and } \beta>\frac12. 
\end{equation}
Under \eqref{scaling-intro},  the 
\emph{formal} limit of system (\ref{plast-PDE-rescal}, \ref{bc}) then consists of 
\begin{itemize}
\item[-] the stationary heat equation
\begin{equation}
\label{stat-heat-intro}
- \mathrm{div}(\condu(\teta)\nabla \teta) =H \qquad \text{ in } \Omega \times (0,T),
\end{equation}
 supplemented with the Neumann condition \eqref{bc-teta};
 \item[-] the  system for perfect plasticity
 \begin{subequations}
\label{RIP-PDE}
\begin{align}
\label{mom-balance-RIP}
 & - \mathrm{div}\sigma = F && \text{ in } \Omega \times (0,T), 
\\
&
\label{flow-rule-RIP}
\partial_{\dot{p}}  \mathrm{R}(\Theta,\dot{p})  \ni \sigma_{\mathrm{D}} && \text{ in } \Omega \times (0,T),
\end{align}
with the boundary conditions \eqref{bc-u-1} and \eqref{bc-u-2}, complemented by the kinematic admissibility condition and Hooke's law
\begin{align}
& 
\label{decomp}\sig u = e + p && \text{ in } \Omega \times (0,T), 
\\
& 
\label{stress-RIP}
\sigma =  \mathbb{C}e && \text{ in } \Omega \times (0,T). 
\end{align}
\end{subequations}
\end{itemize}
In fact, system \eqref{RIP-PDE} has to be weakly formulated in function spaces reflecting 
 the fact that the plastic strain $p$  is  only a  Radon measure on $\Omega$, and so is $\sig{u}$ (so that the displacement variable $u$ is only a function of bounded deformation), and that, in principle, we only have $\BV$-regularity for $t\mapsto p(t)$.
\par
Our asymptotic result, \textbf{\underline{Theorem \ref{mainth:3}}}, states that, 
under suitable conditions on the data $(F_\eps\, g_\eps, \, w_\eps, \, H_\eps, \, h_\eps)_\eps$, 
up to a subsequence the functions 
$(\teta_\eps, u_\eps, e_\eps, p_\eps)_\eps$ converge as $\eps \downarrow 0$  to a quadruple $(\limte, u,e,p)$ 
 such that
\begin{enumerate}
\item $\limte$ 
is constant in space,
\item $(u,e,p)$ comply with the \emph{(global) energetic formulation} of system \eqref{RIP-PDE}, consisting of a global stability condition and of an energy balance;
\item there additionally holds a balance between the energy dissipated through changes of the plastic strain
and the thermal energy
on almost every sub-interval of $(0,T)$, i.e.
\begin{equation}
\label{gift}
\int_\Omega \limte(t) \dd x -  \int_\Omega \limte(s) \dd x
=\mathrm{Var}(p;[s,t])+\int_s^t \int_\Omega \mathsf{H} \dd x \dd r +\int_s^t \int_{\partial\Omega} \mathsf{h} \dd S \dd r
 \text{ for almost all } s< t \in (0,T),  
\end{equation}
with $\mathsf{H}$ and $\mathsf{h}$ the limiting heat sources. 
\end{enumerate}
\noindent
 Observe that \eqref{gift} couples the evolution of the temperature $\Theta$ to that of $p$, and thus of the solution triple $(u,e,p)$. 
 \par
 Finally,   based on   the arguments from \cite{DMSca14QEPP}, in Theorem \ref{mainth:3}  we will also obtain that   $(u,e,p)$ are, ultimately, \emph{absolutely continuous} 
as  functions of time. This is a special feature of the perfectly plastic system, already observed in  \cite{DMDSMo06QEPL}. It is in accordance with the time regularity results proved in \cite{MieThe04RIHM} for energetic solutions to rate-independent systems driven by uniformly convex energy functionals. It is in fact  because of the `convex character' of the 
problem that we retrieve \emph{(global) energetic} solutions, upon taking the vanishing-viscosity and inertia limit, cf.\ also 
\cite[Prop.\ 7]{MRS09}. 
\par
Also in view of the similar vanishing-viscosity analysis developed in \cite{LRTT} in the context of a thermodynamically consistent model for damage, we expect to obtain a different kind of solution when performing the same analysis for thermomechanical systems driven by nonconvex (mechanical) energies. 
We plan to address these studies  in the future.
 \paragraph{\bf Plan of the paper.}
 In \underline{Section \ref{s:2}} we establish all the assumptions on   the thermoviscoplastic system (\ref{plast-PDE}, \ref{bc}) and its data, introduce the two solvability concepts we will address, and state our two existence results, Theorems \ref{mainth:1} \& \ref{mainth:2}. \underline{Section \ref{s:3}} is devoted to the analysis of the time discretization scheme for  (\ref{plast-PDE}, \ref{bc}). In  \underline{Section \ref{s:4}} we pass to the time-continuous limit and conclude the proofs of Thms.\ \ref{mainth:1} \& \ref{mainth:2},
 also relying on a novel, Helly-type compactness result, cf.\ Thm.\ \ref{th:mie-theil} ahead. 
  In \underline{Section \ref{s:5}} we set up the limiting perfectly plastic system and give its (global) energetic formulation. The vanishing-viscosity and inertia analysis is carried out in \underline{Section \ref{s:6}} with Theorem  \ref{mainth:3}, whose proof also relies on some Young measure tools recapitulated in the Appendix.
\begin{notation}[General notation]
\label{not:2.1} \upshape
In what follows, $\R^+$ shall stand for $(0,+\infty)$. We will denote by $\mt^{d\times d}$  
 the space of $d{\times} d$ 
  matrices. We consider $\mt^{d\times d}$   endowed with the  Frobenius inner product 
$\eta : \xi : = \sum_{i j} \eta_{ij}\xi_{ij}$ for two matrices $\eta = (\eta_{ij})$ and $\xi = (\xi_{ij})$, which induces the matrix norm $|\cdot|$. 
$\mt_\sym^{d\times d}$ stands for the subspace of symmetric matrices, and $\mt_\dev^{d\times d}$ for the subspace of symmetric matrices with null trace. In fact, 
$\mt_\sym^{d\times d} = \mt_\dev^{d\times d} \oplus \R I$ ($I$ denoting the identity matrix), since every $\eta \in \mt_\sym^{d\times d}$ can be written as 
\[
\eta = \eta_\dev+ \frac{\mathrm{tr}(\eta)}d I
\]
with $\eta_\dev$ the orthogonal projection of $\eta$ into $\mt_\dev^{d\times d} $. We will refer to $\eta_\dev$ as the deviatoric part of $\eta$.  
\par 
With the symbol $\odot$ we will  denote the symmetrized tensor product  of two vectors $a,\, b \in \R^d$, defined as the symmetric matrix with entries $ \frac{a_ib_j + a_j b_i}2$. Note  that the trace $\mathrm{tr}(a \odot b)$ coincides with the scalar product $a \cdot b$. 
\par
Given a
Banach space $X$
we shall
use  the symbol $\pairing{}{X}{\cdot}{\cdot}$ for the duality
pairing between $X^*$ and $X$; if $X$ is a Hilbert space, $(\cdot,\cdot)_X$ will stand for its inner product. To avoid overburdening notation, we shall often write $\| \cdot\|_X$ both for the norm on $X$, and on the product space  $X \times  \ldots \times X$.  With the symbol $\overline{B}_{1,X}(0)$ we will denote the closed unitary ball in $X$. 
We  shall denote by the symbols
 \[
\text{(i)} \  \mathrm{B}([0,T]; X), \, \qquad \text{(ii)} \  \mathrm{C}^0_{\mathrm{weak}}([0,T];X), \, \qquad \text{(iii)} \  \BV ([0,T]; X)
 \]
  the spaces
of functions from $[0,T]$ with values in $ X$ that are defined at
\emph{every}  $t \in [0,T]$ and:  (i) are measurable on $[0,T]$; (ii) are  \emph{weakly} continuous   on  $[0,T]$; (iii)  have  bounded variation on  $[0,T]$.
\par
Finally, we shall use the symbols
$c,\,c',\, C,\,C'$, etc., whose meaning may vary even within the same   line,   to denote various positive constants depending only on
known quantities. Furthermore, the symbols $I_i$,  $i = 0, 1,... $,
will be used as place-holders for several integral terms (or sums of integral terms) popping in
the various estimates: we warn the reader that we will not be
self-consistent with the numbering, so that, for instance, the
symbol $I_1$ will occur several times with different meanings.
\end{notation}
 \paragraph{\bf Acknowledgements.} I am grateful to the  two anonymous referees for reading this paper very carefully and for several constructive suggestions. 
\section{\bf Main results for the thermoviscoplastic system}
\label{s:2}
First, in Section \ref{ss:2.1}, for the thermoviscoplastic system  (\ref{plast-PDE}, \ref{bc})
we establish all the basic  assumptions on the reference configuration $\Omega$, on  the tensors $\bbC, \, \bbD,\, \bbE$, on the  set of admissible stresses $K$ (and, consequently, on the dissipation potential $\mathrm{R}$),  on the external data $g,\, h, \, f,\, \ell$, and $w$,  and on the initial data $(\teta_0, \, u_0, \, \dot{u}_0, \, e_0, \, p_0)$.  In Section \ref{ss:5.1} later on, we will revisit  and strengthen some of  these conditions in order to deal with the limiting perfectly plastic system.    In view of this, to distinguish the two sets of assumptions, we will  label them by indicating the number of the section (i.e., $2$ for the thermoviscoplastic, and $5$ for the perfectly plastic, system).
\par
Second, in Sec.\ \ref{ss:2.2} we introduce   the weak solvabilty concepts for the (Cauchy problem associated with the) viscoplastic system (\ref{plast-PDE}, \ref{bc}),
and state our existence results in Sec.\ \ref{ss:2.3}. 
\subsection{Setup}
\label{ss:2.1}
\paragraph{{\em The reference configuration}.} 
Let $\Omega \subset \R^d$,  $d\in \{2,3\}$,   be a bounded domain, with Lipschitz boundary; we set $Q: =   (0,T) \times \Omega $.
 The boundary $\partial\Omega $ is given by  
\begin{equation}
\label{Omega-s2}
\tag{2.$\Omega$}
\begin{gathered}
\partial \Omega = \Gamma_\Dir \cup
\Gamma_\Neu \cup \partial\Gamma \quad \text{ with $\Gamma_\Dir, \,\Gamma_\Neu, \, \partial\Gamma$ pairwise disjoint,}
\\
\text{
 $\Gamma_\Dir$ and $\Gamma_\Neu$ relatively open in $\partial\Omega$, and $ \partial\Gamma$ their relative boundary in $\partial\Omega$,}
 \\
\text{
  with Hausdorff measure $\calH^{d-1}(\partial\Gamma)=0$.}
  \end{gathered}
  \end{equation}
  We will denote by $|\Omega|$ the Lebesgue measure of $\Omega$. 
  On the Dirichlet part $\Gamma_\Dir$, assumed  with $\calH^{d-1}(\Gamma_\Dir)>0, $   we shall prescribe the displacement, while on $\Gamma_\Neu$ we will impose a Neumann condition.  The trace of a function $v$ on $\Gamma_\Dir$ or $\Gamma_\Neu$ shall be still  denoted by the symbol $v$.  
  The symbol $H_\Dir^1(\Omega;\R^d)$ shall indicate  the subspace of functions
  of $H^1(\Omega;\R^d)$  with null trace on $\Gamma_\Dir$.
    The symbol $W_\Dir^{1,p}(\Omega;\R^d)$, $p>1,$  shall denote the analogous $W^{1,p}$-space.
In what follows, we shall extensively use Korn's inequality (cf.\ \cite{GeySu86}): for every $1<p<\infty$
 there exists a constant $C_K =C_K(\Omega, p)>0$
 such that there holds
\begin{equation}
\label{Korn}
\| u \|_{W^{1,p}(\Omega;\R^d)} \leq C_K \| \sig u \|_{L^p (\Omega;\mt_\sym^{d \times d})} 
\qquad \text{for all } u \in  W_\Dir^{1,p}(\Omega;\R^d)\,.
\end{equation}
     Finally,  we will use the notation \begin{equation}
\label{label-added}
 W_+^{1,p}(\Omega):= \left\{\zeta \in
W^{1,p}(\Omega)\, : \ \zeta(x) \geq 0  \quad \foraa x \in
\Omega \right\}, \quad \text{ and analogously for }
W_-^{1,p}(\Omega).
\end{equation}
\paragraph{{\em Kinematic admissibility and stress}.}  First of all,   let us formalize the decomposition of the linearized strain $
\sig u $ as the sum of the elastic and the plastic strain. Given a function $w \in H^1(\Omega;\R^d)$,  we say that a triple $(u,e,p)$ is \emph{kinematically admissible with boundary datum $w$}, and write $(u,e,p) \in \mathcal{A}(w)$,  if
\begin{subequations}
\label{kin-adm}
\begin{align}
&
u \in H^1(\Omega;\R^d), \quad e \in L^2(\Omega;\mt_\sym^{d\times d}), \quad p \in L^2(\Omega;\mt_\dev^{d\times d}),
\\
& \sig u = e+p \quad \aein\, \Omega,
\\
& 
u = w \quad \text{on } \Gamma_\Dir.
\end{align}
\end{subequations}
\par
 The elasticity, viscosity, and thermal expansion  tensors  are symmetric and fulfill
\begin{equation}
\label{elast-visc-tensors} 
\tag{2.$\mathrm{T}$}
\begin{gathered}
\bbC , \,  \bbD, \, \bbE
  \in L^\infty(\Omega; \mathrm{Lin}(\mt_\sym^{d \times d}))\,,  \text{ and } 
  \\
  \exists\, C_{\bbC}^1,\,  C_{\bbC}^2, \,  C_{\bbD}^1, \,  C_{\bbD}^2>0  \ \ \foraa x \in \Omega \  \ \forall\, A \in \mt_\sym^{d\times d} \, : \quad \begin{cases}
   & C_{\bbC}^1 |A|^2 \leq  \bbC(x) A : A \leq  C_{\bbC}^2 |A|^2,
   \\
    & C_{\bbD}^1 |A|^2 \leq  \bbD(x) A : A \leq  C_{\bbD}^2 |A|^2,
  \end{cases}
  \end{gathered}  \end{equation}
  where $\mathrm{Lin}(\mt_\sym^{d \times d})$ denotes the space of linear operators from $\mt_\sym^{d \times d}$ to $\mt_\sym^{d \times d}$. 
Observe that with \eqref{elast-visc-tensors} we also encompass in our analysis the case of
an anisotropic and inhomogeneous material. 
Throughout the paper, we will use the short-hand notation 
\begin{equation}
\label{tensor-B}
\bbB: = \bbC \bbE
\end{equation}
for the $(d{\times}d)$-matrix arising from the multiplication of $\bbC$ and $\bbE$. 
\begin{remark}
\label{rmk:lorosi}
\upshape
In \cite{DMSca14QEPP}  the viscosity tensor $\bbD$  was assumed (constant in space and) positive semidefinite, only:  In particular, the case $\bbD \equiv 0$ was encompassed in the existence and vanishing-viscosity analysis. We are not able to extend our  own analysis  in this direction, though. In fact, the coercivity condition required on $\bbD$ (joint with $\sig{\dot u} = \dot e + \dot p$, following from  kinematic admissibility),  will play a crucial role in estimating the term $\iint \teta \bbB{:}  \sig{\dot u} \dd x \dd t$, which  arises from
the mechanical energy balance \eqref{mech-enbal} ahead.
\end{remark}
\paragraph{{\em External heat sources}.}
For the volume and boundary  heat sources $H$ and $h$ we require
\begin{align}
 \label{heat-source} 
 \tag{2.$\mathrm{H}_1$}
 &  H \in L^1(0,T;L^1(\Omega)) \cap L^2 (0,T; H^1(\Omega)^*), &&  H\geq 0 \quad\hbox{a.e.  in } Q\,,
 \\
 \label{dato-h}
  \tag{2.$\mathrm{H}_2$}
 & h \in L^1 (0,T; L^2(\partial \Omega)),  &&  h \geq 0 \quad\hbox{a.e.  in } (0,T)\times \partial \Omega\,.
\end{align}
Indeed, the positivity of $H$ and $h$ is necessary for obtaining the strict positivity of the temperature $\teta$. 
\paragraph{{\em Body force and traction}.}
Our basic conditions on the volume force $F$ and the assigned traction $g$ are \begin{equation}
\label{data-displ}
 \tag{2.$\mathrm{L}_1$}
F\in L^2(0,T; H_\Dir^1(\Omega;\R^d)^*), \qquad g \in L^2(0,T; H_{00,\Gamma_\Dir}^{1/2}(\Gamma_\Neu; \R^d)^*),
\end{equation}
recalling that $ H_{00,\Gamma_\Dir}^{1/2}(\Gamma_\Neu; \R^d)$ is the space of functions $\gamma \in H^{1/2} (\Gamma_\Neu;\R^d)$ such that there exists $\tilde\gamma \in H_\Dir^1(\Omega;\R^d)$ with $\tilde\gamma = \gamma $ in $\Gamma_\Neu$.
\par
Furthermore, for technical reasons that will be expounded in Remark \ref{rmk:diffic-1-test} ahead (cf.\ also the text preceding the proof of Proposition \ref{prop:aprio}),  \underline{in order to allow for a non-zero traction $g$}, also for the viscoplastic system we will need to require  a  \emph{uniform safe load} type condition,
which usually occurs in the analysis of perfectly plastic systems, cf.\  Sec.\ \ref{s:5} later on.  Namely, we impose that there exists a function $\varrho: [0,T] \to L^2(\Omega;\mt_\sym^{d\times d})$ solving for almost all $t\in (0,T)$ the following elliptic problem 
\[
\begin{cases}
- \mathrm{div}(\varrho(t)) = F(t)  & \text{in } \Omega,
\\
\varrho(t) \nu = g(t) & \text{on } \Gamma_\Neu 
\end{cases}
\]
such that 
\begin{equation}
\label{safe-load}
 \tag{2.$\mathrm{L}_2$}
\varrho \in W^{1,1}(0,T;  L^2(\Omega;\mt_\sym^{d\times d})) \qquad \text{and} \qquad  
{\varrho}_\dev
 \in L^1(0,T;L^\infty (\Omega; \mt_\dev^{d\times d}))\,.
\end{equation}
\par
Indeed, condition  \eqref{safe-load} will enter into play only starting from the derivation of a priori estimates on the approximate solutions to
the viscoplastic system,
uniform with respect to the time discretization parameter $\tau$. When not explicitly using \eqref{safe-load}, to shorten notation we will incorporate the volume force $F$ and the traction $g$ into the total load induced by  them, namely the function
$\mathcal{L}: (0,T) \to H_\Dir^1(\Omega;\R^d)^*$ given at $t\in (0,T)$ by 
\begin{equation}
\label{total-load}
\pairing{}{H_\Dir^1(\Omega;\R^d)}{\mathcal{L}(t)}{u}: = \pairing{}{H_\Dir^1(\Omega;\R^d)}{F(t)}{u} + \pairing{}{H_{00,\Gamma_\Dir}^{1/2}(\Gamma_\Neu; \R^d)}{g(t)}{u} \qquad \text{for all } u \in H_\Dir^1(\Omega;\R^d),
\end{equation} 
which fulfills $\mathcal{L} \in L^2(0,T; H_\Dir^1(\Omega;\R^d)^*)$ in view of \eqref{data-displ}.
\paragraph{{\em Dirichlet loading}.}
Finally,  we will suppose that the hard device  $w$ to which the body is subject on $\Gamma_\Dir$  is the trace on $\Gamma_\Dir$  of a function, denoted by the same symbol,  fulfilling
\begin{equation}
\label{Dirichlet-loading}
\tag{2.$\mathrm{W}$} 
w \in  L^1(0,T; W^{1,\infty} (\Omega;\R^d)) \cap W^{2,1} (0,T;H^1(\Omega;\R^d)) \cap H^2(0,T; L^2(\Omega;\R^d))\,.
\end{equation}
We postpone to Remark  \ref{rmk:diffic-1-test}  some explanations on the use of, and need for,  conditions \eqref{Dirichlet-loading}. Let us only mention here that the requirement  $w\in L^1(0,T; W^{1,\infty} (\Omega;\R^d))$  could be replaced by asking for $\bbB{:} \sig w=0$ a.e.\  in $Q$, as imposed, e.g.,  in \cite{Roub-PP}. 
\paragraph{{\em The weak formulation of the momentum balance}.} The variational formulation of  \eqref{mom-balance}, supplemented with the boundary conditions \eqref{bc-u-1} and \eqref{bc-u-2}, reads
\begin{equation}
\label{w-momentum-balance}
\begin{aligned}
\rho\int_\Omega \ddot{u}(t) v \dd x + \int_\Omega \left(\bbD \dot{e}(t) + \bbC e(t) - \teta(t) \bbB \right): \sig v \dd x   &  = \pairing{}{H_\Dir^1(\Omega;\R^d)}{\calL(t)}{v} 
\\ & \qquad  \text{for all } v \in H_\Dir^1(\Omega;\R^d), \ \foraa t \in (0,T)\,.
\end{aligned}
\end{equation}
We will often use the short-hand notation $-\mathrm{div}_{\Dir}$ for the elliptic operator  defined by 
\begin{equation}
\label{div-Gdir}
\pairing{}{ H_\Dir^1(\Omega;\R^d) }{-\mathrm{div}_{\Dir}(\sigma)}{v}: = \int_\Omega \sigma : \sig v \dd x \qquad \text{for all } v \in  H_\Dir^1(\Omega;\R^d) \,.
\end{equation}
\paragraph{{\em The plastic dissipation}.}
Prior to stating our precise assumptions on the multifunction $K: \Omega \times \R^+ \rightrightarrows \mt_\dev^{d \times d}$,  following \cite{Castaing-Valadier77} let us recall  the notions of measurability, lower semicontinuity, and upper semicontinuity, for a general multifunction $ \mathsf{F} : X \rightrightarrows Y$. Although the definitions and results given in  \cite{Castaing-Valadier77} cover much more general situations, for simplicity here we shall confine the discussion to the case of  a topological measurable space $(X, \mathscr{M})$, and a (separable) Hilbert space   $Y$.  For a set $B \subset Y$, we define
\[
\mathsf{F}^{-1}(B): = \{  x\in X\, :  \  \mathsf{F}(x) \cap B \neq \emptyset\}. 
\]
 We say that 
\begin{subequations}
\label{multifuncts-props}
\begin{align}
& 
\label{meas}
\text{$\mathsf{F}$ is measurable if for every open subset $U \subset Y$,  $  \mathsf{F}^{-1}(U) \in \mathscr{M}$;}
\\
& 
\label{lsc}
\text{$\mathsf{F}$ is lower semicontinuous if for every open set $U\subset Y$, the set $  \mathsf{F}^{-1}(U)$ is open; }\\
& 
\label{usc}
\text{$\mathsf{F}$ is upper semicontinuous if for every open set $U\subset Y$, the set $\{ x \in X\, : \ \mathsf{F}(x) \subset U \}$ is open.}
\end{align}
\end{subequations}
Finally, $\mathsf{F}$ is continuous if it is both lower and upper semicontinuous.
\par
Let us now turn back to the multifunction $K : \Omega \times  \R^+ \rightrightarrows \mt_\dev^{d \times d}$. We suppose that
\begin{equation}
\label{measutab-cont-K}
\tag{2.$\mathrm{K}_1$}
\begin{aligned}
&
K : \Omega \times  \R^+ \rightrightarrows \mt_\dev^{d \times d}  &&  \text{ is measurable w.r.t.\ the variables $(x,\teta)$,}
\\
&
K(x, \cdot) : \R^+ \rightrightarrows \mt_\dev^{d \times d}   && \text{ is continuous} \quad \text{for almost all } x \in \Omega.
\end{aligned}
\end{equation}
Furthermore, we require that 
\begin{equation}
\label{elastic-domain}
\tag{2.$\mathrm{K}_2$}
\begin{aligned}
K(x,\teta) \text{ is a convex and compact set in }  \mt_\dev^{d \times d} \qquad \text{for all } \teta \in \R^+, \text{ for almost all } x \in \Omega,
\\
\exists\, 0<c_r<C_R \quad  \foraa x \in \Omega, \ \forall\, \teta \in  \R^+ \, : \quad B_{c_r}(0) \subset  K(x,\teta) \subset B_{C_R}(0).
\end{aligned}
\end{equation}
\par
Therefore, the support function associated with the multifunction $K$, i.e.
\begin{equation}
\label{1-homogeneous-dissip}
\mathrm{R}: \Omega  \times \R^+ \times  \mt_\dev^{d \times d} \to [0,+\infty) \quad \text{defined by } \mathrm{R}(x,\teta, \dot p ): = \sup_{\pi \in K(x,\teta)} \pi : \dot p 
\end{equation}
is positive, with $\mathrm{R}(x,\teta, \cdot) :  \mt_\dev^{d \times d} \to   [0,+\infty)  $ convex and $1$-positively homogeneous for almost all $x \in \Omega$ and for all $\teta \in \R^+$. 
By the first of \eqref{measutab-cont-K},  the function $\mathrm{R}: \Omega  \times \R^+ \times  \mt_\dev^{d \times d} \to [0,+\infty) $ is measurable. Moreover,  by the second of   \eqref{measutab-cont-K},  in view of \cite[Thms.\ II.20, II.21]{Castaing-Valadier77} (cf.\ also  \cite[Prop.\ 2.4]{Sol09}) the function
\begin{subequations}
\label{hypR}
\begin{align}
& 
\label{hypR-lsc}
\mathrm{R}(x,\cdot, \cdot):  \R^+ \times  \mt_\dev^{d \times d} \to [0,+\infty)   \text{ is (jointly) lower semicontinuous}, 
\intertext{for almost all $x \in \Omega$, i.e.\ $\mathrm{R}$ is a \emph{normal integrand},  and} 
& \label{hypR-cont}
\mathrm{R}(x,\cdot, \dot p):   \R^+ \to \R^+ \text{ is continuous for every $\dot p\in \mt_\dev^{d\times d}$}. 
\end{align}
\end{subequations}
Finally, 
it follows from the second of \eqref{elastic-domain}
that  
\begin{subequations}
\label{cons-lin-growth}
\begin{equation}
\label{linear-growth}
c_r|\dot p| \leq \mathrm{R}(x,\teta, \dot p ) \leq C_R |\dot p| \qquad \text{for all } (\teta, \dot p) \in  \R^+ \times  \mt_\dev^{d \times d}  \text{ for almost all }x \in \Omega\,,
\end{equation}  and that
\begin{equation}
\label{bounded-subdiff} \partial_{\dot p} \mathrm{R}(x,\teta, \dot{p}) \subset  \partial_{\dot p} \mathrm{R}(x,\teta, 0) = K(x,\teta) \subset B_{C_R}(0) \qquad \text{for all } (\teta, \dot p)  \in  \R^+ \times  \mt_\dev^{d \times d} \quad \text{for almost all }x \in \Omega.
\end{equation}
\end{subequations}
\par 
Finally, we also introduce the \emph{plastic dissipation  potential} $\calR:L^1(\Omega; \R^+) \times L^1(\Omega;\mt_\dev^{d\times d})$ given by
\begin{equation}
\label{plastic-dissipation-functional}
\calR(\teta, \dot p): = \int_\Omega \mathrm{R}(x,\teta(x), \dot p(x)) \dd x\,.
\end{equation}
\paragraph{{\em The plastic flow rule}.} 
Taking into account the $1$-positive homogeneity of $\mathrm{R}(x,\teta, \cdot)$, which yields the following characterization of $\partial_{\dot p} \mathrm{R}(x,\teta, \dot{p}):  \mt_\dev^{d\times d} \rightrightarrows \mt_\dev^{d\times d}$:
\begin{equation}
\label{characterization-subdiff}
\zeta \in \partial_{\dot p} \mathrm{R}(x,\teta, \dot{p})  \ \ \Leftrightarrow \ \ 
\begin{cases}
& \zeta : \eta \leq  \mathrm{R}(x,\teta, \eta)  \quad \text{for all } \eta \in \mt_\dev^{d\times d}
\\
& \zeta : \dot{p}  =  \mathrm{R}(x,\teta, \dot p),
\end{cases}
 \  \ \Leftrightarrow \  \begin{cases}
 &   \zeta \in  \partial_{\dot p} \mathrm{R}(x,\teta, 0) = K(x,\teta),
  \\
  & \zeta : \dot{p}  \geq   \mathrm{R}(x,\teta, \dot p),
\end{cases}
\end{equation}
 the plastic flow rule 
\begin{equation}
\label{pl-flow}
\partial_{\dot p} \mathrm{R}(x,\teta(t,x), \dot{p}(t,x)) + \dot{p}(t,x) \ni \sigma_\dev(t,x) \qquad \foraa (t,x) \in Q,
\end{equation}
reformulates as
\begin{equation}
\label{reform-pl-flow}
\begin{cases}
 \left( \sigma_\dev(t,x)  -  \dot{p}(t,x) \right) : \eta \leq \mathrm{R}(x,\teta(t,x), \eta)  \quad \text{for all } \eta \in \mt_\dev^{d\times d}  
 \\
  \left( \sigma_\dev(t,x)  -  \dot{p}(t,x) \right) : \dot{p}(t,x)  \geq  \mathrm{R}(x,\teta(t,x), \dot{p}(t,x)) 
\end{cases} \qquad   \foraa (t,x) \in Q\,.
\end{equation}
\paragraph{{\em Cauchy data}.} We will supplement the thermoviscoplastic system with 
initial data
\begin{subequations}
\label{Cauchy-data}
\begin{align}
\label{initial-teta}
&
\teta_0 \in L^1(\Omega),    \text{ fulfilling the strict positivity condition }  \exists\, \teta_*>0: \ \inf_{x\in \Omega} \teta_0(x) \geq \teta_*,
\\
&
\label{initial-u}
u_0 \in H_\Dir^{1} (\Omega;\R^d), \ \dot{u}_0 \in L^2 (\Omega;\R^d),
\\
& 
\label{initial-p}
e_0 \in   L^2(\Omega;\mt_\sym^{d\times d}), \quad  p_0 \in L^2(\Omega;\mt_\dev^{d\times d}) \quad \text{such that } (u_0, e_0, p_0 ) \in \mathcal{A}(w(0))\,.
\end{align}
\end{subequations}
\subsection{Weak solvability concepts for the  thermoviscoplastic system}
\label{ss:2.2}
Throughout this section,  we shall suppose that the functions $\bbC, \ldots, \mathrm{R}$, the data $H,\ldots, w$, and the initial data $(\teta_0, \, u_0, \dot{u}_0, e_0, p_0)$ fulfill the conditions stated in Section \ref{ss:2.1}. We now motivate  the weak solvabilty concepts for the (Cauchy problem associated with the) viscoplastic system
 (\ref{plast-PDE}, \ref{bc}) 
with some heuristic calculations. 
\paragraph{{\em Heuristics for entropic and weak solutions to  system
 (\ref{plast-PDE}, \ref{bc})}.}
As already mentioned in the Introduction, we shall  formulate the heat equation \eqref{heat} 
by means of an entropy inequality and a total energy inequality, featuring  the  stored energy of the system.  The latter is given by  the sum of the internal and of the elastic energies, i.e.
\begin{equation}
\label{stored-energy}
\calE(\teta, u, e, p)= \calE(\teta, e):=  \calF(\teta) + \calQ(e) \quad \text{with } \quad \begin{cases} \calF(\teta) : = \int_\Omega  \teta \dd x, 
\\
\calQ(e): = 
 \frac12 \int_\Omega \bbC e: e \dd x\,. 
 \end{cases}
\end{equation}
\par
Let us formally derive (in particular, without specifying the needed regularity on the solution quadruple $(\teta,u,e,p)$)  the total energy inequality (indeed, we will formally obtain a total  energy \emph{balance}),  starting from the energy estimate associated with system (\ref{plast-PDE}, \ref{bc}).  The latter  consists in testing the momentum balance by $\dot u - \dot w$, the heat equation by $1$, and the plastic flow rule by $\dot p$,  adding the resulting relations and  integrating in space and over a generic interval $(s,t) \subset (0,T)$.  More in detail, the test of \eqref{mom-balance} and of \eqref{flow-rule} yields, after some elementary calculations, 
\begin{equation}
\label{intermediate-mech-enbal}
\begin{aligned}
& 
\frac{\rho}2 \int_\Omega |\dot{u}(t)|^2 \dd x + \int_s^t\int_\Omega \left(  \bbD \dot{e}  + \bbC e - \teta \bbB \right)  : \sig{\dot{u}} \dd x \dd r 
+ \int_s^t \int_\Omega \left(    |\dot p|^2 {+} \mathrm{R}(\teta, \dot p)  \right) \dd x \dd r
\\
&  = \frac{\rho}2 \int_\Omega |\dot{u}(s)|^2 \dd x  + \int_s^t \pairing{}{H_\Dir^1 (\Omega;\R^d)}{\mathcal{L}}{\dot u- \dot w} \dd r  +
\int_s^t \int_\Omega  \left(  \bbD \dot{e}  + \bbC e - \teta \bbB \right)  : \sig{\dot w} \dd x \dd r 
\\
&\quad  +\rho \left( \int_\Omega \dot{u}(t) \dot{w}(t) \dd x -  \int_\Omega \dot{u}(s) \dot{w}(s) \dd x   - \int_s^t \int_\Omega \dot{u}\ddot w \dd x \dd r \right) +
 \int_s^t \int_\Omega \sigma_\dev : \dot{p} \dd x \dd r\,.
\end{aligned}
\end{equation}
Now, taking into account that $ \sig{\dot u}  =\dot e + \dot p$ by the kinematical admissibility condition,  rearranging some terms one has that 
\[
\begin{aligned}
 \int_s^t\int_\Omega \left(  \bbD \dot{e}  + \bbC e  -  \teta \bbB \right)  : \sig{\dot{u}} \dd x \dd r  = &  \int_s^t \int_\Omega \left( \bbD \dot e : \dot e  + \bbC \dot e : e  \right) \dd x \dd r 
   -  \int_s^t \int_\Omega \teta \bbB : \dot e \dd x \dd  r 
 \\
 & \quad + \int_s^t \int_\Omega \left(  \bbD \dot e+ \bbC e  - 
\teta \bbB \right)   : \dot p   \dd x \dd r \,. 
\end{aligned}
\]
Substituting this in \eqref{intermediate-mech-enbal} and  noting that $  \int_s^t \int_\Omega \left(\bbD \dot e  + \bbC e - \teta \bbB \right)  : \dot{p} \dd x \dd r = \int_s^t \int_\Omega \sigma_\dev : \dot{p} \dd x \dd r   $, so that the last term on the right-hand side of \eqref{intermediate-mech-enbal} cancels out,    we get  
 the \emph{mechanical energy balance}, featuring the kinetic and dissipated energies
\begin{equation}
\label{mech-enbal}
\begin{aligned}
& 
\dddn{\frac{\rho}2 \int_\Omega |\dot{u}(t)|^2 \dd x}{kinetic} +\dddn{ \int_s^t\int_\Omega   \left( \bbD \dot e: \dot e   + |\dot p|^2 \right)  \dd x \dd r + \int_s^t \calR(\teta, \dot p) \dd r }{dissipated}
+ \calQ(e(t))
\\
&  =  \frac{\rho}2 \int_\Omega |\dot{u}(s)|^2 \dd x   + \calQ(e(s)) + \int_s^t \pairing{}{H_\Dir^1 (\Omega;\R^d)}{\mathcal{L}}{\dot u{-} \dot w} \dd r  +   \int_s^t \int_\Omega \teta \bbB : \dot e \dd x \dd r 
 \\
& \quad   +\rho \left( \int_\Omega \dot{u}(t) \dot{w}(t) \dd x -  \int_\Omega \dot{u}(s) \dot{w}(s) \dd x - \int_s^t \int_\Omega \dot{u}\ddot w \dd x \dd r \right)     + 
\int_s^t \int_\Omega  \left(  \bbD \dot{e}  + \bbC e - \teta \bbB \right)  : \sig{\dot w} \dd x \dd r, 
 \end{aligned}
\end{equation}
which will also have a significant role for our analysis. 
\par
Summing this with the heat equation tested by $1$ and integrated in time and space gives, after cancelation of some terms,  the \emph{total energy balance} 
\begin{equation}
\label{total-enbal}
\begin{aligned}
& 
\frac{\rho}2 \int_\Omega |\dot{u}(t)|^2 \dd x +\mathcal{E}(\teta(t), e(t)) 
\\
&  = \frac{\rho}2 \int_\Omega |\dot{u}(s)|^2 \dd x +\mathcal{E}(\teta(s), e(s))  + \int_s^t \pairing{}{H_\Dir^1 (\Omega;\R^d)}{\mathcal{L}}{\dot u{-} \dot w}   +\int_s^t \int_\Omega H \dd x \dd r + \int_s^t \int_{\partial\Omega} h \dd S \dd r
\\
& \quad   +\rho \left( \int_\Omega \dot{u}(t) \dot{w}(t) \dd x -  \int_\Omega \dot{u}(s) \dot{w}(s) \dd x   - \int_s^t \int_\Omega \dot{u}\ddot w \dd x \dd r \right)     + 
\int_s^t \int_\Omega \sigma: \sig{\dot w} \dd x \dd r
 \,.\end{aligned}
\end{equation}

As for the \emph{entropy inequality}, let us only mention that it can be formally obtained by multiplying the heat equation \eqref{heat} by $\varphi/\teta$, with $\varphi $ a smooth and \emph{positive} test function. Integrating in  space and over  a generic interval $(s,t) \subset (0,T)$  leads to the identity
\begin{equation}
\label{formal-entropy-eq}
\begin{aligned}
  & \int_s^t \int_\Omega \partial_t \log(\teta) \varphi \dd x \dd r +   \int_s^t \int_\Omega \left( \condu(\teta) \nabla \log(\teta) \nabla \varphi - \condu(\teta) \frac\varphi\teta \nabla \log(\teta) \nabla   \teta  \right) \dd x \dd r  
  \\ & = \int_s^t \int_\Omega \left( H+ \mathrm{R}(\teta,\dot{p}) + |\dot{p}|^2+ \mathbb{D} \dot{e} : 
\dot{e} -\teta \bbB  : \dot{e} \right) \frac{\varphi}\teta \dd x \dd r + \int_s^t \int_{\partial\Omega} h \frac\varphi\teta \dd x \dd r\,.
\end{aligned}
\end{equation}
The entropic solution concept  given in Definition \ref{def:entropic-sols} below will feature the inequality version of \eqref{formal-entropy-eq}, where the first term on the left-hand side is integrated by parts in time, as well as the inequality version of \eqref{total-enbal}. 
\begin{definition}[Entropic solutions to the thermoviscoplastic  system]
\label{def:entropic-sols}
 Given initial data $(\teta_0,u_0, \dot{u}_0, e_0, p_0)$ fulfilling \eqref{Cauchy-data}, we call a quadruple $(\teta,u,e,p)$
an \emph{entropic solution} to the Cauchy problem for system (\ref{plast-PDE}, \ref{bc}),  if
\begin{subequations}
\label{regularity}
\begin{align}
\label{reg-teta}  & \teta \in  L^2(0,T; H^1(\Omega))\cap L^\infty(0,T;L^1(\Omega)),
\\
& \label{reg-log-teta}
\log(\teta) \in  L^2(0,T; H^1(\Omega)),
\\
& \label{reg-u} u \in H^1(0,T; H_\Dir^{1}(\Omega;\R^d)) \cap W^{1,\infty}(0,T; L^2(\Omega;\R^d)) \cap H^2(0,T; H_\Dir^{1}(\Omega;\R^d)^*),
\\
 & \label{reg-e} e \in H^1(0,T; L^2(\Omega;\mt_\sym^{d\times d})),
\\
& \label{reg-p} p \in H^1(0,T; L^2(\Omega;\mt_\dev^{d\times d})),
\end{align}
\end{subequations}
$(u,e,p)$ comply with
 the initial conditions
 \begin{subequations}
 \label{initial-conditions}
\begin{align}
 \label{iniu}  & u(0,x) = u_0(x), \ \ \dot{u}(0,x) = \dot{u}_0(x) & \forae\, x \in
 \Omega,
 \\
  \label{inie}  & e(0,x) = e_0(x)  & \forae\, x \in
 \Omega,
 \\
 \label{inichi}  & p(0,x) = p_0(x) & \forae\, x \in
 \Omega,
\end{align}
\end{subequations}
 (while the initial condition for $\teta$ is implicitly formulated in \eqref{entropy-ineq} and \eqref{total-enineq} below),
and with  
\begin{itemize}
\item[-] the 
\emph{strict positivity} of $\teta$: 
\begin{equation}
\label{teta-strict-pos}
\exists\, \bar\teta>0  \  \foraa (t,x) \in Q\, : \quad \teta(t,x) > \bar\teta;
\end{equation}
\item[-]
the \emph{entropy  inequality}, to hold for almost all $t \in (0,T]$ and almost all
$s\in (0,t)$, and for $s=0$ (where $\log(\teta(0))$ is to be understood as $\log(\teta_0)$), 
\begin{equation}
\label{entropy-ineq}
\begin{aligned}
  & \int_s^t \int_\Omega  \log(\teta) \dot{\varphi} \dd x \dd r -   \int_s^t \int_\Omega \left( \condu(\teta) \nabla \log(\teta) \nabla \varphi - \condu(\teta) \frac\varphi\teta \nabla \log(\teta) \nabla \teta\right)   \dd x \dd r  
  \\ 
& \leq
  \int_\Omega \log(\teta(t)) \varphi(t) \dd x -  \int_\Omega \log(\teta(s)) \varphi(s) \dd x  \\ & \quad
    - \int_s^t \int_\Omega \left( H+ \mathrm{R}(\teta,\dot{p}) + |\dot{p}|^2+ \mathbb{D} \dot{e} : 
\dot{e} -\teta \bbB : \dot{e} \right) \frac{\varphi}\teta \dd x \dd r  - \int_s^t \int_{\partial\Omega} h \frac\varphi\teta \dd x \dd r
\end{aligned}
\end{equation}
for all $\varphi $ in $L^\infty ([0,T]; W^{1,\infty}(\Omega)) \cap H^1 (0,T; L^{6/5}(\Omega))$,  with $\varphi \geq 0$;
\item[-] the \emph{total energy inequality}, to hold  for almost all $t \in (0,T]$ and almost all
$s\in (0,t)$, and for $s=0$:
\begin{equation}
\label{total-enineq}
\begin{aligned}
& 
\frac{\rho}2 \int_\Omega |\dot{u}(t)|^2 \dd x +\mathcal{E}(\teta(t), e(t)) 
\\
&  \leq \frac{\rho}2 \int_\Omega |\dot{u}(s)|^2 \dd x +\mathcal{E}(\teta(s), e(s))  + \int_s^t \pairing{}{H_\Dir^1 (\Omega;\R^d)}{\mathcal{L}}{\dot u{-} \dot w}   +\int_s^t \int_\Omega H \dd x \dd r + \int_s^t \int_{\partial\Omega} h \dd S \dd r
\\
&
\begin{aligned}  \quad   +\rho \left( \int_\Omega \dot{u}(t) \dot{w}(t) \dd x -  \int_\Omega \dot{u}(s) \dot{w}(s) \dd x  - \int_s^t \int_\Omega \dot{u}\ddot w \dd x \dd r \right)   &    + 
\int_s^t \int_\Omega \sigma: \sig{\dot w} \dd x \dd r, 
\end{aligned}
 \end{aligned}
\end{equation}
where
 for $s=0$ we read $\teta(0)=\teta_0$, with the stress $\sigma$ given by the constitutive equation
 \begin{equation}
 \label{stress-consti}
 \sigma= \bbD\dot{e} + \bbC e - \teta \bbB \qquad \aein Q;
 \end{equation}
\item[-] the \emph{kinematic admissibility} condition
\begin{equation}
\label{kin-admis}
(u(t,x), e(t,x), p(t,x)) \in \mathcal{A}(w(t,x)) \qquad \foraa (t,x) \in Q;
\end{equation}
\item[-] the weak formulation \eqref{w-momentum-balance}
of the \emph{momentum balance}; 
\item[-] the \emph{plastic flow rule} \eqref{pl-flow}. 
\end{itemize}
\end{definition}
\begin{remark}
\label{rmk:feasibility}
\upshape
Observe that with the entropy inequality \eqref{entropy-ineq} we are tacitly claiming that, in addition to \eqref{reg-teta} and \eqref{reg-log-teta}, the temperature variable has the following summability properties
\[
\condu(\teta) | \nabla \log(\teta)|^2 \varphi \in L^1(Q), \qquad  \condu(\teta) \nabla \log(\teta) \in L^1(Q)
\]
for every positive admissible  test function  $\varphi$. In fact, we shall retrieve the above properties (and  improve the second one, cf.\ \eqref{further-logteta} ahead),  within the proof of Theorem \ref{mainth:1}. Furthermore,  note that the integral $ \int_\Omega \log(\teta(t)) \varphi(t) \dd x$ makes sense for almost all $t\in(0,T)$, since the estimate
\[
|\log(\teta(t,x))|\leq \teta(t,x) + \frac1{\teta(t,x)} \leq \teta(t,x) + \frac1{\bar\teta} \qquad \foraa (t,x) \in  Q,
\]
(with the second inequality due to \eqref{teta-strict-pos}), and the fact that $\teta \in L^\infty (0,T; L^1(\Omega))$,  guarantee that $\log(\teta) \in  L^\infty (0,T; L^1(\Omega))$ itself. Finally, the requirement that $\varphi \in H^1 (0,T; L^{6/5}(\Omega))$ ensures that $\int_s^t \int_\Omega  \log(\teta) \dot{\varphi} \dd x \dd r $ is  a well-defined integral, since $\log(\teta) \in L^2(0,T;L^6(\Omega))$ by \eqref{reg-log-teta}.
\par
We  refer to \cite[Rmk\ 2.6]{Rocca-Rossi} for a thorough discussion on the consistency between the entropic and the standard, weak formulation of the heat equation \eqref{heat}.  Still, we may mention here that, to obtain the latter  from the former  formulation, one should test 
the entropy inequality  by $\varphi= \teta$. Therefore, $\teta$ should have enough regularity as to make it an admissible test function for 
  \eqref{entropy-ineq}. 
\end{remark} 
\par
In our second solvability concept for the initial-boundary value problem associated with system \eqref{plast-PDE}, the temperature has   the enhanced  time regularity
\eqref{enh-teta-W11} below, which allows us to give an improved  variational formulation of the  heat equation \eqref{heat}.  Observe that, in \cite{Rocca-Rossi} this solution notion was referred to as \emph{weak}. In this paper we will instead prefer the term \emph{weak energy solution}, in order to highlight the validity of the total energy \emph{balance} on \emph{every} interval $[s,t]\subset [0,T]$, cf.\ Corollary \ref{cor:total-enid} below. 
\begin{definition}[Weak energy solutions to the thermoviscoplastic system]
\label{def:weak-sols}
 Given initial data $(\teta_0,u_0, \dot{u}_0, e_0, p_0)$ fulfilling \eqref{Cauchy-data}, we call a quadruple $(\teta,u,e,p)$
a \emph{weak energy solution} to the Cauchy problem for system
 (\ref{plast-PDE}, \ref{bc}),  if
\begin{itemize}
\item[-]
 in addition to  the regularity and summability properties \eqref{regularity}, there holds 
 \begin{equation}
 \label{enh-teta-W11}
 \teta \in W^{1,1}(0,T; W^{1,\infty}(\Omega)^*),
 \end{equation} 
\item[-]
in addition to the initial conditions  \eqref{initial-conditions}, $\teta $ complies with 
\begin{equation}
\label{initeta}
\teta(0) = \teta_0 \qquad \text{ in } W^{1,\infty}(\Omega)^*. \end{equation}
\item[-]
in addition to the strict positivity \eqref{teta-strict-pos}, the kinematic admissibility \eqref{kin-admis}, the weak mometum balance \eqref{w-momentum-balance}, and the flow rule \eqref{pl-flow}, $(\teta,u,e,p)$ comply for almost all $t \in (0,T)$  with  the following weak formulation of the heat equation
\begin{equation} \label{eq-teta}
\begin{aligned}
   &
\pairing{}{W^{1,\infty}(\Omega)}{\dot\teta}{\varphi}
+ \int_\Omega \condu(\teta) \nabla \teta\nabla\varphi \dd
x
\\
& = \int_\Omega \left(H+
\mathrm{R}(\teta, \dot p) + |\dot p |^2  + \mathbb{D} \dot e : \dot{e} - \teta \bbB :  \dot{e} \right) \varphi  \dd x  + \int_{\partial\Omega} h \varphi   \dd S
\quad \text{for all }
\varphi \in   W^{1,\infty}(\Omega). 
\end{aligned}
\end{equation}
\end{itemize}
\end{definition}
Along the lines of Remark \ref{rmk:feasibility}, we may observe that, underlying  the weak formulation \eqref{eq-teta} is  the property $\condu(\teta) \nabla \teta \in L^1(Q;\R^d)$, which shall be in fact (slightly) improved in Theorem \ref{mainth:2}. 
\par
We conclude the section with the following result, under the (tacitly assumed) conditions from  Sec.\ \ref{ss:2.1}. 
\begin{lemma}
\label{cor:total-enid}
\begin{enumerate}
\item
Let $(\teta, u,e,p)$ be either an \emph{entropic} or a   \emph{weak energy solution} to (the Cauchy problem for) system (\ref{plast-PDE}, \ref{bc}). Then,  the functions $(\teta, u,e,p)$ comply with 
the mechanical energy balance \eqref{mech-enbal} for every $0\leq s \leq t \leq T$.
\item 
Let $(\teta, u,e,p)$ be a   \emph{weak energy solution} to (the Cauchy problem for) system (\ref{plast-PDE}, \ref{bc}). Then, the 
total energy \emph{balance}
\begin{equation}
\label{total-enbal-delicate}
\begin{aligned}
& 
\frac{\rho}2 \int_\Omega |\dot{u}(t)|^2 \dd x + \pairing{}{W^{1,\infty}(\Omega)}{\teta(t)}{1} + \calQ(e(t)) \\
&  = \frac{\rho}2 \int_\Omega |\dot{u}(s)|^2 \dd x + \pairing{}{W^{1,\infty}(\Omega)}{\teta(s)}{1} + \calQ(e(s))  + \int_s^t \pairing{}{H_\Dir^1 (\Omega;\R^d)}{\mathcal{L}}{\dot u{-} \dot w}   +\int_s^t \int_\Omega H \dd x \dd r + \int_s^t \int_{\partial\Omega} h \dd S \dd r
\\
& \quad   +\rho \left( \int_\Omega \dot{u}(t) \dot{w}(t) \dd x - \int_\Omega \dot{u}(s) \dot{w}(s) \dd x   - \int_s^t \int_\Omega \dot{u}\ddot w \dd x \dd r \right)     + 
\int_s^t \int_\Omega \sigma: \sig{\dot w} \dd x \dd r
 \end{aligned}
\end{equation}
holds for all $0 \leq s \leq t \leq T$. 
\end{enumerate}
\end{lemma}
\noindent Observe that, since $\teta\in L^\infty(0,T; L^1(\Omega))$, there holds $ \pairing{}{W^{1,\infty}(\Omega)}{\teta(t)}{1}  = \int_\Omega \teta(t) \dd x  = \calF(\teta(t))$ for almost all $t\in (0,T)$ and for $t=0$.  For such $t$, \eqref{total-enbal-delicate}
 may be thus rewritten in terms of the stored energy $\calE$ from \eqref{stored-energy}.
\begin{proof}
The energy balance \eqref{mech-enbal} follows from testing the momentum balance \eqref{w-momentum-balance} by $\dot u-\dot w$, the plastic flow rule by $\dot p$, adding the resulting relations, and integrating in time.
\par
As for \eqref{total-enbal-delicate},
it is sufficient to test the weak formulation \eqref{eq-teta} of the heat equation by $\varphi =1$,  integrate in time taking into account that $\teta \in W^{1,1}(0,T; W^{1,\infty}(\Omega)^*) $, and add  the resulting identity to  \eqref{mech-enbal}. 
\end{proof}
\subsection{Existence results for the thermoviscoplastic system}
\label{ss:2.3}
Our first result states the existence of entropic solutions, under a mild growth condition on the thermal conductivity $\condu$. 
For shorter notation,  in the statement below we shall write   $(2.\mathrm{H})$ in place of \eqref{heat-source}, \eqref{dato-h}, and analogously $(2.\mathrm{L})$, $(2.\mathrm{K})$. 
\begin{maintheorem}
\label{mainth:1}
Assume
\eqref{Omega-s2},  \eqref{elast-visc-tensors},  $(2.\mathrm{H})$,  $(2.\mathrm{L})$,  
 \eqref{Dirichlet-loading},   and $(2.\mathrm{K})$.
In addition,  suppose that 
\begin{equation}
\label{hyp-K}
\tag{2.$\condu_1$}
\begin{aligned}
& \text{the function }   \condu: \R^+ \to \R^+  \  \text{
is
 continuous and}
\\
& \exists \, c_0, \, c_1>0 \quad   \mu>1   \ \
\forall\teta\in \R^+\, :\quad
c_0 (1+ \teta^{\mu}) \leq \condu(\teta) \leq c_1 (1+\teta^{\mu})\,.
\end{aligned}
\end{equation}
Then, for every $(\teta_0, u_0, \dot{u}_0, e_0, p_0) $ satisfying \eqref{Cauchy-data} there exists an entropic solution $(\teta,u,e,p)$ such that, in addition,
 $\teta$ complies with the positivity property 
\begin{equation}
\label{strong-strict-pos}
 \teta(t,x) \geq \bar\teta : = \left( \bar{c} T + \frac1{\teta_*}\right)^{-1} \quad \text{for almost all $(t,x) \in Q$},
\end{equation}
where 
$\teta_*>0$ is  from \eqref{initial-teta} and $\bar{c} := \frac{|\bbB|^2}{2 C_\bbD^1}$, with $C_\bbD^1>0$ from \eqref{elast-visc-tensors}. Finally, 
there holds
\begin{equation}
\label{further-logteta}
\begin{gathered}
\log(\teta) \in L^\infty (0,T;L^p(\Omega))  \quad \text{for all } 1 \leq p <\infty,
\\
\condu(\teta)\nabla \log(\teta) \in  L^{1+\bar\delta}(Q;\R^d)   \text{ with   $ \bar\delta = \frac{\alpha}\mu $ 
and $\alpha \in [(2-\mu)^+, 1)$, and } \qquad  
\\ \condu(\teta)\nabla \log(\teta) \in L^1(0,T;X) \quad \text{with } X=
\left \{ \begin{array}{lll}
 L^{2-\eta}(\Omega;\R^d)  &   \text{ for all } \eta \in (0,1]  & \text{if } d=2,
\\
L^{3/2-\eta}(\Omega;\R^d)  &   \text{ for all } \eta \in (0,1/2]  & \text{if } d=3,
\end{array}
\right.
\end{gathered}
\end{equation}
with $(2-\mu)^+ = \max \{ (2{-}\mu), 0\}$. Therefore,  the entropy inequality \eqref{entropy-ineq} in fact holds for all positive test functions $\varphi \in L^\infty ([0,T]; W^{1,d+\epsilon}(\Omega)) \cap H^1 (0,T; L^{6/5}(\Omega))$, for every $\epsilon>0$. 
 \end{maintheorem}
 \noindent
 The enhanced summability for $\log(\teta)$ in \eqref{further-logteta} ensues from the fact that for every $p \in [1,\infty)$
 there exists $C_p$ such that 
 \[
 |\log(\teta)|^p\leq \teta +C_p \qquad \text{for all } \teta \geq \bar\theta.
 \]
\begin{remark}
\label{rmk:in-LRTT}
\upshape
In \cite{LRTT}  we proved an existence result for a PDE  system modeling \emph{rate-independent } damage in thermoviscoelastic materials, featuring a temperature equation with the same structure as \eqref{heat}.  Also in that context we obtained a  strict positivity property  with the same constant as  in \eqref{strong-strict-pos}. Moreover, 
we showed that, if the heat source function $H$ and the initial temperature $\teta_0$ fulfill
\[
H(t,x) \geq H_*>0 \ \foraa (t,x) \in Q \ \text{ and } \ \teta_0(x) \geq \sqrt{H_*/\bar{c}} \ \foraa x \in \Omega,
\]
with $\bar{c}>0$ from \eqref{strong-strict-pos}, then the enhanced positivity property
\begin{equation}
\label{enh-strict-pos}
 \teta(t,x) \geq \max\{ \bar\teta, \sqrt{H_*/\bar{c}}\} \quad \foraa (t,x) \in Q
\end{equation}
holds. In the setting of the thermoviscoplastic system \eqref{plast-PDE}, too,  it would be possible to prove \eqref{enh-strict-pos}. Observe that, choosing suitable data for the heat equation,  the threshold $\max\{ \bar\teta, \sqrt{H_*/\bar{c}}\}$, and thus the temperature, may be tuned to stay above a given constant. Choosing such a constant as the so-called \emph{Debye temperature} (cf., e.g., \cite[Sec.\ 4.2, p.\ 761]{Wed97LPC}), according to the Debye model one can thus justify the assumption that the heat capacity is constant.
\end{remark}
\par
Under a more stringent growth condition on $\condu$, we obtain the existence of weak energy solutions.
\begin{maintheorem}
\label{mainth:2} Assume \eqref{Omega-s2},  \eqref{elast-visc-tensors},  $(2.\mathrm{H})$,  $(2.\mathrm{L})$,  
 \eqref{Dirichlet-loading},  $(2.\mathrm{K})$, and  
  \eqref{hyp-K}. In addition, 
suppose that 
 the exponent $\mu$ in \eqref{hyp-K} fulfills
\begin{equation}
\label{hyp-K-stronger}
\tag{2.$\condu_2$}
\begin{cases}
\mu \in (1,2) & \text{if } d=2,
\\
\mu \in \left(1, \frac53\right) & \text{if } d=3.
\end{cases}
\end{equation}
Then, for every $(\teta_0, u_0, \dot{u}_0, e_0, p_0) $ satisfying \eqref{Cauchy-data} there exists a weak energy  solution $(\teta,u,e,p)$ to the Cauchy problem for system (\ref{plast-PDE}, \ref{bc})  satisfying 
 \eqref{strong-strict-pos}--\eqref{further-logteta}, as well as 
 \begin{equation}
 \label{further-k-teta}
 \nabla (\hat{\condu}(\teta) )  \in L^{1+\tilde{\delta}}(Q)  \text{ for some $\tilde\delta \in \left(0,\frac13 \right)$},
 \end{equation}  
 with $\hat{\condu}$ a primitive of $\condu$.
 Therefore,
  \eqref{eq-teta} in fact holds for all test functions $\varphi \in W^{1,1+1/{\tilde\delta}}(\Omega)$ and, ultimately, $\teta $ has the enhanced regularity $\teta \in W^{1,1}(0,T;   W^{1,1+1/{\tilde\delta}}(\Omega)^*)$. 
 \end{maintheorem}
As it will be clear from the proof of Thm.\ \ref{mainth:2}, in the case $d=3$  the exponent $ \tilde \delta $ is in fact given by  $ \tilde \delta = \frac{2-3\mu+3\alpha}{3(\mu-\alpha+2)}  $ for all $ \alpha \in (\bar\alpha, 1) $   with $\bar\alpha: = \max\{ \mu-\frac23, (2-\mu)^+\}$:
The condition $\mu <\frac53$ for $d=3$ in fact ensures   that it is possible to choose $\alpha<1$ with $\alpha > \mu-\frac23$. Also, note that   for every $\alpha $ in the prescribed range we have that $\tilde \delta<\tfrac13$, so that  $1+\tfrac1{\tilde\delta} >4$. This yields 
\begin{equation}
\label{for-later-reference}
W^{1,1+1/{\tilde\delta}}(\Omega) \subset L^\infty(\Omega) \qquad \text{for } d \in \{2,3\},
\end{equation}
so that every $\varphi\in W^{1,1+1/{\tilde\delta}}(\Omega)$  can multiply the $L^1$-r.h.s.\ of   the heat equation \eqref{heat} and, moreover, has trace in $L^2(\partial\Omega)$. Therefore, $W^{1,1+1/{\tilde\delta}}(\Omega)$ is an admissible space of test functions for  \eqref{eq-teta}.
Clearly, in the case $d=2$  as well one can explicitly compute $\tilde \delta$, exploiting  the  condition $\mu <2$, leading to  a better range of indexes. 
\par 
 The proofs of Theorems
\ref{mainth:1} and \ref{mainth:2}, developed in Section \ref{s:4},  shall result from passing to the limit in a carefully tailored time discretization scheme of the thermoviscoplastic system (\ref{plast-PDE}, \ref{bc}), analyzed in  detail in  Section  \ref{s:3}. 
\section{Analysis of the thermoviscoplastic system: time discretization}
\label{s:3}
The analysis of the time-discrete scheme  for system  (\ref{plast-PDE}, \ref{bc}) shall  often follow the lines of that developed for the phase transition/damage system analyzed in \cite{Rocca-Rossi}  (cf.\ also the proof of \cite[Thm.\ 2.7]{LRTT}). Therefore, to avoid overburdening the exposition we will not fully develop  all the arguments, but frequently  refer to \cite{Rocca-Rossi, LRTT} for all details. 
\par
In the statement of all the results of this section we will always tacitly assume the conditions on the problem data from  Section \ref{ss:2.1}.
\par
Given
an equidistant partition of $[0,T]$, with time-step $\tau>0$ and
nodes $t_\tau^k:=k\tau$, $k=0,\ldots,K_\tau$,  we
approximate the data  $F$, $g$,  $H$, and $h$ 
 by local means as follows
\begin{equation}
\label{local-means} 
\begin{gathered}
\Ftau{k}:= 
\frac{1}{\tau}\int_{t_\tau^{k-1}}^{t_\tau^k}  F(s)\dd s\,,
\quad 
g_\tau^{k}:= 
\frac{1}{\tau}\int_{t_\tau^{k-1}}^{t_\tau^k}  g(s)\dd s\,,
\quad
 \gtau{k}:= \frac{1}{\tau}\int_{t_\tau^{k-1}}^{t_\tau^k} H(s)
\dd s\,, \quad \htau{k}:= \frac{1}{\tau}\int_{t_\tau^{k-1}}^{t_\tau^k} h(s)
\dd s
\end{gathered}
\end{equation}
for all $k=1,\ldots, K_\tau$.  From the terms $\Ftau{k}$ and $\gtau{k}$ one then defines the elements $ \Ltau{k}$, which are the local-mean approximations of $\calL$. 
Hereafter, given elements $(v_{\tau}^k)_{k=1,\ldots, K_\tau}$ in a Banach space $B$,
we will use the notation
\[
\Dtau{k} v: = \frac{v_{\tau}^{k} - v_{\tau}^{k-1}}\tau, \qquad  \Ddtau{k} v: = \frac{\vtau{k} -2 \vtau{k-1} + \vtau{k-2}}{\tau^2}.
\]
\par
We construct discrete solutions to system  (\ref{plast-PDE}, \ref{bc}) by recursively solving an elliptic system,  cf.\ the forthcoming Problem 
\ref{prob:discrete}, where the weak formulation of the discrete heat equation features the function space
 \begin{equation}
\label{X-space}
 X:= \{ \theta  \in H^1(\Omega)\, : \   \condu(\theta) \nabla \theta  \nabla v  \in L^1(\Omega)   \text{ for all } v \in H^1 (\Omega)\},
 \end{equation}
 and, for $k \in \{1,\ldots, K_\tau\}$,  the elliptic operator
 \begin{equation}
 {A}^k: X  \to  H^1(\Omega)^* \text{  defined by }\\
\pairing{}{H^1(\Omega)}{ {A}^k(\theta) }{v}:=
  \int_\Omega \condu(\theta) \nabla \theta \nabla v \dd x - \int_{\partial \Omega} \htau{k} v \dd S\,.
\end{equation} 
We also mention in advance that, for technical reasons connected both with the proof of existence of discrete solutions to Problem \ref{prob:discrete} (cf.\ the upcoming Lemma 
\ref{l:exist-approx-discr}), and  with the rigorous derivation of 
a priori estimates on them (cf.\ Remark \ref{rmk:comments} below), it will be necessary to add the regularizing term $-\tau \mathrm{div}(|\etau k|^{\gamma- 2} \etau k)$ to the discrete momentum equation, as well as the term $\tau |\ptau k|^{\gamma-2} \ptau k$ to the discrete plastic flow rule, with $\gamma>4$. That is why, we will seek for discrete solutions with  $\etau k\in L^{\gamma} (\Omega;\mt_{\sym}^{d\times d}) $  and $\ptau k  \in L^{\gamma} (\Omega;  \mt_{\dev}^{d\times d})$, giving $\sig{\utau{k}} \in  L^{\gamma} (\Omega;\mt_{\sym}^{d\times d})$  by the kinematic admissibility condition and thus, via Korn's inequality \eqref{Korn}, $\utau k \in W_\Dir^{1,\gamma} (\Omega;\R^d)$.  Because of these regularizations, it will be necessary to supplement the discrete system with  approximate initial data
\begin{subequations}
\label{complete-approx-e_0}
\begin{equation}
\label{approx-e_0}
\begin{aligned}
& 
(e_\tau^0)_\tau  \subset   L^{\gamma} (\Omega;\mt_{\sym}^{d\times d})  & \text{ such that }
\lim_{\tau\down 0} \tau^{1/\gamma} \| e_\tau^0\|_{L^{\gamma} (\Omega;\mt_{\sym}^{d\times d})} =0
 &  \text{ and } e_\tau^0 \to e_0 \text{ in $L^{2} (\Omega;\mt_{\sym}^{d\times d})$},
\\
& 
(p_\tau^0)_\tau  \subset   L^{\gamma} (\Omega;\mt_{\dev}^{d\times d})  & \text{ such that }
\lim_{\tau\down 0} \tau^{1/\gamma} \| p_\tau^0\|_{L^{\gamma} (\Omega;\mt_{\dev}^{d\times d})} =0
 &  \text{ and } p_\tau^0 \to p_0 \text{ in $L^{2} (\Omega;\mt_{\dev}^{d\times d})$}.
\end{aligned}
\end{equation}
By consistency with  the kinematic admissibility condition at time $t=0$, we will also approximate the initial datum $u_0$ with a family 
$(u_\tau^0)_\tau  \subset   W^{1,\gamma} (\Omega;\R^d)$ such that 
\begin{equation}
\label{approx-e_0-3}
(u_\tau^0)_\tau  \subset   W^{1,\gamma} (\Omega;\R^d)   \text{ such that } \lim_{\tau\down 0} \tau^{1/\gamma} \| u_\tau^0\|_{W^{1,\gamma} (\Omega;\R^d)} =0   \text{ and } u_\tau^0 \to u_0 \text{ in $H^{1} (\Omega;\R^d)$}.
\end{equation}
\end{subequations}
In connection with the regularization of the discrete momentum balance,  we will have to approximate the Dirichlet loading $w$ by a family $(w_\tau)_\tau \subset \mathbf{W} \cap 
W^{1,1}(0,T; W^{1,\gamma}(\Omega;\R^d))$, where we have used the place-holder
$ \mathbf{W}:=  L^1(0,T; W^{1,\infty} (\Omega;\R^d)) \cap W^{2,1} (0,T;H^1(\Omega;\R^d)) \cap H^2(0,T; L^2(\Omega;\R^d))$. We will require that 
\begin{equation}
\label{discr-w-tau}
w_\tau \to w \text{ in } \mathbf{W} \text{ as } \tau \downarrow 0, \quad  \text{ as well as }  \quad
\exists\, \alpha_w \in \left(0,\frac1\gamma\right)  \text{ s.t. }    \sup_{\tau>0} \tau^{\alpha_w} \| \sig{\dot{w}_\tau}\|_{L^\gamma (\Omega;\mt_\sym^{d\times d})} \leq C<\infty\,.
\end{equation}
We will then consider the discrete data
\[
 \wtau{k}:=
\frac{1}{\tau}\int_{t_\tau^{k-1}}^{t_\tau^k} w_\tau(s)\dd s\,.
\]
%
\begin{problem}
\label{prob:discrete}
Let $\gamma>4$. Starting from
\begin{align}\label{discr-Cauchy}
\tetau{0}:=\teta_{0}, \qquad \utau{0}:=u^{0}_\tau,\qquad
  \utau{-1}:=u_\tau^{0}-\tau \dot{u}_0,  \qquad \etau{0}: = e_\tau^0, \qquad \ptau{0}:=p_\tau^{0}
\end{align}
 with $\teta_0$ and $\dot{u}_0$ from  \eqref{Cauchy-data} and $(u_\tau^{0}, e_\tau^{0}, p_\tau^{0})$ from 
\eqref{complete-approx-e_0}, 
 find
 $\{(\tetau{k}, \utau{k}, \etau{k}, \ptau{k})\}_{k=1}^{K_\tau}
\subset X \times W_\Dir^{1,\gamma}(\Omega;\R^d) \times   L^{\gamma} (\Omega;\mt_{\sym}^{d\times d}) \times  L^{\gamma} (\Omega;  \mt_{\dev}^{d\times d})$ fulfilling for all $k=1,\ldots, K_\tau$
\begin{subequations}
\label{syst:discr}
\begin{itemize}
\item[-] the discrete heat equation
\begin{equation}
\label{discrete-heat} \begin{aligned}
 &\Dtau{k}\teta + A^k (\tetau{k}) \\ &  = \gtau{k} + \mathrm{R}\left(\tetau{k-1}, \Dtau{k} p  \right) +  \left|  \Dtau{k} p \right|^2+ \bbD   \Dtau{k} e :  \Dtau{k} e -\tetau{k} \bbB : \Dtau k e   \quad \text{in } H^1(\Omega)^*;
\end{aligned} \end{equation}
\item[-] the kinematic admissibility $(\utau{k}, \etau{k}, \ptau{k}) \in \mathcal{A}(\wtau k)$ (in the sense of \eqref{kin-adm}); 
\item[-] the discrete momentum balance 
\begin{equation}
\label{discrete-momentum}
\rho\int_\Omega \Ddtau k u  v \dd x + \int_\Omega \sitau{k} : \sig v \dd x     = \pairing{}{H_\Dir^1(\Omega;\R^d)}{\Ltau k }{v} \quad
 \text{for all } v \in W_\Dir^{1,\gamma}(\Omega;\R^d);
\end{equation}
\item[-] the discrete plastic flow rule
\begin{equation}
\label{discrete-plastic}\zetau k  +  \Dtau kp + \tau |\ptau {k}|^{\gamma-2}\ptau k = \sidevtau{k}, \quad \text{with } \zetau k \in \partial_{\dot p} \mathrm{R} \left(\tetau{k-1},\Dtau k p \right),  \qquad \aein \Omega,
\end{equation}
\end{itemize}
\end{subequations}
where we have used the place-holder
$
\sitau{k}:= \bbD    \Dtau k e  + \bbC \etau k + \tau |\etau{k}|^{\gamma-2} \etau k- \tetau k \bbB\,.
$
\end{problem}
\begin{remark}[Main features of the time-discretization scheme]
\upshape
\label{rmk:comments} 
\upshape
 Observe that  the discrete heat equation  \eqref{discrete-heat} is coupled with the momentum balance \eqref{discrete-momentum} through the implicit term $\tetau{k} $, which therefore contributes to the stress $\sitau{k}$ in \eqref{discrete-momentum} and in \eqref{discrete-plastic}. This makes the  time discretization scheme \eqref{syst:discr}
 fully implicit, as it is not possible to decouple any of the equations from the others. In turn, the `implicit coupling' between the heat  equation and the momentum balance is crucial for the argument leading to the (strict) positivity of the discrete temperatures: we  refer to the proof of  \cite[Lemma 4.4]{Rocca-Rossi} for all details. In fact, the time discretization schemes in \cite{Roub-Bartels-2,Roub-PP} are fully implicit as well,  again  in view of the positivity of the temperature (though the arguments there are different, based on the approach via the enthalpy transformation in the heat equation).
 \par
 The role of the terms $-\tau \mathrm{div}(|\etau k|^{\gamma- 2} \etau k)$  and $\tau |\ptau k|^{\gamma-2} \ptau k$, added to the discrete momentum equation and   plastic flow rule, respectively, is to `compensate' the quadratic terms on the right-hand side of \eqref{discrete-heat}.  More precisely, they ensure that the pseudomonotone operator by means of which  we will reformulate
 our approximation of 
  system \eqref{syst:discr}, c.f.\ \eqref{pseudo-monot} ahead, is coercive, in particular w.r.t.\  the $H^1(\Omega)$-norm  in the variable $\teta$. This will allow us to  apply a result from the theory of pseudomonotone operators in order to obtain the existence of solutions to \eqref{pseudo-monot} and, a fortiori, to \eqref{syst:discr}. 
\end{remark}
\begin{proposition}[Existence of discrete solutions]
\label{prop:exist-discr}
Under  the growth condition \eqref{hyp-K}, Problem \ref{prob:discrete} admits a solution  $\{(\tetau{k}, \utau{k}, \etau{k}, \ptau{k})\}_{k=1}^{K_\tau}$. Furthermore, any solution  to  Problem \ref{prob:discrete} fulfills
\begin{equation}
\label{discr-strict-pos}
\tetau{k}\geq \bar{\teta}>0 \qquad \text{for all } k =1, \ldots, K_\tau, 
\end{equation}
with  $\bar{\teta}$ from \eqref{strong-strict-pos}.  
\end{proposition} 
Along the lines of \cite{Rocca-Rossi, LRTT}, we will prove the existence of a solution to Problem \ref{prob:discrete}
by
\begin{enumerate}
\item constructing an approximate problem where the thermal conductivity coefficient $\condu$ is truncated and, accordingly, so are the occurrences of $\teta$ in the thermal expansion terms coupling the discrete heat and momentum equations, cf.\ system \eqref{syst:approx-discr} below; 
\item proving the existence of a solution to the approximate discrete problem by resorting to a general existence result from \cite{Roub05NPDE} for elliptic systems featuring pseudomonotone operators;
\item passing to the limit with the truncation parameter.
\end{enumerate}
As the statement of Proposition \ref{prop:exist-discr}
suggests, the positivity property 
\eqref{discr-strict-pos}
 can be proved for \emph{all} discrete solutions to Problem  \ref{prob:discrete} (i.e.\ not only for those deriving from the aforementioned
 approximation procedure). Since its proof can be carried out by repeating the arguments for positivity in 
 \cite{Rocca-Rossi, LRTT},   we choose to omit it and refer to these papers for all details. We shall instead focus on the existence argument, dwelling  on the parts which differ from  \cite{Rocca-Rossi, LRTT} with some detail. 
 \par
 The \underline{proof of Proposition \ref{prop:exist-discr}} will be split in  some steps:
 \\
\textbf{Step $1$: existence for the approximate discrete system.} 
We introduce the truncation operator 
\begin{equation}
\label{def-truncation-m} 
\mathcal{T}_M : \R \to \R, \qquad 
\mathcal{T}_M(r):= \left\{ \begin{array}{ll} -M &
\text{if } r <-M,
\\
r   & \text{if } |r| \leq M,
\\
M  & \text{if } r >M,
\end{array}
\right.
\end{equation}
and define
\begin{align}
& 
\label{def-k-m} \condu_M(r):= \condu(\mathcal{T}_M(r)) := \left\{ \begin{array}{ll} \condu(-M) & \text{if
} r <-M,
\\
\condu(r)   & \text{if } |r| \leq M,
\\
\condu(M) & \text{if } r >M,
\end{array}
\right.
\\
&
\label{M-operator}
{A}_M^k: H^1(\Omega)  \to H^1(\Omega)^*
  \text{   by }
\pairing{}{H^1(\Omega)} { {A}_M^k(\theta)}{v}:= \int_\Omega \condu_M(\theta) \nabla \theta  \nabla v \dd x -
\int_{\partial \Omega} \htau{k}v \dd S.
\end{align}
For later use, we observe that, thanks to \eqref{hyp-K}  there still holds $\condu_M(r) \geq c_{0} $ for all $r \in \R$, and therefore
\begin{equation}
\label{ellipticity-retained} \forall\, \delta>0 \  \ \exists\, C_\delta>0\, \ \forall\,  \theta \in H^1(\Omega)\, : \qquad 
\pairing{}{H^1(\Omega)}{ \mathcal{A}_M^{k}  (\theta)}{\theta} \geq c_0 \int_\Omega |\nabla \theta|^2 \dd x - \delta \|\theta\|_{L^2(\partial\Omega)}^2 - C_\delta \| \htau{k} \|_{L^2(\partial\Omega)}^2\,. 
\end{equation}
\par 
The approximate version of system \eqref{syst:discr} reads (to avoid overburdening notation, for the time being we will not highlight the dependence of the solution quadruple on the truncation parameter $M$):
\begin{subequations}
\label{syst:approx-discr}
\begin{align}
\label{heat:approx-discr}
&
 \begin{aligned}
 &\Dtau{k}\teta + A_M^k (\tetau{k}) \\ &  = \gtau{k} + \mathrm{R}\left(\tetau{k-1}, \Dtau{k} p  \right) +  \left|  \Dtau{k} p \right|^2+ \bbD   \Dtau{k} e :  \Dtau{k} e -\mathcal{T}_M(\tetau{k}) \bbB : \Dtau k e   \quad \text{in } H^1(\Omega)^*;
\end{aligned} 
\\
& \label{mom:approx-discr}
\rho\int_\Omega  \Ddtau k {u} v \dd x + \int_\Omega \sitau{k} : \sig v \dd x     = \pairing{}{H_\Dir^1(\Omega;\R^d)}{\Ltau k }{v} 
\quad  \text{for all } v \in W_\Dir^{1,\gamma}(\Omega;\R^d);
 \\
 & \label{plasr:approx-discr} 
\zetau{k} +  \Dtau kp + \tau |\ptau {k}|^{\gamma-2}\ptau k  =  \simdevtau{k} ,\quad \text{with } \zetau k \in   \partial_{\dot p} \mathrm{R} \left(\tetau{k-1},\Dtau k p \right), \qquad \aein \Omega,
  \end{align}
  \end{subequations}
  coupled with 
the kinematic admissibility 
\begin{equation}
\label{discr-kin-adm}
(\utau{k}, \etau{k}, \ptau{k}) \in \mathcal{A}(\wtau k),
\end{equation}
where now
\[
\sitau{k}:= \bbD    \Dtau k e  + \bbC \etau k + \tau |\etau{k}|^{\gamma-2} \etau k- \calT_M(\tetau k) \bbB\,.
\]
\par
The following result states the existence of solutions to system \eqref{syst:approx-discr} for $k\in \{1,\ldots, K_\tau\} $ fixed: in its proof, we make use of the higher order terms added to the discrete momentum equation and plastic flow rule.  
\begin{lemma}
\label{l:exist-approx-discr} 
Under  the growth condition \eqref{hyp-K}, there exists $\bar\tau>0$ such that  for $0<\tau< \bar \tau$ and  for every $k=1,\ldots, K_\tau$ there exists a solution $(\tetau{k}, \utau{k}, \etau{k}, \ptau{k}) \in H^1(\Omega)  \times W_\Dir^{1,\gamma}(\Omega;\R^d) \times L^\gamma(\Omega;\mt_\sym^{d\times d})  \times L^\gamma(\Omega;\mt_\dev^{d\times d})  $ to system \eqref{syst:approx-discr}, such that $\tetau k$ complies with the positivity property \eqref{discr-strict-pos}. 
\end{lemma}
\begin{proof}
The positivity  \eqref{discr-strict-pos} follows from the same argument developed in the proof of  \cite[Lemma 4.4]{Rocca-Rossi}. As for existence: For fixed $k \in \{1,\ldots, K_\tau\}$,  
we reformulate system \eqref{syst:approx-discr}, coupled with \eqref{discr-kin-adm},  as 
\begin{equation}
\label{pseudo-monot}
\partial\Psi_k (\tetau k , \utau k -\wtau k , \ptau k) +\mathscr{A}_k (\tetau k , \utau k -\wtau k , \ptau k ) \ni \mathscr{B}_k,
\end{equation}
where the  elliptic operator $\mathscr{A}_k : \mathbf{B} \to \mathbf{B}^*$, with  $\mathbf{B}: =  H^1(\Omega)  \times W_\Dir^{1,\gamma}(\Omega;\R^d) \times L^\gamma(\Omega;\mt_\dev^{d\times d})$,  is given component-wise by 
\begin{subequations}
\label{oper-scrA}
\begin{align}
& 
\label{oper-scrA-1}\begin{aligned}
\mathscr{A}_k^1 (\teta, \tilde u, p): =  & \teta+ A_M^k(\teta) - \mathrm{R}(\tetau {k-1}, p-\ptau{k-1})- \frac1\tau |p|^2   - \frac2\tau p : \ptau{k-1} \\ &\quad  -\frac1\tau \bbD \left(\sig{\tilde u + \wtau k} -p \right){:} \left(\sig{\tilde u + \wtau k} -p \right)
-\frac2\tau \bbD \left(\sig{\tilde u + \wtau k} -p \right){:}\etau{k-1} 
 \\ & \quad+ \mathcal{T}_M (\teta) \bbB \left(\sig{\tilde u + \wtau k} -p - \etau{k-1} \right),
\end{aligned}
\\
& 
\label{oper-scrA-2}\begin{aligned}
\mathscr{A}_k^2 (\teta, \tilde u, p): =   \rho (\tilde u -\wtau k) - \mathrm{div}_{\Dir}\Big(  & \tau \bbD \left( \sig{\tilde u + \wtau k} -p   \right) +\tau^2 \bbC   \left( \sig{\tilde u + \wtau k} -p   \right) \\ & + \tau^3 \left|   \sig{\tilde u + \wtau k} -p    \right|^{\gamma-2}   \left( \sig{\tilde u + \wtau k} -p   \right) -\tau^2 \mathcal{T}_M(\teta) \bbB \Big),
\end{aligned}
\\
& 
\label{oper-scrA-3}\begin{aligned}
\mathscr{A}_k^3(\teta, \tilde u, p): =  p + \tau^2 |p|^{\gamma-2} p
& 
- \Big(   \bbD \left( \sig{\tilde u + \wtau k} -p   \right) +\tau  \bbC   \left( \sig{\tilde u + \wtau k} -p   \right) \\  & \qquad \qquad + \tau^2 \left|   \sig{\tilde u + \wtau k} -p    \right|^{\gamma-2}   \left( \sig{\tilde u + \wtau k} -p   \right) -\tau \mathcal{T}_M(\teta) \bbB \Big)_\dev,
\end{aligned}
\end{align}
with $-\mathrm{div}_{\Dir}$ defined by \eqref{div-Gdir}, 
 \end{subequations}
while the vector $\mathscr{B}_k \in \mathbf{B}^*$ on the right-hand side  of \eqref{pseudo-monot} has components
\begin{subequations}
\label{vector-B}
\begin{align}
&
\label{vect-B-1}
\mathscr{B}_k^1: = \gtau k + \frac1\tau |\ptau{k-1}|^2 + \frac1\tau \bbD \etau{k-1}: \etau{k-1}, 
\\
&
\label{vect-B-2}
\mathscr{B}_k^2: = \Ltau k +2\rho\utau{k-1} - \rho\utau{k-1} - \mathrm{div}_{\Dir}(\tau \bbD \etau{k-1}), 
\\
&
\label{vect-B-3}
\mathscr{B}_k^3: = \ptau{k-1} - (\bbD \etau{k-1})_\dev,
\end{align}   \end{subequations} 
 and $\partial \Psi_k : \mathbf{B} \rightrightarrows \mathbf{B}^*$ is the subdifferential of the lower semicontinuous and convex potential
$
\Psi_k(\teta, \tilde u, p): = \mathrm{R}(\tetau{k-1}, p -\ptau{k-1}).  
$
  We shall therefore prove the existence of a solution to the abstract subdifferential inclusion \eqref{pseudo-monot} by  applying the existence result \cite[Cor.\ 5.17]{Roub05NPDE}, which amounts to verifying that $\mathscr{A}_k : \mathbf{B} \to \mathbf{B}^*$ is coercive and  pseudomonotone.  The latter property means  that (cf., e.g.,  \cite{Roub05NPDE}) it is bounded and
  fulfills the following  for all $(\eta_m)_m,\eta,\zeta \in  \mathbf{B}$:
  \begin{equation}
  \label{def-pseudomon}
  \left.
  \begin{array}{rr}
&  \eta_m \weakto \eta, 
  \\
  & \limsup_{m\to\infty}  \pairing{}{\mathbf{B}}{\mathscr{A}_k (\eta_m)}{\eta_m-\eta} \leq 0 
  \end{array}
  \right\} \ \Rightarrow 
   \pairing{}{\mathbf{B}}{\mathscr{A}_k (\eta)}{\eta-\zeta} \leq  \liminf_{m\to\infty} \pairing{}{\mathbf{B}}{\mathscr{A}_k (\eta_m)}{\eta_m-\zeta}\,.
  \end{equation}
  \par
  To check coercivity, we compute
  \begin{equation}
  \label{rough-coerc-1}
  \begin{aligned}
  \pairing{}{\mathbf{B}}{\mathscr{A}_k (\teta, \tilde u, p)}{(\teta,\tilde u, p)} & = \pairing{}{H^1(\Omega)}{\mathscr{A}_k^1 (\teta, \tilde u, p)}{\teta} +  \pairing{}{W_\Dir^{1,\gamma}(\Omega;\R^d)}{\mathscr{A}_k^2 (\teta, \tilde u, p)}{\tilde u}  + \int_\Omega  \mathscr{A}_k^3 (\teta, \tilde u, p) : p \dd x 
\\ &   \stackrel{(1)}{\geq} \|\teta\|_{L^2(\Omega)}^2 + c_0 \|\nabla \teta\|_{L^2(\Omega)}^2 + \rho \| \tilde  u \|_{L^2(\Omega)}^2 + \left( \tau C_\bbD^1 + \tau^2 C_\bbC^1 \right) 
  \| \sig{\tilde u} +\sig{\wtau k}  - p \|_{L^2(\Omega)}^2 \\ & \quad + \tau^3    \| \sig{\tilde u} +\sig{\wtau k}  - p \|_{L^\gamma(\Omega)}^\gamma + \| p \|_{L^2(\Omega)}^2 + \tau^2 \| p \|_{L^\gamma(\Omega)}^\gamma  + I_1 + I_2  +I_3, 
  \end{aligned}
\end{equation}
where (1) follows from  \eqref{elast-visc-tensors} and \eqref{ellipticity-retained}. Taking into account \eqref{linear-growth}, again \eqref{elast-visc-tensors}, and the fact that $|\calT_M(\teta) | \leq M $ a.e.\ in $\Omega$,  we have
\begin{subequations}
\label{est_I-terms}
\begin{equation}
\label{est-I-1}
\begin{aligned}
I_1  &  = -\delta \| \teta \|_{L^2(\partial\Omega)}^2 - C_\delta   \| \htau k \|_{L^2(\partial\Omega)}^2  \\ & \quad  - C_R\| p -\ptau{k-1}\|_{L^2(\Omega)} \|\teta\|_{L^2(\Omega)} - \frac1\tau \| p \|_{L^4(\Omega)}^2 \| \teta\|_{L^2(\Omega)}  -\frac2\tau \|\ptau{k-1}\|_{L^4(\Omega)}\| p \|_{L^4(\Omega)} \| \teta\|_{L^2(\Omega)}   \\ & \quad - \frac{C_\bbD^2}\tau    \| \sig{\tilde u} +\sig{\wtau k}  - p \|_{L^4(\Omega)}^2  \| \teta\|_{L^2(\Omega)}  - C  \| \sig{\tilde u} +\sig{\wtau k}  - p \|_{L^4(\Omega)} \| \etau {k-1}\|_{L^4(\Omega)} \| \teta\|_{L^2(\Omega)}  \\ & \quad - C \| \sig{\tilde u} +\sig{\wtau k}  - p \|_{L^2(\Omega)}^2 - C \| \etau {k-1}\|_{L^2(\Omega)}^2 \,,
\end{aligned}
\end{equation}
 with $\delta>0$ to be specified later, 
as well as 
\begin{equation}
\label{est-I-2}
\begin{aligned}
I_2  & = -\rho \| \tilde u \|_{L^2(\Omega)}  \| \wtau k  \|_{L^2(\Omega)}  - C \| \sig{\tilde u} + \sig{\wtau k} -p \|_{L^2(\Omega)} \|\sig{\wtau k} -p \|_{L^2(\Omega)}\\ & \qquad  -\tau^3 \int_\Omega |\sig{\tilde u} + \sig{\wtau k} -p|^{\gamma-1} |\sig{\wtau k} -p | \dd x - C \int_\Omega |\sig{\tilde u}| \dd x,
\end{aligned}
\end{equation}
and 
\begin{equation}
\label{est-I-3}
\begin{aligned} 
I_3  & =- C \| \sig{\tilde u} + \sig{\wtau k} -p\|_{L^2(\Omega)} \| p \|_{L^2(\Omega)}  -\tau^2\int_\Omega |  \sig{\tilde u} + \sig{\wtau k} -p|^{\gamma-1} |p| \dd x  - C \int_\Omega |p|\dd x\,.
\end{aligned}
\end{equation}
Now, with straightforward calculations it is possible to absorb the negative terms $I_1, \, I_2,\, I_3$ into the positive terms on the right-hand side of \eqref{rough-coerc-1}: without entering into details, let us only observe  that, for example, the sixth term on the right-hand side of \eqref{est-I-1}
 can be estimated by means of Young's inequality as 
\[
\begin{aligned}
 - \frac{C_\bbD^2}\tau    \| \sig{\tilde u} +\sig{\wtau k}  - p \|_{L^4(\Omega)}^2 \| \teta\|_{L^2(\Omega)} & 
  \geq 
 -\delta \| \teta\|_{L^2(\Omega)}^2 - C   \| \sig{\tilde u} +\sig{\wtau k}  - p \|_{L^4(\Omega)}^4
\\ &  \geq -\delta \| \teta\|_{L^2(\Omega)}^2
  - \frac{\tau^3}2  \| \sig{\tilde u} +\sig{\wtau k}  - p \|_{L^\gamma(\Omega)}^\gamma - C,
  \end{aligned}
\] using that $\gamma>4$.
The fourth term can be dealt with  in the same way, so that one of the resulting terms is  absorbed into   $\tau^2 \| p \|_{L^\gamma(\Omega)}^\gamma$.  The other terms  contributing to $I_1$, $I_2$, and
$I_3$
can be handled analogously. 
Let us now observe that the positive terms on the right-hand side of \eqref{rough-coerc-1} bound the 
desired norms of $\teta$, $\tilde u$, $ p$. Indeed,  also taking into account that, again by Young's inequality
\[
 \| \sig{\tilde u} +\sig{\wtau k}  - p \|_{L^2(\Omega)}^2  \geq c \|  \sig{\tilde u}  \|_{L^2(\Omega)}^2 - C \| \sig{\wtau k}  \|_{L^2(\Omega)}^2 - \frac{\tau^2}4\|p \|_{L^\gamma(\Omega)}^\gamma - C, 
\]
and repeatedly using  the well-known estimate $(a+b)^{\gamma} \leq 2^{\gamma-1} (a^\gamma+b^\gamma)$ for all $a,\, b \in [0,+\infty)$, which gives 
\[
\begin{aligned}
 \tau^3    \| \sig{\tilde u} +\sig{\wtau k}  - p \|_{L^\gamma(\Omega)}^\gamma  +\frac{ \tau^2}4 \| p \|_{L^\gamma(\Omega)}^\gamma & \geq    \frac{\tau^3}{2^{\gamma-1}}   \| \sig{\tilde u} +\sig{\wtau k}  \|_{L^\gamma(\Omega)}^\gamma  +\left( \frac{ \tau^2}4-\tau^3\right)  \| p \|_{L^\gamma(\Omega)}^\gamma  
 \\ & \geq   \frac{\tau^3}{2^{2\gamma-2}}   \| \sig{\tilde u}   \|_{L^\gamma(\Omega)}^\gamma  +\frac{ \tau^2}8 \| p \|_{L^\gamma(\Omega)}^\gamma -  \frac{\tau^3}{2^{\gamma-1}} \|\sig{\wtau k}  \|_{L^\gamma(\Omega)}^\gamma   \end{aligned}
\]
(where we have also used that, for $\tau < \bar\tau := 1/8$, there holds $\tau^2/8 \geq \tau^3 $), 
we end up with
\[
  \pairing{}{\mathbf{B}}{\mathscr{A}_k (\teta, \tilde u, p)}{(\teta,\tilde u, p)} \geq 
c \left(  \|\teta\|_{H^1(\Omega)}^2 +  \| \tilde  u \|_{L^2(\Omega)}^2 +  \|  \sig{\tilde u}  \|_{L^2(\Omega)}^2  +  \|  \sig{\tilde u}  \|_{L^\gamma(\Omega)}^\gamma   + \| p \|_{L^2(\Omega)}^2 + \| p \|_{L^\gamma(\Omega)}^\gamma \right) - C
\]
for two positive constants $c$ and $C$, depending on $\tau$, on $M$, and on $w$. Thanks to Korn's inequality \eqref{Korn}, this shows the coercivity of $\mathscr{A}_k$. Its pseudomonotonicity \eqref{def-pseudomon} ensues from standard arguments.
  Indeed, one can observe that 
 $\mathscr{A}_k$ is given by the sum  of either bounded, radially continuous, monotone mappings (cf.\ e.g.\   \cite[Def.\ 2.3]{Roub05NPDE}), which
 are pseudomonotone \cite[Lemma\ 2.9]{Roub05NPDE}, or of totally continuous mappings. In fact, perturbations of pseudomonotone mappings by totally continuous ones remain pseudomonotone, \cite[Cor. \ 2.12]{Roub05NPDE}.  Therefore, we are in a position to apply   \cite[Cor.\ 5.17]{Roub05NPDE} and thus conclude
  the  existence of solutions to system \eqref{syst:approx-discr}.
\end{subequations}
\end{proof}
\textbf{Step $2$: a priori estimates on the solutions of the approximate discrete system.} Let now 
\[
(\tetaum k, \utaum k, \etaum k,\ptaum k)_{M}
\]
 be a family of solutions to system \eqref{syst:approx-discr}. The following result collects a series of  a priori estimates uniform w.r.t.\ the parameter $M$ (but not w.r.t.\ $\tau$): a crucial ingredient to derive them will be a discrete version of the total energy inequality \eqref{total-enineq}, cf.\ \eqref{discr-total-ineq} below,  featuring the discrete total energy
\begin{equation}
\label{discr-total-energy}
\calE_\tau (\teta,e,p) : = \int_\Omega \teta \dd  x +\frac12 \int_\Omega \bbC e: e \dd x +\frac\tau\gamma \int_\Omega \left(|e|^\gamma + |p|^\gamma \right)  \dd x\,.
\end{equation}
\begin{lemma}
\label{l:aprio-M} Let $ k \in \{1,\ldots, K_\tau\} $ and $\tau \in (0,\bar \tau)$  be fixed. 
Under  the growth condition \eqref{hyp-K}, the solution quadruple  $(\tetaum k, \utaum k, \etaum k,\ptaum k) $ to \eqref{syst:approx-discr} satisfies 
\begin{equation}
\label{discr-total-ineq}
\begin{aligned}
& \frac{\rho}2 \int_\Omega \left|\frac{\utaum k - \utau{k-1}}{\tau} \right|^2 \dd x +  \calE_\tau (\tetaum k,\etaum k,\ptaum k) \\  &  \leq  \frac{\rho}2 \int_\Omega \left| \frac{\utau{k-1} - \utau{k-2}}\tau \right |^2 \dd x +  \calE_\tau (\tetau {k-1},\etau {k-1},\ptau{k-1})   
+ \tau \int_\Omega \gtau k \dd x    + \tau \int_{\partial\Omega} \htau{k} \dd x 
  \\ & \quad    + \tau \pairing{}{H_\Dir^{1}(\Omega;\R^d)}{\Ltau k}{\frac{\utaum k - \utau{k-1}}{\tau} {-} \Dtau kw}+ \tau \int_\Omega  \simtau {k} : \sig{\Dtau kw}  \\ & \quad  +\rho \int_\Omega \left(  \frac{\utaum k - \utau{k-1}}{\tau}  - \Dtau {k-1} u \right)   \Dtau kw \dd x\,.
\end{aligned}
\end{equation}
Moreover, there exists a constant $C>0$ such that  for all $M>0$
\begin{subequations}
\label{estimates-M-indep} 
\begin{align}
& 
\label{est-M-indep1}
\| \tetaum k\|_{L^1(\Omega)} + \| \utaum k \|_{L^2(\Omega;\R^d)} + \| \etaum k  \|_{L^2(\Omega;\mt_\sym^{d\times d})}  
\leq C,
\\
& \label{est-M-indep2}
\tau^{1/\gamma}   \| \utaum k \|_{W^{1,\gamma} (\Omega;\R^d)} 
+  \tau^{1/\gamma}   \| \etaum k \|_{L^\gamma(\Omega;\mt_\sym^{d\times d})}
+  \tau^{1/\gamma}   \| \ptaum k \|_{L^\gamma(\Omega;\mt_\dev^{d\times d})}  \leq C, 
\\ & 
\label{est-M-indep3}
\| \tetaum k \|_{H^1(\Omega)} \leq C, 
\\
& \label{est-M-indep4}
\| \zetau{k} \|_{L^\infty(\Omega;\mt_\dev^{d\times d})} \leq C,
\end{align}
\end{subequations}
where $\zetau k \in \partial_{\dot p} \mathrm{R} (\tetau{k-1}, (\ptaum{k}{ -} \ptau{k-1})/{\tau})$ fulfills \eqref{plasr:approx-discr}. 
\end{lemma}
\begin{proof} Inequality \eqref{discr-total-ineq} follows by multiplying \eqref{heat:approx-discr} by $\tau$ and integrating it in space, testing \eqref{mom:approx-discr} by $\utaum k -\wtau k - (\utau {k-1} - \wtau {k-1})$, and testing \eqref{plasr:approx-discr} by $\ptaum k - \ptaum{k-1}$. We add  the resulting relations and develop  the following estimates 
\begin{subequations}
\label{est-mimick}
\begin{align}
&
\label{est-mimick-1}
\begin{aligned}
&
\frac{\rho}{\tau^2} \int_\Omega \left(\utaum {k} {-} \utau{k-1} {-} (\utau{k-1}{ -} \utau{k-2} )   \right) (\utaum{k} {- }  \utau{k-1} )  \dd x \\ & \geq  \frac{\rho}2 \left\| \frac{\utaum k - \utau{k-1}}{\tau} \right\|_{L^2(\Omega)}^2 -  \frac{\rho}2 \left\| \frac{\utau {k-1} - \utau{k-2}}{\tau} \right\|_{L^2(\Omega)}^2,
\end{aligned}
\\
&    
\label{est-mimick-2}
\begin{aligned}
&
\int_\Omega \bbD \left(\frac{\etaum k - \etau{k-1}}\tau
\right){:} \sig{\utaum k - \utau{k-1}} \dd x 
\\ &  = \tau \int_\Omega \bbD   \frac{\etaum k - \etau{k-1}}\tau {:}\frac{\etaum k - \etau{k-1}}\tau  \dd x +\int_\Omega \bbD   \frac{\etaum k - \etau{k-1}}\tau{ :} (\ptaum k{ -} \ptau {k-1}) \dd x, 
\end{aligned}
\\
& 
\label{est-mimick-3}
\begin{aligned}
\int_\Omega \bbC \etaum k {:}   \sig{\utaum k - \utau{k-1}} \dd x   & = \int_\Omega  \bbC \etaum {k} : (\etaum k - \etau{k-1}) +   \bbC \etaum {k} : (\ptaum k {-} \ptau{k-1}) \dd x 
\\  & \geq \int_\Omega \left(  \tfrac12 \bbC  \etaum {k} {:}  \etaum {k}  - \tfrac12 \bbC \etau{k-1}{:} \etau{k-1} +\bbC \etaum {k} : (\ptaum k {-} \ptau{k-1}) \right) \dd x 
 \,, \end{aligned}
\\
& 
\label{est-mimick-4}
\begin{aligned}
&
\int_\Omega |\etaum k|^{\gamma-2} \etaum k   :   \sig{\utaum k - \utau{k-1}}  \dd x 
\\ & =  \int_\Omega |\etaum k|^{\gamma-2} \etaum k   :   (\etaum k {-} \etau{k-1})  \dd x + \int_\Omega   |\etaum k|^{\gamma-2} \etaum k   :   (\ptaum k {-} \ptau{k-1})  \dd x  
\\ & \quad \geq\int_\Omega \left(  \tfrac1{\gamma} |\etaum k|^{\gamma} {-} \tfrac1{\gamma} |\etau {k-1}|^\gamma {+}  |\etaum k|^{\gamma-2} \etaum k   :   (\ptaum k{ -} \ptau{k-1})  \right) \dd x \,. \end{aligned}
\end{align}
\end{subequations}
Observe that  \eqref{est-mimick-2}--\eqref{est-mimick-4} mimic  the calculations on the time-continuous level leading to \eqref{total-enbal} and in fact rely on the kinematic admissibility condition. The terms on the  right-hand side of  \eqref{est-mimick-2} cancel with the fourth term on the r.h.s.\ of \eqref{heat:approx-discr}, multiplied by $\tau$, and with the analogous term deriving from \eqref{plasr:approx-discr}, tested by $\ptaum{k} - \ptau{k-1}$. In the same way, the last terms on the r.h.s.\ of   \eqref{est-mimick-3} and 
 \eqref{est-mimick-4} cancel with the ones coming from \eqref{plasr:approx-discr}. In fact, it can be easily checked that, with the exception of $\tau \gtau k$, all the terms on the r.h.s.\ of \eqref{heat:approx-discr} cancel out: for instance, $\tau \int_\Omega \mathrm{R} (\tetau{k-1}, \ptaum {k} - \ptau{k-1}) \dd x $  cancels with the term $ \int_\Omega \zetaum k {:} ( \ptaum k {-} \ptau {k-1}) \dd x $ in view of \eqref{characterization-subdiff}. In this way, we conclude \eqref{discr-total-ineq}. 
 \par
 In order to derive estimates \eqref{est-M-indep1}--\eqref{est-M-indep2}, we observe that the first four  terms on the right-hand side of \eqref{discr-total-ineq} are bounded,  depending on the quantities
 $ \| \utau{k-1}\|_{L^2(\Omega;\R^d)} $, $\calE_\tau (\tetau {k-1},\etau {k-1},\ptau{k-1}), $ 
$\|\gtau k\|_{L^1(\Omega)}$, $ \|   \htau{k} \|_{L^2(\partial\Omega)}$, 
 whereas the remaining ones  can be controlled by the ones on  the left-hand  side. In fact,
 we have
\[
\begin{aligned}
 &  
 \begin{aligned}
\left|   \tau \pairing{}{H_\Dir^{1}(\Omega;\R^d)}{\Ltau k}{\frac{\utaum k - \utau{k-1}}{\tau} {-} \Dtau kw } \right|   & \stackrel{(1)}{\leq} \delta \| \utaum k {-}  \wtau k \|_{H^1(\Omega;\R^d)}^2 + \delta  \| \utau { k-1}{-} \wtau{k-1} \|_{H^1(\Omega;\R^d)}^2  + C_\delta \| \Ltau k \|_{H^1(\Omega;\R^d)^*}^2  \\ & \stackrel{(2)}{\leq}   \delta C_K   \| \sig{\utaum k}\|_{L^2(\Omega;\mt_\sym^{d\times d} )}^2 +  C   \\ & \stackrel{(3)}{\leq} 2 \delta C_K^2   \| {\etaum k}\|_{L^2(\Omega;\mt_\sym^{d\times d} )}^2 +  2 \delta C_K   \| {\ptaum k}\|_{L^2(\Omega;\mt_\dev^{d\times d} )}^2+C, \end{aligned}
\\
&
\begin{aligned}
 & \left| \tau \int_\Omega  \simtau {k} : \sig{\Dtau kw} \dd x \right|    \\  & \stackrel{(4)}{\leq}  \frac{\delta}\tau   \|\etaum k - \etau{k-1}\|_{L^2(\Omega;\mt_\sym^{d\times d})}^2 + \delta \tau \|\etaum k  \|_{L^2(\Omega;\mt_\sym^{d\times d})}^2 + C_\delta  \|  \sig{\Dtau kw} \|_{L^2(\Omega;\mt_\sym^{d\times d})}^2   \\ & \quad +   C  \|  \sig{\Dtau kw} \|_{L^\infty(\Omega;\mt_\sym^{d\times d})} \left(  \int_\Omega | \tetaum k|  + |\etaum{k}|^{\gamma-1} \dd x \right) \end{aligned}
\\
 &  \begin{aligned} \left| \rho \int_\Omega \left(  \frac{\utaum k - \utau{k-1}}{\tau}  {-} \Dtau {k-1} u \right)   \Dtau kw \dd x  \right|  & \leq \frac{\rho}4  \int_\Omega \left|\frac{\utaum k - \utau{k-1}}{\tau} \right|^2 \dd x  + \frac{\rho}4\| \Dtau {k-1} u\|_{L^2(\Omega;\R^d)}^2  \\ & \qquad + \rho  \| \Dtau {k} w\|_{L^2(\Omega;\R^d)}^2\,, \end{aligned}
\end{aligned}
\]
where $\delta>0$ in (1) and in the other estimates is an arbitrary positive constant, to be specified later, while (2) ensues from Korn's inequality \eqref{Korn} and from the bounds on the quantities $\| \Ltau k \|_{H_\Dir^{1}(\Omega;\R^d)^*}$, 
$\| \wtau{k}\|_{H^{1}(\Omega;\R^d)}$, $\| \wtau{k-1}\|_{H^{1}(\Omega;\R^d)}$,  $ \| \utau{k-1}\|_{H^1(\Omega;\R^d)} $ , and (3) from the kinematic admissibility condition.   For (4) we have used that 
$\simtau{k}:= \bbD    \Dtau k e  + \bbC \etau k + \tau |\etau{k}|^{\gamma-2} \etau k- \calT_M(\tetau k) \bbB\,,$ as well as the fact that $\|\calT_M(\tetaum k )\|_{L^1(\Omega)} \leq \| \tetaum k \|_{L^1(\Omega)}$. It is now immediate  to check that the terms on the right-hand sides of the above estimates are either bounded, due to our assumptions,  or can be absorbed into the left-hand side of \eqref{discr-total-ineq},
suitably tuning the positive constant $\delta$.
 All in all, we conclude that 
\[
 \int_\Omega \left|\frac{\utaum k - \utau{k-1}}{\tau} \right|^2 \dd x +  \calE_\tau (\tetaum k,\etaum k,\ptaum k)\leq C
\]
for a constant independent of $M$.  Estimates \eqref{est-M-indep1} and \eqref{est-M-indep2} then ensue, also taking into account Korn's inequality. 
\par
Estimate \eqref{est-M-indep3} is proved in two steps, by  testing \eqref{heat:approx-discr} first by $\calT_M(\tetaum k)$, and secondly by $\tetaum k$. We refer to the proof of \cite[Lemma 4.4]{Rocca-Rossi} for all the calculations.
\par Estimate \eqref{est-M-indep4}  follows from the fact that $\zetaum k \in \partial_{\dot p} \mathrm{R}(\tetau {k-1}, (\ptaum{k} {-}\ptau{k-1})/\tau)$ and from \eqref{bounded-subdiff}. 
\end{proof}  
\textbf{Step $3$: limit passage in the approximate discrete system.}  With the following result we conclude the proof of Proposition \ref{prop:exist-discr}. From now on, we suppose that  $M\in \N\setminus\{0\}$. 
\begin{lemma}
\label{l:3.6}
 Let $ k \in \{1,\ldots, K_\tau\} $  and $\tau\in (0,\bar\tau)$  be fixed. Under  the growth condition \eqref{hyp-K}, there exist a (not relabeled) subsequence
  of   $(\tetaum k, \utaum k, \etaum k,\ptaum k)_{M}$  and of $(\zetaum k)_M$, and a quadruple $(\tetau k, \utau k, \etau k, \ptau k) \in H^1(\Omega)  \times W_\Dir^{1,\gamma}(\Omega;\R^d) \times L^\gamma(\Omega;\mt_\sym^{d\times d})  \times L^\gamma(\Omega;\mt_\dev^{d\times d})  $ and  $\zetau k \in   L^\infty(\Omega;\mt_\dev^{d\times d}) $,   such that the following convergences hold as $M\to\infty$
 \begin{subequations}
 \label{conves-as-M}
 \begin{align}
 \label{conves-as-M-teta}
 &  \tetaum k \weakto \tetau k && \text{in } H^1(\Omega), 
 \\
  \label{conves-as-M-u}
 &  \utaum k \to \utau  k  &&  \text{in } W_\Dir^{1,\gamma}(\Omega;\R^d),
 \\
  \label{conves-as-M-e}
 &  \etaum k \to \etau  k  &&  \text{in } L^{\gamma}(\Omega;\mt_\sym^{d\times d}),
 \\
  \label{conves-as-M-p}
 &  \ptaum k \to \ptau  k &&  \text{in } L^{\gamma}(\Omega;\mt_\dev^{d\times d}),
  \\
  \label{conves-as-M-zeta}
 &  \zetaum k \weaksto \zetau k  &&  \text{in } L^{\infty}(\Omega;\mt_\dev^{d\times d}),
 \end{align}
\end{subequations}
and the quintuple $(\tetau k, \utau k, \etau k, \ptau k,\zetau k)$ fulfill system \eqref{syst:discr}.
\end{lemma}
\begin{proof}
It follows from estimates \eqref{estimates-M-indep}  that convergences \eqref{conves-as-M-teta}, \eqref{conves-as-M-zeta}, and the weak versions of \eqref{conves-as-M-u}--\eqref{conves-as-M-p} hold  as $M\to\infty$, along a suitable subsequence. Moreover, there exist $\varepsilon_\tau^k  \in  L^{\gamma/(\gamma{-}1)}(\Omega;\mt_\sym^{d\times d})$ and $\pi_\tau^k \in 
 L^{\gamma/(\gamma{-}1)}(\Omega;\mt_\dev^{d\times d})$ such that 
 \[
 |\etaum k |^{\gamma-2} \etaum  k  \weakto \varepsilon_\tau^k \quad \text{in } L^{\gamma/(\gamma{-}1)}(\Omega;\mt_\sym^{d\times d}), \qquad \qquad  |\ptaum k |^{\gamma-2} \ptaum  k  \weakto \pi_\tau^k \quad \text{in } L^{\gamma/(\gamma{-}1)}(\Omega;\mt_\dev^{d\times d})\,.
 \]
 Furthermore, from \eqref{conves-as-M-teta} one deduces that $\tetaum k \to \tetau k $ strongly in $L^{3\mu+6 -\rho}(\Omega)$ for all $\rho \in (0, 3\mu+5]$. Hence, it is not difficult to conclude that  
 \begin{equation}
 \label{strong-conv-trunc}
 \calT_M(\tetaum k) \to \tetau k \quad \text{ in $L^{3\mu+6 -\rho}(\Omega) $ for all $\rho \in (0, 3\mu+5] $.} 
 \end{equation}
 \par
   With these convergences at hand, it is possible to pass to the limit in \eqref{mom:approx-discr}--\eqref{plasr:approx-discr} and prove that the functions $(\tetau k, \utau k, \etau k, \ptau k, \zetau k, \varepsilon_\tau^k, \pi_\tau^k)$ fulfill 
 \begin{equation} \label{not-still-complete}
\begin{aligned}  & 
 \rho \Ddtau k u  -\mathrm{div}_{\Dir}(  \bar{\sigma}_\tau^k ) =  \Ltau k \qquad \text{in } H_\Dir^1(\Omega)^*,
 \\& 
\zetau k  +  \Dtau kp + \pi_\tau^k   =  (  \bar{\sigma}_\tau^k) _\dev \qquad  \aein \Omega, \end{aligned}
 \end{equation}
 with
  $  \bar{\sigma}_\tau^k  = \bbD \Dtau k e + \bbC \etau k +  \varepsilon_\tau^k - \tetau k \bbB $.
   In  order to conclude the discrete momentum equation and plastic flow rule, it thus remains to show that 
 \begin{equation}
 \label{identifications-to-show}
 \varepsilon_\tau^k  = |\etau k |^{\gamma-2} \etau k, \qquad    \pi_\tau^k  = |\ptau k |^{\gamma-2} \ptau k,  \qquad  \zetau k \in \partial_{\dot p} \mathrm{R}(\tetau {k-1}, \ptau k - \ptau{k-1}) \quad  \aein \Omega. 
 \end{equation} 
 With this aim,  on the one hand we observe that 
 \begin{equation}
 \label{limsup_M}
 \begin{aligned}  & 
 \limsup_{M\to \infty} \left( \int_\Omega \zetaum k {:} \ptaum {k}  \dd x +\tau \int_\Omega |\ptaum k |^\gamma \dd x + \tau  \int_\Omega |\etaum k |^\gamma \dd x \right) \\ &  \stackrel{(1)}{\leq}   \limsup_{M\to \infty}  \left ( - \int_\Omega \frac{\ptaum k - \ptau{k-1}}\tau : \ptaum{k} \dd x + \int_\Omega \simdevtau  k : \ptaum{k} \dd x  +\tau   \int_\Omega |\etaum k |^\gamma \dd x   \right)  \\ &  \stackrel{(2)}{\leq}  - \int_\Omega \frac{\ptau k - \ptau{k-1}}\tau : \ptau{k} \dd x + \limsup_{M\to \infty}\int_\Omega  \dddn{\simtau  k : \sig{\utaum k}}{$=\simtau  k : \sig{\utaum k{-}\wtau k}  +\simtau k{:} \sig{\wtau k}$}  -   \simtau k  : \etaum k +  \tau |\etaum k |^\gamma   \dd x     \\ & 
 \begin{aligned}
  \stackrel{(3)}{= }  - \int_\Omega \frac{\ptau k - \ptau{k-1}}\tau : \ptau{k} \dd x + \limsup_{M\to \infty} \Big( &  - \int_\Omega  \rho \frac{\utaum k - 2\utau{k-1} + \utau {k-2}}{\tau^2}  (\utaum{k}{-}\wtau k) \dd x \\ & \quad  + \pairing{}{H_\Dir^{1}(\Omega;\R^d)}{\Ltau k}{\utaum{k} {-}\wtau k}  +\int_\Omega \simtau k{:} \sig{\wtau k} \dd x  \\ &  \quad -\int_\Omega \left(  \bbD \frac{\etaum k - \etau{k-1}}{\tau}  {+} \bbC \etaum{k} { -} \calT_M (\tetaum k) \bbB \right){:} \etaum k \dd x  \Big)  \end{aligned}   \\ & 
  \begin{aligned}
   \stackrel{(4)}{\leq} 
  &  - \int_\Omega \frac{\ptau k - \ptau{k-1}}\tau : \ptau{k} \dd x  - \rho \int_\Omega \Ddtau k u  (\utau k{-} \wtau k)  \dd x + \pairing{}{H_\Dir^{1}(\Omega;\R^d)}{\Ltau k}{\utau{k} {-}\wtau k} +\int_\Omega \sitau k {:} \sig{\wtau k} \dd x  \\ &   -\int_\Omega \left(  \bbD \frac{\etau k - \etau{k-1}}{\tau}  {+ }\bbC \etau{k}  {-} \tetau k \bbB \right){:} \etau k \dd x 
  \end{aligned}    \\ &  \stackrel{(5)}{=}  \int_\Omega \zetau k {:} \ptau {k}  \dd x + \int_\Omega |\pi_\tau^k|^\gamma \dd x +   \int_\Omega |\varepsilon_\tau^k| ^\gamma \dd x.
  \end{aligned}
 \end{equation}
In \eqref{limsup_M}, 
 (1) follows from testing \eqref{plasr:approx-discr} by $\ptaum k$, (2) from  the weak convergence  $\ptaum k \to \ptau k $ in $L^2(\Omega;\mt_\dev^{d\times d})$ and  the discrete admissibility condition, (3) from rewriting the term $\int_\Omega \simtau  k : \sig{\utaum k{-}\wtau k}  \dd x $ in terms of  \eqref{mom:approx-discr}   tested
 by $\utaum k{-} \wtau k$, and from using the explicit expression of  $\simtau k$ (which leads to the cancelation of the term $\int_\Omega \tau | \etaum k|^\gamma \dd x$), (4) from the previously proved convergences via lower semicontinuity arguments,  and (5) from repeating the above calculations
 in the frame of system \eqref{not-still-complete}, fulfilled by the 
limiting seventuple 
 $(\tetau k, \utau k, \etau k, \ptau k, \zetau k, \varepsilon_\tau^k, \pi_\tau^k)$.
   On the other hand, we have that 
  \begin{equation}
 \label{liminf_M}
 \begin{aligned} 
 \liminf_{M\to \infty}  \int_\Omega \zetaum k {:} \ptaum {k}  \dd x \geq  \int_\Omega \zetau k {:} \ptau {k}  \dd x, \qquad   &  \liminf_{M\to \infty}  \int_\Omega |\ptaum k |^\gamma \dd x \geq \int_\Omega |\pi_\tau^k|^\gamma \dd x,  \\ & \liminf_{M\to \infty}  \int_\Omega |\etaum k |^\gamma \dd x \geq  \int_\Omega |\varepsilon_\tau^k| ^\gamma \dd x\,, \end{aligned} 
 \end{equation}
 where the second and the third inequalities follow from the weak convergence of $(\ptaum k)_M$ and $(\etaum k)_M$ to $\pi_\tau^k$ and $\varepsilon_\tau^k$, whereas the   first inequality ensues from 
 \[  
 \begin{aligned} 
   \liminf_{M\to \infty}  \left(  \int_\Omega \zetaum k {:} \ptaum {k}  {-} \zetau k {:} \ptau {k} \right)   \dd x &   \geq   \liminf_{M\to \infty}    \int_\Omega \zetaum k {:}  ( \ptaum {k}  {-} \ptau k) \dd x +    \liminf_{M\to \infty}  \int_\Omega (\zetaum k {-} \zetau k) { :} \ptau {k}   \dd x  \\ &  \stackrel{(1)}{\geq}   \liminf_{M\to \infty}\int_\Omega \left( \mathrm{R}(\tetau{k-1}, \ptaum{k} - \ptau{k-1}) -   \mathrm{R}(\tetau{k-1}, \ptau{k} - \ptau{k-1}) \right)  \dd x  \stackrel{(2)}{\geq} 0 \end{aligned} 
 \]
 with (1) due to the fact that $\zetaum k \in \partial_{\dot p} \mathrm{R}(\tetau{k-1}, \ptaum{k} - \ptau{k-1}) $ and from $\zetaum k \weaksto \zetau k$ in $L^\infty (\Omega;\mt_\dev^{d\times d})$ as $M\to\infty$, and (2) following from the lower semicontinuity  w.r.t.\ to the weak $L^2(\Omega;\mt_\dev^{d\times d})$-convergence of the integral  functional $ p \mapsto \int_\Omega \mathrm{R} (\tetau{k-1}, p - \ptau{k-1}) \dd x $. 
 Combining \eqref{limsup_M} and \eqref{liminf_M} we obtain that  
\[     \left\{ 
\begin{array}{lll}
\lim_{M\to\infty} \int_\Omega \zetaum k {:} \ptaum {k}  \dd x = \int_\Omega  \zetau k {:} \ptau {k} \dd x  &\stackrel{(1)}{ \Rightarrow} & \zetau k \in  \partial_{\dot p} \mathrm{R}(\tetau{k-1}, \ptau{k} - \ptau{k-1}) \quad  \aein \Omega, 
\\  
 \lim_{M\to \infty}  \int_\Omega |\ptaum k |^\gamma \dd x  =  \int_\Omega |\pi_\tau^k|^\gamma \dd x  & \Rightarrow & \ptaum k \to \pi_\tau^k \qquad \text{in } L^\gamma(\Omega;\mt_\sym^{d\times d}),
 \\
  \lim_{M\to \infty}  \int_\Omega |\etaum k |^\gamma \dd x  =  \int_\Omega |\varepsilon_\tau^k|^\gamma \dd x & \Rightarrow & \etaum k \to \varepsilon_\tau^k \qquad \text{in } L^\gamma(\Omega;\mt_\dev^{d\times d})\,,
  \end{array} \right.
\] 
with (1) due to Minty's trick, cf.\ also
\cite[Lemma 1.3, p. 42]{barbu76}.
  Hence we conclude  convergences \eqref{conves-as-M-e}--\eqref{conves-as-M-p} (and \eqref{conves-as-M-u}, via the kinematic admissibility $\sig{\utaum k } = \etaum k + \ptaum k$ and Korn's inequality),  as well as \eqref{identifications-to-show}.  Therefore  $(\tetau k, \utau k, \etau k, \ptau k,\zetau k) $ fulfill 
the discrete momentum balance  \eqref{discrete-momentum} and flow rule \eqref{discrete-plastic}.
\par 
Exploiting convergences \eqref{conves-as-M}  we pass to the limit  as $M\to \infty$ on the right-hand side of \eqref{heat:approx-discr}. In order to take the limit of the elliptic operator on the left-hand side,  we repeat the argument from the proof of \cite[Lemma 4.4]{Rocca-Rossi}. Namely, we observe that,  due to convergence \eqref{strong-conv-trunc}, 
$\condu_M(\tetaum k) = \condu(\calT_M(\tetaum k)) \to \condu(\tetau k)$ in $L^q(\Omega)$ for all $1\leq q<3+\tfrac6{\mu}$, and combine this with the fact that $\nabla \tetaum k \weakto \nabla \teta $ in $L^2(\Omega)$, and with the fact that, by comparison in  \eqref{heat:approx-discr}, $(\mathcal{A}_M^k(\tetaum k))_M$ is bounded in $H^1(\Omega)^*$. All in all, we conclude that $\mathcal{A}_M^k(\tetaum k) \weakto \calA^k(\tetau k)$ in  $H^1(\Omega)^*$ as $M\to\infty$, yielding the discrete heat equation \eqref{discrete-heat}. 
\end{proof}

\section{Proof of Theorems \ref{mainth:1} and \ref{mainth:2}}
\label{s:4}
In the statements of all of the results of this section, leading to the proofs of Thms.\ \ref{mainth:1} \& \ref{mainth:2}, we will always tacitly assume the conditions on the problem data from  Section \ref{ss:2.1}.
 \par
We start by fixing some notation for the approximate solutions. 
\begin{notation}[Interpolants]
\upshape
For a given Banach space $B$ and a
$K_\tau$-tuple $( \mathfrak{h}_\tau^k )_{k=0}^{K_\tau}
\subset B$, we  introduce
 the left-continuous and  right-continuous piecewise constant, and the piecewise linear interpolants
 of the values  $\{ \mathfrak{h}_\tau^k
\}_{k=0}^{K_\tau}$, i.e.\
\[
\left.
\begin{array}{llll}
& \pwc  {\mathfrak{h}}{\tau}: (0,T] \to B  & \text{defined by}  &
\pwc {\mathfrak{h}}{\tau}(t): = \mathfrak{h}_\tau^k,
\\
& \upwc  {\mathfrak{h}}{\tau}: (0,T] \to B  & \text{defined by}  &
\upwc {\mathfrak{h}}{\tau}(t) := \mathfrak{h}_\tau^{k-1},
\\
 &
\pwl  {\mathfrak{h}}{\tau}: (0,T] \to B  & \text{defined by} &
 \pwl {\mathfrak{h}}{\tau}(t):
=\frac{t-t_\tau^{k-1}}{\tau} \mathfrak{h}_\tau^k +
\frac{t_\tau^k-t}{\tau}\mathfrak{h}_\tau^{k-1}
\end{array}
\right\}
 \qquad \text{for $t \in
(t_\tau^{k-1}, t_\tau^k]$,}
\]
setting $\pwc  {\mathfrak{h}}{\tau}(0)= \upwc  {\mathfrak{h}}{\tau}(0)= \pwl  {\mathfrak{h}}{\tau}(0): = \mathfrak{h}_\tau^0$. 
We also introduce the piecewise linear interpolant    of the values
$\{ \Dtau  k {\mathfrak{h}}  = \tfrac{\mathfrak{h}_\tau^k - \mathfrak{h}_{\tau}^{k-1})}\tau\}_{k=1}^{K_\tau}$ (which are  the
values taken by  the piecewise constant function $\pwl
{\dot{\mathfrak{h}}}{\tau}$), viz.
\[
\pwwll  {\mathfrak{h}}{\tau}: (0,T) \to B  \ \text{ defined by } \ \pwwll
{\mathfrak{h}}{\tau}(t) :=\frac{(t-t_\tau^{k-1})}{\tau}
\Dtau k  {\mathfrak{h}}  +
\frac{(t_\tau^k-t)}{\tau} \Dtau {k-1} {\mathfrak{h}}   \qquad \text{for $t \in
(t_\tau^{k-1}, t_\tau^k]$.}
\]
Note that $ {\partial_t \pwwll  {{\mathfrak{h}}}{\tau}}(t) =  \Ddtau k {\mathfrak{h}}  $ for $t \in
(t_\tau^{k-1}, t_\tau^k]$.

Furthermore, we   denote by  $\pwc{\mathsf{t}}{\tau}$ and by
$\upwc{\mathsf{t}}{\tau}$ the left-continuous and right-continuous
piecewise constant interpolants associated with the partition, i.e.
 $\pwc{\mathsf{t}}{\tau}(t) := t_\tau^k$ if $t_\tau^{k-1}<t \leq t_\tau^k $
and $\upwc{\mathsf{t}}{\tau}(t):= t_\tau^{k-1}$ if $t_\tau^{k-1}
\leq t < t_\tau^k $. Clearly, for every $t \in (0,T)$ we have
$\pwc{\mathsf{t}}{\tau}(t) \downarrow t$ and
$\upwc{\mathsf{t}}{\tau}(t) \uparrow t$ as $\tau\to 0$.
\end{notation}

In view of \eqref{heat-source}, \eqref{dato-h}, and \eqref{data-displ} it is easy to check that the piecewise constant
interpolants $(\pwc  H{\tau}
)_{\tau}$, $(\pwc  h{\tau} )_{\tau}$, and $(\pwc \calL{\tau})_\tau$   of the values 
$\gtau{k}$,  $\htau{k}$, and $\Ltau k $, cf.\  \eqref{local-means}, fulfill as $\tau \down
0$
\begin{subequations}
\label{convs-interp-data}
\begin{align}
\label{converg-interp-g}  & \pwc H{\tau}  \to H
  \text{ in $L^1(0,T;L^1(\Omega))\cap L^2(0,T;H^1(\Omega)^*)$.}
  \\
  \label{converg-interp-h}  & \pwc h{\tau}  \to h
  \text{ in $L^1(0,T;L^2(\partial\Omega))$,}
  \\
  \label{converg-interp-L}  & \pwc \calL{\tau} \to \calL \text{ in $ L^2(0,T; H_\Dir^1(\Omega;\R^d)^*).$} 
\end{align}
Furthermore, it follows from \eqref{Dirichlet-loading} and \eqref{discr-w-tau} that 
\begin{equation}
\label{converg-interp-w}
\begin{gathered}
\pwc w\tau \to w \quad \text{in  $L^1(0,T; W^{1,\infty} (\Omega;\R^d))$}, \qquad \pwl w\tau \to w \quad \text{in $W^{1,p} (0,T; H^1(\Omega;\R^d))$ for all } 1 \leq p<\infty, \\ \pwwll w\tau \to w \quad \text{in  $W^{1,1}(0,T; H^1(\Omega;\R^d)) \cap H^1(0,T; L^2(\Omega;\R^d))$},
\\
\sup_{\tau>0}   \tau^{\alpha_w} \| \sig{\pwl{\dot w}{\tau}}\|_{L^\gamma(\Omega;\mt_\sym^{d\times d})} \leq C <\infty \text{ with } \alpha_w \in (0,\tfrac1\gamma)\,.
 \end{gathered}
\end{equation}
\end{subequations}
\par We  now reformulate the discrete system \eqref{syst:discr}  in terms of the approximate solutions constructed interpolating the   discrete solutions $(\tetau k,\utau k, \etau k,\ptau k)_{k=1}^{K_\tau}$. Therefore, we have
\begin{subequations}
\label{syst-interp}
\begin{align}
&
\label{eq-teta-interp}
\begin{aligned} & 
\partial_t \pwl \teta{\tau}(t) 
+   \mathcal{A}^{\frac{\bar{\mathsf{t}}_\tau(t)}{\tau}}( \pwc{\teta}\tau (t) )    \\ &\quad = \pwc H{\tau}(t)+ \mathrm{R}(\upwc \teta\tau(t), \pwl{\dot p}\tau (t)) + |\pwl{\dot p}\tau (t)|^2 + \bbD \pwl{\dot e}\tau(t) {:}   \pwl{\dot e}\tau(t)  - \pwc \teta\tau(t) \bbB : \pwl{\dot e}\tau(t), 
 \quad \text{in $H^1(\Omega)^*$,}
\end{aligned}
\\
&
 \label{eq-u-interp}
 \begin{aligned}
\rho\int_\Omega  \partial_t\pwwll {\uu}{\tau}(t)  v \dd x + \int_\Omega \pwc\sigma \tau(t)  {:} \sig{v} \dd x  = \pairing{}{H_\Dir^{1}(\Omega;\R^d)}{\pwc\calL \tau(t)}{v}
\quad  \text{for all } v \in W_\Dir^{1,\gamma}(\Omega;\R^d),
\end{aligned}\\
& \label{eq-p-interp}
\begin{aligned}
\pwc\zeta\tau(t)+ \pwl{\dot p}\tau (t) + \tau | \pwc p\tau(t)|^{\gamma-2} \pwc p\tau(t)  =  (\pwc\sigma \tau(t)  )_\dev \quad \aein\, \Omega
  \end{aligned}
\end{align}  for almost all $t\in (0,T)$, 
with  $\pwc\zeta\tau \in \partial_{\dot p} \mathrm{R} (\upwc \teta\tau, \pwl{\dot p}\tau )  $  a.e.\ in $Q$, and where we have used the notation
\begin{equation}
\label{sigma-interp}
\pwc\sigma\tau: = \bbD \pwl {\dot e}\tau + \bbC \pwc e\tau + \tau |\pwc e\tau|^{\gamma-2} \pwc e\tau - \pwc\teta\tau \bbB. 
\end{equation}
\end{subequations}
\par
We now show that the approximate solutions fulfill  the approximate versions of the
entropy inequality
\eqref{entropy-ineq}, 
 of the
 total  energy  inequality \eqref{total-enineq},  and of the mechanical energy (in)equality
 \eqref{mech-enbal}.
   These discrete inequalities will have a pivotal role in the derivation of a priori estimates on the  approximate solutions. Moreover, we will take their limit  in order to obtain the entropy and total energy inequalities prescribed by the notion of \emph{entropic solution}, cf.\ Definition \ref{def:entropic-sols}.  
 \par
 For stating the discrete entropy inequality  \eqref{entropy-ineq-discr} below, we need to introduce  \emph{discrete} test functions.  For technical reasons, we will need to pass to the limit with test functions enjoying a slightly stronger  time regularity than that required by Def.\ \ref{def:entropic-sols}. Namely,  we fix a positive   test function $\varphi $, with  $\varphi  \in \rmC^0 ([0,T]; W^{1,\infty}(\Omega)) \cap H^1 (0,T; L^{6/5}(\Omega)) $. We  set  
  \begin{equation}
 \label{discrete-tests-phi}
   \varphi_\tau^k:= 
\frac{1}{\tau}\int_{t_\tau^{k-1}}^{t_\tau^k}  \varphi(s)\dd s \qquad \text{for } k=1, \ldots, K_\tau,
 \end{equation}
and consider the piecewise constant and  linear interpolants
$\pwc \varphi\tau$ and $\pwl \varphi\tau$ of the values
$(\varphi_\tau^k)_{k=1}^{K_\tau}$.
It can be shown that  the following convergences hold as $\tau \to 0$
\begin{equation}
\label{convergences-test-interpolants}
\pwc \varphi\tau  \to \varphi \quad
\text{ in } L^\infty (0,T; W^{1,\infty}(\Omega)) \text{ and }  \quad \partial_t \pwl \varphi\tau  \to \partial_t \varphi \quad \text{ in } L^2 (0,T; L^{6/5}(\Omega)).
\end{equation} 
Observe that the first convergence property easily follows from the fact that the map $\varphi: [0,T]\to  W^{1,\infty}(\Omega)$ is uniformly continuous. 
  \begin{lemma}[Discrete entropy, mechanical, and  total energy inequalities]
\label{lemma:discr-enid}
The interpolants of the   discrete solutions
 $(\tetau{k},  \utau{k}, \etau k,  \ptau{k} )_{k=1}^{K_\tau}$ to Problem \ref{prob:discrete}
 fulfill
 \begin{itemize}
 \item[-]
  the \emph{discrete} entropy inequality 
 \begin{equation}\label{entropy-ineq-discr}
\begin{aligned}
& \int_{\pwc{\mathsf{t}}{\tau}(s)}^{\pwc{\mathsf{t}}{\tau}(t)}
\int_\Omega \log(\upwc\teta\tau (r)) 
\pwl {\dot\varphi}\tau (r) \dd x \dd r - \int_{\pwc{\mathsf{t}}{\tau}(s)}^{\pwc{\mathsf{t}}{\tau}(t)}  \int_\Omega  \condu(\pwc\teta\tau (r)) \nabla \log(\pwc\teta\tau(r))  \nabla \pwc\varphi\tau(r)  \dd x \dd r \\ & 
\leq \int_\Omega \log(\pwc\teta\tau(t))
\pwc\varphi\tau(t) \dd x - \int_\Omega
\log(\pwc\teta\tau(s)) 
\pwc\varphi\tau(s) \dd x
 -  \int_{\pwc{\mathsf{t}}{\tau}(s)}^{\pwc{\mathsf{t}}{\tau}(t)}  \int_\Omega \condu(\pwc \teta\tau(r)) \frac{\pwc\varphi\tau(r)}{\pwc \teta\tau(r)}
\nabla \log(\pwc \teta\tau(r)) \nabla \pwc \teta\tau (r) \dd x \dd r\\ & \quad 
 -
  \int_{\pwc{\mathsf{t}}{\tau}(s)}^{\pwc{\mathsf{t}}{\tau}(t)}   \int_\Omega \left( \pwc H{\tau} (r)  + \mathrm{R}(\upwc \teta\tau(r), \pwl{\dot p}\tau(r)) + |\pwl{\dot p}\tau(r)|^2 + \bbD \pwl{\dot e}\tau (r) {:} \pwl{\dot e}\tau(r) - \pwc\teta\tau(r) \bbB {:}    \pwl{\dot e}\tau(r)   \right)
 \frac{\pwc\varphi\tau(r)}{\pwc\teta\tau(r)} \dd x \dd r   \\ & \quad -  \int_{\pwc{\mathsf{t}}{\tau}(s)}^{\pwc{\mathsf{t}}{\tau}(t)}   \int_{\partial\Omega} \pwc h{\tau} (r)   \frac{\pwc\varphi\tau(r)}{\pwc\teta\tau(r)} \dd S \dd r
\end{aligned}
\end{equation}
for all $0 \leq s \leq t \leq T$ and
for all $\varphi \in \mathrm{C}^0 ([0,T]; W^{1,\infty}(\Omega)) \cap H^1 (0,T; L^{6/5}(\Omega)) $ with
$\varphi \geq 0$;
\item[-] the \emph{discrete} total energy inequality for all $ 0 \leq s \leq t
\leq  T$, viz.
\begin{equation}
\label{total-enid-discr}
\begin{aligned} & 
\frac{\rho}2 \int_\Omega | \pwl{\dot u}\tau (t)|^2 \dd x +  \calE_\tau(\pwc\teta\tau(t),\pwc e\tau(t),\pwc p\tau(t))  \\ & \leq \frac{\rho}2 \int_\Omega | \pwl{\dot u}\tau (s)|^2 \dd x 
 +  \calE_\tau(\pwc\teta\tau(s),\pwc e \tau(s),\pwc p\tau(s)) +   \int_{\pwc{\mathsf{t}}{\tau}(s)}^{\pwc{\mathsf{t}}{\tau}(t)} \pairing{}{H_\Dir^1(\Omega;\R^d)}{\pwc \calL \tau (r) }{\pwl{\dot u}\tau (r) {-} \pwl{ \dot w}\tau (r)} \dd r
 \\ & \quad
 + \int_{\pwc{\mathsf{t}}{\tau}(s)}^{\pwc{\mathsf{t}}{\tau}(t)} \left( 
\int_\Omega  \pwc H\tau  \dd x + 
 \int_{\partial\Omega} \pwc h\tau  \dd S \right)  \dd r \\ & \quad 
 +\rho \int_\Omega \pwl{\dot u}\tau(t) \pwl{\dot w }\tau(t) \dd x -  \rho \int_\Omega \pwl{\dot u}\tau(s)   \pwl{\dot w }\tau(s) \dd x 
  -  \rho  \int_{\pwc{\mathsf{t}}{\tau}(s)}^{\pwc{\mathsf{t}}{\tau}(t)} \int_\Omega  \pwl{\dot u}\tau(r-\tau) \partial_t \pwwll{ w}{\tau} (r)  \dd x \dd r \\ & \quad    +   \int_{\pwc{\mathsf{t}}{\tau}(s)}^{\pwc{\mathsf{t}}{\tau}(t)}\int_\Omega \pwc \sigma\tau (r) {:} \sig{\pwl{\dot w}\tau (r) }  \dd x \dd r 
 \end{aligned}
\end{equation}
 with the discrete total energy functional $\calE_\tau$ from 
 \eqref{discr-total-energy}; 
 \item[-] the \emph{discrete} mechanical energy inequality  for all $ 0 \leq s \leq t
\leq  T$, viz.
 \begin{equation}
\label{mech-ineq-discr}
\begin{aligned} & 
\frac{\rho}2 \int_\Omega | \pwl{\dot u}\tau (t)|^2 \dd x 
+  \int_{\pwc{\mathsf{t}}{\tau}(s)}^{\pwc{\mathsf{t}}{\tau}(t)} \int_\Omega \left( \bbD \pwl{\dot e}\tau(r) {:} \pwl{\dot e}\tau(r)  + \mathrm{R}(\upwc \teta\tau(r), \pwl{\dot p}\tau(r)) + |\pwl{\dot p}\tau(r)|^2  \right)  \dd x \dd r 
+ \frac12\int_\Omega \bbC \pwc e\tau(t) {:}  \pwc e\tau(t)  \dd x
\\ 
& \qquad + \frac{\tau}\gamma\int_\Omega \left(  |\pwc e\tau(t)|^\gamma +  |\pwc p\tau(t)|^\gamma \right) \dd x 
\\ &  \leq 
 \frac{\rho}2 \int_\Omega | \pwl{\dot u}\tau (s)|^2 \dd x 
  + \frac12\int_\Omega \bbC \pwc e\tau(s) {:}  \pwc e\tau(s)  \dd x
   + \frac{\tau}\gamma\int_\Omega \left(  |\pwc e\tau(s)|^\gamma +  |\pwc p\tau(s)|^\gamma \right) \dd x 
   \\
   &\quad
+
    \int_{\pwc{\mathsf{t}}{\tau}(s)}^{\pwc{\mathsf{t}}{\tau}(t)} \pairing{}{H_\Dir^1(\Omega;\R^d)}{\pwc \calL \tau (r) }{\pwl{\dot u}\tau (r){-} \pwl{\dot w}\tau (r)  } \dd r    + \int_{\pwc{\mathsf{t}}{\tau}(s)}^{\pwc{\mathsf{t}}{\tau}(t)} \int_\Omega \pwc\teta\tau(r) \bbB {:} \pwl{\dot e}\tau \dd x  \dd r 
   +\rho \int_\Omega \pwl{\dot u}\tau(t) \pwl{\dot w }\tau(t) \dd x
    \\ & \quad
    -  \rho \int_\Omega \pwl{\dot u}\tau(s)   \pwl{\dot w }\tau(s) \dd x 
    - \rho  \int_{\pwc{\mathsf{t}}{\tau}(s)}^{\pwc{\mathsf{t}}{\tau}(t)}  \int_\Omega \pwl{\dot u}\tau(r-\tau)  \partial_t \pwwll{w}\tau (r)  \dd x \dd r     +   \int_{\pwc{\mathsf{t}}{\tau}(s)}^{\pwc{\mathsf{t}}{\tau}(t)}\int_\Omega \pwc \sigma\tau (r) {:} \sig{\pwl{\dot w}\tau (r) }  \dd x \dd r\,.  
 \end{aligned}
\end{equation} 
 \end{itemize}
 \end{lemma}
 Below we will only outline the argument for Lemma \ref{lemma:discr-enid}, referring to the proof of the analogous \cite[Prop.\ 4.8]{Rocca-Rossi} for most of the details. Let us only mention in advance that we will make use of the 
 \emph{discrete by-part integration} formula, holding
for all $K_\tau$-uples  $\{\mathfrak{h}_\tau^k \}_{k=0}^{K_\tau} \subset B,\, \{ v_\tau^k \}_{k=0}^{K_\tau} \subset B^*$ in a given Banach space $B$:
\begin{equation}
\label{discr-by-part} \sum_{k=1}^{K_\tau} \tau
\pairing{}{B}{\vtau{k}}{\Dtau{k}{\mathfrak{h}}} =
\pairing{}{B}{\vtau{K_\tau}}{\btau{K_\tau}}
-\pairing{}{B}{\vtau{0}}{\btau{0}} -\sum_{k=1}^{K_\tau}\tau\pairing{}{B}{\dtau{k}{v} }{ \btau{k-1}}\,,
\end{equation}
as well as of the following inequality, 
satisfied by any \underline{concave}
 (differentiable)
  function $\psi: \mathrm{dom}(\psi) \to \R$:
 \begin{equation}
 \label{inequality-concave-functions}
 \psi(x) - \psi(y) \leq \psi'(y) (x-y) \qquad \text{for all } x,\, y \in \mathrm{dom}(\psi).
 \end{equation}
\begin{proof}[Sketch of the proof of Lemma \ref{lemma:discr-enid}]
The entropy inequality \eqref{entropy-ineq-discr} follows from testing  \eqref{discrete-heat} by $\frac{\varphi_\tau^k}{\tetau{k}}$, for
$k \in \{1,\ldots,K_\tau\}$  fixed, with $\varphi \in \mathrm{C}^0 ([0,T]; W^{1,\infty}(\Omega)) \cap H^1 (0,T; L^{6/5}(\Omega)) $  an arbitrary positive test function.  Observe that  $\frac{\varphi_\tau^k}{\tetau{k}} \in H^1(\Omega)$, as $\tetau k \in H^1(\Omega)$ is bounded away from zero by a strictly positive constant, cf.\ \eqref{discr-strict-pos}.   We then obtain 
\[
\begin{aligned} & 
\int_\Omega \left(   \gtau{k} + \mathrm{R}\left(\tetau{k-1}, \Dtau{k} p  \right) +  \left|  \Dtau{k} p \right|^2+ \bbD   \Dtau{k} e :  \Dtau{k} e -\tetau{k} \bbB : \Dtau k e \right)  \frac{\varphi_\tau^k}{\tetau{k}} \dd x  + \int_{\partial\Omega} \htau k   \frac{\varphi_\tau^k}{\tetau{k}} \dd S
\\ & = 
\int_\Omega \frac{\tetau k - \tetau{k-1}}{\tau} \frac{\varphi_\tau^k}{\tetau{k}}  \dd x + \int_\Omega \condu(\tetau k) \nabla \tetau{k} \nabla \left(   \frac{\varphi_\tau^k}{\tetau{k}}  \right)  \dd x 
\\ & \stackrel{(1)}{\leq}
\int_\Omega  \frac{\log(\tetau k) - \log(\tetau{k-1})}{\tau} \varphi_\tau^k \dd x + \int_\Omega \left( \frac{\condu(\tetau k)}{\tetau k} \nabla \tetau{k} \nabla \varphi_\tau^k - \frac{\condu(\tetau k)}{|\tetau k|^2} |\nabla \tetau{k}|^2 \varphi_\tau^k  \right)  \dd x 
\end{aligned}
\]
where (1) follows from \eqref{inequality-concave-functions}
with $\psi=\log$.
 Then, one sums the above inequality, multiplied by $\tau$, over $k=m,\ldots, j$, for any couple of indices $m,\, j \in \{1, \ldots, K_\tau\}$, and uses the discrete by-part integration formula \eqref{discr-by-part} to deal with the term 
$
\sum_{k=m}^j  \frac{\log(\tetau k) - \log(\tetau{k-1})}{\tau} \varphi_\tau^k\,.
$
This leads to \eqref{entropy-ineq-discr}. 
\par
As for the discrete total energy inequality, with the very same calculations developed in the proof of Lemma \ref{l:aprio-M}, one shows that the solution quadruple $(\tetau k, \utau k,\etau k, \ptau k)$ to system \eqref{syst:discr} fulfills the energy inequality \eqref{discr-total-ineq}. Note  that the two inequalities,  i.e.\ the one for system  \eqref{syst:discr} and inequality \eqref{discr-total-ineq} for the   truncated version \eqref{syst:approx-discr} of \eqref{syst:discr}, in fact coincide since they neither involve the elliptic operator in the discrete heat equation, nor the thermal expansion terms coupling the heat equation and the momentum balance, which are the terms affected by the truncation procedure. Then, \eqref{total-enid-discr} ensues by adding \eqref{discr-total-ineq} over the index $k =m,\ldots, j$, for any couples of indices $m,\, j \in \{1,\ldots, K_\tau\}$.
\par  The mechanical energy inequality \eqref{mech-ineq-discr} is derived by subtracting from  \eqref{total-enid-discr} the discrete heat equation  \eqref{discrete-heat}  multiplied by $\tau$ and integrated over $\Omega$. 
\end{proof}



\subsection{A priori estimates}
\label{ss:3.2}
The following result collects the a priori estimates on the approximate solutions of system \eqref{syst-interp}. Let us mention in advance  that, along the footsteps of \cite{Rocca-Rossi}, we shall derive from the discrete entropy inequality \eqref{entropy-ineq-discr} a \emph{weak version} of the estimate on the total variation of 
 $\log(\pwc \teta \tau)$, cf.\ \eqref{aprio_Varlog} and \eqref{var-notation}
below, which will play a crucial role in the compactness arguments for the approximate temperatures $(\pwc \teta \tau)_\tau$.    
\begin{proposition}
\label{prop:aprio} 
Assume \eqref{hyp-K}. Then, there exists a constant $S>0$ such that for all $\tau>0$ the following estimates hold
\begin{subequations}
\label{aprio}
\begin{align}
& \label{aprioU1}
\|\pwc \uu{\tau}\|_{L^\infty(0,T;H^1(\Omega;\R^d))}
 \leq S,
\\
& \label{aprioU2}
\|\pwl \uu{\tau}\|_{H^1(0,T; H^1(\Omega;\R^d) ) \cap
W^{1,\infty}(0,T; L^2(\Omega;\R^d))}
 \leq S,
\\
& \label{aprioU3} \|\pwwll \uu{\tau}
\|_{L^2(0,T; H^1(\Omega;\R^d) )  \cap L^\infty (0,T;  L^2(\Omega;\R^d)) \cap W^{1,\gamma/(\gamma-1)}(0,T;W^{1,\gamma}(\Omega;\R^d)^*)} \leq S,
\\
& \label{aprioE1}
\|\pwc
e{\tau}\|_{L^\infty(0,T;L^2(\Omega;\mt_\sym^{d\times d}))} \leq S,
\\
& \label{aprioE2}
\|\pwl
e{\tau}\|_{H^1(0,T;L^2(\Omega;\mt_\sym^{d\times d}))} \leq S,
\\
& \label{aprioE3}
\tau^{1/\gamma}\|\pwc
e{\tau}\|_{L^\infty(0,T;L^\gamma(\Omega;\mt_\sym^{d\times d}))} \leq S,
\\
& \label{aprioP1}
\|\pwc
p{\tau}\|_{L^\infty(0,T;L^2(\Omega;\mt_\dev^{d\times d}))} \leq S,
\\
& \label{aprioP2}
\|\pwl
p{\tau}\|_{H^1(0,T;L^2(\Omega;\mt_\dev^{d\times d}))} \leq S,
\\
& \label{aprioP3}
\tau^{1/\gamma}\|\pwc
p{\tau}\|_{L^\infty(0,T;L^\gamma(\Omega;\mt_\dev^{d\times d}))} \leq S,
\\
&
\label{log-added}
\| \log(\pwc \teta{\tau}) \|_{L^2 (0,T; H^1(\Omega))}
 \leq S,
\\
& \label{aprio6-discr}
 \|\pwc
\teta{\tau}\|_{L^2(0,T; H^1(\Omega)) \cap L^{\infty}(0,T;L^{1}(\Omega))} \leq S,
\\
&  \label{est-temp-added?-bis}
  \| (\pwc \teta\tau)^{(\mu+\alpha)/2} \|_{L^2(0,T; H^1(\Omega))}, \,  \| (\pwc \teta\tau)^{(\mu-\alpha)/2} \|_{L^2(0,T; H^1(\Omega))} \leq C \quad \text{for all } \alpha \in [(2{-}\mu)^+, 1)
\\ 
& \label{aprio_Varlog}
\sup_{\varphi \in W^{1,d+\epsilon}(\Omega), \ \| \varphi \|_{W^{1,d+\epsilon}(\Omega)}\leq 1}
\mathrm{Var}(\pairing{}{W^{1,d+\epsilon}(\Omega)}{\log(\pwc\teta\tau)}{\varphi}; [0,T]) \leq S\ \quad \text{for every } \epsilon>0.
\end{align}
Furthermore, if  $\condu$ fulfills \eqref{hyp-K-stronger}, there holds in addition
\begin{align}
\label{aprio7-discr}
&
\sup_{\tau>0} \| \pwl \teta{\tau} \|_{\mathrm{BV} ([0,T]; W^{1,\infty} (\Omega)^*)} \leq S.
\end{align}
\end{subequations}
\end{proposition}
\noindent The starting point in the proof is the discrete total energy inequality \eqref{total-enid-discr}, giving rise to  
the second of \eqref{aprioU2}, the second of \eqref{aprioU3},  \eqref{aprioE1}, \eqref{aprioE3}, 
\eqref{aprioP3}, and the second of \eqref{aprio6-discr}: we will detail the related calculations, in particular showing how the terms 
arising from the external forces $F$ and $g$, and those
involving the Dirichlet loading $w$  can be handled. Let us also refer to the upcoming Remark \ref{rmk:diffic-1-test} for more comments.
\par
 The \emph{dissipative} estimates, i.e.\ the first  of  \eqref{aprioU2} and \eqref{aprioU3}, \eqref{aprioE2}, and \eqref{aprioP2}, then follow from the  discrete mechanical energy inequality \eqref{mech-ineq-discr}. The remaining estimates on the approximate temperature can be performed with the very same arguments as in the proof of \cite[Prop.\ 4.10]{Rocca-Rossi}, to which we shall refer for all details.
\begin{proof}
\textbf{First a priori estimate:} We write the total energy inequality \eqref{total-enid-discr} for $s=0$ and estimate the terms on its right-hand side: 
\begin{equation}
\label{very-1st-step}
\begin{aligned} 
\frac{\rho}2 \int_\Omega | \pwl{\dot u}\tau (t)|^2 \dd x +  \calE_\tau(\pwc\teta\tau(t),\pwc e\tau(t), \pwc p\tau(t))  \leq 
I_1+I_2+I_3+I_4+I_5+I_6,
\end{aligned}
\end{equation}
with 
\[
I_1 = 
\frac{\rho}2 \int_\Omega | {\dot u}_\tau(0)|^2 \dd x 
 +  \calE_\tau(\pwc \teta\tau(0),\pwc e \tau(0), \pwc p\tau(0))  \leq C
 \]
 thanks to \eqref{initial-teta}, \eqref{complete-approx-e_0},  and \eqref{discr-Cauchy}.  To estimate $I_2$ we use the safe load condition \eqref{safe-load}, namely
 \begin{equation}
 \label{here:safe-loads}
 \begin{aligned}
 I_2 &  =    \int_{0}^{\pwc{\mathsf{t}}{\tau}(t)} \pairing{}{H_\Dir^{1}(\Omega;\R^d)}{ \pwc \calL \tau (r) }{\pwl{\dot u}\tau (r){-} \pwl{\dot w}\tau (r) }  \dd r  
 \\ &   =  \int_{0}^{\pwc{\mathsf{t}}{\tau}(t)} \pairing{}{H_\Dir^{1}(\Omega;\R^d)}{ \pwc F \tau (r) }{\pwl{\dot u}\tau (r) {-} \pwl{\dot w}\tau (r)  }  \dd r  +  \int_{0}^{\pwc{\mathsf{t}}{\tau}(t)} \pairing{}{H_{00,\Gamma_\Dir}^{1/2}(\Gamma_\Neu;\R^d)}{ \pwc g \tau (r) }{\pwl{\dot u}\tau (r)  {-} \pwl{\dot w}\tau (r)  }  \dd r  
 \\ & =  \int_{0}^{\pwc{\mathsf{t}}{\tau}(t)}  \int_\Omega \pwc{\varrho}\tau (r) : \left(  \sig{\pwl{\dot u}\tau (r)} {-} \sig{\pwl {\dot w}\tau(r)} \right)  \dd x \dd r \\ & 
   \stackrel{(1)}{=}   \int_{0}^{\pwc{\mathsf{t}}{\tau}(t)}  \int_\Omega \pwc{\varrho}\tau (r)  : {\pwl{\dot e}\tau (r)}  \dd x \dd r  +  \int_{0}^{\pwc{\mathsf{t}}{\tau}(t)}  \int_\Omega \pwc{\varrho}\tau (r): {\pwl{\dot p}\tau (r)}  \dd x \dd r  - \int_{0}^{\pwc{\mathsf{t}}{\tau}(t)}  \int_\Omega \pwc{\varrho}\tau (r)  :  \sig{\pwl {\dot w}\tau(r) }\dd x \dd r 
   \\ & 
    \doteq I_{2,1} + I_{2,2} +I_{2,3}
   \end{aligned}
   \end{equation}
   where $\pwc F\tau, \,  \pwc g\tau,\, \pwc \varrho\tau, \, \pwl \varrho\tau$ denote the approximations of $F, \, g, \, \varrho$. Equality  (1) follows from   the kinematic admissibility condition $\sig{\pwl{\dot u}\tau}  = \pwl{\dot e}\tau  + \pwl  {\dot p}\tau$.
    Observe that, thanks to \eqref{safe-load}, there holds
 \begin{equation}
 \label{converg-rho-interp}
  \| \pwc \varrho\tau \|_{L^\infty(0,T;L^2(\Omega;\mt_\sym^{d\times d}))} +  \| \pwl \varrho\tau \|_{W^{1,1}(0,T;L^2(\Omega;\mt_\sym^{d\times d}))}+
    \| (\pwc \varrho\tau)_\dev \|_{L^1(0,T;L^\infty(\Omega;\mt_\dev^{d\times d}))} \leq C\,.
 \end{equation}
   Now,  using   the discrete by-part integration formula \eqref{discr-by-part}  we see that 
   \[
   \begin{aligned} 
    I_{2,1} & =    -   \int_{0}^{\pwc{\mathsf{t}}{\tau}(t)}  \int_\Omega \pwl{\dot \varrho}\tau (r) {:} \upwc e{\tau}(r) \dd x \dd r 
 +\int_\Omega   \pwc{\varrho}\tau(t) : \pwc e\tau(t) \dd x -  \int_\Omega   \pwc{\varrho}\tau(0): {\pwc e\tau(0)} \dd x 
 \\
 &    \stackrel{(2)}{\leq}      \int_{0}^{\pwc{\mathsf{t}}{\tau}(t)}   \|  \pwl{\dot \varrho}\tau (r)  \|_{L^2(\Omega;\mt_\sym^{d\times d})}  \|   \upwc e{\tau}(r)\|_{L^2(\Omega;\mt_\sym^{d\times d})} \dd r +\frac{C_\bbC^1}{16} \| \pwc e\tau(t) \|_{L^2(\Omega;\mt_\sym^{d\times d})}^2+ C   \| \pwc \varrho\tau(t) \|_{L^2(\Omega;\mt_\sym^{d\times d})}^2 +C
 \end{aligned}
 \]
 where estimate (2)  follows from Young's inequality.
 The choice of the coefficient $\tfrac{C_\bbC^1}{16}$ will allow us to absorb the second term into the left-hand side of \eqref{very-1st-step}, taking into account
 the coercivity property \eqref{elast-visc-tensors} of $\bbC$, which ensures that  $\calE_\tau(\pwc\teta\tau(t),\pwc e\tau(t), \pwc p\tau(t))$ on the left-hand side of \eqref{very-1st-step} bounds $ \| \pwc e\tau(t)\|_{L^2(\Omega;\mt_\sym^{d\times d})}^2$. 
 As for $I_{22}$, using the discrete flow rule \eqref{eq-p-interp} and taking into account  the expression of $(\pwc \sigma\tau)_\dev$  we gather
 \[
    \begin{aligned} 
    I_{2,2} &  =  \int_{0}^{\pwc{\mathsf{t}}{\tau}(t)}  \int_\Omega ( \pwc{\varrho}\tau (r))_\dev \left( \bbD \pwl{\dot e}\tau(r) + \bbC \pwc e\tau(r)  + \tau|\pwc e\tau(r) |^{\gamma-2} \pwc e\tau(r)  - \pwc\teta\tau(r) \bbB - \pwc\zeta\tau(r) - \tau | \pwc p\tau(r)|^{\gamma-2} \pwc p\tau(r)    \right)  \dd x \dd r  \\ & \doteq I_{2,2,1} +  I_{2,2,2} + I_{2,2,3} + I_{2,2,4} + I_{2,2,5}  + I_{2,2,6}
 \end{aligned}
  \]
 and we estimate the above terms as follows. First, for $I_{2,2,1}$ we resort to the by-parts integration formula \eqref{discr-by-part} with the very same calculations as in the estimate of the integral term $I_{2,1}$. Second, we estimate
 \[
 I_{2,2,2} \leq  \frac{C_\bbC^1}{16} \| \pwc e\tau(t) \||_{L^2(\Omega;\mt_\sym^{d\times d})}^2+ C   \| \pwc \varrho\tau(t) \||_{L^2(\Omega;\mt_\sym^{d\times d})}^2\,.
 \]
In the estimate of $I_{2,2,3}$  we use  H\"older's inequality
 \[
 I_{2,2,3} \leq  \frac{\tau \gamma} 2    \int_{0}^{\pwc{\mathsf{t}}{\tau}(t)}     \| (\pwc \varrho\tau (r))_\dev \|_{L^\gamma (\Omega;\mt_\dev^{d\times d})}^\gamma \dd r +\frac{ \tau }{2\gamma} \int_{0}^{\pwc{\mathsf{t}}{\tau}(t)}   \| \pwc e\tau (r)\|_{L^\gamma (\Omega;\mt_\sym^{d\times d})}^\gamma \dd r\,.
 \]
 For $I_{2,2,4}$ we resort to 
  estimate  \eqref{converg-rho-interp} for $ (\pwc \varrho\tau)_\dev$ in $L^1(0,T; L^\infty (\Omega;\mt_\dev^{d\times d}))$, so that 
\[
I_{2,2,4} \leq  C  \int_{0}^{\pwc{\mathsf{t}}{\tau}(t)}  \| (\pwc \varrho\tau (r))_\dev \|_{L^\infty (\Omega;\mt_\dev^{d\times d})} \| \pwc \teta\tau(r) \|_{L^1(\Omega)} \dd r;
\]
 again, this term will be estimated via Gronwall's inequality,
 taking into account  that  $\calE_\tau(\pwc\teta\tau(t),\pwc e\tau(t), \pwc p\tau(t))$ on the left-hand side of \eqref{very-1st-step} bounds $\| \pwc\teta\tau(t)\|_{L^1(\Omega)}$.
  Finally, since $\| \pwc \zeta \tau (t) \|_{L^\infty(\Omega;\mt_\dev^{d\times d})}  \leq C_R$ thanks to \eqref{bounded-subdiff}, we find that $I_{2,2,5} \leq  C_R   \int_{0}^{\pwc{\mathsf{t}}{\tau}(t)}  \| \pwc \varrho\tau (r) \|_{L^1 (\Omega;\mt_\dev^{d\times d})} \dd r \leq C$
  by \eqref{converg-rho-interp},
  while with H\"older's inequality we have 
 \[
 I_{2,2,6} \leq  \frac{\tau \gamma} 2  \int_{0}^{\pwc{\mathsf{t}}{\tau}(t)}     \| (\pwc \varrho\tau (r))_\dev \|_{L^\gamma (\Omega;\mt_\dev^{d\times d})}^\gamma \dd r +\frac{ \tau }{2\gamma} \int_{0}^{\pwc{\mathsf{t}}{\tau}(t)}   \| \pwc p\tau (r)\|_{L^\gamma (\Omega;\mt_\dev^{d\times d})}^\gamma \dd r\,.
 \]
   This concludes the estimation of $I_{2,2}$. 
   Finally, we have 
   \[
   I_{2,3} \leq   \int_{0}^{\pwc{\mathsf{t}}{\tau}(t)}  \| \pwc{\varrho}\tau (r) \|_{L^2(\Omega;\mt_\sym^{d\times d})}   \| \sig{\pwl {\dot w}\tau(r)}  \|_{L^2(\Omega;\mt_\sym^{d\times d})}  \dd r \leq C
   \]
   in view of \eqref{converg-rho-interp} and  \eqref{converg-interp-w}, which provides a  bound for $\pwl w \tau$, and we have thus handled all 
    the terms contributing to $I_2$. 
We also have 
\[
\begin{aligned}
I_3 =   \int_{0}^{\pwc{\mathsf{t}}{\tau}(t)} \left( 
\int_\Omega  \pwc H\tau  \dd x + 
 \int_{\partial\Omega} \pwc h\tau  \dd S \right)  \dd r  \leq \|   \pwc H\tau \|_{L^1(0,T; L^1(\Omega))} +  \|   \pwc h\tau \|_{L^1(0,T; L^2(\partial\Omega))} \leq C,
 \end{aligned}
 \]
due to \eqref{convs-interp-data};
\[
\begin{aligned}
 I_4 &  = \rho \int_\Omega \pwl{\dot u}\tau(t) \pwl{\dot w }\tau(t) \dd x -  \rho \int_\Omega \dot{u}_0   \pwl{\dot w }\tau(0) \dd x 
  -  \rho  \int_{0}^{\pwc{\mathsf{t}}{\tau}(t)}  \pwl{\dot u}\tau(r-\tau) \partial_t \pwwll{ w}\tau (r)  \dd x \dd r  \\ & \stackrel{(1)}{\leq}
  C +  \frac{\rho}8 \int_\Omega | \pwl{\dot u}\tau (t)|^2 \dd x + 2\rho  \int_\Omega | \pwl{\dot w}\tau (t)|^2 \dd x +  \rho  \int_{0}^{\pwc{\mathsf{t}}{\tau}(t)-\tau} \|  \pwl{\dot u}\tau(s) \|_{L^2(\Omega;\R^d)}   \|  \partial_t\pwwll{w}\tau(s+\tau) \|_{L^2(\Omega;\R^d)}  \dd s,
\end{aligned}
\]
 where (1) follows from  \eqref{initial-u},
 \eqref{discr-Cauchy}, and \eqref{converg-interp-w},
   and  we are tacitly assuming that $\pwl {\dot u}\tau $  extends identically to zero on the interval $(-\tau,0)$. Moreover, 
 \begin{equation}
 \label{here:w}
\begin{aligned}
 I_ 5 &=  
  \int_{0}^{\pwc{\mathsf{t}}{\tau}(t)}\int_\Omega \left( \bbD \pwl{\dot e}\tau (r)  + \bbC \pwc e\tau (r)  -\pwc\teta\tau(r) \bbB \right) {:} \sig{\pwl{\dot w}\tau (r) }  \dd x \dd r \\ &  
\stackrel{(2)}{\leq}  \int_\Omega \bbD \pwc e\tau(t) : \sig{\pwl{\dot w}\tau (t) }   \dd x -  \int_\Omega \bbD \pwc e\tau(0) :  \sig{\pwl{\dot w}\tau (0) }   \dd x  - \int_{0}^{\pwc{\mathsf{t}}{\tau}(t)}   \int_\Omega \bbD \pwc e\tau(r-\tau): \sig{\partial_t \pwwll w\tau(r)}  \dd x \dd r  
\\
 & \quad +  C_\bbC^2  \int_{0}^{\pwc{\mathsf{t}}{\tau}(t)} \| \pwc e\tau (r) \|_{L^2(\Omega;\mt_\sym^{d\times d}) }  \|  \sig{\pwl{\dot w}\tau (r) }   \|_{L^2(\Omega;\mt_\sym^{d\times d}) } \dd r + C   \int_{0}^{\pwc{\mathsf{t}}{\tau}(t)} \| \pwc\teta\tau(r) \|_{L^1(\Omega)} \|   \sig{\pwl{\dot w}\tau (r) } \|_{L^\infty (\Omega;\mt_\sym^{d\times d})} \dd r 
 \\
 & \stackrel{(3)}{\leq}  C + \frac{C_\bbC^1}8   \int_\Omega |\pwc e\tau(t) |^2 \dd x +   C    \int_{0}^{\pwc{\mathsf{t}}{\tau}(t)    }   \left( \| \sig{\partial_t \pwwll w\tau(r)}\|_{L^2(\Omega;\mt_\sym^{d\times d}) }+   \|  \sig{\pwl{\dot w}\tau (r) }   \|_{L^2(\Omega;\mt_\sym^{d\times d}) } \right)    \| \pwc e\tau(r) \|_{L^2(\Omega;\mt_\sym^{d\times d}) }  \dd r \\ & \quad+   C   \int_{0}^{\pwc{\mathsf{t}}{\tau}(t)} \| \pwc\teta\tau(r) \|_{L^1(\Omega)} \|   \sig{\pwl{\dot w}\tau (r) } \|_{L^\infty (\Omega;\mt_\sym^{d\times d})} \dd r 
 \end{aligned}
 \end{equation}
 where (2) follows from integrating by parts the term $\iint   \bbD \pwl{\dot e}\tau  {:} \sig{\pwl{\dot w}\tau   }$ (again, setting $\pwc e\tau \equiv 0 $ on $(-\tau, 0)$),  and (3) by Young's inequality, with the coefficient $ \frac{C_\bbC^1}8  $ chosen in such a way as to absorb the second term 
  on the right-hand side into   the left-hand side of \eqref{very-1st-step}. 
Collecting all  of the above estimates and taking into account the coercivity properties of $\calE_
\tau$, as well as the bounds provided by   \eqref{converg-rho-interp} and  \eqref{converg-interp-w},
we get 
\[
\begin{aligned}
&
\frac38 {\rho} \int_\Omega | \pwl{\dot u}\tau (t)|^2 \dd x +  \| \pwc\teta\tau(t)\|_{L^1(\Omega)} + \frac{1}4 C_{\bbC}^1  \| \pwc e\tau(t)\|_{L^2(\Omega;\mt_\sym^{d\times d})}^2  + \frac\tau{2\gamma }
\| \pwc e\tau(t)\|_{L^\gamma(\Omega;\mt_\sym^{d\times d})}^\gamma +  \frac\tau{2\gamma} 
\| \pwc p\tau(t)\|_{L^\gamma(\Omega;\mt_\dev^{d\times d})}^\gamma
\\ & 
  \leq   C + 
 \int_{0}^{\pwc{\mathsf{t}}{\tau}(t)}   \|  \pwl{\dot \varrho}\tau (r)  \|_{L^2(\Omega;\mt_\sym^{d\times d})}  \|   \upwc e{\tau}(r)\|_{L^2(\Omega;\mt_\sym^{d\times d})} \dd r +
  C  \int_{0}^{\pwc{\mathsf{t}}{\tau}(t)}  \| (\pwc \varrho\tau (r))_\dev \|_{L^\infty (\Omega;\mt_\dev^{d\times d})} \| \pwc \teta\tau(r) \|_{L^1(\Omega)} \dd r\\ & 
  \quad 
+  \rho  \int_{0}^{\pwc{\mathsf{t}}{\tau}(t)-\tau}   \|  \partial_t\pwwll{w}\tau(s+\tau) \|_{L^2(\Omega;\R^d)}  \|  \pwl{\dot u}\tau(s) \|_{L^2(\Omega;\R^d)}   \dd  s 
\\ &  \quad 
+   \int_{0}^{\pwc{\mathsf{t}}{\tau}(t)    }   \left( \| \sig{\partial_t \pwwll w\tau(r)}\|_{L^2(\Omega;\mt_\sym^{d\times d}) }+   \|  \sig{\pwl{\dot w}\tau (r) }   \|_{L^2(\Omega;\mt_\sym^{d\times d}) } \right)    \| \pwc e\tau(r) \|_{L^2(\Omega;\mt_\sym^{d\times d}) }  \dd r \\ & \quad+   C   \int_{0}^{\pwc{\mathsf{t}}{\tau}(t)}  \|   \sig{\pwl{\dot w}\tau (r) } \|_{L^\infty (\Omega;\mt_\sym^{d\times d})} \| \pwc\teta\tau(r) \|_{L^1(\Omega)}  \dd r  \,.
\end{aligned}
\]
 Applying  a suitable version of Gronwall's Lemma,  
 we conclude that 
\[
\|\pwl{\dot u}\tau \|_{L^\infty(0,T; L^2(\Omega;\R^d))}  +  \calE_\tau(\pwc\teta\tau(t),\pwc e\tau(t), \pwc p\tau(t))  \leq C, 
\]
whence  the second of \eqref{aprioU2}, the second of \eqref{aprioU3},  \eqref{aprioE1}, \eqref{aprioE3}, 
\eqref{aprioP3},  and  the second of \eqref{aprio6-discr}. 
\begin{remark}
\upshape
\label{rmk:diffic-1-test} 
\noindent
The safe load condition \eqref{safe-load} is crucial for  handling $ \int_{0}^{\pwc{\mathsf{t}}{\tau}(t)} \pairing{}{H_\Dir^{1}(\Omega;\R^d)}
{ \pwc \calL \tau  }{\pwl{\dot u}\tau  {-}  \pwl{\dot w}\tau}  \dd r  $ on the r.h.s.\ of \eqref{very-1st-step}, cf.\ \eqref{here:safe-loads}. In fact, this term involves the \emph{dissipative} variable $\pwl{\dot u}\tau $, whose $L^2(\Omega;\R^d)$-norm, \emph{only}, is estimated by
 the r.h.s. of  \eqref{very-1st-step}. Condition \eqref{safe-load} then  allows us to rewrite the above integral in terms of the functions $\pwc \varrho \tau$ and $\pwl {\dot e}\tau$, 
 $\pwl p\tau$, and the resulting integrals are then treated via integration by parts, leading to quantities that can be controlled by the l.h.s.\ of \eqref{very-1st-step}. 
 \par
 Without \eqref{safe-load}, the term $ \int_{0}^{\pwc{\mathsf{t}}{\tau}(t)} \pairing{}{H_\Dir^{1}(\Omega;\R^d)}{ \pwc \calL \tau }{\pwl{\dot u}\tau {-}   \pwl{\dot w}\tau }   \dd r  $ could be treated only by supposing that $g \equiv 0$,   and that $F\in L^2(Q;\R^d)$. 
 \par
 The estimates for the term $ \int_{0}^{\pwc{\mathsf{t}}{\tau}(t)}\int_\Omega \pwc\sigma\tau {:} \sig{\pwl{\dot w}\tau  }  \dd x \dd r $, cf.\ \eqref{here:w}, 
unveil the role of the condition $w\in L^1(0,T; W^{1,\infty} (\Omega;\R^d))$, which allows us to control the term $ \int_{0}^{\pwc{\mathsf{t}}{\tau}(t)}\int_\Omega 
 \pwc\teta\tau \bbB{:} \sig{\pwl{\dot w}\tau  }  \dd x \dd r $ exploiting the $L^1(\Omega)$-bound provided by the l.h.s.\ of \eqref{very-1st-step}. Alternatively, one could 
 impose some sort of `compatibility' between the thermal expansion tensor $\bbB = \bbC \bbE$ and the Dirichlet loading $w$, by requiring that $\bbB {:} \sig{\dot w} \equiv 0$, cf.\ 
 \cite{Roub-PP}.
 Analogously, the condition $w \in W^{2,1}(0,T;H^1(\Omega;\R^d))$ has been used in the estimation of the term $I_5$, cf.\ \eqref{here:w}.
  \end{remark}
\textbf{Second a priori estimate:} we test
\eqref{discrete-heat} by $(\tetau{k})^{\alpha-1}$, with $\alpha \in (0,1)$, thus obtaining   
\begin{equation}
\label{ad-est-temp1}
\begin{aligned}
& 
\int_\Omega \left(  \gtau{k} + \mathrm{R}\left(\tetau{k-1}, \Dtau{k} p  \right) +  \left|  \Dtau{k} p \right|^2+ \bbD   \Dtau{k} e :  \Dtau{k} e   \right)  (\tetau{k})^{\alpha-1} \dd x  \\ & \quad
- \int_\Omega \condu(\tetau k) \nabla \tetau k \nabla  (\tetau{k})^{\alpha-1} \dd  x + 
\int_{\partial\Omega} \htau k  (\tetau{k})^{\alpha-1} \dd S  
\\
& 
\leq \int_\Omega  \left( \frac1{\alpha}\frac{(\tetau k)^\alpha - (\tetau {k-1})^\alpha}{\tau} + 
\tetau{k} \bbB : \Dtau k e (\tetau{k})^{\alpha-1} \right)  \dd  x
\end{aligned}
\end{equation}
where we have applied the concavity inequality \eqref{inequality-concave-functions}, with the choice
$\psi(\teta)=\tfrac1{\alpha}\teta^\alpha$, to estimate the term $\frac1\tau \int_\Omega (\tetau{k} {-} \tetau{k-1}) (\tetau k)^{\alpha-1} \dd x $. 
Therefore,  multiplying by $\tau$, summing over the index $k$ and  neglecting some  positive terms on the left-hand side of \eqref{ad-est-temp1}, we obtain
 for all $t \in (0,T]$
 \begin{equation} 
 \label{ad-est-temp2}
 \begin{aligned}
 &
 \frac{4(1-\alpha)}{\alpha^2} \int_0^{\pwc{\mathsf{t}}{\tau}(t)} \int_\Omega \condu(\pwc \teta{\tau}) |\nabla ((\pwc \teta{\tau})^{\alpha/2}) |^2 \dd x \dd s  +  \int_0^{\pwc{\mathsf{t}}{\tau}(t)}  \int_\Omega   
  C_\bbD^1 | \pwl {\dot e}\tau|^2 (\pwc \teta \tau)^{\alpha-1}  
 \dd x \dd s
\\
&
\leq I_1+I_2+I_3,
 \end{aligned}
 \end{equation}
 with 
 \begin{equation}
 \label{est-temp-I1}
  I_1 =   \frac1\alpha\int_\Omega  (\pwc \teta\tau(t) )^{\alpha} \dd x \leq \frac1\alpha \| \pwc \teta\tau\|_{L^\infty (0,T; L^1(\Omega))} + C \leq C
 \end{equation}
 via Young's inequality (using that  $\alpha \in (0,1)$) and the second of \eqref{aprio6-discr};   similarly $ I_2= -  \frac1\alpha\int_\Omega  (\teta_0)^{\alpha} \dd x \leq \frac1\alpha \| \teta_0\|_{L^1(\Omega)} + C$, whereas 
   \begin{equation}
 \label{est-temp-I3}
 I_3=  \int_0^{\pwc{\mathsf{t}}{\tau}(t)}  \int_\Omega   \pwc \teta\tau(t) \bbB : \pwl{\dot e}\tau(t)  ( \pwc \teta\tau(t) )^{\alpha-1}  \dd  x \leq 
\frac{ C_\bbD^1 }4   \int_0^{\pwc{\mathsf{t}}{\tau}(t)}  \int_\Omega  | \pwl {\dot e}\tau|^2 (\pwc \teta \tau)^{\alpha-1}   \dd x \dd s + C  \int_0^{\pwc{\mathsf{t}}{\tau}(t)}  \int_\Omega
(\pwc \teta \tau)^{\alpha+1}   \dd x \dd s\,.
 \end{equation}
 All in all, absorbing the first term on the right-hand side of \eqref{est-temp-I3} into  the left-hand side of \eqref{ad-est-temp2} and taking into account the growth condition \eqref{hyp-K} on $\condu$,  which yields  with easy calculations 
 that 
 \begin{equation}
 \label{needs-be-observed}
 \int_0^{\pwc{\mathsf{t}}{\tau}(t)} \int_\Omega \condu(\pwc \teta{\tau}) |\nabla ((\pwc \teta{\tau})^{\alpha/2}) |^2 \dd x \dd s  \stackrel{(1)}{\geq} c  \int_0^{\pwc{\mathsf{t}}{\tau}(t)} \int_\Omega | \pwc \teta \tau|^{\mu+\alpha-2}  | \nabla  \pwc \teta \tau |^2 \dd x \dd s = c  \int_0^{\pwc{\mathsf{t}}{\tau}(t)} \int_\Omega  | \nabla  (\pwc \teta \tau)^{(\mu+\alpha)/2} |^2 \dd x \dd s, 
  \end{equation}
 we conclude from \eqref{ad-est-temp2} that 
 \begin{equation}
  \label{ad-est-temp3}
 c   \int_0^{\pwc{\mathsf{t}}{\tau}(t)} \int_\Omega  | \nabla  (\pwc \teta \tau)^{(\mu+\alpha)/2} |^2 \dd x \dd s \leq C +  C  \int_0^{\pwc{\mathsf{t}}{\tau}(t)}  \int_\Omega
(\pwc \teta \tau)^{\alpha+1}   \dd x \dd s\,.
 \end{equation}
 From now on, the  calculations follow exactly the same lines as those developed in \cite[(3.8)--(3.12)]{Rocca-Rossi} for the analogous estimate,  in turn based on the ideas from \cite{FPR09}. While referring to \cite{Rocca-Rossi} for all details, let us just give the highlights. Setting  $ \pwc \xi\tau : = (\pwc \teta\tau \vee 1)^{(\mu+\alpha)/2}$, we  deduce from \eqref{ad-est-temp3} the following inequality
  \begin{equation}
  \label{ad-est-temp4}
  \begin{aligned}
 \int_0^{\pwc{\mathsf{t}}{\tau}(t)} \int_\Omega  | \nabla  \pwc \xi\tau  |^2 \dd x \dd s &  \leq C +  C  \int_0^{\pwc{\mathsf{t}}{\tau}(t)}  \| \pwc \xi\tau\|_{L^q(\Omega)}^q    \dd s
 \\ & 
  \leq C +
  \frac12  \int_0^{\pwc{\mathsf{t}}{\tau}(t)} \int_\Omega  | \nabla  \pwc \xi\tau  |^2 \dd x \dd  s  +
  C 
 \int_0^{\pwc{\mathsf{t}}{\tau}(t)}  \| \pwc \xi\tau\|_{L^r(\Omega)}^s \dd s + C   \int_0^{\pwc{\mathsf{t}}{\tau}(t)}  \| \pwc \xi\tau\|_{L^r(\Omega)}^q \dd s,
 \end{aligned}
 \end{equation}
 with $q \in [1,6)$  satisfying $\tfrac{\mu+\alpha}{2} \geq \tfrac{\alpha+1}{q}$.
 The very last estimate ensues from 
 the Gagliardo-Nirenberg inequality, which  in fact yields 
 \begin{equation}
 \label{GN}
 \|  \pwc \xi\tau \|_{L^q(\Omega)} \leq C_{\mathrm{GN}} \| \nabla \pwc \xi\tau \|_{L^2(\Omega;\R^d)}^\theta \| \pwc \xi\tau \|_{L^r(\Omega)}^{1-\theta} +  C \| \pwc \xi\tau \|_{L^r(\Omega)} \qquad \text{for } \theta \in (0,1) \text{ s.t. } \frac1q= \frac\theta 6 +\frac{1-\theta}r
 \end{equation}
 with   $r \in [1,q]$. Then, $s$ in \eqref{ad-est-temp4}  is a third exponent, related to $q$ and $r$ via \eqref{GN}.
 In \cite{Rocca-Rossi} it is shown that the exponents 
$q$ and $r$ can be chosen in such a way as to have $ \| \pwc \xi\tau\|_{L^\infty(0,T; L^r(\Omega))} \leq C   \| \pwc \teta\tau\|_{L^\infty (0,T; L^1(\Omega))} + C \leq C $ thanks to the  second of \eqref{aprio6-discr}.
In particular, one has to impose that $1\leq r \leq \frac2{\mu+\alpha}$. 
 Inserting this into \eqref{ad-est-temp4} one concludes that $\| \nabla  \pwc \xi\tau  \|_{L^2(0,T; L^2(\Omega;\R^d))} \leq C$. All in all, this argument yields a bound for $ \pwc \xi\tau$ in $L^2(0,T; H^1(\Omega)) \cap L^\infty (0,T; L^r(\Omega))$.  Since $ \pwc \xi\tau  = (\pwc \teta\tau \vee 1)^{(\mu+\alpha)/2}$, we ultimately conclude that 
 \begin{equation}
 \label{est-temp-added?}
 \| (\pwc \teta\tau)^{(\mu+\alpha)/2} \|_{L^2(0,T; H^1(\Omega)) \cap L^\infty (0,T; L^r(\Omega))} \leq C. 
 \end{equation}
Then, from inequality (1) in \eqref{needs-be-observed} we deduce that 
 $ \int_0^{T} \int_\Omega  | \nabla  \pwc \teta\tau  |^2 \dd x \dd s \leq C $
 provided that 
 \begin{equation}
 \label{restriction-alpha}
 \mu+\alpha-2 \geq 0 \ \text{whence the constraints} \ \alpha>0  \text{ and } \alpha \geq 2-\mu.
 \end{equation}
 From the constraints $\frac2{\mu+\alpha} \geq 1$
 and $\frac{\mu+\alpha}2 \geq 1$ in \eqref{restriction-alpha} we deduce that $\frac{\mu+\alpha}2 = 1$. Ultimately, $r=1$ and  $\alpha = 2-\mu$. 
  Thus,
  \eqref{est-temp-added?}
yields 
  the first of  \eqref{aprio6-discr}. 
 Interpolating between the two estimates in \eqref{aprio6-discr}  via the Gagliardo-Niremberg inequality gives 
 \begin{equation}
 \label{Gagliardo-Nir}
 \| \pwc \teta\tau \|_{L^h (Q)} \leq C \quad \text{with } h=\frac83 \text{ if } d=3 \text{ and } h=3 \text{ if } d=2.
 \end{equation}
 Estimate \eqref{log-added} follows from taking into account that 
\[
\| \log(\pwc \teta\tau)\|_{L^2(0,T; H^1(\Omega))} \leq C  \left( 1+ \|\pwc \teta\tau\|_{L^2(0,T; H^1(\Omega))} \right) 
\]
thanks to the strict positivity \eqref{discr-strict-pos}. 
For later use, let us point out that, in the end, we recover the bound
\[
 \| (\pwc \teta\tau)^{(\mu{+}\alpha)/2} \|_{L^2(0,T;H^1(\Omega))} \leq C 
\] 
for arbitrary $\alpha \in (0,1)$. For this, it is sufficient to observe that the second term on the right-hand side of \eqref{ad-est-temp3} now fulfills
\[
 \int_0^{\pwc{\mathsf{t}}{\tau}(t)}  \int_\Omega
(\pwc \teta \tau)^{\alpha+1}   \dd x \dd s \leq C
\]
thanks to estimate \eqref{Gagliardo-Nir}. Hence, \eqref{ad-est-temp3}  yields $  \int_0^{\pwc{\mathsf{t}}{\tau}(t)} \int_\Omega  | \nabla  (\pwc \teta \tau)^{(\mu+\alpha)/2} |^2 \dd x \dd s\leq C$, 
whence the  $L^2(0,T;H^1(\Omega))$-bound for $ (\pwc \teta\tau)^{(\mu{+}\alpha)/2}$ via the Poincar\'e inequality.  
Then, taking into account that $  (\pwc \teta\tau)^{(\mu-\alpha)/2} \leq  (\pwc \teta\tau)^{(\mu+\alpha)/2}  +1 $ a.e.\ in $Q$ and using 
that
 \[
\int_\Omega |\nabla  (\pwc \teta\tau)^{(\mu-\alpha)/2}|^2 \dd x  = C  \int_\Omega  ( \pwc \teta\tau)^{\mu-\alpha - 2} | \nabla \pwc \teta \tau|^2 \dd x \leq \frac{C}{\bar\teta^{2\alpha}} 
 \int_\Omega  ( \pwc \teta\tau)^{\mu+\alpha - 2} | \nabla \pwc \teta \tau|^2 \dd x \leq C,
 \]
 we conclude  estimate \eqref{est-temp-added?-bis}.
\par\noindent
\textbf{Third a priori estimate:} We consider the mechanical energy inequality \eqref{mech-ineq-discr} written for $s=0$. We estimate the terms on  its right-hand side  by the very same calculations developed in the \emph{First a priori estimate} for the right-hand side terms of \eqref{total-enid-discr}. We also  use that 
\[
 \int_{0}^{\pwc{\mathsf{t}}{\tau}(t)} \int_\Omega \pwc\teta\tau \bbB {:} \pwl{\dot e}\tau \dd x  \dd r \leq   \delta \int_0^{\pwc{\mathsf{t}}{\tau}(t)} \int_\Omega | \pwl{\dot e}\tau |^2 \dd x \dd r + C_\delta  \| \pwc \teta \tau \|_{L^2(0,T; L^2(\Omega))}^2
\]
via Young's inequality, with the constant $\delta>0$ chosen in such a way as to absorb the term $\iint   | \pwl{\dot e}\tau |^2$ into the left-hand side of  \eqref{mech-ineq-discr}. Since $ \| \pwc \teta \tau \|_{L^2(0,T; L^2(\Omega))}^2 \leq C$ by the previously proved \eqref{aprio6-discr}, we ultimately conclude that the  terms on the right-hand side of  \eqref{mech-ineq-discr} are all bounded, uniformly w.r.t.\ $\tau$. This leads to \eqref{aprioE2} and \eqref{aprioP2}, whence 
\eqref{aprioP1}, as well as \eqref{aprioU1} and the first of 
\eqref{aprioU2} by kinematic admissibility. 
\par\noindent
\textbf{Fourth a priori estimate:} 
It follows from estimates \eqref{aprioE1}, \eqref{aprioE2}, \eqref{aprioE3}, and  \eqref{aprio6-discr} that the stresses $(\pwc \sigma\tau )_\tau$ are uniformly bounded in
$L^{\gamma/(\gamma{-}1)}(Q; \mt_\sym^{d\times d})$. Therefore, also taking into account \eqref{converg-interp-L},    a comparison argument in the discrete momentum balance  \eqref{eq-u-interp} yields   that the derivatives $(\partial_t \pwwll u\tau)_\tau$ are bounded in $L^{\gamma/(\gamma{-}1)} (0,T;W^{1,\gamma}(\Omega;\R^d)^*)$, whence the third of \eqref{aprioU3}.
\par\noindent 
\textbf{Fifth a priori estimate:} We will now sketch the argument for \eqref{aprio_Varlog}, referring to the proof of \cite[Prop.\ 4.10]{Rocca-Rossi} for all details. Indeed,  let us fix a partition $0 =\sigma_0 < \sigma_1 < \ldots < \sigma_J =T$ of the interval $[0,T]$. From the discrete  entropy inequality
\eqref{entropy-ineq-discr} written on the  interval $[\sigma_{i-1},\sigma_i]$
 and for a \emph{constant-in-time} test function 
 we deduce that
 \begin{equation}
  \label{ell-i-pos-neg}
 \begin{aligned}
 &
 \int_\Omega (\log(\pwc\teta\tau (\sigma_{i}))  - \log(\pwc\teta\tau (\sigma_{i-1})) ) \varphi \dd x + \Lambda_{i,\tau} (\varphi)   \geq 0 && \text{for all } \varphi \in W_+^{1,d+\epsilon}(\Omega),
 \\
   &
 \int_\Omega ( \log(\pwc\teta\tau (\sigma_{i-1}))-  \log(\pwc\teta\tau (\sigma_{i}))) \varphi \dd x - \Lambda_{i,\tau} (\varphi) \geq 0 && \text{for all } \varphi \in W_-^{1,d+\epsilon}(\Omega),
 \end{aligned} 
 \end{equation}
 where we have used the place-holder
 \begin{equation}
 \label{pl-ho-varphi}
\begin{aligned}
\Lambda_{i,\tau} (\varphi) =
 &   \int_{\pwc{\mathsf{t}}{\tau}(\sigma_{i-1})}^{\pwc{\mathsf{t}}{\tau}(\sigma_i)}  \int_\Omega  \condu(\pwc\teta\tau) \nabla \log(\pwc\teta\tau) \nabla \varphi \dd x \dd r + \int_{\pwc{\mathsf{t}}{\tau}(\sigma_{i-1})}^{\pwc{\mathsf{t}}{\tau}(\sigma_i)} \int_\Omega \bbB {:} \pwl{\dot e}\tau  \varphi  \dd x \dd r
\\
& \quad -  \int_{\pwc{\mathsf{t}}{\tau}(\sigma_{i-1})}^{\pwc{\mathsf{t}}{\tau}(\sigma_i)} \int_\Omega  \condu(\pwc\teta\tau) \frac{\varphi}{\pwc\teta\tau} \nabla (\log(\pwc\teta\tau)) \nabla \pwc\teta\tau \dd x \dd r -  \int_{\pwc{\mathsf{t}}{\tau}(\sigma_{i-1})}^{\pwc{\mathsf{t}}{\tau}(\sigma_i)}   \int_{\partial\Omega} \pwc h{\tau}  \frac{\varphi}{\pwc\teta\tau} \dd S \dd r
 \\
 &\quad - \int_{\pwc{\mathsf{t}}{\tau}(\sigma_{i-1})}^{\pwc{\mathsf{t}}{\tau}(\sigma_i)} \int_\Omega \left(\pwc H\tau+ \mathrm{R}(\upwc \teta\tau, \pwl {\dot p}\tau) + | \pwl {\dot p}\tau|^2 + \bbD \pwl{\dot e}\tau{:}  \pwl{\dot e}\tau  \right)
 \frac{\varphi}{\pwc\teta\tau} \dd x \dd r.
\end{aligned}
\end{equation}
Arguing as in the proof of \cite[Prop.\ 4.10]{Rocca-Rossi}, 
from \eqref{ell-i-pos-neg} we deduce that 
\begin{equation}
\label{genialata}
\begin{aligned} & 
\sum_{i=1}^J \left|  \pairing{}{W^{1,d+\epsilon}(\Omega)}{\log(\pwc\teta\tau (\sigma_{i}))  - \log(\pwc\teta\tau (\sigma_{i-1})) }{\varphi} \right| 
\\ & \leq  \sum_{i=1}^J   
\int_\Omega (\log(\pwc\teta\tau (\sigma_{i}))  - \log(\pwc\teta\tau (\sigma_{i-1}))) |\varphi| \dd x
+ \Lambda_{i,\tau} (|\varphi|) + |\Lambda_{i,\tau} (\varphi^+)| + |\Lambda_{i,\tau} (\varphi^-)| 
\end{aligned} 
\end{equation}
for all $\varphi \in W^{1,d+\epsilon}(\Omega)$. 
Then, we infer the bound  \eqref{aprio_Varlog} by 
 estimating the terms on the right-hand side of \eqref{genialata}, uniformly w.r.t.\ $\varphi$. 
 In particular, to handle the second, fourth, and fifth integral terms  arising from $\Lambda_{i,\tau} (\varphi)$ (cf.\
  \eqref{pl-ho-varphi}), we  use the previously proved estimates  \eqref{aprioE2}, \eqref{aprioP2}, 
 as well as the bounds provided by \eqref{converg-interp-g} and   \eqref{converg-interp-h}  on $\pwc H\tau$ and $\pwc h \tau$, cf.\ \cite{Rocca-Rossi} for all details. Let us only comment on the estimates for the first and third integral terms 
 on the r.h.s.\ of \eqref{pl-ho-varphi}. 
 We remark that for every $\varphi \in W^{1,d+\epsilon}(\Omega)$ we have 
  \begin{equation}
  \label{est-kappa-teta-log}
 \begin{aligned}
 &  \left| \int_{\pwc{\mathsf{t}}{\tau}(\sigma_{i-1})}^{\pwc{\mathsf{t}}{\tau}(\sigma_i)}  \int_\Omega  \condu(\pwc\teta\tau) \nabla \log(\pwc\teta\tau) \nabla \varphi \dd x \dd r   \right|  
\\ 
 & \stackrel{(1)}{\leq} C  \int_{\pwc{\mathsf{t}}{\tau}(\sigma_{i-1})}^{\pwc{\mathsf{t}}{\tau}(\sigma_i)}  \int_\Omega  \left(| \pwc\teta\tau |^{\mu-1} | \nabla  \pwc\teta\tau | + \frac1{\pwc\teta\tau}   | \nabla  \pwc\teta\tau |  \right) |\nabla \varphi|  \dd x \dd r
 \\
  &  \stackrel{(2)}{\leq} C   \int_{\pwc{\mathsf{t}}{\tau}(\sigma_{i-1})}^{\pwc{\mathsf{t}}{\tau}(\sigma_i)}   \int_\Omega  |\left( \pwc\teta\tau \right)^{(\mu+\alpha-2)/2}  \nabla  \pwc\teta\tau |  \left( \pwc\teta\tau \right)^{(\mu-\alpha)/2}  |\nabla \varphi|   + \frac1{\bar\teta}   | \nabla  \pwc\teta\tau |  |\nabla \varphi|  \dd x \dd r 
   \\
  &  \stackrel{(3)}{\leq} C   \int_{\pwc{\mathsf{t}}{\tau}(\sigma_{i-1})}^{\pwc{\mathsf{t}}{\tau}(\sigma_i)}   \|\left( \pwc\teta\tau \right)^{(\mu+\alpha-2)/2}  \nabla  \pwc\teta\tau \|_{L^2(\Omega;\R^d)} \, \|  \left( \pwc\teta\tau \right)^{(\mu-\alpha)/2} \|_{L^{d^\star}(\Omega)} \| \nabla \varphi \|_{L^{d+\epsilon}(\Omega;\R^d)} \dd r  \\ & \quad +  C   \int_{\pwc{\mathsf{t}}{\tau}(\sigma_{i-1})}^{\pwc{\mathsf{t}}{\tau}(\sigma_i)}    \| \nabla  \pwc\teta\tau \|_{L^2(\Omega;\R^d)}  \|\nabla \varphi\|_{L^2(\Omega;\R^d)} \dd r 
\end{aligned} 
 \end{equation}
 where (1) follows  from the growth condition \eqref{hyp-K} on $\condu$, (2) from the discrete positivity property \eqref{discr-strict-pos}, and (3) from H\"older's inequality, in view of  the continuous embedding 
 \begin{equation}
 \label{Sobolev-embedding}
 H^1(\Omega) \subset L^{d^\star}(\Omega) \quad \text{with } d^\star \begin{cases} 
 \in [1,\infty) & \text{if } d=2,
 \\
 = 6 & \text{if } d=3.
 \end{cases}
 \end{equation}
 Therefore, observe that only $\varphi \in W^{1,d+\epsilon}(\Omega)$, with $\epsilon>0$, is needed. 
 Then, we use estimates \eqref{aprio6-discr} and  \eqref{est-temp-added?-bis} to bound the terms on the r.h.s.\ of \eqref{est-kappa-teta-log}.  As for the third term on the r.h.s.\ of \eqref{pl-ho-varphi}, we use that 
 \[
 \begin{aligned}
 &
\left|   \int_{\pwc{\mathsf{t}}{\tau}(\sigma_{i-1})}^{\pwc{\mathsf{t}}{\tau}(\sigma_i)} \int_\Omega  \condu(\pwc\teta\tau) \frac{\varphi}{\pwc\teta\tau} \nabla (\log(\pwc\teta\tau)) \nabla \pwc\teta\tau \dd x \dd r \right|
\\ 
 & \stackrel{(4)}{\leq} C  \int_{\pwc{\mathsf{t}}{\tau}(\sigma_{i-1})}^{\pwc{\mathsf{t}}{\tau}(\sigma_i)}  \int_\Omega  \left(| \pwc\teta\tau |^{\mu-2} | \nabla  \pwc\teta\tau |^2 + \frac1{\bar\teta^2}   | \nabla  \pwc\teta\tau |^2\right) | \varphi|  \dd x \dd r
 \\ 
 & \stackrel{(5)}{\leq} C  \| \varphi \|_{L^\infty(\Omega)}  \int_{\pwc{\mathsf{t}}{\tau}(\sigma_{i-1})}^{\pwc{\mathsf{t}}{\tau}(\sigma_i)}  \int_\Omega   | \pwc\teta\tau |^{\mu+\alpha-2} | \nabla  \pwc\teta\tau |^2 + | \nabla  \pwc\teta\tau |^2  \dd x \dd r,
\end{aligned} 
 \]
 with (4) due to  \eqref{hyp-K} and the positivity property \eqref{discr-strict-pos},  and (5) following from the estimate $| \pwc\teta\tau |^{\mu-2} \leq  | \pwc\teta\tau |^{\mu+\alpha-2}+1 $, combined with the fact that $\varphi \in W^{1,d+\epsilon}(\Omega) \subset L^\infty(\Omega)$. Again, we conclude via the bounds  \eqref{aprio6-discr} and  \eqref{est-temp-added?-bis}. 
\par\noindent 
\textbf{Sixth a priori estimate:} 
Under the stronger condition \eqref{hyp-K-stronger}, we multiply the discrete heat equation \eqref{discrete-heat} by a function  $\varphi \in W^{1,\infty}(\Omega)$. Integrating in space  we thus obtain for almost all $t\in (0,T)$
\begin{equation}
\label{analog-7th-est}
\left|\int_\Omega \pwl{\dot \teta}\tau (t)  \varphi \dd x \right| 
\leq \left| \int_\Omega \condu(\pwc \teta\tau(t)) \nabla  \pwc \teta\tau(t) \nabla \varphi \dd x \right| + \left| \int_\Omega \pwc J{\tau}(t) \varphi \dd x \right|  + \left| \int_{\partial\Omega} \pwc h\tau(t)  \varphi \dd S\right| \doteq I_1+I_2+I_3\,,
\end{equation}
where we have used the place-holder $  \pwc J{\tau}(t): =   \pwc H\tau(t) + \mathrm{R}\left(\upwc\teta\tau(t), \pwl {\dot p}\tau(t) \right) +  \left|  \pwl {\dot p}\tau(t)  \right|^2+ \bbD   \pwl {\dot e}\tau(t)  : \pwl {\dot e}\tau(t) -\pwc\teta\tau(t) \bbB : \pwl {\dot e}\tau(t) $. Now, in view of  \eqref{converg-interp-g}  for $\pwc H\tau$ and of estimates 
 \eqref{aprioE2}, \eqref{aprioP2},   and \eqref{aprio6-discr}, it is clear that 
 \[
I_2\leq \mathcal{J}_\tau (t) \| \varphi\|_{L^\infty(\Omega)} \quad \text{with $\mathcal{J}_\tau (t): =  \| \pwc J\tau(t)\|_{L^1(\Omega)} $.} 
 \]
Observe that the   family $(\mathcal{J}_\tau)_\tau$ is uniformly bounded in $L^1(0,T)$. The third term on the r.h.s.\ of  \eqref{analog-7th-est} is analogously bounded thanks to \eqref{converg-interp-h}. As for the first one, we use that 
\[
I_1\leq    C\|(\pwc \teta\tau)^{(\mu-\alpha+2)/2} \|_{L^2(\Omega)} \| (\pwc \teta\tau)^{(\mu+\alpha-2)/2} \nabla  \pwc \teta\tau\|_{L^2(\Omega;\R^d)} \|\nabla \varphi\|_{L^\infty (\Omega;\R^d)}
+   C \| \nabla \pwc \teta\tau\|_{L^2(\Omega;\R^d)}  \|\nabla \varphi\|_{L^2 (\Omega;\R^d)},
\]  based on the growth condition \eqref{hyp-K} on $\condu$.  By  \eqref{hyp-K-stronger} we have 
 $\mu <5/3$ if $d=3$, and $\mu<2$ if $d=2$. Since $\alpha$ can be chosen arbitrarily close to $1$, from \eqref {Gagliardo-Nir}
 we gather that  $(\pwc \teta\tau)^{(\mu-\alpha+2)/2}$ is bounded in $L^2(Q)$. Therefore, also  taking into account \eqref{est-temp-added?} and \eqref{aprio6-discr} we infer that $I_1\leq \mathcal{K}_\tau(t)  \| \varphi\|_{L^\infty(\Omega)}$  with $(\mathcal{K}_\tau)_\tau$ bounded in $L^1(0,T)$. Hence, 
 estimate \eqref{aprio7-discr}
 follows. 
\end{proof}
\subsection{Passage to the limit}
\label{ss:3.3}	
In this section, we conclude the proof of Theorems	\ref{mainth:1} \& \ref{mainth:2}.
 First of all, from the a priori
estimates obtained in Proposition \ref{prop:aprio} we deduce the convergence (along a subsequence, in suitable topologies) of the approximate solutions, to a quadruple $(\teta,u,e,p)$. In the proofs of Thm.\  \ref{mainth:1} (\ref{mainth:2}, respectively), we then proceed to  show  that  $(\teta,u,e,p)$ is  an \emph{entropic} (a \emph{weak energy}, respectively) solution to (the Cauchy problem for) system  (\ref{plast-PDE}, \ref{bc}), by passing to the limit in the approximate system \eqref{syst-interp}, and in the discrete entropy and total energy inequalities. Let us mention that, in order to recover the  kinematic admissibility, the weak momentum balance, and the plastic flow rule, we will follow an approach different from  that developed in \cite{DMSca14QEPP}.
The latter paper exploited   a reformulation of the (discrete) momentum balance and flow rule in terms of a mechanical energy balance, and a variational inequality,
based on the results from \cite{DMDSMo06QEPL}. Let us point out that
 it would be possible to repeat this argument in the present setting as well. Nonetheless, the limit passage procedure that we will develop in Step 2 of the proof of Thm.\ \ref{mainth:1} will lead us to conclude, via careful $\limsup$-arguments, additional strong convergences that will allow us to take the limit of the quadratic terms on the r.h.s.\ of the heat equation \eqref{heat}.
\par
Prior to our
compactness statement
for the sequence of approximate solutions, 
we recall here a  compactness result, akin to the Helly Theorem and  tailored to the bounded variation type estimate \eqref{aprio_Varlog}, which will have a pivotal role in establishing the convergence properties for  (a subsequence of) the approximate temperatures. 
Theorem \ref{th:mie-theil} below was proved in \cite{Rocca-Rossi}, 
cf.\ Thm.\ A.5 therein, 
with the exception  of convergence \eqref{enhSav}. We will  give its proof  in the Appendix, and in doing so we will shortly recapitulate the argument for \cite[Thm.\ A.5]{Rocca-Rossi}. Since in the proof we shall resort to a compactness result from the theory of Young measures, also invoked in the proof of Thm.\ \ref{mainth:3} ahead, we shall  recall such result, together with some basics of the theory, in the Appendix.
\begin{theorem}
\label{th:mie-theil}
Let $\bsV$ and $\bsY$ be two (separable) reflexive Banach spaces  such that
$\bsV \subset \bsY^*$ continuously. Let
 $(\ell_k)_k \subset L^p
(0,T;\bsV) \cap \mathrm{B} ([0,T];\bsY^*)$ be  bounded  in $L^p
(0,T;\bsV) $ and suppose in addition  that
\begin{align}
\label{ell-n-0}
&
\text{$(\ell_k(0))_k\subset \bsY^*$ is bounded},
\\
&
\label{BV-bound}
\exists\, C>0 \  \ \forall\, \varphi \in \overline{B}_{1,\bsY}(0)  \ \  \forall\, k \in \N\, : \quad
 \mathrm{Var}(\pairing{}{\bsY}{\ell_k}{ \varphi}; [0,T] )  \leq C,
\end{align}
where, for  given $\ell \in
\mathrm{B}([0,T];\bsY^*)$ and $\varphi \in \bsY$ we set
\begin{equation}
\label{var-notation}
\begin{aligned}
\mathrm{Var}(\pairing{}{\bsY}{\ell}{ \varphi}; [0,T] ) : =  \sup \{  \sum_{i=1}^J
\left |\pairing{}{\bsY}{\ell(\sigma_{i})}{ \varphi} - \pairing{}{\bsY}{\ell(\sigma_{i-1})}{ \varphi}  \right|
\, :  \
0 =\sigma_0 < \sigma_1 < \ldots < \sigma_J =T \} \,.
\end{aligned}
\end{equation}
\par
Then, there exist  a  (not relabeled) subsequence
 $(\ell_{k})_k$
 and a function $\ell \in L^p (0,T;\bsV) \cap L^\infty (0,T; \bsY^*) $ 
 such that  as $k\to \infty$
 \begin{align}
 \label{weak-LpB}
 &
 \ell_{k} \weaksto \ell \quad \text{ in } L^p (0,T;\bsV) \cap L^\infty (0,T;\bsY^*),
 \\
\label{weak-ptw-B}
&
\ell_{k}(t) \weakto \ell(t) \quad \text{ in } \bsV\quad \foraa t \in (0,T).
\end{align}
\par
Furthermore, for almost all $t \in (0,T)$ and any sequence $(t_k)_k \subset [0,T]$ with $t_k \to t$ there holds 
\begin{equation}
\label{enhSav}
 \ell_{k}(t_k) \weakto \ell(t)  \qquad \text{ in $\bsY^*$. }
 \end{equation}
\end{theorem}
\par
We are now in the position to prove the following compactness result where, in particular, we show that, along a subsequence, the sequences $(\pwc\teta\tau)_\tau$ and $(\upwc\teta\tau)_\tau$ converge,  in suitable topologies, to the \emph{same} limit $\teta$. This is not a trivial consequence of the obtained a priori estimates, as no bound on the total variation of the functions $\pwc\teta\tau$ is available. In fact, this fact stems from  the `generalized $\BV$' estimate  \eqref{aprio_Varlog}, via the convergence property  \eqref{enhSav} from Theorem \ref{th:mie-theil}.  
\begin{lemma}[Compactness]
\label{l:compactness}
Assume \eqref{hyp-K}. Then, for any sequence $\tau_k \downarrow 0$ there exist a (not relabeled) subsequence and a quintuple $(\teta, u, e, p,\zeta)$ such that the following convergences hold
\begin{subequations}
\label{convergences-cv}
\begin{align}
&
\label{cvU1}
 \pwl \uu{\tau_k} \weaksto \uu  &&  \text{ in $H^1(0,T;  H^1(\Omega;\R^d)) \cap
W^{1,\infty}(0,T;L^2(\Omega;\R^d))$,}
\\
&
\label{cvU2}
  \pwc \uu{\tau_k},\, \upwc \uu{\tau_k} \to \uu &&
 \text{ in $L^\infty(0,T;H^{1-\epsilon}(\Omega;\R^d)) $ for all $\epsilon \in (0,1]$,}
 \\
 &
 \label{cvU3}
 \pwl \uu{\tau_k} \to \uu  && \text{ in $\mathrm{C}^0([0,T];H^{1-\epsilon}(\Omega;\R^d))  $ for all $\epsilon \in (0,1]$,}
 \\
  &
 \label{cvU3-bis}
\pwwll{u}{\tau_k} \to \dot {\uu}  && \text{ in $\mathrm{C}_{\mathrm{weak}}^0([0,T];L^2(\Omega;\R^d)) \cap L^2(0,T; H^{1-\epsilon}(\Omega;\R^d))$ for all $\epsilon \in (0,1]$,}
\\
&
\label{cvU4}
\partial_t \pwwll {\uu}{\tau_k} \weakto \ddot u &&
 \text{ in $L^{\gamma/(\gamma-1)}(0,T;W^{1,\gamma}(\Omega;\R^d)^*) $,}
 \\
&
\label{cvE1}
\pwc e{\tau_k} \weaksto e && \text{ in $L^\infty(0,T; L^2(\Omega;\mt_\sym^{d\times d}))$,}
\\
&
\label{cvE2}
\pwl e{\tau_k} \weakto e && \text{ in $H^1(0,T; L^2(\Omega;\mt_\sym^{d\times d}))$,}
\\
&
\label{cvE3-bis}
\pwl e{\tau_k} \to e && \text{ in $\rmC_{\mathrm{weak}}^0([0,T]; L^2(\Omega;\mt_\sym^{d\times d}))$,}
\\
&
\label{cvE3}
\tau|\pwc e{\tau_k}|^{\gamma-2}  \pwc e{\tau_k} \to 0 && \text{ in $L^{\infty}(0,T; L^{\gamma/(\gamma{-}1)}(\Omega; \mt_\sym^{d\times d}))$,}
 \\
&
\label{cvP1}
\pwc p{\tau_k} \weaksto p && \text{ in $L^\infty(0,T; L^2(\Omega;\mt_\sym^{d\times d}))$,}
\\
&
\label{cvP2}
\pwl p{\tau_k} \weakto p && \text{ in $H^1(0,T; L^2(\Omega;\mt_\sym^{d\times d}))$,}
\\
&
\label{cvP3-bis}
\pwl p{\tau_k} \to p && \text{ in $\rmC_{\mathrm{weak}}^0([0,T]; L^2(\Omega;\mt_\dev^{d\times d}))$,} \\
&
\label{cvP3}
\tau|\pwc p{\tau_k}|^{\gamma-2}  \pwc p{\tau_k} \to 0 && \text{ in $L^{\infty}(0,T;   L^{\gamma/(\gamma{-}1)}(\Omega; \mt_\dev^{d\times d}))$,}
\\
 &
 \label{cvT1}
 \pwc \teta{\tau_k}\weakto \teta && \text{ in $L^2 (0,T; H^1(\Omega))$},
 \\
&
\label{cvT2}
\log(\pwc\teta{\tau_k})  \weaksto  \log(\teta)  &&
\text{ in } L^2 (0,T; H^1(\Omega))  \cap  L^\infty (0,T; W^{1,d+\epsilon}(\Omega)^*)   \quad \text{for every } \epsilon>0,
\\
&
\label{cvT4}
 \log(\pwc\teta{\tau_k}(t)) \weakto \log(\teta(t))   &&   \text{ in $H^1(\Omega)$ for almost all $t \in (0,T)$,}  
\\
&
\label{cvuT4}
 \log(\upwc\teta{\tau_k}(t)) \weakto \log(\teta(t))   &&   \text{ in $H^1(\Omega)$ for almost all $t \in (0,T)$,}  
\\
&
\label{cvT5}
\pwc\teta{\tau_k}\to \teta &&  \text{  in $L^h(Q)$
for all $h\in [1,8/3) $ for $d=3$ and all $h\in [1, 3)$ if $d=2$,}
\\
&
\label{cvuT5}
\upwc\teta{\tau_k}\to \teta &&  \text{  in $L^h(Q)$
for all $h\in [1,8/3) $ for $d=3$ and all $h\in [1, 3)$ if $d=2$,}
\\
&
\label{cvT8}
(\pwc\teta{\tau_k})^{(\mu+\alpha)/2}\weakto \teta^{(\mu+\alpha)/2} &&  \text{  in $L^2(0,T;H^1(\Omega))$  for every $\alpha \in [(2{-}\mu)^+,1)$, } 
\\
&
\label{cvT9}  (\pwc\teta{\tau_k})^{(\mu-\alpha)/2}\weakto \teta^{(\mu-\alpha)/2} &&   \text{  in $L^2(0,T;H^1(\Omega))$ for every $\alpha \in [(2{-}\mu)^+,1)$, }
\\
& 
\label{cvZ}
\pwc\zeta{\tau_k} \weaksto \zeta && \text{ in $L^\infty (Q;\mt_\dev^{d \times d})$}.
\end{align}
The triple $(u,e,p)$ complies with the kinematic admissibility condition  \eqref{kin-admis}, while $\teta$  also fulfills
\begin{equation}
\label{additional-teta}
\teta \in L^\infty (0,T; L^1(\Omega)) \text{ and }
 \teta \geq \bar{\teta} \text{ a.e.\ in $Q$}
 \end{equation}
  with $\bar{\teta}$  from \eqref{discr-strict-pos}.
\par
Furthermore, under condition \eqref{hyp-K-stronger} we also have $\teta \in \mathrm{BV} ([0,T]; W^{1,\infty} (\Omega)^*)$, and
\begin{align}
& \label{cvT6}
\pwc\teta{\tau_k} \to \teta && \text{ in } L^2 (0,T; Y) \text{ for all $Y$ such that $H^1(\Omega) \Subset Y \subset W^{1,\infty} (\Omega)^*$},
\\
&
\label{cvT7}
\pwc\teta{\tau_k}(t) \weaksto \teta(t) && \text{ in } W^{1,\infty} (\Omega)^* \text{ for all } t \in [0,T].
\end{align}
\end{subequations}
\end{lemma}  Let us  mention beforehand  that, in the proof of Thm.\ \ref{mainth:1} we will obtain further convergence properties for the sequences of approximate solutions,
 cf.\ also Remark \ref{rmk:energy-conv} ahead. 
\begin{proof}[Sketch of the proof]
Convergences \eqref{cvU1}--\eqref{cvU3}, \eqref{cvE1}--\eqref{cvE3-bis},  \eqref{cvP1}--\eqref{cvP3-bis}, and \eqref{cvZ} follow from the a priori estimates in Proposition \ref{prop:aprio} via well known weak and strong  compactness results (cf.\ e.g.\ \cite{Simon87}), also taking into account
that 
\begin{equation}
\label{stability-e-k}
\| \pwc e{\tau_k} {-} \pwl e{\tau_k}\|_{L^\infty (0,T; L^2(\Omega;\R^d))} \leq C \tau_k^{1/2} \to 0  \qquad \text{ as $k\to\infty$},
\end{equation}
and the analogous relations involving $\pwc p{\tau_k},\, \pwl p{\tau_k}$, etc.
 Passing to the limit as $k\to\infty$ in the discrete kinematic admissibility condition $(\pwc u{\tau_k}(t), \pwc e{\tau_k}(t), \pwc p{\tau_k}(t)) \in \calA(\pwc w{\tau_k}(t))$ for  a.a.\ $t \in (0,T)$,  also in view of convergence \eqref{converg-interp-w} for $\pwc w{\tau_k}$, we conclude that the triple $(u,e,p)$ is admissible. 
In view of estimate \eqref{aprioU3}  for  $(\pwwll u{\tau_k})_k$,
again by the Aubin-Lions type compactness results from \cite{Simon87} we conclude that  there exists $v$  such that $\pwwll u{\tau_k} \to v$ in $ L^2 (0,T;H^{1-\epsilon}(\Omega;\R^d)) \cap  \rmC_{\mathrm{weak}}^0 ([0,T]; L^2(\Omega;\R^d))$ for every $\epsilon \in (0,1]$. Taking into account that
\begin{equation}
\label{stability}
 \| \pwwll \uu{\tau_k} - \pwl {\dot \uu}{\tau_k}\|_{L^\infty
(0,T;  W^{1,\gamma}(\Omega;\R^d)^* )}\leq \tau_k^{1/\gamma}  \| \partial_t \pwwll
 {\uu}{\tau_k} \|_{L^{\gamma/(\gamma{-}1)} (0,T;W^{1,\gamma}(\Omega;\R^d)^*)} \leq S {\tau_k}^{1/\gamma},
 \end{equation} 
 we conclude that $v = \dot u$, whence \eqref{cvU3-bis}.
 It then follows from \eqref{stability}
that 
\begin{equation}
\label{dotu-quoted-later}
\pwl {\dot \uu}{\tau_k}(t) \weakto \dot{u}(t) \qquad \text{in }L^2(\Omega;\R^d) \quad \text{for every } t \in [0,T]. 
\end{equation}
  Moreover, thanks to  \eqref{stability} we 
  identify the weak limit  of $\partial_t \pwwll {\uu}{\tau_k}$  in  $L^{\gamma/(\gamma-1)}(0,T;W^{1,\gamma}(\Omega;\R^d)^*) $  with $\ddot u$, and \eqref{cvU4} ensues.   In order to prove \eqref{cvE3} (an analogous argument yields \eqref{cvP3}), it is sufficient to observe that 
 \[
 \| \tau |\pwc{e}{\tau_k}|^{\gamma-2}  \pwc{e}{\tau_k} \|_{L^\infty(0,T; L^{\gamma/(\gamma{-}1)}(\Omega; \mt_\sym^{d\times d}))}  = \tau^{1/\gamma} \left( \tau^{1/\gamma}  \| \pwc{e}{\tau_k} \|_{L^\infty(0,T; L^{\gamma}(\Omega; \mt_\sym^{d\times d}))} \right)^{\gamma{-}1} \to 0
 \]
 thanks to estimate \eqref{aprioE3}.
 \par
 For the convergences of the functions $(\pwc \teta{\tau_k})_k$, we briefly recap the arguments from the proof of \cite[Lemma 5.1]{Rocca-Rossi}. On account of estimates \eqref{log-added} and \eqref{aprio_Varlog} we can  apply the compactness 
 Theorem \ref{th:mie-theil} to the functions $\ell_k = \log(\pwc\teta{\tau_k})$, in the setting of  the spaces $\bsV= H^1(\Omega)$,  $\bsY= W^{1,d+\epsilon}(\Omega)$, and with $p=2$.
Hence we conclude that, up to a subsequence  the functions $\log(\pwc\teta{\tau_k})$ weakly$^*$  converge
to some $\lambda  $ in $  L^2 (0,T; H^1(\Omega)) \cap L^\infty(0,T; W^{1,d+\epsilon}(\Omega)^*)$ for all $\epsilon>0$,  i.e.\  \eqref{cvT2}, and that
$\log(\pwc\teta{\tau_k}(t)) \weakto \lambda(t)$ in $H^1(\Omega)$ for almost all $t\in (0,T)$, i.e.\  \eqref{cvT4}. Therefore, up to a further subsequence
we have  $\log(\pwc\teta{\tau_k}) \to \lambda$ almost everywhere in $Q$.  Thus,
$
\pwc\teta{\tau_k} \to \teta:= e^{\lambda} 
$ almost everywhere in $Q$. Convergences \eqref{cvT1} and \eqref{cvT5} then  follow from estimates \eqref{aprio6-discr} and \eqref{Gagliardo-Nir}, respectively. 
An immediate lower semicontinuity argument  combined with estimate  \eqref{aprio6-discr}  allows us to conclude \eqref{additional-teta}; the strict positivity of $\teta$ follows from \eqref{discr-strict-pos}. 	
Concerning convergence \eqref{cvT8},  we use \eqref{cvT5} to deduce that $(\pwc \teta{\tau_k})^{(\mu+\alpha)/2} \to \teta^{(\mu+\alpha)/2} $ in $L^{2h/(\mu+\alpha)}(Q)$ for  $h$ as in \eqref{cvT5}. Since $(\pwc \teta{\tau_k})^{(\mu+\alpha)/2} $ is itself bounded in $L^2(0,T;H^1(\Omega)$ by estimate \eqref{est-temp-added?-bis},  \eqref{cvT8} ensues, and so does \eqref{cvT9} by a completely analogous argument. 
\par Let us now address  convergences \eqref{cvuT4} and  \eqref{cvuT5}  for the sequence   $(\upwc \teta{\tau_k})_k$. On the one hand, 
observe that estimates \eqref{log-added}--\eqref{aprio_Varlog} also hold for  $(\upwc \teta{\tau_k})_k$. Therefore, we may apply 
  Thm.\  \ref{th:mie-theil} to the functions $ \log(\upwc\teta{\tau_k}) $ and conclude that there exists $\underline \lambda $ such that 
$\log(\upwc\teta{\tau_k})  \weaksto \underline \lambda$ in $ L^2 (0,T; H^1(\Omega)) \cap L^\infty(0,T; W^{1,d+\epsilon}(\Omega)^*)$ for all $\epsilon>0$, as well as 
$\log(\upwc\teta{\tau_k}(t)) \weakto \underline\lambda(t)$ in $H^1(\Omega)$ for almost all $t\in (0,T)$. 
On the other hand, 
since $\upwc \teta{\tau_k} (t) = \pwc \teta{\tau_k}(t-\tau_k)$  for almost all $ t \in (0,T)$, from \eqref{enhSav} 
 we conclude that $ \log(\upwc \teta{\tau_k} (t)  )  \weakto \log(\teta(t)) $ in $W^{1,d+\epsilon}(\Omega)^*$ for almost all $t \in (0,T)$. Hence we identify $\underline\lambda(t) = \log(\teta(t))$  for almost all $t \in (0,T)$. Then, convergences   \eqref{cvuT4} and  \eqref{cvuT5} ensue from the very same arguments as for the sequence $(\pwc\teta{\tau_k})_k$ (in fact, the analogue
  of \eqref{cvT2} also holds for $\log(\upwc \teta{\tau_k})_k$). 
\par
Finally, under condition \eqref{hyp-K-stronger},  we  can also count on the  $\mathrm{BV}$-estimate \eqref{aprio7-discr} for $(\pwc\teta\tau)_\tau$. 
We may then apply \cite[Lemma 7.2]{DMDSMo06QEPL}, which generalizes the classical Helly Theorem to functions with values in the dual of a separable Banach  space,  and conclude  the pointwise convergence \eqref{cvT7}. 
Convergence \eqref{cvT6} follows from  estimate
 \eqref{aprio7-discr}  combined with \eqref{aprio6-discr},  via   an Aubin-Lions type compactness result for $\BV$-functions (see, e.g.,  \cite[Chap.\ 7, Cor.\ 4.9]{Roub05NPDE}).
\end{proof} 
\par
We are now in the position to develop the \underline{\bf proof of Theorem \ref{mainth:1}}.   Let $(\tau_k)$ be a  vanishing   sequence of time steps,  and let 
\[
(\pwc\teta{\tau_k},  \upwc\teta{\tau_k}, \pwl{\teta}{\tau_k}, \pwc u{\tau_k}, \pwl u{\tau_k}, \pwwll  u{\tau_k},  \pwc e{\tau_k}, \pwl e{\tau_k}, \pwc p{\tau_k}, \pwl p{\tau_k}, \pwc \zeta{\tau_k})_k,
\]
be a sequence of  solutions to the approximate PDE system \eqref{syst-interp} for which the convergences stated in Lemma \ref{l:compactness} hold to a quintuple $(\teta,u,e,p,\zeta)$.
We  will pass to the limit in the time-discrete versions of the momentum balance and of the plastic flow rule, in the discrete entropy inequality and in the discrete total energy inequality, to conclude that $(\teta,u,e,p)$ is an entropic solution to the thermoviscoplastic system
in the sense of Def.\ \ref{def:entropic-sols}.
\par
\noindent
\emph{Step $0$: ad the initial conditions \eqref{initial-conditions} and the kinematic admissibility \eqref{kin-admis}.} It was shown in Lemma \ref{l:compactness} that the limit triple $(u,e,p)$ is kinematically admissible. Passing to the limit in the initial conditions \eqref{discr-Cauchy}, on account of \eqref{complete-approx-e_0}   and of  the pointwise convergences \eqref{cvU3}, 
 \eqref{cvE3-bis},  \eqref{cvP3-bis},  and \eqref{dotu-quoted-later},
 we conclude that the triple $(u,e,p)$ comply with initial conditions 
 \eqref{initial-conditions}.
\par\noindent
\emph{Step $1$: ad the momentum balance \eqref{w-momentum-balance}.} Thanks to convergences \eqref{cvE1}--\eqref{cvE3} and \eqref{cvT1} we have that 
\begin{equation}
\label{cvS}
\pwc \sigma{\tau_k} = \bbD \pwl{\dot e}{\tau_k} + \bbC \pwc e{\tau_k} + \tau |\pwc e{\tau_k}|^{\gamma-2} \pwc e{\tau_k}  - \pwc{\teta}{\tau_k} \bbB \weakto \sigma = \bbD \dot e +\bbC e - \teta  \bbB  \qquad \text{ in 
$L^{\gamma/(\gamma-1)} (Q;\mt_\sym^{d\times d})$.}
\end{equation}
 Combining this with convergence \eqref{cvU4} and with \eqref{converg-interp-L} for $(\pwc \calL{\tau_k})_k$, we pass to the limit in the discrete momentum balance \eqref{eq-u-interp} and conclude that $(\teta,u,e)$ fulfill \eqref{w-momentum-balance} with test functions in $W_\Dir^{1,\gamma}(\Omega;\R^d)$.
 By comparison in  \eqref{w-momentum-balance}  we conclude that $\ddot{u} \in L^2(0,T; H_{\Dir}^1(\Omega;\R^d)^*)$, whence \eqref{reg-u}.  Moreover, a density argument yields that 
 \eqref{w-momentum-balance} holds with test functions in $H_{\Dir}^1(\Omega;\R^d)$. This concludes the proof of  the momentum
 balance. 
 \par\noindent 
\emph{Step $2$: ad the plastic flow rule \eqref{pl-flow}.}  Convergences \eqref{cvP2}--\eqref{cvP3}, \eqref{cvZ}, and \eqref{cvS} ensure that the functions $(\teta, e,p,\zeta)$ fulfill 
\begin{equation}
\label{prelim-flow-rule}
\zeta +\dot p  =  \sigma_\dev \qquad \aein\, Q.
\end{equation}
In order to conclude \eqref{pl-flow} it remains to show that $\zeta \in \partial_{\dot p} \mathrm{R}(\teta, \dot p)$ a.e.\ in $Q$, which  can be reformulated  via  \eqref{characterization-subdiff}.  In turn, the latter relations are equivalent to 
\begin{equation}
\label{def-subdiff}
\begin{cases}
\iint_Q \zeta{:} \eta \dd x \dd t \leq \int_0^T \calR(\teta(t), \eta(t)) \dd t \quad \text{for all } \eta \in L^2(Q;\mt_{\dev}^{d\times d}), 
\\
   \iint_Q \zeta{:} \dot{p} \dd x \dd t  \geq \int_0^T \calR(\teta(t), \dot{p}(t)) \dd t. 
   \end{cases}
\end{equation} 
To obtain \eqref{def-subdiff} we will pass to the limit in the analogous relations satisfied at level $k$, namely
\begin{equation}
\label{def-subdiff-k}
\begin{cases}
\iint_Q  \pwc\zeta{\tau_k} {:} \eta \dd x \dd t \leq \int_0^T   \calR(\upwc\teta{\tau_k}(t),  \eta(t)) \dd t \quad \text{for all } \eta \in L^2(Q;\mt_{\dev}^{d\times d}), \\  \iint_Q  \pwc\zeta{\tau_k}  {:}  \pwl { \dot p}{\tau_k}  \dd x \dd t  \geq \int_0^T  \calR(\upwc\teta{\tau_k}(t), \pwl { \dot p}{\tau_k}(t)) \dd t. \end{cases}
\end{equation}
 With this aim, we use conditions \eqref{hypR} on the dissipation metric $\mathrm{R}$.   In order to pass to the limit in the first  of \eqref{def-subdiff-k}  for a  fixed $\eta \in L^2(Q;\mt_{\dev}^{d\times d})$,   we use convergence \eqref{cvZ} for $(\pwc\zeta{\tau_k})_k$,   and the fact that 
\[
\lim_{k\to\infty} \iint_Q \mathrm{R}(\upwc \teta{\tau_k}, \eta) \dd x \dd t =   \iint_Q \mathrm{R}(\teta, \eta) \dd x \dd t.
\]
  The latter limit passage follows from convergence \eqref{cvuT5} for $\upwc\teta{\tau_k}$ which, combined with the continuity property \eqref{hypR-cont}, gives that 
$\mathrm{R}(\upwc \teta{\tau_k}, \eta) \to \mathrm{R}(\teta,\eta)$ almost everywhere in $Q$. Then we use 
 the dominated convergence theorem, taking into account that for every $k\in \N$ we have  $ \mathrm{R}(\upwc \teta{\tau_k}, \eta) \leq C_R |\eta|$ a.e.\ in $Q$ thanks to \eqref{linear-growth}. 
\par
As for the second inequality in \eqref{def-subdiff-k}, 
   we use  \eqref{hypR-lsc}  and the convexity of the map $\dot p \mapsto \mathrm{R}(\teta, \dot p)$, combined with convergences  \eqref{cvP2} and \eqref{cvuT5},  to conclude via the Ioffe theorem \cite{Ioff77LSIF} that 
\begin{equation}\label{lscR}
\liminf_{k\to \infty} \int_0^T \calR (\upwc \teta{\tau_k}(t), \pwl {\dot p}{\tau_k}(t)) \dd t \geq \int_0^T \calR(\teta(t), \dot{p}(t)) \dd t\,.
\end{equation}
Secondly, we show that 
\begin{equation} \label{limsup-cond}
\limsup_{k\to \infty}  \iint_Q   \pwc\zeta{\tau_k} {:} \pwl { \dot p}{\tau_k} \dd x \dd t \leq  \iint_Q   \zeta {:} { \dot p} \dd x \dd t.
\end{equation}
For \eqref{limsup-cond} we repeat the same argument developed to obtain \eqref{identifications-to-show} in the proof of Lemma \ref{l:3.6}, and observe that 
\[
\begin{aligned}  & 
\limsup_{k\to \infty}  \left( \iint_Q    \pwc\zeta{\tau_k} {:} \pwl { \dot p}{\tau_k} \dd x \dd t + \iint_Q   |\pwl { \dot p}{\tau_k}|^2 \dd x \dd t +  
\frac{\tau_k}\gamma \int_\Omega |\pwc p{\tau_k}(T)|^\gamma \dd x +
 \iint_Q  \bbD \pwl{\dot e}{\tau_k}{:}   \pwl{\dot e}{\tau_k}\dd x \dd t \right) 
\\ & \stackrel{(1)}{=}  \limsup_{k\to \infty}  \left(  \iint_Q \left( \bbD  \pwl{\dot e}{\tau_k} + \bbC \pwc e{\tau_k} +\tau  | \pwc e{\tau_k}|^{\gamma-2}  \pwc e{\tau_k} - \pwc\teta{\tau_k} \bbB  \right){:}  \pwl { \dot p}{\tau_k} \dd x \dd t  +   \iint_Q  \bbD \pwl{\dot e}{\tau_k}{:}   \pwl{\dot e}{\tau_k}\dd x \dd t \right)  + \dddn{\lim_{k\to\infty} \frac{\tau_k}{\gamma} \int_\Omega   |p_{\tau_k}^0|^\gamma \dd x }{$=0$}
\\ & 
\stackrel{(2)}\leq  \limsup_{k\to \infty}   \dddn{\iint_Q \left( \bbD  \pwl{\dot e}{\tau_k} + \bbC \pwc e{\tau_k} +\tau  | \pwc e{\tau_k}|^{\gamma-2}  \pwc e{\tau_k} - \pwc\teta{\tau_k} \bbB  \right){:}  \sig{\pwl { \dot u}{\tau_k} - \pwl {\dot w}{\tau_k}}  \dd x \dd t }{$=    \int_0^T \pairing{}{H_\Dir^1(\Omega;\R^d)^*}{ \pwc \calL{\tau_k} }{ \pwl {\dot u}{\tau_k} {-}\pwl {\dot w}{\tau_k}}\dd t  - \rho\iint_Q \partial_t \pwwll u{\tau_k} ( \pwl {\dot u}{\tau_k} {-}\pwl {\dot w}{\tau_k})) \dd x\dd t $  }+  \limsup_{k\to \infty}  \iint_Q \pwc{\sigma}{\tau_k} {:} \sig{\pwl {\dot w}{\tau_k}} \dd x \dd t   \\ & \quad -   \liminf_{k\to \infty}  \iint_Q \left(   \bbC \pwc e{\tau_k} +\tau  | \pwc e{\tau_k}|^{\gamma-2}  \pwc e{\tau_k} - \pwc\teta{\tau_k} \bbB  \right) {:}  \pwl { \dot e}{\tau_k} \dd x \dd t  
\\ & \stackrel{(3)}\leq \int_0^T \pairing{}{H_\Dir^1(\Omega;\R^d)^*}{ \calL}{ \dot{u} {-} \dot w} \dd t - \iint_Q\left( \rho    \ddot{u} ( \dot u {-}\dot w) {+} \sigma{:} \sig{\dot w} {+} \bbC e {:} \dot e {- }\teta \bbB {:} \dot e \right) \dd x \dd t\\ &  \stackrel{(4)}= \iint_Q \zeta{:} \dot p  \dd x \dd t + 
\iint_Q |\dot p|^2 \dd x \dd t + \iint_Q \bbD \dot e{:} \dot e \dd x \dd t, \end{aligned} 
\]
where (1) follows from testing the discrete flow rule \eqref{eq-p-interp} by $\pwl {\dot p}{\tau_k}$,  (2) from the kinematic admissibility condition, yielding $\pwl {\dot p}{\tau_k} = \sig{\pwl{\dot u}{\tau_k}} - \pwl {\dot e}{\tau_k} = \sig{\pwl{\dot u}{\tau_k}{-} \pwl{\dot w}{\tau_k}} - \pwl {\dot e}{\tau_k} + \sig{\pwl{\dot u}{\tau_k}}$, which also leads to the cancellation of the term $\iint_Q  \bbD \pwl{\dot e}{\tau_k}{:}   \pwl{\dot e}{\tau_k}$,
and from condition  \eqref{approx-e_0} on the sequence $(p_{\tau_k}^0)_k$. 
 The limit passage in  (3)  follows
 \begin{itemize}
 \item
  from convergence \eqref{converg-interp-L} for $(\pwc \calL{\tau_k})_k$,
  \item
   from \eqref{cvU1},
   \item
 from convergence   \eqref{converg-interp-w} for $(\pwl w{\tau_k})_k$, 
   combined with the stress convergence \eqref{cvS} and with \eqref{cvE3}, which yield
    \[
     \int_{0}^{\pwc{\mathsf{t}}{\tau}(t)} \int_\Omega  \tau |\pwc e{\tau_k}|^{\gamma-2} \pwc e{\tau_k} : \sig{\pwl {\dot w}{\tau_k}} \dd d \dd r
     = \int_{0}^{\pwc{\mathsf{t}}{\tau}(t)} \int_\Omega  \tau^{1-\alpha_w} |\pwc e{\tau_k}|^{\gamma-2} \pwc e{\tau_k} : \tau^{\alpha_w}\sig{\pwl {\dot w}{\tau_k}} \dd d \dd r  \to 0,
     \] 
so that 
    \begin{equation}
    \label{houston}
     \lim_{k\to\infty}   \int_{0}^{\pwc{\mathsf{t}}{\tau}(t)} \int_\Omega 
     \pwc \sigma{\tau_k} {:} \sig{\pwl {\dot w}{\tau_k}} \dd x \dd t 
     = \int_0^t \int_\Omega \sigma {:} \sig{\dot w} \dd x \dd t\,,
    \end{equation}
\item
  from \eqref{A-below}:
\begin{equation}
\label{A-below}
\begin{aligned}
& 
\limsup_{k\to\infty}\left(  -\iint_Q \rho \partial_t \pwwll u{\tau_k}  ( \pwl {\dot u}{\tau_k} {-} \pwl {\dot w}{\tau_k} )   \dd x\dd t \right) \\ &  \leq   - \liminf_{k\to\infty} \tfrac\rho2 \int_\Omega  |\pwl {\dot u}{\tau_k} (T)|^2 \dd x  +\tfrac\rho2 \int_\Omega  |\pwl {\dot u}{\tau_k} (0)|^2 \dd x  -  \lim_{k\to\infty} \rho \iint_Q   \partial_t \pwwll u{\tau_k}   \pwl {\dot w}{\tau_k}    \dd x\dd  t   \\ & \stackrel{(A)}{\leq} -\tfrac\rho2 \int_\Omega | \dot u(T)|^2 \dd x +\rho \int_\Omega |\dot{u}_0|^2 \dd x  -\rho \iint_Q \ddot{u} \dot w \dd x \dd t 
\end{aligned}
\end{equation}
with (A) due to \eqref{cvU4},   \eqref{converg-interp-w}, and \eqref{dotu-quoted-later},
\item from    \eqref{B-below}:
 \begin{equation}
\label{B-below}
\begin{aligned}
& 
\begin{aligned}
-\liminf_{k\to\infty}
\iint_Q    \bbC \pwc e{\tau_k}   {:}  \pwl { \dot e}{\tau_k} \dd x \dd t   &  \leq -  \liminf_{k\to\infty} \int_\Omega \tfrac12  \bbC  \pwc e{\tau_k} (T){:} \pwc e{\tau_k} (T) \dd x +   \int_\Omega \tfrac12  \bbC e_0{:} e_0 \dd x \\ &  \stackrel{(B)}{\leq}- \int_\Omega \tfrac12  \bbC  e(T){:} e(T) \dd x +   \int_\Omega \tfrac12  \bbC e_0{:} e_0, 
\end{aligned}
\\
& 
-\liminf_{k\to\infty}
\iint_Q     \tau  | \pwc e{\tau_k}|^{\gamma-2}  \pwc e{\tau_k}   {:}  \pwl { \dot e}{\tau_k} \dd x \dd t   \leq -  \liminf_{k\to\infty} \int_\Omega \tfrac\tau\gamma  | \pwc e{\tau_k} (T)|^\gamma\dd x +  \lim_{k \to\infty} \int_\Omega \tfrac{\tau_k}\gamma  | e_{\tau_k}^0|^\gamma  \dd x  \stackrel{(C)}{\leq} 0,
\\
& \lim_{k\to\infty} \iint_Q  \pwc\teta{\tau_k} \bbB   {:}  \pwl { \dot e}{\tau_k} \dd x \dd t \stackrel{(D)}{ = } \iint_Q \teta \bbB {:} \dot e \dd x \dd t,
\end{aligned}
\end{equation}
with (B) due to \eqref{cvE3-bis}, (C) due to  \eqref{cvE3} and \eqref{approx-e_0},  and (D) due to \eqref{cvE2} and \eqref{cvT5}. 
\end{itemize}
 Finally, (4) follows from testing \eqref{w-momentum-balance} by $\dot u - \dot w$, and \eqref{prelim-flow-rule} by $\dot p$. 
From the thus obtained $\limsup$-inequality, arguing in the very same way as in the proof of Lemma \ref{l:3.6},  we conclude that
\begin{equation}
\label{strong-convergences}
\begin{aligned}
& 
\lim_{k \to\infty} \iint_Q    \pwc\zeta{\tau_k} {:} \pwl { \dot p}{\tau_k} \dd x \dd t =   \iint_Q    \zeta {:} { \dot p} \dd x \dd t,
\\	
& 
\pwl{\dot p}{\tau_k} \to \dot p && \text{ in } L^2(Q;\mt_\dev^{d\times d}), 
\\
& 
\pwl{\dot e}{\tau_k} \to \dot e && \text{ in } L^2(Q;\mt_\sym^{d\times d})\,.
\end{aligned}
\end{equation}	
Hence, combining the first of \eqref{strong-convergences} with \eqref{lscR}, we take the limit in the second inequality in \eqref{def-subdiff-k}. 
  All in all, we deduce \eqref{def-subdiff}. Hence, the functions $(\teta, e,p,\zeta)$ fulfill the plastic flow rule  \eqref{pl-flow}. 
\par\noindent 
\emph{Step $3$:  enhanced convergences.}
For later use, observe that \eqref{strong-convergences}  give 
\begin{equation}
\label{cvE3-quater}
\pwl{e}{\tau_k} \to  e \quad  \text{ in } H^1(0,T;L^2(\Omega; \mt_\sym^{d\times d}))\,, \qquad 
\pwl{p}{\tau_k} \to  p \quad  \text{ in } H^1(0,T;L^2(\Omega; \mt_\dev^{d\times d}))\,.
\end{equation}
Moreover, by the kinematic admissibility condition we deduce the strong convergence of $\sig{\pwl {\dot u}{\tau_k}}$ in $L^2(Q;\mt_\sym^{d\times d})$,
 hence, by Korn's inequality,  
 \begin{equation}
\label{cvU-quater}
\pwl u{\tau_k} \to u \quad  \text{ in } H^1(0,T; H^1(\Omega;\R^d)).
\end{equation}
Finally, repeating the $\limsup$ argument leading to \eqref{strong-convergences} on a generic interval $[0,t]$, we find that 
\begin{equation}
\label{cvU-quinque}
\pwl{\dot u}{\tau_k}(t) \to \dot{u}(t) \quad \text{ in } L^2(\Omega;\R^d) \quad \text{ for every } t \in [0,T].
\end{equation}
All in all, also on account of \eqref{cvE3} and \eqref{cvP3},   we have that convergence \eqref{cvS} improves to a  strong one. 
	Therefore, from \eqref{eq-p-interp} we deduce that 
\[
\pwc \zeta{\tau_k} = (\pwc \sigma{\tau_k})_\dev - \pwl{\dot{p}}{\tau_k} - \tau_k |\pwc p{\tau_k}|^{\gamma-2} \pwc p{\tau_k} \to \sigma_\dev - \dot p = \zeta \qquad \text{a.e.\ in } Q.
\]
We will use this to pass to the limit in the pointwise inequality 
\[
\pwc \zeta{\tau_k}(t,x){:} \left( \dot{p}(t,x){-} \pwl{\dot{p}}{\tau_k}(t,x) \right) +\mathrm{R}(\upwc \teta{\tau_k}(t,x),  \pwl{\dot{p}}{\tau_k}(t,x)) \leq  \mathrm{R}(\upwc \teta{\tau_k}(t,x),  \dot p(t,x)) \qquad \foraa (t,x)\in Q.
\]
Indeed, in view of \eqref{strong-convergences}, which gives $\lim_{k\to\infty} \pwc \zeta{\tau_k}{:} ( \dot{p}{-} \pwl{\dot{p}}{\tau_k}) =0 $ a.e.\ in $Q$,  of convergence \eqref{cvuT5} for $\upwc \teta{\tau_k}$, and of  the continuity property \eqref{hypR-cont}, from the above inequality we conclude that 
\[
\limsup_{k\to\infty}  \mathrm{R}(\upwc \teta{\tau_k}(x,t),  \pwl{\dot{p}}{\tau_k}(x,t)) \leq  \mathrm{R}(\teta(x,t),  \dot p(x,t)) \qquad \foraa (x,t) \in Q.
\]
Combining this with the lower semicontinuity inequality which derives from \eqref{hypR-lsc}, we  ultimately have that $\mathrm{R}(\upwc \teta{\tau_k},  \pwl{\dot{p}}{\tau_k}) \to  \mathrm{R}(\teta,  \dot p) $ a.e.\ in $Q$, hence
\begin{equation}
\label{dominated-R}
   \mathrm{R}(\upwc\teta{\tau_k}, \pwl {\dot p}{\tau_k}) \to \mathrm{R}(\teta, \dot p) \quad \text{ in } L^2 (Q)
   \end{equation}
   by the dominated convergence theorem.
\par\noindent 
\emph{Step $4$: ad the entropy inequality \eqref{entropy-ineq}.} Let us fix a positive test  function $\varphi \in \rmC^0 ([0,T]; W^{1,\infty}(\Omega)) \cap H^1(0,T; L^{6/5}(\Omega))$ for \eqref{entropy-ineq}, and approximate it with the discrete test functions from \eqref{discrete-tests-phi}: their interpolants $\pwc \varphi\tau, \, \pwl \varphi\tau$ comply with convergences \eqref{convergences-test-interpolants} and with the discrete entropy inequality \eqref{entropy-ineq-discr}, where we pass to the limit. We take the limit of the first integral term on the left-hand side of  \eqref{entropy-ineq-discr} based on convergence \eqref{cvT2} for $\log(\upwc\teta{\tau_k})$.  
\par For the second integral term, 
we will prove that 
\begin{equation}
\label{weak-nabla-logteta}
\condu (\pwc\teta{\tau_k}) \nabla \log(\pwc\teta{\tau_k}) \weakto \condu(\teta) \nabla \log(\teta) \qquad \text{ in } 
  L^{1+\bar\delta}(Q;\R^d)   \text{ with   $ \bar\delta = \frac{\alpha}\mu $ 
and $\alpha \in [(2-\mu)^+, 1)$.}  \end{equation}
 First of all,  let us prove that $(\condu (\pwc\teta{\tau_k}) \nabla \log(\pwc\teta{\tau_k}))_k $  is bounded in $L^{1+\bar\delta}(Q;\R^d)  $. To this aim, we argue as in the proof of the \emph{Fifth a priori estimate} from Prop.\ \ref{prop:aprio} and observe that 
\[
\left| \condu (\pwc\teta{\tau_k}) \nabla \log(\pwc\teta{\tau_k}) \right|  \leq C 
\left(|\pwc\teta{\tau_k}|^{\mu-1} +\frac1{\bar \teta} \right)    |\nabla \pwc\teta{\tau_k} | \qquad \aein \, Q,
\]
by the growth condition \eqref{hyp-K} and the positivity \eqref{strong-strict-pos}. Let us now focus on the first term on  the r.h.s.:  with H\"older's inequality we   have that, for a positive exponent $r$, 
\[
\begin{aligned} 
\iint_Q  \left(|\pwc\teta{\tau_k}|^{\mu-1}     |\nabla \pwc\teta{\tau_k} |  \right)^r \dd x \dd t
& 
 \leq \| ( |\pwc\teta{\tau_k}|^{(\mu-\alpha)/2} )^r\|_{L^{2/(2{-}r)}(Q)} 
\|(   |\pwc\teta{\tau_k}|^{(\mu+\alpha-2)/2} |\nabla\pwc\teta{\tau_k}| )^r\|_{L^{2/r}(Q;\R^d)}  \\ & \leq  C  \| ( |\pwc\teta{\tau_k}|^{(\mu-\alpha)/2} )^r\|_{L^{2/(2{-}r)}(Q)},
\end{aligned}
\]
where the second inequality follows from the estimate for 
$ |\pwc\teta{\tau_k}|^{(\mu+\alpha-2)/2} \nabla\pwc\teta{\tau_k} $ in $L^2(Q;\R^d)$ thanks to \eqref{est-temp-added?-bis}. The latter also yields a bound for $ (\pwc\teta{\tau_k})^{(\mu+\alpha)/2} $ in $L^{2}(Q)$,  hence an estimate for  $ (\pwc\teta{\tau_k})^{(\mu-\alpha)/2} $ in $L^{2(\mu+\alpha)/(\mu-\alpha)}(Q)$. Therefore, for $r = (\mu+\alpha)/\mu = 1+\alpha/\mu$ we obtain that $  \| ( |\pwc\teta{\tau_k}|^{(\mu-\alpha)/2} )^r\|_{L^{2/(2{-}r)}(Q)} \leq C$, and the estimate for $\condu (\pwc\teta{\tau_k}) \nabla \log(\pwc\teta{\tau_k}) $ follows. 
For  the proof of  convergence \eqref{weak-nabla-logteta}, relying on convergences \eqref{cvT1}--\eqref{cvT5},   we refer to 
\cite[Thm.\ 1]{Rocca-Rossi}.  
Therefore we conclude the first of \eqref{further-logteta}. 
\par
To take the limit in the right-hand side terms in  the entropy inequality \eqref{entropy-ineq-discr}, for the first two integrals we use convergence \eqref{cvT4} combined with \eqref{convergences-test-interpolants}. A lower semicontinuity argument also based on the Ioffe theorem \cite{Ioff77LSIF} and on convergences \eqref{convergences-test-interpolants},   \eqref{cvT2},   and \eqref{cvT5} gives that
\[
\begin{aligned} 
\limsup_{k\to\infty} \left(  -  \int_{\pwc{\mathsf{t}}{\tau_k}(s)}^{\pwc{\mathsf{t}}{\tau_k}(t)}  \int_\Omega \condu(\pwc \teta{\tau_k}) \frac{\pwc\varphi{\tau_k}}{\pwc \teta{\tau_k}}
\nabla \log(\pwc \teta{\tau_k})  \nabla \pwc \teta{\tau_k} \dd x \dd r \right) 
 & = - \liminf_{k\to\infty} \int_{\pwc{\mathsf{t}}{\tau_k}(s)}^{\pwc{\mathsf{t}}{\tau_k}(t)}  \int_\Omega \condu(\pwc \teta{\tau_k}) \pwc\varphi{\tau_k} | \nabla  \log(\pwc \teta{\tau_k}) |^2 \dd x \dd r
\\
  & \leq  -  \int_s^t \int_\Omega \condu(\teta) \varphi
|\nabla \log(\teta)|^2\dd x \dd r,
\end{aligned}
\] 
which allows us to deal with  the third integral term on the r.h.s.\ of   \eqref{entropy-ineq-discr}. 
 We take the limit of the fourth  integral term taking into account convergences \eqref{converg-interp-g},    \eqref{cvT5}, 
 which yields
 \[
 \frac1{\pwc\teta{\tau_k}} \to \frac1\teta \qquad \text{ in } L^p(Q) \quad \text{for all } 1 \leq p<\infty,
 \]
 since $\left| \frac1{\pwc\teta{\tau_k}} \right| \leq \frac1{\bar\teta}$ a.e.\ in $Q$,  as well as 
 the previously established strong  convergences \eqref{strong-convergences} and \eqref{dominated-R}. 
 Finally,
 since $\nabla \left( \tfrac1{\pwc\teta{\tau_k}} \right) =
 \tfrac{\nabla  \pwc\teta{\tau_k}}{|\pwc\teta{\tau_k}|^2}$, combining \eqref{strong-strict-pos} with estimate 
 \eqref{aprio6-discr} we infer that $( \tfrac1{\pwc\teta{\tau_k}} )_k$ is  bounded in $L^2(0,T;H^1(\Omega))$. 
 All in all, we have 
  \begin{equation}
 \label{weak-1-teta}
  \frac1{\pwc\teta{\tau_k}} \weakto \frac1\teta \qquad \text{ in } L^2(0,T;H^1(\Omega)),
\end{equation}
which allows us to pass to the limit in the fifth integral term, in combination with convergence \eqref{converg-interp-h}.
 \par
 Ultimately, we establish the summability property $\condu(\teta) \nabla \log(\teta)$ in $L^{1}(0,T;X)$, with $X$ from  \eqref{further-logteta}, by combining the  facts that $\teta^{(\mu+\alpha-2)/2}\nabla\teta \in L^2(Q;\R^d)$ thanks to convergence \eqref{cvT8}, with the information that $\teta^{(\mu-\alpha)/2} \in L^2(0,T; H^1(\Omega))$ by \eqref{cvT9},  and by arguing in the very same way as in the proof of the \emph{Fifth a priori estimate} from Prop.\ \ref{prop:aprio}.  In view of  \eqref{further-logteta}, the entropy inequality \eqref{entropy-ineq}
  in fact makes sense for all positive test functions $\varphi $ in $H^1(0,T; L^{6/5}(\Omega)) \cup L^\infty(0,T;W^{1,d+\epsilon}(\Omega))$ with $\epsilon>0$. Therefore, with a density argument we conclude it for this larger test space.
 \par\noindent \emph{Step $5$: ad the total energy inequality \eqref{total-enineq}.} It is deduced by passing to the limit in the discrete total energy inequality \eqref{total-enid-discr}. For  the first integral term on the left-hand side, we use  that $\pwl {\dot u}{\tau_k}(t) \weakto \dot u(t)$  in $L^2(\Omega;\R^d)$ for all $t\in [0,T]$, cf.\ \eqref{dotu-quoted-later}.   For the second term we observe that 
 $\liminf_{k\to\infty} \calE_{\tau_k} (\pwc\teta{\tau_k}(t), \pwc e{\tau_k}(t)) \geq \calE (\teta(t), e(t))$  for \emph{almost all} $t\in (0,T)$ by convergence \eqref{cvT5} for $\pwc\teta{\tau_k}$ and 
by \eqref{cvE3-quater}, combined with \eqref{stability-e-k}.  The limit passage on the right-hand side, for almost all $s \in (0,t)$,  follows from \eqref{cvU-quater},  again \eqref{cvT5} and  \eqref{cvE3-quater}, from \eqref{houston},
and from convergences
\eqref{convs-interp-data} and 
for the interpolants $(\pwc H{\tau_k})_k, \, (\pwc h{\tau_k})_k, \, (\pwc \calL{\tau_k})_k, \, (\pwc w{\tau_k})_k$. 
\par\noindent
This concludes the proof of Theorem \ref{mainth:1}.
   \QED
   \medskip 
   \par
   \noindent
   We now briefly sketch the   \underline{\bf proof of Theorem \ref{mainth:2}}.  The limit passage in the discrete momentum balance and in the plastic flow rule, cf.\ \eqref{eq-u-interp}
   and \eqref{eq-p-interp},  follows from the arguments in the proof of Thm.\ \ref{mainth:1}.
   \par
   As for the heat equation, we shall  as a  first step prove that the limit quadruple $(\teta, u,e,p)$ complies with 
   \begin{equation}
\begin{aligned}
\label{eq-teta-interm}   &
\pairing{}{W^{1,\infty}(\Omega)}{\teta(t)}{\varphi(t)}
-\int_0^t\int_\Omega \teta \varphi_t \dd x \dd s     +\int_0^t \int_\Omega \condu(\teta) \nabla \teta\nabla\varphi \dd
x \dd s  
\\ &\quad  = \int_\Omega \teta_0 \varphi(0) \dd x  +
\int_0^t\int_\Omega \left(H+ \mathrm{R}(\teta, \dot p)+ |\dot p|^2 +\bbD \dot e {:} \dot e -\teta \bbB \dot e \right)\varphi \dd x \dd s  +  \int_0^t \int_\Omega \int_{\partial\Omega} h \varphi \dd S \dd s\,.
\end{aligned}
\end{equation}
for all   test functions $
\varphi\in \rmC^0([0,T]; W^{1,\infty}(\Omega))\cap
H^1(0,T;L^{6/5}(\Omega))  $ and for all  $ t\in (0,T]. $
With this aim,  we  pass to the limit in the approximate temperature equation \eqref{eq-teta-interp}, tested by the approximate test functions from \eqref{discrete-tests-phi}, where we integrate by parts in time   the term $\int_0^t \int_\Omega \pwl{\dot \teta}{\tau_k} \pwc{\varphi}{\tau_k}\dd x \dd r$. For this limit passage, we exploit  convergences \eqref{convergences-test-interpolants}
    as well as 
   \eqref{cvT6} for $(\upwc \teta{\tau_k})_k$  and \eqref{cvT7}.
   \par
  For the limit passage in the term $\iint_Q \condu(\pwc \teta{\tau_k}) \nabla  \pwc \teta{\tau_k} \nabla \pwc \varphi{\tau_k} \dd x \dd t$ we prove that 
   $\condu(\pwc \teta{\tau_k}) \nabla  \pwc \teta{\tau_k}  \weakto \condu(\teta) \nabla \teta$ in $L^{1+\tilde\delta}(Q;\R^d)$, with $\tilde \delta>0$ given by \eqref{further-k-teta}.  Let us check the bound 
   \begin{equation}
   \label{bound-tilde-delta}
   \| \condu(\pwc \teta{\tau_k}) \nabla  \pwc \teta{\tau_k} \|_{L^{1+\tilde\delta}(Q;\R^d)} \leq C, 
   \end{equation}
by again resorting to estimates \eqref{aprio6-discr} and \eqref{est-temp-added?-bis}.  Indeed, by \eqref{hyp-K} we have that 
   \begin{equation}
   \label{4.47}
   | \condu(\pwc \teta{\tau_k}) \nabla  \pwc \teta{\tau_k} | \leq C| \pwc \teta{\tau_k}|^{(\mu-\alpha+2)/2}  | \pwc \teta{\tau_k}|^{(\mu+\alpha-2)/2}  |\nabla  \pwc \teta{\tau_k} |+ C | \nabla  \pwc \teta{\tau_k} | \qquad \aein\, Q,
   \end{equation}
and  we estimate the first term on the r.h.s.\ by observing that  $| \pwc \teta{\tau_k}|^{(\mu+\alpha-2)/2}  |\nabla  \pwc \teta{\tau_k} |$ is bounded in $L^2(Q)$ thanks to   \eqref{est-temp-added?-bis}. On the other hand, in the case $d=3$, to which we confine the discussion,  by  interpolation arguments $\pwc\teta{\tau_k}$ is bounded in $L^h(Q)$ for every $1 \leq h<\frac83$. Therefore, for $\alpha>\mu-\frac23$ (so that  $\mu-\alpha+2<\frac83$), the functions   $(| \pwc \teta{\tau_k}|^{(\mu-\alpha+2)/2})_k$ are bounded in $L^r(Q)$ 
with $1\leq r< \frac{16}{3(\mu-\alpha+2)}$. Then, \eqref{bound-tilde-delta} follows from  \eqref{4.47} via the H\"older inequality. The corresponding weak convergence can be  proved
 arguing in the very same way as in the proof of \cite[Thm.\ 2]{Rocca-Rossi}, to which we refer the reader.  
 Therefore we conclude that 
 $\condu(\teta) \nabla \teta \in L^{1+\tilde\delta}(Q;\R^d)$. Observe that $\condu(\teta) \nabla \teta = \nabla (\hat{\condu}(\teta))$ thanks to \cite{Marcus-Mizel}.  Since $\hat{\condu}(\teta)$ itself is a function in $L^{1+\tilde\delta}(Q) $ (for $d=3$, this follows from the fact that $\hat{\condu}(\teta) \sim \teta^{\mu+1} \in L^{h/(\mu+1)}(Q)$ for  every $1 \leq h<\frac83$), we conclude \eqref{further-k-teta}.
 \par
 The limit passage on the r.h.s.\ of  the discrete heat equation \eqref{eq-teta-interp} results from   \eqref{converg-interp-g}, from \eqref{cvT6},  the strong convergences \eqref{strong-convergences}, and \eqref{dominated-R}. 
 \par
 All in all, we obtain \eqref{eq-teta-interm}, whence  for every $\varphi \in W^{1,\infty}(\Omega)$  and every $0 \leq s \leq t \leq T$
 \[
 \begin{aligned}
 &
\pairing{}{W^{1,\infty}(\Omega)}{\teta(t)-\teta(s)}{\varphi}  \\ &   = 
-\int_s^t \int_\Omega \condu(\teta) \nabla \teta\nabla\varphi \dd
x \dd r
+
\int_s^t\int_\Omega \left(H+ \mathrm{R}(\teta, \dot p)+ |\dot p|^2 +\bbD \dot e {:} \dot e -\teta \bbB \dot e \right) \varphi \dd x \dd r +  \int_s^t \int_\Omega \int_{\partial\Omega} h \varphi \dd S \dd r\,.
\end{aligned}
\]
 From this we easily conclude 
 the enhanced regularity \eqref{enh-teta-W11}.  Thanks to \cite[Thm.\ 7.1]{DMDSMo06QEPL}, the absolutely continuous function $\teta:  [0,T] \to  W^{1,\infty}(\Omega)^*$ admits at almost all $t\in (0,T)$  the derivative $\dot{\teta}(t)$, which turns out to be the limit  as $h \to 0$ of the incremental quotients $\frac{\teta(t+h)-\teta(t)}h$, w.r.t.\ the  weak$^*$-topology of $W^{1,\infty}(\Omega)^*$. Therefore, the enhanced weak formulation of the heat equation \eqref{eq-teta} follows.
 \par
 Recall (cf.\ the comments following the statement of Thm.\ \ref{mainth:2}) that $\tilde{\delta}$ is small enough as to ensure that $W^{1,1+1/\tilde{\delta}}(\Omega) \subset L^\infty(\Omega)$. Therefore, the terms on the r.h.s.\ of  the heat equation \eqref{heat} can be multiplied by test functions $\varphi \in W^{1,1+1/\tilde{\delta}}(\Omega)$. 
 Thanks to  \eqref{further-k-teta}, also the second term on the l.h.s.\ of   \eqref{heat} admits such test functions. Therefore by comparison we conclude that $\dot{\teta} \in L^1(0,T;   W^{1,1+1/\tilde{\delta}}(\Omega)^*)$ and,  with a density argument, extend  the weak formulation \eqref{eq-teta} to this (slightly) larger  space of test functions. 
This finishes the proof of Theorem \ref{mainth:2}. 
\QED
\begin{remark}[Energy convergences for the approximate solutions]
\label{rmk:energy-conv}
\upshape
As a by-product of the proofs of Theorems \ref{mainth:1} and \ref{mainth:2}, we  improve  convergences \eqref{convergences-cv}  of the approximate solutions  to an entropic/weak energy solution of the thermoviscoplastic system. More specifically, it follows from \eqref{cvT5}  and \eqref{strong-convergences}--\eqref{cvU-quinque} that we have the convergence of the kinetic energies
\[
\frac{\varrho}2 \int_\Omega |\pwl {\dot u}{\tau_k}(t)|^2 \dd x \to  \frac{\varrho}2 \int_\Omega  |\dot u(t)|^2 \dd x \qquad \text{for all } t \in [0,T],
\]
of the 
dissipated energies
\[
\int_0^T \int_\Omega \bbD \pwl {\dot e}{\tau_k}{:} \pwl {\dot e}{\tau_k} \dd x  \dd t \to  \int_0^T \int_\Omega \bbD \dot e{:} \dot e \dd x  \dd t, \qquad \int_0^T \calR (\upwc \teta{\tau_k}, \pwl {\dot p}{\tau_k}) \dd t \to \int_0^T \calR (\teta, \dot p) \dd t,
\] and of the thermal and mechanical energies
\[
\calF(\pwc \teta{\tau_k}(t)) \to \calF(\teta(t)) \quad \foraa t \in (0,T), \qquad \calQ(\pwc e{\tau_k}) \to \calQ(e) \qquad \text{uniformly in } [0,T]. 
\]
\end{remark}
\section{\bf Setup for the perfectly plastic system}
\label{s:5}
As already mentioned in the Introduction,  in the vanishing-viscosity limit  of the thermoviscoplastic system  we will obtain the \emph{(global) energetic formulation} for the perfectly plastic system,  coupled with  the stationary limit of the heat equation. Prior to performing this asymptotic analysis, in this section we gain further insight into the concept of energetic solution for  perfect plasticity. 
\par
For the energetic formulation to be fully meaningful,  in Sec.\ \ref{ss:5.1} we need to strengthen the  assumptions,  previously given in Section \ref{ss:2.1},
on the reference configuration $\Omega$,
 on the elasticity tensor $\bbC$,  
 and on the elastic domain $
x  \in \Omega \rightrightarrows K(x) \subset \mt_\dev^{d\times d}$ 
(indeed,
 we will drop the dependence of $K$ on the -spatially and temporally- nonsmooth variable $\teta$). 
 Instead,  we will weaken the regularity requirements on the Dirichlet loading $w$. 
 \par
Preliminarily, let us recall some basic facts about the space of functions with bounded deformation in  $\Omega$.
\paragraph{{\em The space $\BD(\Omega;\R^d)$}} It  is defined by
\begin{equation}
\BD(\Omega;\R^d): = \{ u \in L^1(\Omega;\R^d)\, : \ \sig{u} \in\rmM(\Omega;\mt_\sym^{d\times d}) \},
\end{equation}
with $\rmM(\Omega;\mt_\sym^{d\times d})$ the space of Radon measures on $\Omega$ with values in $\mt_\sym^{d\times d}$,  with norm $\| \lambda\|_{\rmM(\Omega;\mt_\sym^{d\times d}) }: = |\lambda|(\Omega)$ and  $|\lambda|$ the variation of the measure. Recall that, by the Riesz representation theorem, 
 $\rmM(\Omega;\mt_\sym^{d\times d})$ can be identified with the dual of the space $\mathrm{C}_0(\Omega;\mt_\sym^{d\times d})$ of the continuous functions $\varphi: \Omega \to \mt_\sym^{d\times d}$ such that the sets $\{ |\varphi|\geq c\}$ are compact for every $c>0$. 
The space $ \BD(\Omega;\R^d)$ is endowed with the graph norm
\[
\| u \|_{\BD(\Omega;\R^d)}: = \| u \|_{L^1(\Omega;\R^d)}+ \| \sig{u}\|_{\rmM(\Omega;\mt_\sym^{d\times d}) },
\]
which makes it a Banach space. It turns out that  $\BD(\Omega;\R^d)$ is the dual of a normed space, cf.\ \cite{Temam-Strang80}.
\par
In addition to the strong convergence induced by $\norm{\cdot}{\BD(\Omega;\R^d)}$, this duality defines  a notion of weak$^*$ convergence on $\BD(\Omega;\R^d)$ :
 a sequence $(u_k)_k$ converges weakly$^*$ to $u$ in $\BD(\Omega;\R^d)$ if $u_k\weakto u$ in $L^1(\Omega;\R^d)$ and $\sig{u_k}\weaksto \sig{u}$ in $ \rmM(\Omega;\mt_\sym^{d\times d})$. Every bounded sequence in $\BD(\Omega;\R^d)$ has  a weakly$^*$ converging subsequence and, furthermore, a subsequence converging weakly in $L^{d/(d{-}1)}(\Omega;\R^d)$ and strongly in $L^{p}(\Omega;\R^d)$  for every $1\leq p <\frac d{d-1}$. 
\par
 Finally, we recall that for  every $u \in \BD(\Omega;\R^d)$  the trace $u|_{\partial\Omega}$ is well defined as an element in $L^1(\partial\Omega;\R^d)$, and that (cf.\ \cite[Prop.\ 2.4, Rmk.\ 2.5]{Temam83}) a Poincar\'e-type inequality holds:
\begin{equation}
\label{PoincareBD}
\exists\, C>0  \ \ \forall\, u \in \BD(\Omega;\R^d)\, : \ \ \norm{u}{L^1(\Omega;\R^d)} \leq C \left( \norm{u}{L^1(\Gamma_\Dir;\R^d)} + \norm{\sig u}{\rmM(\Omega;\mt_\sym^{d\times d})}\right)\,.
\end{equation}
\subsection{Setup}
\label{ss:5.1}
\par
Let us now detail the basic assumptions on the data of the perfectly plastic system. We postpone to the end of Section \ref{ss:5.2} a series of   comments on the outcome of the conditions given below, as well as on  the possibility of weakening some of them.
\paragraph{{\em The reference configuration}.} 
For technical reasons related to the definition of the stress-strain duality,  cf.\ Remark \ref{rmk:stress-strain-dual} later on,  in addition to  conditions \eqref{Omega-s2} required in Sec.\ \ref{ss:2.1}, we will suppose from now on that 
\begin{equation}
\label{bdriesC2}
\tag{5.$\Omega$}
\partial\Omega \text{ and } \partial\Gamma \text{ are of class } \rmC^2\,.
\end{equation}
The latter requirement means that for every $x \in \partial\Gamma$ there exists a $ \rmC^2$-diffeomorphism defined in an open neighborhood of $x$ that maps $\partial\Omega$ into a $(d{-}1)$-dimensional plane, and $\partial\Gamma$ into a $(d{-}2)$-dimensional plane.
\paragraph{{\em Kinematic admissibility and stress}.}  Given a function $w\in H^1(\Omega;\R^d)$, we say that a triple $(u,e,p)$ is 
 \emph{kinematically admissible with boundary datum $w$ for the perfectly plastic system} (\emph{kinematically admissible}, for short), and write $(u,e,p) \in \mathcal{A}_{\BD}(w)$,  if
\begin{subequations}
\label{kin-adm-BD}
\begin{align}
&
\label{kin-adm-BD-a}
u \in \BD(\Omega;\R^d), \quad e \in L^2(\Omega;\mt_\sym^{d\times d}), \quad p \in \rmM(\Omega {\cup} \Gamma_\Dir;\mt_\dev^{d\times d}),
\\
& 
\label{kin-adm-BD-b}
\sig u = e+p,
\\
\label{kin-adm-BD-c}
& 
p = (w -u) \odot \nu \mathscr{H}^{d-1} \quad \text{on } \Gamma_\Dir. 
\end{align}
\end{subequations} 
  Observe that 
  \eqref{kin-adm-BD-a}
reflects the fact that the plastic strain is now a measure that can concentrate on Lebesgue-negligible sets. 
Furthermore,   \eqref{kin-adm-BD-c} relaxes the Dirichlet condition $w=u$ on $\Gamma_\Dir$ imposed  by the kinematic admissibility condition \eqref{kin-adm} and represents a plastic slip (mathematically described by the singular part of the measure $p$) occurring on 
$\Gamma_\Dir$. It can be checked that  $\calA(w) \subset \calA_{\BD}(w)$.
In the proof of Theorem \ref{mainth:3} we will make use of the following closedness property, proved in \cite[Lemma 2.1]{DMDSMo06QEPL}.
\begin{lemma}
\label{l:closure-kin-adm}
Assume \eqref{Omega-s2}  and \eqref{bdriesC2}. Let $(w_k)_k\subset H^1(\Omega;\R^d)$ and $(u_k,e_k,p_k )\in  \mathcal{A}_{\BD}(w_k)$ for every $k\in \N$. Assume that 
\begin{equation}
\label{ptw-convs}
\begin{aligned}
&
w_k\weakto w_\infty \text{ in } H^1(\Omega;\R^d), \qquad u_k\weaksto u_\infty \text{ in } \BD(\Omega;\R^d), 
\\
& e_k\weakto e_\infty \text{ in } L^2(\Omega;\mt_\sym^{d\times d}), \qquad p_k \weaksto p_\infty \text{ in }  \rmM(\Omega \cup \Gamma_\Dir;\mt_\dev^{d\times d}).
\end{aligned}
\end{equation}
Then, $(u_\infty, e_\infty, p_\infty) \in \mathcal{A}_{\BD}(w_\infty)$.
\end{lemma}
\par
In the perfectly plastic system, the stress is given by $\sigma = \bbC e$, cf.\ \eqref{stress-RIP}. Following \cite{DMDSMo06QEPL, Sol09, FraGia2012, Sol14},  in addition to \eqref{elast-visc-tensors}, we suppose that    for almost all  $x\in \Omega$ the  elastic tensor $\bbC(x) \in \mathrm{Lin}(\mt_\sym^{d\times d})$ maps the orthogonal spaces $\mt_\dev^{d\times d}$ and $\R I$ into themselves. Namely, there exist functions
\begin{equation}
\label{bbC-s:5}
\tag{5.$\mathrm{T}$}
\bbC_\dev \in L^\infty (\Omega;  \mathrm{Lin}(\mt_\sym^{d\times d}) )  \text{ and }  \eta \in L^\infty(\Omega;\R^+) \text{ s.t. }  \forall\, A \in \mt_\sym^{d\times d} \ \ 
\bbC(x)A = \bbC_\dev(x) A_\dev + \eta(x) \mathrm{tr}(A) I,
\end{equation}
with $I $ the identity matrix. 
\par 
\paragraph{{\em Body force, traction, and Dirichlet loading}.}  Along the footsteps of  \cite{DMDSMo06QEPL}, we  enhance our conditions on $F$ and $g$ (cf.\ \eqref{data-displ}), by requiring that 
\begin{equation}
\label{data-displ-s5}
 \tag{5.$\mathrm{L}_1$}
 F 
\in \mathrm{AC} ([0,T]; L^d(\Omega;\R^d)),\qquad g \in \AC ([0,T]; L^\infty (\Gamma_\Neu;\R^d)).
\end{equation}
Therefore, for every $t\in [0,T]$  the element $\calL(t)$ defined by \eqref{total-load} belongs to $\BD(\Omega;\R^d)^*$, and moreover (see \cite[Rmk.\ 4.1]{DMDSMo06QEPL})
$\dot{\calL}(t)$ exists in  $ \BD(\Omega;\R^d)^*$
 for almost all $t\in (0,T)$.
 Furthermore,  we shall strengthen the   safe load condition from \eqref{safe-load} to
 \begin{equation}
 \label{safe-load-s5}
 \varrho \in W^{1,1}(0,T; L^2(\Omega;\mt_\sym^{d\times d})) \quad \text{ and  } \varrho_\dev \equiv 0,
 \end{equation}
 cf.\ Remark \ref{rmk:honest} below. 
 \begin{remark}
 \label{rmk:honest}
 \upshape
 The second of \eqref{safe-load-s5} is required only for consistency with 
  the upcoming conditions \eqref{safe-load-eps}  on the stresses $(\varrho_\eps)_\eps$, associated    via the safe load condition with  the 
forces $(F_\eps)_\eps$ and $(g_\eps)_\eps$  for the thermoviscoplastic systems approximating the perfectly plastic one. 
The feasibility of \eqref{safe-load-s5} is completely open, though. Hence, we might as well 
 confine the discussion to the case the body force $F$ and the assigned traction $g$ are null. We have chosen not to do so because
 \eqref{safe-load-s5} is the natural counterpart to  \eqref{safe-load-eps}.
 \end{remark}
Further, we consider  the body to be   solicited by a hard device $w$ on the Dirichlet boundary $\Gamma_\Dir$, for which  we suppose
\begin{equation}
\label{Dir-load-Sec5}
\tag{5.$\mathrm{W}$}
w \in \mathrm{AC}([0,T]; H^1(\Omega;\R^d)),
\end{equation}
which is  a weaker requirement than \eqref{Dirichlet-loading}.
\paragraph{{\em The plastic dissipation.}} Since the plastic strain (and, accordingly, the plastic strain rate) is  a measure on $\Omega \cup \Gamma_
\Dir$, from now on we will suppose that the multifunction $K$ is defined on $\Omega \cup \Gamma_\Dir$. Furthermore, 
following \cite{Sol09}, we will require that 
\begin{equation}
\label{ass-K1-sec-5}
\tag{5.$\mathrm{K}_1$}
K : \Omega \cup \Gamma_\Dir \rightrightarrows \mt_\dev^{d\times d}   \text{ is continuous }
\end{equation}
in the sense specified by \eqref{multifuncts-props} and that 
\begin{equation}
\label{ass-K2-sec-5}
\tag{5.$\mathrm{K}_2$}
\begin{aligned}
&
   K(x)  \  \text{ is a convex and compact subset of $\mt_\dev^{d\times d}$ for all } x \in \Omega \cup \Gamma_\Dir \text{ and} \\ 
&
\exists\, 0<c_r<C_R \  \ \ \forall\, x \in   \Omega \cup \Gamma_\Dir\, :   \ \ \  B_{c_r}(0) \subset K(x) \subset B_{C_R}(0).
 \end{aligned}
\end{equation}
In order to state the stress constraint $\sigma_\dev \in K$ a.e.\ in $\Omega$ 
 in a more compact form, we also introduce the set
\begin{equation}
\label{realization-K}
\calK(\Omega): = \{ \zeta \in L^2(\Omega; \mt_\dev^{d\times d})\, : \zeta(x) \in K(x)  \  \foraa x \in \Omega\}.
\end{equation}
We will denote by $\mathrm{P}_{\mathcal{K}(\Omega)} :  L^2(\Omega; \mt_\dev^{d\times d}) \to  L^2(\Omega; \mt_\dev^{d\times d})$ the projection operator 
onto the closed convex  set $\mathcal{K}(\Omega)$
induced by 
the projection operators onto the sets $K(x)$, namely, for a given $\sigma \in  L^2(\Omega; \mt_\dev^{d\times d})$,
\begin{equation}
\label{proj-K-Omega}
\xi = \mathrm{P}_{\mathcal{K}(\Omega)}(\sigma) \qquad \text{if and only if} \qquad \xi(x) = \mathrm{P}_{K(x)}(\sigma(x)) \ \foraa x \in \Omega.
\end{equation}
\par
We introduce the support function $ \mathrm{R}: \Omega \times  \mt_\dev^{d\times d}  \to [0,+\infty) $ associated with the multifunction $K$. In order to define the related dissipation potential, we have to resort to the theory of convex function of measures   \cite{Goffman-Serrin64,Reshetnyak68}, since the tensor $p$  and its rate $\dot p$
are  now  Radon measures on $\Omega \cup \Gamma_\Dir$. Therefore, with  every $\dot{p} \in \mathrm{M}(\Omega{\cup}  \Gamma_\Dir; \mt_\sym^{d\times d})$ (for convenience, we will keep to the notation $\dot p$ for the independent variable in the plastic dissipation potential) we associate the nonnegative Radon measure  $\overline{\mathrm{R}}(\dot{p})$ defined by 
\begin{equation}
\label{Radon-meas-R}
\overline{\mathrm{R}}(\dot{p})(B): = \int_B  \mathrm{R} \left(x, \frac{\dot p}{|\dot p|} (x) \right)  \dd |\dot p|(x) \qquad \text{for every  Borel set } B \subset  \Omega{\cup}  \Gamma_\Dir,
\end{equation}
with ${\dot p}/|\dot p|$ the Radon-Nykod\'ym derivative of the measure $\dot p$ w.r.t.\ its variation $|\dot p|$. We then consider the \emph{plastic dissipation potential} 
\begin{equation}
\label{pl-diss-pot}
\calR:  \mathrm{M}(\Omega{\cup}  \Gamma_\Dir; \mt_\sym^{d\times d}) \to [0,+\infty) \text{ defined by }  \calR(\dot p): =  \overline{\mathrm{R}}(\dot{p})(\Omega{\cup} \Gamma_\Dir).
\end{equation}
Observe that the definition of the functional $\calR$ on $\mathrm{M}(\Omega{\cup}  \Gamma_\Dir; \mt_\sym^{d\times d}) $ is consistent with that given on 
$ L^1(\Omega;\mt_\sym^{d\times d})$ in \eqref{plastic-dissipation-functional} (in the case the yield surface $K$ does not depend on $\teta$), namely
  \begin{equation}
\label{consistency-measure-L1} \calR(\dot p) = \int_\Omega \mathrm{R}(x,\dot p(x)) \dd x  \qquad \text{ if }  \dot p \in L^1(\Omega;\mt_\sym^{d\times d}).
\end{equation}
This justifies the abuse in the  notation for $\mathcal{R}$. 
\par
It follows from the lower semicontinuity of $x \rightrightarrows K(x)$ that its support function 
$ \mathrm{R} $ is   lower semicontinuous on $\Omega \times  \mt_\dev^{d\times d}$.
Since $\mathrm{R}(x,\cdot)$ is also convex and $1$-homogeneous, Reshetnyak's Theorem (cf., e.g., \cite[Thm.\ 2.38]{AmFuPa05FBVF}) applies to ensure that the functional $\calR$ from \eqref{pl-diss-pot} is lower semicontinuous w.r.t.\  the weak$^*$-topology on   $\mathrm{M}(\Omega{\cup}  \Gamma_\Dir; \mt_\sym^{d\times d}) $. Accordingly, the induced  total variation functional, defined for every function $p : [0,T] \to  \mathrm{M}(\Omega{\cup}  \Gamma_\Dir; \mt_\sym^{d\times d})$ by 
\begin{equation}
\mathrm{Var}_{\calR}(p; [a,b]) : = \sup\left \{  \sum_{i=1}^N \calR (p(t_i) - p(t_{i-1})) \, : \ a= t_0<t_1< \ldots< t_{N-1} = t_N = b \right\}
\end{equation}
for  $ [a,b] \subset [0,T]$,
is lower semicontiuous  w.r.t.\ the pointwise (in time) convergence of $p$ in  the weak$^*$ topology of $\mathrm{M}(\Omega{\cup}  \Gamma_\Dir; \mt_\sym^{d\times d})$.
\begin{remark}
\upshape
\label{rmk:added-sol}
The dependence of the constraint set on a (spatially and/or temporally) discontinuous additional variable poses considerable difficulties in 
 handling  the plastic dissipation potential, in that Reshetnyak's Theorem is no longer applicable. To bypass this, in the non-associative plasticity models considered in, e.g., \cite{BabFraMor12,DalDesSol11}, such additional variable has been mollified; very recently, in \cite{CO17} a  Reshetnyak-type lower semicontinuity result has been obtained in the case plastic dissipation potential also depends on a damage variable $z\in \mathrm{B}([0,T];W^{1,d}(\Omega))$ (with $d$ the space dimension of the problem).
\par
Here we  prefer to avoid mollifying the temperature variable, and the result from  \cite{CO17}  does not apply due to the poor regularity of $\teta$. That is why, we have dropped the dependence of $K$ on $\teta$. 
\end{remark}
\par
Finally, we  recall \cite[Thm.\ 7.1]{DMDSMo06QEPL}, stating that for every $p \in \mathrm{AC}([0,T]; \mathrm{M}(\Omega{\cup}  \Gamma_\Dir; \mt_\sym^{d\times d})) $,  there exists 
\[
\dot{p}(t): = \mathrm{weak}^*-\lim_{h \to 0} \frac{p(t+h) - p(t)}{h} \qquad \foraa t \in (0,T), 
\] 
where the limit is w.r.t.\ the weak$^*$-topology of $\mathrm{M}(\Omega{\cup}  \Gamma_\Dir; \mt_\sym^{d\times d})$. Moreover, there holds
\begin{equation}
\label{consistency-dotp}
\mathrm{Var}_{\calR}(p; [a,b]) = \int_a^b  \calR(\dot p(t)) \dd t \qquad \text{for all } [a,b] \subset [0,T],
\end{equation}
cf.\  \cite[Thm.\ 7.1]{DMDSMo06QEPL} and \cite[Thm.\ 3.6]{Sol09}.
\paragraph{{\em Cauchy data}.} We will supplement the perfectly plastic  system with 
initial data
\begin{subequations}
\label{Cauchy-data-s5}
\begin{align}
&
\label{initial-u-s5}
u_0 \in \mathrm{BD}(\Omega;\R^d),
\\
& 
\label{initial-p-s5}
e_0 \in   L^2(\Omega;\mt_\sym^{d\times d}), \quad  p_0 \in \mathrm{M}(\Omega{\cup}\Gamma_\Dir;\mt_\dev^{d\times d}) \quad \text{such that } (u_0, e_0, p_0 ) \in \mathcal{A}_{\BD}(w(0))\,.
\end{align}
\end{subequations}
\subsection{Energetic solutions to the perfectly plastic system}
\label{ss:5.2}
Throughout this section we will tacitly suppose the validity of conditions \eqref{bdriesC2}, \eqref{bbC-s:5}, 
\eqref{data-displ-s5},  \eqref{Dir-load-Sec5}, and
 \eqref{ass-K1-sec-5}--\eqref{ass-K2-sec-5}.  
We are now in the position to give the notion of energetic solution (or \emph{quasistatic evolution}) for the perfectly plastic system (in the isothermal case). 
\begin{definition}[Global energetic solutions to the perfectly plastic system]
\label{def:en-sols-PP}
Given initial data $(u_0,e_0,p_0)$ fulfilling \eqref{Cauchy-data-s5},  we call a triple $(u,e,p)$ a \emph{global energetic solution} to the Cauchy problem for system 
\eqref{RIP-PDE}, with boundary datum $w$ on $\Gamma_\Dir$,  if
\begin{subequations}
\label{reg-eneerg-sols}
\begin{align}
& 
\label{reg-u-RIP}
u \in \mathrm{BV}([0,T]; \BD(\Omega;\R^d)),
\\
& 
\label{reg-e-RIP}
e \in \mathrm{BV}([0,T]; L^2(\Omega;\mt_\sym^{d\times d})), \\
& 
\label{reg-p-RIP}
p \in \mathrm{BV}([0,T]; \mathrm{M}(\Omega; \mt_\dev^{d\times d})),
\end{align}
\end{subequations}
$(u,e,p)$ comply with the initial conditions
\begin{equation}
\label{Cauchy-RIP}
u(0,x)= u_0(x), \qquad  e(0,x)= e_0(x), \qquad  p(0,x)= p_0(x) \qquad \foraa x \in \Omega,
\end{equation}
and with the following conditions \emph{for every} $t \in [0,T]$: 
\begin{itemize}
\item[-] \emph{kinematic admissibility}: $(u(t), e(t), p(t)) \in  \mathcal{A}_\BD(w(t))$; 
\item[-]
 \emph{global stability}:
\begin{equation}
\label{glob-stab-5}
\tag{$\mathrm{S}$}
\calQ(e(t)) - \pairing{}{\BD(\Omega;\R^d)}{\calL(t)}{u(t)} \leq \calQ(\tilde e) + \calR(\tilde p - p(t))   - \pairing{}{\BD(\Omega;\R^d)}{\calL(t)}{\tilde u}  \qquad \text{for all } (\tilde u, \tilde e, \tilde p) \in \mathcal{A}_\BD(w(t))
\end{equation}
 (recall definition \eqref{total-load}  of  the total loading function $\calL$); 
\item[-]
 \emph{energy balance}:
\begin{equation}
\label{enbal-5}
\tag{$\mathrm{E}$}
\begin{aligned}
\calQ(e(t)) + \mathrm{Var}_\calR (p; [0,t]) &   = \calQ(e_0) 
+  \int_0^t \int_\Omega \sigma : \sig{\dot{w}} \dd x \dd s - \int_0^t  \pairing{}{\BD(\Omega;\R^d)}{\calL}{\dot w}  \dd s 
\\
&  \quad + \pairing{}{\BD(\Omega;\R^d)}{\calL(t)}{u(t)} -  \pairing{}{\BD(\Omega;\R^d)}{\calL(0)}{u_0} - \int_0^t  \pairing{}{\BD(\Omega;\R^d)}{\dot{\calL}}{u}  \dd s  
 \end{aligned}
\end{equation}
with the stress $\sigma$ given by $\sigma(t)  = \bbC e(t)$ for every $t \in [0,T]$.
\end{itemize}
\end{definition}
\begin{remark}
\label{rmk:save-the-day}
\upshape
It follows from  \cite[Thm.\ 4.4]{DMDSMo06QEPL}  (cf.\ also \cite[Rmk.\ 5]{DMSca14QEPP}), that \eqref{enbal-5} is equivalent to the condition that 
\begin{equation}
\label{alternat-enbal-RIS}
\begin{aligned}
&
p \in \mathrm{BV}([0,T]; \mathrm{M}(\Omega; \mt_\dev^{d\times d})),
\\
&
\begin{aligned}
&
\calQ(e(t)) + \mathrm{Var}_\calR (p; [0,t]) 
- \int_\Omega \varrho(t) : (e(t) {-} \sig{w(t)} ) \dd x
\\
 &   = \calQ(e_0)  - \int_\Omega \varrho(0) : (e_0 {-} \sig{w(0)} ) \dd x + \int_0^t \int_\Omega \sigma : \sig{\dot{w}} \dd x \dd s
-\int_0^t\int_\Omega \dot{\varrho}: (e{-} \sig{w}) \dd x \dd s
\end{aligned}
\end{aligned}
\end{equation}
for every $t\in [0,T]$.
\end{remark}
\noindent In the above definition,  for consistency with the standard concept of energetic solution, we have required only $\BV$-time regularity for the functions $(e,p)$ (and, accordingly, for $u$). On the other hand, an important feature of perfect plasticity is that, due to its \emph{convex character}, the maps $t\mapsto u(t), \, t \mapsto e(t), \, t \mapsto p(t)$ are ultimately \emph{absolutely continuous} on $[0,T]$. 
In fact,
 the following result ensures that, if  a triple $(u,e,p)$  complies with  \eqref{glob-stab-5} and \eqref{enbal-5} \emph{at almost all} $t \in (0,T)$, then it satisfies said conditions \emph{for every} $t\in [0,T]$, and in addition  the maps $t\mapsto u(t), \, t \mapsto e(t), \, t \mapsto p(t)$ are absolutely continuous. The proof of Thm.\ \ref{th:dm-scala} 
 below follows from that for
  \cite[Thm.\ 5]{DMSca14QEPP}, since the argument  carries over to the present case of a spatially-dependent dissipation metric $\mathrm{R}$.
\begin{theorem}
\label{th:dm-scala}
Let $S \subset [0,T]$ be a set of full measure containing $0$. Let $(u,e,p) : S \to  \BD(\Omega;\R^d) \times  L^2(\Omega;\mt_\sym^{d\times d})\times  \mathrm{M}(\Omega; \mt_\dev^{d\times d})$ be measurable and bounded functions satisfying the Cauchy conditions \eqref{Cauchy-RIP} with a triple $(u_0, e_0, p_0) $ as in  \eqref{Cauchy-data-s5} and fulfilling the stability condition
\eqref{glob-stab-5} at time $t=0$, as well as the kinematic admissibility, the global stability  condition, and the energy balance \emph{for every} $t \in S$. Suppose in addition that $p\in \BV([0,T];  \mathrm{M}(\Omega; \mt_\dev^{d\times d})).$
\par
Then, the pair $(u,e)$ extends to an absolutely continuous function $(u,e) \in \AC ([0,T]; \BD(\Omega;\R^d) \times  L^2(\Omega;\mt_\sym^{d\times d}))$. Moreover, $p \in  \AC([0,T];  \mathrm{M}(\Omega; \mt_\dev^{d\times d}))$ and the triple $(u,e,p)$ is a global energetic solution to the perfectly plastic system in the sense of Definition \ref{def:en-sols-PP}. 
\end{theorem}
In the proof of  the forthcoming Theorem \ref{mainth:3}, we will also make use of the following result, first obtained in \cite[Thm.\ 3.6]{DMDSMo06QEPL} in the homogeneous case, and extended to a spatially-dependent yield surface in \cite[Thm.\ 3.10]{Sol09}.
\begin{lemma}
\label{l:for-stability}
Let $S \subset [0,T]$ be a set of full measure containing $0$. Let $(u,e,p) : S \to  \BD(\Omega;\R^d) \times  L^2(\Omega;\mt_\sym^{d\times d})\times  \mathrm{M}(\Omega; \mt_\dev^{d\times d})$  fulfill the kinematic admissibility   at $t \in S$. Then, the following conditions are equivalent 
\begin{compactitem}
\item[-] $(u(t), e(t),p(t))$ comply with the global stability condition \eqref{glob-stab-5} at $t$; 
\item[-] the stress  $\sigma(t) = \bbC e(t)$  satisfies $\sigma(t) \in \calK(\Omega)$ and  the boundary value problem
\begin{equation}
\label{BVprob-stress}
 \begin{cases}
 -\mathrm{div}_{\Dir}(\sigma(t)) = F(t)  & \text{in } \Omega, 
 \\
 \sigma(t) \nu = g(t) & \text{on } \Gamma_\Neu,
 \end{cases}
 \end{equation}
 with the operator $-\mathrm{div}_{\Dir}$ from   \eqref{div-Gdir}.  
 \end{compactitem}
\end{lemma}
\begin{remark}
\label{rmk:stress-strain-dual} 
\upshape
In the proof of Thm.\ \ref{mainth:3}, Lemma \ref{l:for-stability} will play a pivotal role in the argument for the global stability condition \eqref{glob-stab-5}.
In the spatially homogeneous case addressed in \cite{DMDSMo06QEPL}, the proof of the  analogue of Lemma \ref{l:for-stability}  relies on a careful definition of the duality between the (deviatoric part of the) stress $\sigma$, which is typically not continuous as a function of the space variable, and the strain $\sig u$, as well as the plastic strain $p$, which  in turn are just measures. In particular, the regularity conditions \eqref{bdriesC2} on $\partial\Omega$ and $\partial\Gamma$ entail the validity of a by-part integration formula (cf.\ 
\cite[Prop.\ 2.2]{DMDSMo06QEPL}), which is at the core of the proof of 
\cite[Thm.\ 3.6]{DMDSMo06QEPL}. Another crucial point is the validity of the inequality  (between measures)
\begin{equation}
\label{Radon-meas-ineqs}
\overline{\mathrm{R}}(\dot{p}) \geq [\sigma_\dev{:} \dot{p}],
\end{equation}
where $\overline{\mathrm{R}}(\dot{p})$ is the Radon measure defined by \eqref{Radon-meas-R}, and the measure $ [\sigma_\dev{:} \dot{p}]$  (we refer to \cite[Sec.
 2]{DMDSMo06QEPL}  for its definition),  `surrogates' the duality between $\sigma_\dev$ and $\dot{p}$. 
 \par
 In \cite{Sol09} it was shown that, if  $K: \Omega{\cup}\Gamma_\Dir \rightrightarrows \mt_\dev^{d\times d}$ is continuous, then \eqref{Radon-meas-ineqs} holds also in the spatially heterogeneous case and, based on that,
  Lemma \ref{l:for-stability}
  (cf.\ \cite[Thm.\ 3.10]{Sol09}), 
   was derived. 
   \par
     However, it was observed in  \cite{FraGia2012} that the continuity of $K$ is a quite restrictive condition for applications to heterogeneous materials.   The authors carried out the analysis under a   much weaker, and more mechanically feasible, set of conditions on the multifunction $K$ by 
  adopting a slightly different approach to the proof of existence of `quasistatic evolutions'.  In particular, 
  their argument for obtaining the stability condition  \eqref{glob-stab-5} did not rely on 
    Lemma \ref{l:for-stability}. Rather, it was based on the construction of a suitable recovery sequence in the time discrete-to-continuous limit passage. 
    Unfortunately, it seems to us that,  for the asymptotic analysis developed in the upcoming Section 
\ref{s:6}, this argument could not be exploited to recover \eqref{glob-stab-5} in the vanishing-viscosity and inertia limit. That is why, we have to stay with 
the  continuity requirement \eqref{ass-K1-sec-5} on $K$. 
\end{remark}
\section{From the thermoviscoplastic to the perfectly plastic system}
\label{s:6}
In this section we address  the limiting behavior of  \emph{weak energy solutions} to the thermoviscoplastic system  (\ref{plast-PDE}, \ref{bc})  as  the rate of the external loads $F,\, g, \, w$ and of the heat sources $H,\, h$ becomes slower and slower. Accordingly, we will rescale time by a factor $\eps>0$. Before detailing our conditions on the data $F,\, g,\, w,\, H $, and $h$ for performing the vanishing-viscosity analysis of  system 
 (\ref{plast-PDE}, \ref{bc}), let us  specify that, already on the ``viscous'' level, we will confine the discussion to   the  case in which
  \[
  \text{ the elastic domain $K$ does not depend on  $\teta$ but only on $x \in \Omega$,  and  fulfills \eqref{ass-K1-sec-5}--\eqref{ass-K2-sec-5},}
   \]
   cf.\ Remark \ref{rmk:added-sol}. 
   Moreover,    we will suppose that the 
   thermal expansion tensor 
   also depends on $\eps$,  i.e.\ 
   $\bbE =  \bbE_\eps$, with the scaling given by  \eqref{scaling-intro}. 
%
   Hence, the tensors $\bbB_\eps: = \bbC \bbE_\eps$ have the form
   \begin{equation}
   \label{scalingbbB}
    \bbB_\eps = \eps^{\beta} \bbB \qquad \text{with } \beta >\frac12 \text{ and }  \bbB \in L^\infty (\Omega; \R^{d\times d}). 
   \end{equation}
   We postpone to Remark \ref{rmk:conds-forces}
   ahead
   some comments on the role of condition \eqref{scalingbbB}. 
\par
Let  $(H_\eps, h_\eps,F_\eps, g_\eps,  w_\eps)_\eps$  be a family of data for system   (\ref{plast-PDE}, \ref{bc}),  and let us  rescale them by the factor $\eps>0$, thus introducing 
\[
H^\eps(t) : =  H_\eps \left( \eps t \right), \quad 
 h^\eps(t) : =  h_\eps \left( \eps t \right), \quad 
F^\eps(t) : =  F_\eps \left( \eps t \right), \quad 
g^\eps(t) : =  g_\eps \left( \eps t \right),\quad
 w^\eps(t) : =  w_\eps \left( \eps t \right) \qquad t \in \left[0,\frac T\eps \right].
\]
Correspondingly, we denote by $(\teta^\eps, u^\eps, e^\eps, p^\eps)$  a weak energy solution to the thermoviscoplastic  system, with the tensor  $ \bbB_\eps $ from 
\eqref{scalingbbB}, starting from initial data $(\teta^0_\eps, u^0_\eps, e^0_\eps, p^0_\eps)$: under the conditions on the functions $(H_\eps, h_\eps,F_\eps, g_\eps,  w_\eps)_\eps$ and  the data   $(\teta^0_\eps, u^0_\eps, e^0_\eps, p^0_\eps)_\eps$ stated in Sec.\ \ref{ss:2.1},  the existence of $(\teta^\eps, u^\eps, e^\eps, p^\eps)$   is ensured by Thm.\ \ref{mainth:2}. We further rescale the functions $(\teta^\eps, u^\eps, e^\eps, p^\eps)$  in such a way as to have them defined on the interval $[0,T]$, by  setting 
\[
\teta_\eps(t): = \teta^\eps \left( \frac t \eps\right), \quad  u_\eps(t): = u^\eps \left( \frac t \eps\right),  \quad e_\eps(t): = e^\eps \left( \frac t \eps\right), \quad   p_\eps(t): = p^\eps \left( \frac t \eps\right), \qquad t \in [0,T].
\]
\par
For later reference, here we state  the defining properties  of weak energy solutions in terms of the rescaled quadruple $(\teta_\eps, u_\eps, e_\eps, p_\eps)$, taking into account the improved formulation of the heat equation provided in Theorem \ref{mainth:2}.  In addition to the kinematic admissibility $\sig{u_\eps}= e_\eps + p_\eps$, we have:
\begin{subequations}
\label{weak-form-eps}
\begin{compactitem}
\item[-] strict positivity: $\teta_\eps \geq \bar\teta>0$ a.e.\ in $\Omega$, with $\bar\teta$ given by \eqref{strong-strict-pos};
\item[-] weak formulation of the heat equation, for almost all $t \in (0,T)$ and all test functions $\varphi \in  W^{1,1+1/\tilde{\delta}}(\Omega)$, with $\tilde\delta>0$ such that \eqref{further-k-teta} holds:
\begin{equation}
\label{eq-teta-eps}
\begin{aligned}
 &\eps \pairing{}{W^{1,1+1/\tilde{\delta}}(\Omega)}{\dot{\teta}_\eps(t)}{\varphi}
+ \int_\Omega \condu(\teta_\eps(t)) \nabla \teta_\eps(t)\nabla\varphi \dd
x
\\
& = \int_\Omega \left(H_\eps(t)+
\eps\mathrm{R}(x,\dot p_\eps (t)) + \eps^2 \dot{p}_\eps(t) : \dot{p}_\eps(t)  + \eps^2\mathbb{D} \dot{e}_\eps (t) :\dot{e}_\eps (t) -\eps \teta_\eps (t) \bbB_\eps :  \dot{e}_\eps (t) \right) \varphi  \dd x  \\ & \quad  + \int_{\partial\Omega} h_\eps (t)  \varphi   \dd S;
\end{aligned}
\end{equation}
\item[-] weak momentum balance  for almost all $t \in (0,T)$ and all test functions $v \in H_\Dir^1(\Omega;\R^d)$
\begin{equation}
\label{w-momentum-balance-eps}
\begin{aligned}
\rho \eps^2 \int_\Omega \ddot{u}_\eps(t) v \dd x + \int_\Omega \left(\bbD \eps \dot{e}_\eps(t) + \bbC e_\eps(t) - \teta_\eps(t) \bbB_\eps \right): \sig v \dd x   &  = \pairing{}{H_\Dir^1(\Omega;\R^d)}{\calL_\eps(t)}{v} 
\end{aligned}
\end{equation}  
for almost all $t\in (0,T)$,
with $\calL_\eps$ defined from $F_\eps$ and $g_\eps$ as in \eqref{total-load};
\item[-] the plastic flow rule  \eqref{flow-rule-rescal}, rewritten (cf.\ \eqref{flow-rule-rewritten}) for later use  as  
\begin{equation}
\label{pl-flow-rule-eps}
\sie -\eps {\dot p}_\eps = \mathrm{P}_{\mathcal{K}(\Omega)}(\siedev) \qquad \aein\, Q,
\end{equation}
with $\sie =  \eps\bbD  \dot{e}_\eps + \bbC e_\eps - \teta_\eps \bbB_\eps$ and the projection operator  $ \mathrm{P}_{\mathcal{K}(\Omega)}   $ from \eqref{proj-K-Omega}.
 \end{compactitem}
We also record the 
\begin{compactitem}
\item[-] (rescaled) mechanical energy balance
\begin{equation}
\label{mech-enbal-rescal}
\begin{aligned}
& 
\frac{\rho \eps^2}2 \int_\Omega |\dot{u}_\eps(t)|^2 \dd x +   \eps\int_0^t\int_\Omega   \bbD \dot{e}_\eps: \dot{ e}_\eps   \dd x \dd r 
+ \frac{\eps}2 \int_0^t\int_\Omega  |\dot{p}_\eps|^2   \dd x \dd r  + \frac1{2\eps} \int_0^t d^2(\siedev, \calK(\Omega)) \dd t  \\ 
& \quad 
+ \int_0^t \calR( \dot{p}_\eps) \dd r 
+ \calQ(e_\eps(t))
\\
&  =  \frac{\rho \eps^2}2 \int_\Omega |\dot{u}_\eps^0|^2 \dd x + \calQ(e_\eps^0)  + \int_0^t \pairing{}{H_\Dir^1 (\Omega;\R^d)}{\mathcal{L}_\eps}{\dot{u}_\eps{-} \dot{ w}_\eps} \dd r  +   \int_0^t \int_\Omega \teta_\eps \bbB_\eps : \dot{ e}_\eps \dd x \dd r 
 \\
& \quad   +\rho\eps^2 \left( \int_\Omega \dot{u}_\eps(t) \dot{w}_\eps(t) \dd x -  \int_\Omega \dot{u}_\eps^0 \dot{w}_\eps(0) \dd x - \int_0^t \int_\Omega \dot{u}_\eps\ddot{w}_\eps \dd x \dd r \right)     + 
\int_0^t \int_\Omega  \sie  : \sig{\dot{w}_\eps} \dd x \dd r 
 \end{aligned}
\end{equation}
for every $t \in [0,T]$,
where we have used \eqref{pl-flow-rule-eps},  yielding that 
\[
\eps | \dot{p}_\eps|^2 =  \frac{\eps}2 | \dot{p}_\eps|^2  +  \frac{1}{2\eps}\dddn{ | \sie {-} \mathrm{P}_{\mathcal{K}(\Omega)}(\siedev) |^2}{$=  d^2(\siedev, \calK(\Omega)) $ }  \qquad \aein \, Q,
\]
with $d(\cdot, \calK(\Omega))$  the distance function from the closed and convex set $\calK(\Omega)$.
\end{compactitem}
Finally,  adding \eqref{mech-enbal-rescal} with \eqref{eq-teta-eps} tested by $\frac1{\eps}$ we obtain the 
\begin{compactitem}
\item[-] (rescaled) total energy balance  
 \begin{equation}
\label{total-enbal-rescal}
\begin{aligned}
& 
\frac{\rho \eps^2}2 \int_\Omega |\dot{u}_\eps(t)|^2 \dd x +\pairing{}{W^{1,\infty}}{\teta_\eps(t)}{1} +\frac12 \int_\Omega \bbC  e_\eps(t){:}   e_\eps(t) \dd x 
\\
&  = \frac{\rho\eps^2}2 \int_\Omega |\dot{u}_\eps^0|^2 \dd x +\mathcal{E}(\teta_\eps^0, e_\eps^0)  + \int_0^t \pairing{}{H_\Dir^1 (\Omega;\R^d)}{\mathcal{L}_\eps}{\dot {u}_\eps{-} \dot{w}_\eps} \dd r    +\int_0^t \int_\Omega\frac1\eps H_\eps \dd x \dd r + \int_0^t \int_{\partial\Omega} \frac1\eps  h_\eps \dd S \dd r
\\
& \quad   +\rho\eps^2 \left( \int_\Omega \dot{u}_\eps(t) \dot{w}_\eps(t) \dd x -  \int_\Omega \dot{u}_\eps^0 \dot{w}_\eps(0) \dd x - \int_0^t \int_\Omega \dot{u}_\eps\ddot{w}_\eps \dd x \dd r \right)     + 
\int_0^t \int_\Omega  \sie  : \sig{\dot{w}_\eps} \dd x \dd r
\end{aligned}
\end{equation}
for every $t \in [0,T]$.
\end{compactitem}
\end{subequations}
Indeed, 
as in the proof of Theorems \ref{mainth:1} and \ref{mainth:2},  
\eqref{total-enbal-rescal} will be the starting point in the derivation of the  priori estimates, \emph{uniform} w.r.t.\ the parameter $\eps$, on the functions $(\teta_\eps, u_\eps, e_\eps, p_\eps)_\eps$, under the following
\paragraph{\emph{Hypotheses on the data $(H_\eps, h_\eps, F_\eps, g_\eps, w_\eps)_\eps$ and on the initial data $(\teta_\eps^0, u_\eps^0,  \dot{u}_\eps^0, e_\eps^0, p_\eps^0)_\eps$.} }
Since it will be  necessary to  start with  \eqref{total-enbal-rescal} in this temperature-dependent setting, we shall have to assume that
the families of data $(H_\eps)_\eps$ and   $(h_\eps)_\eps$ converge to zero in the sense
 that there exists a constant $\overline{C}>0$ such that for every $\eps>0$
 \[
 \|H_\eps\|_{L^1(0,T; L^1(\Omega))} \leq \overline{C}\eps, \qquad  \|h_\eps\|_{L^1(0,T; L^1(\partial\Omega))} \leq \overline{C}\eps.
 \]
 We will in fact need to strengthen this in order to pass to the limit, as $\eps\down 0$, in \eqref{total-enbal-rescal}, by assuming that there exist $\mathsf{H}\in L^1(0,T; L^1(\Omega)), \ \mathsf{h} \in L^1(0,T; L^1(\partial\Omega))$ such that 
\begin{subequations}
\label{data-eps}
\begin{equation}
\label{heat-sources-eps}
 \frac{1}{\eps} H_\eps \weakto \mathsf{H} \text{ in }  L^1(0,T; L^1(\Omega)), \quad  \frac{1}{\eps} h_\eps \weakto \mathsf{h} \text{ in }  L^1(0,T; L^1(\partial\Omega)).
\end{equation}
As for the body and surface forces, for every $\eps>0$ the functions $F_\eps$ and $g_\eps$  have to comply with \eqref{data-displ} and the safe-load condition \eqref{safe-load}, with associated stresses $\varrho_\eps$.
We impose that there exist $F$ and $g $ as in \eqref{data-displ-s5} to which  $(F_\eps)_\eps$ and $(g_\eps)_\eps$   converge in topologies that we choose not to specify, and that the sequence $(\varrho_\eps)_\eps$ suitably  converge to the stress $\varrho$ from \eqref{safe-load-s5}  (hence, with $\varrho_\dev \equiv 0$) associated with $F$ and $g$, namely
\begin{equation}
\label{safe-load-eps} 
\begin{aligned}
 & 
  \varrho_\eps \to \varrho   \qquad \text{in } 
W^{1,1}(0,T; L^2(\Omega; \mt_\sym^{d\times d})),
  \\ & 
\| (\varrho_\eps )_\dev \|_{L^1(0,T; L^\infty (\Omega; \mt_\sym^{d\times d}))} \leq  \overline{C}\eps, \qquad 
\| ({\varrho}_\eps)_\dev \|_{W^{1,1}(0,T; L^2 (\Omega; \mt_\sym^{d\times d}))} \leq  \overline{C}\eps\,.
\end{aligned}
\end{equation}
For later use, let us record here that, since $\pairing{}{H_\Dir^1(\Omega;\R^d)}{\calL_\eps(t)}{v} = \int_\Omega \varrho_\eps(t) {:} \sig{v} \dd x $ for every $v \in H_\Dir^1(\Omega;\R^d)$ by the safe load condition,  and analogously for  $\dot{\calL}_\eps$, it follows from \eqref{safe-load-eps} that
\begin{equation}
\label{total-load-eps-bound}
\calL_\eps \to \calL  \quad \text{in } W^{1,1} (0,T;H_\Dir^1(\Omega;\R^d)^*),
\end{equation}
with $\calL$ the total load associated with $F$ and $g$. 
\par
Furthermore, we impose that the Dirichlet loadings $(w_\eps)_\eps \subset  L^1(0,T; W^{1,\infty} (\Omega;\R^d)) \cap W^{2,1} (0,T;H^1(\Omega;\R^d)) \cap H^2(0,T; L^2(\Omega;\R^d)) $ (cf.\ \eqref{Dirichlet-loading}),  fulfill
\begin{equation}
\label{est-w-eps}
\eps \|  \dot{w}_\eps \|_{L^\infty (0,T;H^1(\Omega;\R^d))} \leq \overline{C}, \quad  \eps \|\ddot{w}_\eps  \|_{L^1(0,T;H^1(\Omega;\R^d))} \leq \overline{C}, \quad  \eps^{1/2} \|  \sig{\dot{w}_\eps}  \|_{L^1(0,T;L^\infty(\Omega;\mt_\sym^{d\times d}))} \leq \overline{C}, 
\end{equation}
and that there exists $w \in H^1(0,T;  H^1(\Omega;\R^d))$ (cf.\ \eqref{Dir-load-Sec5}) such that 
\begin{equation}
\label{conv-Dir-loads}
w_\eps \to w \quad \text{in } H^1(0,T;  H^1(\Omega;\R^d)).
\end{equation}
Finally, for the Cauchy data $(\teta_\eps^0, u_\eps^0, \dot{u}_\eps^0,  e_\eps^0, p_\eps^0)_\eps$ we impose 
 the  convergences
\begin{equation}
\label{convs-init-data}
\begin{aligned} 
& 
\eps \|  \dot{u}_\eps^0\|_{L^2(\Omega;\R^d)} \to 0,   \quad 
\teta_\eps^0 \weakto \teta_0 \text{ in } L^1(\Omega),
\\
& u_\eps^0 \weaksto u_0 \text{ in } \BD(\Omega;\R^d), \qquad e_\eps^0 \to e_0 \text{ in } L^2(\Omega;\mt_\sym^{d\times d}), \qquad   p_\eps^0 \weaksto p_0 \text{ in } \mathrm{M}(\Omega{\cup}\Gamma_\Dir;\mt_\dev^{d\times d}).
\end{aligned}
\end{equation}
Observe that, since  $(u_\eps^0, \dot{u}_\eps^0,  e_\eps^0, p_\eps^0) \in \calA(w_\eps(0)) \subset  \calA_{\BD}(w_\eps(0))$,  convergences \eqref{conv-Dir-loads} and \eqref{convs-init-data}, combined with Lemma  \ref{l:closure-kin-adm},    ensure that the  triple $(u_0,e_0,p_0) $ is in $ \calA_{\BD}(w(0))$.
\end{subequations}
\par
We are now in the position to give our asymptotic result, stating the convergence (along a sequence $\eps_k \downarrow 0$) of a family of solutions to the thermoviscoplastic system, to a quadruple $(\limte, u,e,p)$ such that $(u,e,p)$ is an energetic solution to the plastic system, while the limit temperature $\limte$ is constant in space, but still time-dependent. Furthermore, 
we find that  $(\limte,u,e,p)$ fulfill a further energy balance, cf.\ \eqref{anche-la-temp} ahead, from which we deduce a balance between the energy dissipated by the plastic strain, and the thermal energy. 
 \begin{maintheorem}
\label{mainth:3}
Let the reference configuration $\Omega$ and the elasticity tensor $\bbC $ comply with
\eqref{Omega-s2},
   \eqref{bdriesC2} and  \eqref{elast-visc-tensors}, \eqref{bbC-s:5}, respectively.
Moreover, assume 
 \eqref{ass-K1-sec-5} and \eqref{ass-K2-sec-5}. 
  Let $(\teta_\eps, \ue,e_\eps,\pe)_\eps$ be a family of \emph{weak energy solutions} to the  rescaled thermoviscoplastic systems 
(\ref{plast-PDE-rescal}, \ref{bc}), with heat conduction coefficient $\condu$ fulfilling \eqref{hyp-K} and \eqref{hyp-K-stronger}, tensors $\bbB_\eps$ satisfying  \eqref{scalingbbB},  with data $(H_\eps, h_\eps, F_\eps, g_\eps, w_\eps)_\eps$ fulfilling 
conditions \eqref{data-eps}, and initial data $(\teta_\eps^0, u_\eps^0, e_\eps^0, p_\eps^0)_\eps$   converging 
as in \eqref{convs-init-data}
to a triple $(u_0,e_0,p_0)$ fulfilling the stability condition
\eqref{glob-stab-5} at time $t=0$.
\par
Then, for every vanishing sequence $(\eps_k )_k $ there exist a (not relabeled) subsequence $(\teta_{\eps_k}, \uek,e_{\eps_k},\pek)_k$  and  
 $\Theta \in L^\infty(0,T;L^\infty(\Omega))$, 
 $  u \in  L^\infty (0,T; \BD(\Omega;\R^d) )$, $ e \in L^\infty(0,T; L^2(\Omega;\R^d))$, $  p \in \BV([0,T];  \mathrm{M}(\Omega{\cup} \Gamma_\Dir;\mt_\dev^{d\times d}) )$ such that 
\begin{enumerate}
\item the following convergences hold as $k\to\infty$:
\begin{subequations}
\label{convs-eps}
\begin{align}
\label{convs-eps-teta}
&
\tetaek \weakto \limte&&  \text{ in } 
L^h(Q)  
 &&  \text{ for every } h 
 \in \begin{cases}
 [1,3] & \text{ if } d =2,
 \\
 [1,8/3] & \text{ if } d=3, 
 \end{cases}
\\
\label{convs-eps-u}
& 
\uek(t) \weaksto u(t)  &&  \text{ in } \BD(\Omega;\R^d)  && \foraa t\in (0,T), 
\\
\label{convs-eps-e}
& 
e_{\eps_k}(t) \to e(t)  &&  \text{ in } L^2(\Omega;\mt_\sym^{d\times d})  && \foraa t\in (0,T), 
\\
\label{convs-eps-p}
& 
p_{\eps_k}(t) \weaksto p(t)  &&  \text{ in } \mathrm{M}(\Omega{\cup} \Gamma_\Dir;\mt_\dev^{d\times d})  && \text{for every } t\in [0,T];
\end{align}
\end{subequations}
\item  $\limte $ is strictly positive and  constant in space;
\item $(u,e,p)$ is a \emph{global energetic solution to the perfectly plastic system}, with initial  and boundary   data $(u_0,e_0,p_0)$ and $w$, and the enhanced time regularity
\begin{equation}
\label{sono-AC}
u\in \AC ([0,T]; \BD(\Omega;\R^d)), \qquad e \in \AC ([0,T]; L^2(\Omega;\mt_{\sym}^{d\times d})), \qquad p \in  \AC ([0,T]; 
 \mathrm{M}(\Omega{\cup} \Gamma_\Dir;\mt_\dev^{d\times d}));
\end{equation}
\item the quadruple $(\teta,u,e,p)$ fulfills the additional  energy balance
\begin{equation}
\label{anche-la-temp}
\begin{aligned}
\calE(\limte(t), e(t)) & = \calE(\teta_0, e_0) +\int_0^t \int_\Omega \mathsf{H} \dd x \dd s +\int_0^t \int_{\partial\Omega} \mathsf{h} \dd S \dd s 
+ \int_0^t \int_\Omega \sigma : \sig{\dot{w}} \dd x \dd s
\\
& \quad
%
-\int_0^t \pairing{}{H^1(\Omega;\R^d)}{\calL}{\dot w} \dd s
+ \pairing{}{\BD(\Omega;\R^d)}{\calL(t)}{u(t)} 
\\
& \quad - \pairing{}{\BD(\Omega;\R^d)}{\calL(0)}{u_0} - \int_0^t \pairing{}{\BD(\Omega;\R^d)}{\dot{\calL}}{u} \dd s,
 \end{aligned}
\end{equation}
for almost all $t\in (0,T)$, 
and therefore there holds
\begin{equation}
\label{balance-of-dissipations}
\begin{aligned}
&
|\Omega| (\Theta(t) - \Theta(s))= \calF(\limte(t)) -  \calF(\limte(s)) 
\\
&
=  \mathrm{Var}_{\calR}(p;[s,t]) +\int_s^t \int_\Omega \mathsf{H} \dd x \dd r +\int_s^t \int_{\partial\Omega} \mathsf{h} \dd S \dd r 
 \qquad \text{for almost all } s, t \in (0,T) \text{ with } s \leq t. 
 \end{aligned}
\end{equation}
\end{enumerate}
\end{maintheorem}
\noindent
Observe that, by virtue of convergence \eqref{convs-eps-teta}, the limiting temperature $\limte$ inherits the strict positivity property $\limte(t) \geq \bar\teta>0$ for almost all $t\in (0,T)$. 
\begin{remark}
\label{rmk:2weak} 
\upshape
Under the same conditions as for Thm.\ \ref{mainth:3},
the vanishing-viscosity and inertia analysis for \emph{entropic} solutions to the  rescaled thermoviscoplastic system
(\ref{plast-PDE-rescal}, \ref{bc}) would lead to a considerably weaker formulation of the 
limiting system. Indeed, the energy balance \eqref{anche-la-temp} would be replaced by an upper energy estimate. Accordingly, it would no longer be possible to deduce \eqref{balance-of-dissipations}, which provides information on the evolution of the limiting temperature. 
\end{remark}
\begin{remark}[An alternative scaling condition on the heat conduction coefficient $\condu$]
\upshape
\label{rmk:alternative-scaling}
For the vanishing-viscosity and inertia analysis carried out in the frame of the damage system analyzed in \cite{LRTT}, a scaling condition  on the heat conduction coefficients $\condu_\eps$, allowed to depend on $\eps$, 
 was  exploited, in place of \eqref{scalingbbB}. Namely,  it was supposed that 
\begin{equation}
\label{scaling-lrtt}
\kappa_\eps(\teta) = \frac1{\eps^2} \kappa(\teta) \qquad \text{with } \kappa \in \rmC^0(\R^+) \text{ satisfying \eqref{hyp-K-stronger}}.
\end{equation}
This reflects the view that, for the limit system, if a change of heat is caused at some spot in the material, then heat must be conducted all over the body with infinite speed. 
In fact, \eqref{scaling-lrtt}  as well led us to show that the limit temperature is constant in space, like in the present case. 
\par
This scaling condition was combined with the requirement that the Dirichlet boundary $\Gamma_\Dir$ coincides with the whole $\partial\Omega$, and that the Dirichlet loading $w$ is null,
in order to deduce 
\begin{compactenum}
\item
the convergence (along a subsequence) of the temperatures $\teta_\eps$ to a spatially constant function $\Theta$;
\item the strong convergence $\eps e(\dot{u}_\eps) \to 0 $ in $L^2(Q;\mt_\sym^{d\times d})$, by means of a careful argument strongly relying on the homogeneous character of the Dirichlet boundary conditions.
\end{compactenum}
In this way, in \cite{LRTT} we bypassed one of the  major difficulties in the asymptotic analysis, 
  namely the presence of the 
 thermal expansion term $
\iint  \tetae \bbB{:}\dot{e}_\eps 
\dd x \dd t $ on the r.h.s.\ of the rescaled mechanical energy balance, which in turn  is the starting point for the derivation of  a priori estimates uniform w.r.t.\ $\eps$ for  the dissipative variables $\dot{e}_\eps$ and $\dot{p}_\eps$.
\par
In the present context, we have decided not to develop  the approach  based on condition \eqref{scaling-lrtt}. In fact, it would have forced us to take null Dirichlet loadings for the limit perfectly plastic system,  and this, in combination with the strong safe load condition \eqref{safe-load-s5}, would have been too restrictive. 
\end{remark}
We will develop the proof of Theorem \ref{mainth:3} in the ensuing Sec.\ \ref{ss:6.1}.
\subsection{Proof of Theorem \ref{mainth:3}} \label{ss:6.1}
We start by deriving a series of a priori estimates, 
\emph{uniform} w.r.t.\ the parameter $\eps$, for  \emph{a distinguished class} of weak energy solutions to system (\ref{plast-PDE-rescal}, \ref{bc}).  In fact, 
in the derivation of these estimates we will perform the same tests as in the proof of  Prop.\ \ref{prop:aprio}, in particular the test of the heat equation by $\tetae^\alpha$, with $\alpha \in [2-\mu, 1)$, $\alpha>0$.  Since $\tetae^\alpha$ is not an admissible test function for  the rescaled heat equation \eqref{eq-teta-eps} due to its insufficient spatial regularity, the calculations related to this test can  be rendered rigorously only on the time discrete level, and the resulting a priori estimates in fact only hold for the \underline{weak energy solutions 
 arising from the time discretization scheme}. That is why, in Proposition \ref{prop:aprio-eps} below we will only claim that \emph{there exist} a family of weak energy solutions for which suitable a priori estimates hold. 
\begin{proposition}[A priori estimates uniform w.r.t.\ $\eps$]
\label{prop:aprio-eps}
Assume \eqref{Omega-s2},
   \eqref{bdriesC2} and  \eqref{elast-visc-tensors}, \eqref{bbC-s:5}.  Assume  conditions  \eqref{hyp-K} and \eqref{hyp-K-stronger} on $\condu$,   \eqref{scalingbbB} on the tensors $\bbB_\eps$, and \eqref{data-eps} on the  data $(H_\eps, h_\eps, F_\eps, g_\eps, w_\eps)_\eps$  and  $(\teta_\eps^0, u_\eps^0, e_\eps^0, p_\eps^0)_\eps$. 
\par
Then, there exist a constant $C>0$ and a family $(\teta_\eps, \ue,e_\eps,\pe)_\eps$  of \emph{weak energy solutions} to the  rescaled thermoviscoplastic systems
 (\ref{plast-PDE-rescal}, \ref{bc}), such that the following estimates hold  uniformly w.r.t.\ the parameter $\eps>0$: 
\begin{subequations}
\label{aprio-eps}
\begin{align}
&
\label{aprio-eps-u}
\| \sig{u_\eps}\|_{L^\infty (0,T; L^1(\Omega;\mt_\sym^{d\times d}))} + \eps^{1/2} \| \sig{\dot{u}_\eps}\|_{L^2 (Q;\mt_\sym^{d\times d})} 
\\ & \qquad \nonumber + \eps \|  \dot{u}_\eps\|_{L^\infty(0,T;L^2(\Omega;\R^d))}  + \eps^2 \| \ddot{u}_\eps \|_{L^2(0,T; H_\Dir^{1}(\Omega;\R^d)^*)} \leq C, 
\\
&
\label{aprio-eps-e}
\| e_\eps\|_{L^\infty (0,T; L^2(\Omega;\mt_\sym^{d\times d}))} + \eps^{1/2} \| \dot{e}_\eps\|_{L^2 (Q;\mt_\sym^{d\times d})}  \leq C, 
\\
&
\label{aprio-eps-p}
\| p_\eps\|_{L^\infty (0,T; L^1(\Omega;\mt_\sym^{d\times d}))}  + \| \dot{p}_\eps\|_{L^1 (Q;\mt_\sym^{d\times d})}  + \eps^{1/2} \| \dot{p}_\eps\|_{L^2 (Q;\mt_\sym^{d\times d})}  \leq C, 
\\
& 
\label{aprio-eps-dist}
\frac1{\eps^{1/2}} \| d((\sie)_\dev, \mathcal{K}(\Omega)) \|_{L^2(0,T)} \leq C,
\\
&
\label{aprio-eps-teta}
\| \tetae\|_{L^\infty (0,T; L^1(\Omega))} 
+ \|\tetae\|_{L^h(Q)} + \frac1{\eps^{1/2}}\|\nabla \tetae\|_{L^2(Q;\R^d)}
 \leq C \qquad \text{for } h = \begin{cases}
 3 & \text{if } d=2,
 \\
 \frac83 & \text{if } d=3.
 \end{cases}
\end{align}
 \end{subequations}
\end{proposition}
\begin{proof} 
We will follow the outline of the proof of Prop.\ \ref{prop:aprio}, referring to it for all details. 
\par \noindent 
\textbf{First a priori estimate:}  We start from the rescaled total energy balance \eqref{total-enbal-rescal} and estimate the  terms on its right-hand side. It follows from \eqref{convs-init-data} that $\eps^2 \| \dot{u}_\eps^0 \|_{L^2(\Omega;\R^d)}^2 \leq C$ and $\calE (\teta_\eps^0, e_\eps^0) \leq C$.  
As for the third term on the r.h.s., we use the safe load condition, yielding
\begin{equation}
\label{will-be-quoted}
\begin{aligned}
 \int_0^t \pairing{}{H_\Dir^1 (\Omega;\R^d)}{\mathcal{L}_\eps}{\dot {u}_\eps{-} \dot{w}_\eps}   \dd r  &  = \int_0^t \int_\Omega \varrho_\eps{:}  \left( \sig{\dot{u}_\eps}{-} \sig{\dot{w}_\eps} \right) \dd x \dd r    \\ &  \stackrel{(1)}{=}   \int_0^t \int_\Omega \varrho_\eps {:} \dot{e}_\eps \dd x \dd r  +  \int_0^t \int_\Omega \varrho_\eps : \dot{p}_\eps \dd x \dd r - \int_0^t \int_\Omega \varrho_\eps {:} \sig{\dot{w}_\eps} \dd x \dd r
 \\ &
   \stackrel{(2)}{=}  -  \int_0^t \int_\Omega \dot{\varrho_\eps} {:} e_\eps \dd x \dd r 
 + \int_\Omega \varrho_\eps(t) {:} e_\eps(t) \dd x -  \int_\Omega \varrho_\eps(0) {:} e_\eps^0 \dd x  + \int_0^t \int_\Omega (\varrho_\eps)_\dev \dot{p}_\eps \dd x \dd r
 \\
 & \qquad 
 -
  \int_0^t \int_\Omega \varrho_\eps {:} \sig{\dot{w}_\eps} \dd x \dd r
 \\ & \doteq I_1 + I_2+I_3+I_4 +I_5
 \end{aligned}
 \end{equation}
 with (1) due to  the kinematic admissibility condition, and (2) following from  integration by parts, and the fact that $ \dot{p}_\eps \in \mt_\dev^{d \times d}$ a.e.\ in $Q$.
 We  estimate 
\[
\begin{aligned}
& \left| I_1\right| \leq     \int_0^t \| \dot{\varrho_\eps}  \|_{L^2(\Omega;\mt_\sym^{d\times d})} \|  e_\eps  \|_{L^2(\Omega;\mt_\sym^{d\times d})} \dd r,
\\ & \left| I_2 \right| \stackrel{(3)} \leq  C \| e_\eps(t)\|_{L^2(\Omega;\mt_\sym^{d\times d})} \leq \frac{C_{\bbC}^1}{16}  \| e_\eps(t)\|_{L^2(\Omega;\mt_\sym^{d\times d})}^2 + C,
\\ & 
 \left| I_3 \right| \stackrel{(4)} \leq  C \| e_\eps^0\|_{L^2(\Omega;\mt_\sym^{d\times d})}\leq C, 
\end{aligned}
\] 
where (3) and (4) follow from  the bound provided for   $\| \varrho_\eps \|_{L^\infty (0,T;L^2(\Omega;\mt_\sym^{d\times d}))}$ by condition \eqref{safe-load-eps}, and from \eqref{convs-init-data}. Instead, for $I_4$ we use  the plastic flow rule \eqref{pl-flow-rule-eps}, rewritten as $ \eps \dot{p}_\eps  =( \sigma_\eps)_\dev - \zeta_\eps$, with $\zeta_\eps \in \partial_{\dot p} \mathrm{R}(\cdot, \dot{p}_\eps) $ a.e.\ in $Q$. Then,
\[
I_4 =   \int_0^t \int_\Omega \tfrac 1\eps (\varrho_\eps)_\dev{:} ( \sigma_\eps)_\dev \dd x -  \int_0^t \int_\Omega  \tfrac 1\eps  (\varrho_\eps)_\dev{:} \zeta_\eps \dd x \doteq I_{4,1} + I_{4,2},
\]
 and 
 \[
  I_{4,1} = \int_0^t \int_\Omega \tfrac 1\eps (\varrho_\eps)_\dev{:} \left( \bbD \dot{e}_\eps + \bbC e_\eps-\tetae \bbB_\eps \right)_\dev \dd x \dd r \doteq I_{4,1,1} +  I_{4,1,2} + I_{4,1,3}
\]
with 
\[
\begin{aligned}
I_{4,1,1} &  = -  \int_0^t \int_\Omega \tfrac 1\eps (\dot{\varrho}_\eps)_\dev{:} \bbD e_\eps \dd x \dd r +  \int_\Omega \tfrac 1\eps (\varrho_\eps(t))_\dev{:}\bbD e_\eps(t) \dd x   -  \int_\Omega \tfrac 1\eps (\varrho_\eps(0))_\dev{:}\bbD e_\eps^0 \dd x 
\\ & 
\leq  C \int_0^t    \tfrac 1\eps \|   (\dot{\varrho}_\eps)_\dev \|_{L^2(\Omega;\mt_\dev^{d\times d})}\|  e_\eps   \|_{L^2(\Omega;\mt_\sym^{d\times d})} \dd r +   \frac C{\eps^2} \|   ({\varrho}_\eps)_\dev \|_{L^\infty(0,T;L^2(\Omega;\mt_\dev^{d\times d}))}^2  \\ & \quad +  \frac{C_{\bbC}^1}{16}  \| e_\eps(t)\|_{L^2(\Omega;\mt_\sym^{d\times d})}^2  + C  \|  e_\eps^0  \|_{L^2(\Omega;\mt_\sym^{d\times d})}^2,
\end{aligned}
\]
and, analogously, 
\[
\begin{aligned} 
&
\left| I_{4,1,2}  \right| \leq   C  \int_0^t    \tfrac 1\eps \|   \varrho_\eps \|_{L^2(\Omega;\mt_\dev^{d\times d})}\|  e_\eps   \|_{L^2(\Omega;\mt_\sym^{d\times d})} \dd r,
\\ & 
\left| I_{4,1,3}  \right| \leq   C   \int_0^t    \tfrac 1\eps \|   \varrho_\eps \|_{L^\infty(\Omega;\mt_\dev^{d\times d})} \| \tetae\|_{L^1(\Omega)} \dd r\,.
\end{aligned}
\]
Instead, for the term $I_{4,2}$ we use that  $|  \zeta_\eps  | \leq C_R$ by \eqref{elastic-domain}, so that $|I_{4,2}| \leq \tfrac {C_R}{\eps} \|   (\varrho_\eps)_\dev \|_{L^1(Q;\mt_\dev^{d\times d})}$. 
Finally,
\[
I_5 \leq \int_0^t  \|   \varrho_\eps \|_{L^2(\Omega;\mt_\dev^{d\times d})}\|  \sig{\dot{w}_\eps}   \|_{L^2(\Omega;\mt_\sym^{d\times d})} \dd r  \leq C
\]
thanks to \eqref{safe-load-eps} and \eqref{conv-Dir-loads}. 
The fourth and the fifth terms on the r.h.s.\ of  \eqref{total-enbal-rescal} are bounded thanks to condition \eqref{heat-sources-eps}.  We estimate  the sixth 
 term by 
 \[
   \frac{\rho\eps^2}4 \int_\Omega| \dot{u}_\eps(t)|^2 \dd x  +  \rho \eps  \int_0^t \| \dot{u}_\eps\|_{L^2(\Omega;\R^d)} \eps \|\ddot{w}_\eps  \|_{L^2(\Omega;\R^d)}  \dd r  +  C +  C\eps^2 \|  \dot{w}_\eps \|_{L^\infty (0,T;L^2(\Omega;\R^d)}^2   
   \]
  where we have also used \eqref{convs-init-data}.  
  As for the last term on the r.h.s.\ of  \eqref{total-enbal-rescal}, 
  arguing in the very same way as in the proof of Prop.\ \ref{prop:aprio}, we estimate 
  \[
\int_0^t \int_
\Omega  (\eps \bbD \dot{e}_\eps + \bbC e_\eps -\tetae \bbB_\eps)   : \sig{\dot{w}_\eps} \dd x \dd r  \doteq I_{6,1}+ I_{6,2}+I_{6,3},
\]
with 
\[
\begin{aligned}
 I_{6,1} &  = -  \int_0^t \int_
\Omega \eps \bbD e_\eps  : \sig{\ddot{w}_\eps} \dd x \dd r + \int_\Omega \eps \bbD   e_\eps(t)   : \sig{\dot{w}_\eps(t)} \dd x  - \int_\Omega \eps \bbD  e_\eps^0   : \sig{\dot{w}_\eps(0)} \dd x  \\ &  \leq \int_0^T  \| e_\eps   \|_{L^2(\Omega;\mt_\sym^{d\times d})} \eps  \| \sig{\ddot{w}_\eps}\|_{L^2(\Omega;\mt_\sym^{d\times d})}  \dd r + \frac{C_{\bbC}^1}{16}  \| e_\eps(t)\|_{L^2(\Omega;\mt_\sym^{d\times d})}^2  +C \eps^2 \|  \sig{\dot{w}_\eps} \|_{L^\infty(0,T; L^2(\Omega;\mt_\sym^{d\times d}))}^2  + C,
\end{aligned}
\]
where we have again used \eqref{convs-init-data}. We also have 
\[
\begin{aligned}
&
\left|  I_{6,2}  \right| \leq C \int_0^t   \| e_\eps   \|_{L^2(\Omega;\mt_\sym^{d\times d})} \|   \sig{\dot{w}_\eps}    \|_{L^2(\Omega;\mt_\sym^{d\times d})} \dd r,
\\ & 
\left|  I_{6,3}  \right| \leq C\int_0^t \| \tetae\|_{L^1(\Omega)}  \eps^{1/2} \|  \sig{\dot{w}_\eps}  \|_{L^\infty(\Omega;\mt_\sym^{d\times d})} \dd r,
\end{aligned}
\]
thanks to  the scaling \eqref{scalingbbB} of the tensors $\bbB_\eps$. 
\par
All in all, taking into account  the bounds provided by conditions \eqref{data-eps},  we obtain  
\[
\begin{aligned}
&
\eps^2 \int_\Omega | \dot{u}_\eps(t)|^2 \dd x + \int_\Omega \tetae(t) \dd x + \int_\Omega |e_\eps(t)|^2 \dd x
\\ & 
\begin{aligned}
 \leq 
C 
+ C \int_0^t  \Big( & \| \dot{\varrho}_\eps\|_{L^2(\Omega;\mt_\sym^{d\times d})}  {+}\tfrac1{\eps} \|( \dot{\varrho}_\eps)_\dev\|_{L^2(\Omega;\mt_\dev^{d\times d})}
 {+}\tfrac1\eps \| \varrho_\eps\|_{L^2(\Omega;\mt_\sym^{d\times d})}
 \\
 &
  {+}  \|   \sig{\dot{w}_\eps}    \|_{L^2(\Omega;\mt_\sym^{d\times d})}   {+} \eps \|\sig{\ddot{w}_\eps}\|_{L^2(\Omega;\mt_\sym^{d\times d})}\Big) \| e_\eps\|_{L^2(\Omega;\mt_\sym^{d\times d})} \dd r  
  \end{aligned}
\\ & \quad + C \int_0^t \eps \| \ddot{w}_\eps\|_{L^2(\Omega;\R^d)}   \eps \| \dot{u}_\eps\|_{L^2(\Omega;\R^d)}  \dd r +  C \int_0^t \left( \tfrac1\eps \| \varrho_\eps\|_{L^\infty(\Omega;\mt_\sym^{d\times d})} + \eps^{1/2} \| \sig{\dot{w}_\eps }\|_{L^\infty(\Omega;\mt_\sym^{d\times d})} \|\right) \|\tetae\|_{L^1(\Omega)} \dd r.
\end{aligned}
\]
Applying    Gronwall's Lemma and again exploiting  \eqref{data-eps}, 
 we obtain $\eps \| \dot{u}_\eps\|_{L^\infty(0,T;L^2(\Omega;\R^d))} + \sup_{t\in [0,T]}\calE(\tetae(t), e_\eps(t)) \leq C$, whence  the third bound in \eqref{aprio-eps-u}, and the first bounds in \eqref{aprio-eps-e} and \eqref{aprio-eps-teta} .
\par \noindent	
\textbf{Second a priori estimate:} We (formally) test the rescaled heat equation \eqref{eq-teta-eps} by $\tetae^{\alpha-1}$ and integrate on $(0,t),$  thus retrieving the (formally written)  analogue of \eqref{ad-est-temp2}, namely
\begin{equation}
\label{eps-analog}
\begin{aligned}
&
c \int_0^t \int_\Omega \condu(\tetae) | \nabla (\tetae^{\alpha/2})|^2 \dd x \dd r + \eps^2 C_{\bbD}^2   \int_\Omega |\dot{e}_\eps|^2  \tetae^{\alpha-1} \dd x \dd r 
\\
&
\leq \eps \int_0^t \int_\Omega \dot{\teta}_\eps  \tetae^{\alpha-1} \dd x \dd r  + \eps \int_0^t \int_\Omega \tetae \bbB_\eps{:} \dot{e}_\eps \tetae^{\alpha-1} \dd x \dd r
\\
& = \frac\eps\alpha \int_\Omega (\tetae(t))^\alpha \dd x -  \frac\eps\alpha \int_\Omega(\teta_\eps^0)^\alpha \dd x + \eps \int_0^t \int_\Omega \tetae   \eps^\beta  \bbB{:} \dot{e}_\eps \tetae^{\alpha-1} \dd x \dd r \doteq I_1+I_2+I_3
\end{aligned}
\end{equation}
in view of the scaling \eqref{scalingbbB} for $\bbB_\eps$. 
The first two integral terms on the r.h.s.\ can be treated in the same way as in \eqref{est-temp-I1},
taking into account the previously proved bound for $\| \tetae\|_{L^\infty(0,T;L^1(\Omega))}$. We thus obtain 
\begin{equation}
\label{eps-analog-1}
I_1+I_2 \leq C\eps\,.
\end{equation}
 Again, we estimate
\[
I_3 \leq 
\frac{\eps^2 C_{\bbD}^2 }4  \int_\Omega |\dot{e}_\eps|^2 \tetae^{\alpha-1} \dd x \dd r  +  C  \eps^{2\beta}  \int_0^t \int_\Omega \tetae^{\alpha+1} \dd x \dd r\,.
\]
While the first term in the above formula is absorbed into the  l.h.s.\ of \eqref{eps-analog}, the  second one  is handled by the very same arguments in the proof of Prop.\ \ref{prop:aprio}.
In this way, also taking into account \eqref{eps-analog-1},  we obtain, 
\begin{equation}
\label{L2H1-ests}
\|\nabla  \tetae\|_{L^2(Q;\R^d)}^2 + \| \nabla(\tetae)^{(\mu{+}\alpha)/2}\|_{L^2(Q;\R^d))}^2
+ \| \nabla(\tetae)^{(\mu{-}\alpha)/2}\|_{L^2(Q;\R^d))}^2 \leq C\eps +   C'\eps^{2\beta}, 
\end{equation}
whence, in particular, the third bound in \eqref{aprio-eps-teta}. The second bound follows from interpolation, cf.\ \eqref{Gagliardo-Nir}.
\par \noindent	
\textbf{Third a priori estimate:} We now address the (rescaled) mechanical energy balance \eqref{mech-enbal-rescal}. The scaling \eqref{scalingbbB} of $\bbB_\eps$ yields for  the third integral term on the right-hand side
\begin{equation}
\label{scaling-BBeps}
\left| \int_0^t \int_\Omega \tetae \bbB_\eps \dot{e}_\eps \dd x \dd r \right|  \leq 
\int_0^t \int_\Omega \tetae |\bbB | \eps^{1/2}|\dot{e}_\eps|  \dd x \dd r  \leq\int_0^t \int_\Omega |\tetae|^2 \dd x \dd r +  \frac{\eps}4 \int_0^t \int_\Omega |\dot{e}_\eps|^2  \dd x \dd r, 
\end{equation}
so that the latter term  can be absorbed into the left-hand side. The remaining terms on  the r.h.s.\ are handled by the very same calculations developed for the \emph{First a priori estimate}. Therefore, 
from the bounds for the terms on the l.h.s.\ of  \eqref{mech-enbal-rescal},
we conclude  the second of \eqref{aprio-eps-e},  as well as  \eqref{aprio-eps-p} and thus,  by kinematic admissibility, the first two bounds in  \eqref{aprio-eps-u}. We also infer  \eqref{aprio-eps-dist}. 
\par \noindent	
\textbf{Fourth a priori estimate:}  The  last bound in \eqref{aprio-eps-u} follows from a comparison argument  in the rescaled momentum balance \eqref{w-momentum-balance-eps}, taking into account the previously proved estimates, as well as the uniform  bound \eqref{total-load-eps-bound} for $\calL_\eps$.  
\end{proof}
\begin{remark}
 \label{rmk:conds-forces}
\upshape Condition  \eqref{safe-load-eps}, imposing that  the functions $(\varrho_\eps)_\dev$ tend to zero (w.r.t.\ suitable norms) has been crucial to compensate the blowup of the bounds for $
\dot{p}_\eps$, in  the estimate of the term 
$\int_0^t \int_\Omega (\varrho_\eps)_\dev \dot{p}_\eps \dd x \dd r $ contributing to 
$
\int_0^t \pairing{}{H_\Dir^1 (\Omega;\R^d)}{\mathcal{L}_\eps}{\dot {u}_\eps}   \dd r $
on the right-hand side of the total energy balance \eqref{total-enbal-rescal}.
A close perusal at the calculations for handling $
\int_0^t \pairing{}{H_\Dir^1 (\Omega;\R^d)}{\mathcal{L}_\eps}{\dot {u}_\eps}   \dd r $ reveals that, taking  the tractions $g_\eps$ null would not have allowed us to avoid
\eqref{safe-load-eps}, either (unlike for the thermoviscoplastic system, cf.\ Remark \ref{rmk:diffic-1-test}).
\par
For estimating the term $
\iint  \tetae \bbB {:} \eps^{1/2}\dot{e}_\eps 
\dd x \dd t $ it would in fact be  just sufficient that the thermal expansion tensors $\bbB_\eps$ scale like $\eps^{1/2}$: As we will see in the proof of Theorem \ref{mainth:3}, 
the (slightly) stronger scaling condition from \eqref{scalingbbB} is necessary for the limit passage as $\eps\downarrow 0$ in the mechanical energy equality. 
\end{remark}
\subsection{Proof of Theorem \ref{mainth:3}}
\label{ss:6.2}
We split the arguments in some steps. 
\par\noindent
\emph{Step $0$: compactness.}
It follows from \eqref{aprio-eps-teta} that 
$\nabla \tetaek \to 0$ in $L^2(Q;\R^d)$. Therefore, also taking into account the other bounds in \eqref{aprio-eps-teta}, we infer that, up to a subsequence the functions $(\tetaek)_k $ weakly  converge to a spatially constant 
strictly positive
function $\limte \in L^h(Q)$, with $h $ as in 
 \eqref{convs-eps-teta}.  In fact, we find that  $\limte \in L^\infty(0,T;L^\infty(\Omega))$  
 since for every $ t \in (0,T)$ and  (sufficiently small) $r>0$
\[
\int_{t-r}^{t+r} \| \limte\|_{L^1(\Omega)} \dd s \leq \liminf_{k\to\infty} \int_{t-r}^{t+r} \| \tetaek\|_{L^1(\Omega)} \dd s \leq  2 r C,
\]
where the first inequality follows from $\tetaek \weakto \limte$ in $L^1(Q) $ and the second estimate from bound \eqref{aprio-eps-teta}. Then, it suffices to take the limit as $r\down 0$ at every Lebesgue point of the  function $t\mapsto  \| \limte(t)\|_{L^1(\Omega)}  = \limte(t) |\Omega|$. 

On account of the continuous  embedding $L^1(\Omega;\mt_\dev^{d\times d}) \subset \mathrm{M}(\Omega{\cup}\Gamma_\Dir;\mt_\dev^{d\times d})$ we gather from 
\eqref{aprio-eps-p} that the functions $\pe$ have uniformly bounded variation  in $ \mathrm{M}(\Omega{\cup}\Gamma_\Dir;\mt_\dev^{d\times d})$. Therefore, a generalization of Helly Theorem for functions with values in the dual of a separable space, cf.\  \cite[Lemma 7.2]{DMDSMo06QEPL},  yields  that there exists $p\in \BV([0,T]; \mathrm{M}(\Omega{\cup}\Gamma_\Dir;\mt_\dev^{d\times d}))$ such that convergence \eqref{convs-eps-p} holds and, by the lower semicontinuity of the variation functional $\Var_\calR$, that
\begin{equation}
\label{lsc-Var-R}
 \mathrm{Var}_{\calR}(p; [a,b])  \leq \liminf_{k\to\infty}  \mathrm{Var}_{\calR}(\pek; [a,b])  \qquad \text{for every } [a,b]\subset [0,T]. 
\end{equation}
 For later use, we remark that, in view of  estimate \eqref{aprio-eps-p}  on $(\dot{p}_\eps)_\eps$, 
\begin{equation}
\label{vanno-a-zero-p}
\eps_k \int_0^T \calR(\dot{p}_{\eps_k}) \dd t  \to 0, \qquad \eps_k  \dot{p}_{\eps_k} \to 0 \text{ in } L^2(Q; \mt_\dev^{d\times d}).
\end{equation}
In fact, we even have that 
\begin{equation}
\label{vanno-a-zero-deb-p}
 \eps_k^{1/2}  \dot{p}_{\eps_k} \weakto 0 \text{ in } L^2(Q; \mt_\dev^{d\times d}).
\end{equation}
Indeed,  by  \eqref{aprio-eps-p}   there exists $\varpi \in L^2(Q; \mt_\dev^{d\times d})$ such that   $ \eps_k^{1/2}  \dot{p}_{\eps_k} \weakto \varpi$ in $ L^2(Q; \mt_\dev^{d\times d})$. 
We now show that $\varpi \equiv 0$. With this aim, we observe that, on the one hand  the weak convergence in   $ L^2(Q; \mt_\dev^{d\times d})$ entails that
 \[
\int_\Omega \xi(x) \left( \int_0^t \eps_k^{1/2}  \dot{p}_{\eps_k}(s,x) \dd s \right)  \dd x \to \int_\Omega \xi(x) \left( \int_0^t \varpi(s,x) \dd s \right)  \dd x
\]
for every $t \in (0,T)$ and
 $\xi \in L^2(\Omega; \mt_\dev^{d\times d})$, i.e.\ $\int_0^t \eps_k^{1/2}  \dot{p}_{\eps_k} \dd s  \weakto \int_0^t \varpi \dd s $  in $L^2(\Omega; \mt_\dev^{d\times d})$. 
On the other hand, we have that 
\[
\left\| \int_0^t \eps_k^{1/2}  \dot{p}_{\eps_k} \dd r   \right\|_{L^1(\Omega; \mt_\dev^{d\times d})} =  \left\|  \eps_k^{1/2}  p_{\eps_k} (t) -  \eps_k^{1/2} p_{\eps_k}^0    \right\|_{L^1(\Omega; \mt_\dev^{d\times d})} \to 0 
\] 
in view of   
estimate \eqref{aprio-eps-p}.
 Hence,  \eqref{vanno-a-zero-deb-p} ensues.
\par
Up to a further subsequence, we have  
\begin{equation}
\label{convs-e-linfty}
e_{\eps_k} \weaksto e \quad \text{ in $L^\infty (0,T; L^2(\Omega; \mt_\sym^{d\times d}))$.}
\end{equation}
Due to  \eqref{aprio-eps-e},  with the same arguments as for \eqref{vanno-a-zero-deb-p} we have that 
\begin{equation}
\label{vanno-a-zero-e}
 \eps_k  \dot{e}_{\eps_k} \to 0 \text{ in } L^2(Q; \mt_\sym^{d\times d}), \qquad \eps_k^{1/2}  \dot{e}_{\eps_k} \weakto 0 \text{ in } L^2(Q; \mt_\sym^{d\times d}).
\end{equation}
\par
  We combine the estimate 
for $\sig{u_\eps}$ in $L^\infty (0,T;  L^1(\Omega; \mt_\sym^{d\times d}))$ with the fact that the trace of $u_\eps $  on $\Gamma_\Dir$ (i.e., the trace of $w_\eps$) is bounded in $L^\infty(0,T;L^1(\Gamma_\Dir;\R^d))$ thanks to \eqref{conv-Dir-loads}. Then, via the Poincar\'e-type inequality \eqref{PoincareBD} we conclude that $(u_\eps)_\eps$  is bounded 
in  $L^\infty (0,T;  \BD(\Omega;\R^d))$, which embeds continuosly into $L^\infty (0,T; L^{d/(d{-}1)} (\Omega;\R^d))$.  Therefore, up to a subsequence \begin{equation}
\label{convs-u-linfty}
u_{\eps_k} \weaksto u \quad \text{ in }  L^\infty (0,T; L^{d/(d{-}1)} (\Omega;\R^d)). 
\end{equation}
Again via  inequality \eqref{PoincareBD} combined with estimate \eqref{est-w-eps} on $\dot{w}_\eps$, we deduce from the estimate for  $\eps^{1/2} \sig{\dot{u}_\eps}$ in $ L^2(Q; \mt_\sym^{d\times d})
$ that $ \eps^{1/2} \dot{u}_\eps$ is bounded in $L^2 (0,T;  \BD(\Omega;\R^d))$, hence in $L^2(0,T; L^{d/(d{-}1)} (\Omega;\R^d))$. Therefore, taking into account \eqref{convs-u-linfty}, we get that  
\[
\eps^{1{/}2}  \dot{u}_{\eps_k} \weakto 0 \quad \text{ in }  L^2 (0,T; L^{d/(d{-}1)} (\Omega;\R^d)).
\]
Thus,
\eqref{aprio-eps-u}   also yields
\begin{equation}
\label{vanno-a-zero-u}
 \eps_k  \dot{u}_{\eps_k} \weaksto 0 \text{ in } L^\infty(0,T; L^2(\Omega;\R^d)), \qquad \eps_k^{2}  \ddot{u}_{\eps_k} \weakto 0 \text{ in } L^2(0,T; H_{\Dir}^1(\Omega;\R^d)^*). 
\end{equation}
\par\noindent
\emph{Step $1$: ad the global stability condition \eqref{glob-stab-5} for almost all $t\in (0,T)$.} We will exploit Lemma \ref{l:for-stability} and check that 
\begin{enumerate}
\item
the stress $\sigma$
belongs to the elastic domain $\mathcal{K}(\Omega)$;
\item
it complies with the boundary value problem \eqref{BVprob-stress};
\item the triple $(u,e,p)$ is kinematically admissible.
\end{enumerate}
\par
\textbf{Ad (1):} It follows from the scaling \eqref{scalingbbB} of the tensors $\bbB_\eps$ and from estimate \eqref{aprio-eps-teta} on $(\tetae)_\eps$ that   the term 
$\tetae \bbB_\eps $ strongly converges to $0$ in $L^2(Q;\mt_\sym^{d\times d})$. Therefore, also taking into account convergences \eqref{convs-e-linfty} and 
 \eqref{vanno-a-zero-e},  we deduce that 
 \begin{equation}
 \label{stress-eps-k}
 \sigma_{\eps_k} \weakto \sigma = \bbC e   \qquad \text{ in $L^2 (Q; \mt_\sym^{d\times d})$.}
 \end{equation}
 Hence, 
 \[
\int_0^T d^2( \sigma_\dev, \mathcal{K}(\Omega)) \dd t \leq \liminf_{k\to\infty}  \int_0^T d^2( (\sigma_{\eps_k})_\dev, \mathcal{K}(\Omega)) \dd t   =0,
 \]
 where the last equality follows from  estimate \eqref{aprio-eps-dist} deduced from the  (rescaled) mechanical energy balance \eqref{mech-enbal-rescal}. Therefore, the limit stress $\sigma$ complies with the admissibility  condition $\sigma(t) \in   \mathcal{K}(\Omega)$ for almost all $t\in (0,T)$. 
 \par
 \textbf{Ad (2):}
 Exploiting convergence \eqref{total-load-eps-bound}  for the loads $(\calL_{\eps_k})_k$ and \eqref{vanno-a-zero-u} for the inertial terms $(\ddot{u}_{\eps_k})_k$,  we can pass to the limit in the rescaled momentum balance \eqref{w-momentum-balance-eps} and deduce that $\sigma$ complies with 
 \[
 \int_\Omega \sigma(t) {:} \sig{v} \dd x = \pairing{}{H_\Dir^1(\Omega;\R^d)}{\calL(t)}v = \int_\Omega F(t) v \dd x + \int_{\Gamma_\Neu} g(t) v \dd S\qquad \foraa t \in (0,T),
 \]
 whence \eqref{BVprob-stress}. 
 \par
 \textbf{Ad (3):}
 In order to prove that $(u(t), e(t), p(t)) \in \mathcal{A}_{\BD} (w(t))$ we will make use of  the closedness property guaranteed by Lemma \ref{l:closure-kin-adm},  and pass to the limit in the condition $(\ue(t), e_{\eps}(t), \pe(t)) \in  \mathcal{A}(w_\eps(t)) \subset \mathcal{A}_{\BD} (w_\eps(t))$  for almost all $t\in (0,T)$.  However, we cannot directly apply  Lemma \ref{l:closure-kin-adm}  as, at the moment, we cannot count on \emph{pointwise-in-time} convergences for the functions $(u_{\eps_k})_k$ and $(e_{\eps_k})_k$. In order to extract more information from the weak convergences \eqref{convs-e-linfty}  and \eqref{convs-u-linfty}, we resort to the Young measure compactness result  stated in the upcoming Theorem
  \ref{thm.balder-gamma-conv}.  Indeed, up to a further extraction, with the sequence $(u_{\eps_k}, e_{\eps_k})_k$, bounded in  $L^\infty (0,T; \bsX)$ with $\bsX = L^{d/{(d{-}1)}}(\Omega;\R^d) \times L^2(\Omega;\mt_\sym^{d\times d}),$  we can associate a limiting Young measure $ \bfmu\in \mathscr{Y}(0,T; \bsX)$  such that for almost all $t\in (0,T)$ the probability measure $\mu_t$ is concentrated on the set  $\bsL_t $ 
 of the limit points of $(u_{\eps_k}(t), e_{\eps_k}(t))_k$  w.r.t.\ the weak topology of $\bsX$,  and we have the following representation formulae for the limits $u$ and $e$ (cf.\ \eqref{eq:35}) 
\[
(u(t), e(t) ) =  \int_{\bsX} (\mathsf{u}, \mathsf{e})  \dd \mu_t(\mathsf{u}, \mathsf{e}) \qquad \foraa t \in (0,T).
\]
Furthermore, for almost all $t\in (0,T)$  let us  consider  the \emph{marginals} of $\mu_t$, namely the probability measures $\mu_t^1$ on $L^{d/{(d{-}1)}}(\Omega;\R^d)$, and $\mu_t^2$ on $L^2(\Omega;\mt_\sym^{d\times d})$, defined by taking the push-forwards of $\mu_t$ through the projection maps $\pi_1: \bsX \to L^{d/{(d{-}1)}}(\Omega;\R^d)$, and $\pi_2:  \bsX \to  L^2(\Omega;\mt_\sym^{d\times d})$, i.e.\ $\mu_t^i =( \pi_i)_\#\mu_t $ for $i=1,2,$ with $( \pi_i)_\#\mu_t$ defined by $( \pi_i)_\#\mu_t(B): = \mu_t (\pi_i^{-1}(B))$ for every $B \subset L^{d/{(d{-}1)}}(\Omega;\R^d) $ and $B\subset L^2(\Omega;\mt_\sym^{d\times d})$, respectively. Therefore,
\begin{subequations}
   \label{baycenters}
      \begin{equation}
u(t) = \pi_1 \left(  \int_{\bsX} (\mathsf{u}, \mathsf{e})  \dd \mu_t(\mathsf{u}, \mathsf{e}) \right) =   \int_{\bsX} \pi_1(\mathsf{u}, \mathsf{e})  \dd \mu_t(\mathsf{u}, \mathsf{e})  = \int_{ L^{d/{(d{-}1)}}(\Omega;\R^d)} \mathsf{u} \dd \mu_t^1(\mathsf{u}),
\end{equation}
and, analogously, 
      \begin{equation}
e(t) = \int_{L^2(\Omega;\mt_\sym^{d\times d}) } \mathsf{e} \dd \mu_t^2(\mathsf{e}).
\end{equation}
\end{subequations}
By
\eqref{e:concentration} in 
 Theorem \ref{thm.balder-gamma-conv},  the measure $\mu_t^1$ ($\mu_t^2$, respectively) is concentrated on $\bsU_t: = \pi_1(\bsL_t)$, the set of the weak-$L^{d/{(d{-}1)}}(\Omega;\R^d)$ limit points of $(u_{\eps_k}(t))_k$   (on $\bsE_t: = \pi_2(\bsL_t)$, the set of the weak-$L^2(\Omega;\mt_\sym^{d\times d}) $ limit points of $(e_{\eps_k}(t))_k$,  respectively).
 We now combine \eqref{baycenters} with the following information on the sets $\bsU_t$ and $\bsE_t$. Namely,  
 we have (1):
 \begin{subequations}
 \label{viva_Y_meas}
  \begin{equation}
 \label{vYm1}
 \bsU_t \subset \BD(\Omega;\R^d) \qquad \foraa t \in (0,T). 
 \end{equation}
 Indeed, pick $\mathsf{u} \in  \bsU_t $ and a  subsequence $u_{\eps_{k_j}^t} (t)$, possibly depending on $t$, such that $u_{\eps_{k_j}^t} (t)\weakto \mathsf{u} $ in $ L^{d/{(d{-}1)}}(\Omega;\R^d)$. Since $(u_\eps)_\eps$ is bounded in $L^\infty(0,T; \BD(\Omega;\R^d))$, we may suppose that the sequence $(u_{\eps_{k_j}^t} )$ is bounded in $\BD(\Omega;\R^d)$ and, a fortiori, weakly$^*$-converges to $\mathsf{u}$ in $ \BD(\Omega;\R^d)$, whence \eqref{vYm1}.
  Ultimately, 
 \[
 u(t) =  \int_{ L^{d/{(d{-}1)}}(\Omega;\R^d)} \mathsf{u} \dd \mu_t^1(\mathsf{u}) =  \int_{\BD(\Omega;\R^d)} \mathsf{u} \dd \mu_t^1(\mathsf{u})
 \]
 Furthermore, (2):
 \begin{equation}
 \label{vYm2}
 \sig{\mathsf{u}} = \mathsf{e} + p(t)  \qquad \text{for every } (\mathsf{u}, \mathsf{e}) \in \bsU_t \times \bsE_t \text{ and } \foraa t \in (0,T). 
 \end{equation}
 This follows from passing to the limit in the kinematic admissibility condition $\sig{u_{\eps_k}(t)} = e_{\eps_k}(t)+ p_{\eps_k}(t)$, taking into account the pointwise convergence \eqref{convs-eps-p}. 
 Finally, (3): 
  \begin{equation}
 \label{vYm3}
p(t) = (w(t){-}\mathsf{u}) {\otimes} \nu\mathscr{H}^{d-1} \qquad \text{ on } \Gamma_\Dir \qquad \text{for every } \mathsf{u} \in  \bsU_t  \text{ and } \foraa t \in (0,T), 
\end{equation}
\end{subequations}
 which ensues from
 Lemma \ref{l:closure-kin-adm}, also taking into account convergence \eqref{conv-Dir-loads} for $(w_\eps)_\eps$.
 \par  Then, integrating  \eqref{vYm2} w.r.t.\ the measure $\mu_t $, using that
 \[
 \iint_{\bsX}  \sig{\mathsf{u}}  \dd \mu_t (\mathsf{u}, \mathsf{e}) = \sig{ \int_{\bsX} \mathsf{u} \dd \mu_t(\mathsf{u})} = \sig{\int_{ L^{d/{(d{-}1)}}(\Omega;\R^d)} \mathsf{u} \dd \mu_t^1(\mathsf{u}) } = \sig{u(t)}
 \]
 by the  linearity of the operator $\sig{\cdot}$, and arguing analogously for the other terms in \eqref{vYm2}, we conclude that $\sig{u(t)} = e(t)+p(t)$. The boundary condition on $\Gamma_\Dir$ follows from integrating \eqref{vYm3}.  
 This concludes the proof of  the kinematic admissibility condition, and thus of \eqref{glob-stab-5}, \emph{for almost all $t \in (0,T)$.}
 \par\noindent
\emph{Step $2$: ad the upper energy estimate in \eqref{enbal-5} for almost all $t\in (0,T)$.} We shall now prove the inequality  $\leq $ in \eqref{enbal-5}. With 
this aim, we pass to the limit in the (rescaled) mechanical energy balance \eqref{mech-enbal-rescal}, integrated on a generic interval $(a,b) \subset (0,T)$. Taking into account that the first four terms on the l.h.s.\ are positive, we have that
\begin{equation}
\label{u-e-est}
\begin{aligned}
\liminf_{k\to\infty} \int_a^b  (\text{l.h.s.\ of   \eqref{mech-enbal-rescal}}) \dd t  &  \geq \liminf_{k\to\infty} \int_a^b \int_0^t  \calR(\dot{p}_{\eps_k})  \dd s \dd t +   \liminf_{k\to\infty}   \int_a^b \calQ(e_{\eps_k}(t))  \dd t  \\ & \geq \int_a^b   \mathrm{Var}_\calR (p; [0,t]) \dd t  + \int_a^b   \calQ(e(t))  \dd t.
\end{aligned}
\end{equation}
The first $\liminf$-inequality follows from the  fact that 
\[
\liminf_{k\to\infty}  \int_0^t  \calR(\dot{p}_{\eps_k})  \dd s \stackrel{\eqref{consistency-dotp}}{ =} \liminf_{k\to\infty}  \mathrm{Var}_\calR (\pek; [0,t]) \stackrel{\eqref{lsc-Var-R}}{\geq  }\mathrm{Var}_\calR (p; [0,t]) \quad \text{for every } t \in [0,T]
\]
and from the Fatou Lemma. The second one for the elastic energy  is due to  the weak  convergence \eqref{convs-e-linfty} for the sequence $(e_{\eps_k})_k$. 
\par
As for the r.h.s.\ of   \eqref{mech-enbal-rescal}, we have that 
\begin{equation}
\label{l-e-est}
\begin{aligned}
\lim_{k\to\infty}  \int_a^b  (\text{r.h.s.\ of   \eqref{mech-enbal-rescal}}) \dd t  & =   \int_a^b \Big( \calQ(e_0)  - \int_\Omega \varrho(0) : (e_0 {-} \sig{w(0)} ) \dd x + \int_0^t \int_\Omega \sigma : \sig{\dot{w}} \dd x \dd s   
\\
 &  \qquad   + \int_\Omega \varrho(t) : (e(t) {-} \sig{w(t)} ) \dd x 
-\int_0^t\int_\Omega \dot{\varrho}: (e{-} \sig{w}) \dd x \dd s
\Big) \dd t\,.
\end{aligned}
\end{equation}
In fact, the  term $ \tfrac{\rho \eps_k^2}2 \int_\Omega |\dot{u}_{\eps_k}^0|^2 \dd x $ on the r.h.s.\ of \eqref{mech-enbal-rescal} tends to zero by  \eqref{convs-init-data}. For the term  $\iint  \pairing{}{}{\calL_{\eps_k}}{\dot{u}_{\eps_k} {-} \dot{w}_{\eps_k}}   $ we use the safe-load condition, yielding
\[
\int_a^b \int_0^t \pairing{}{H_\Dir^1(\Omega;\R^d)}{\calL_{\eps_k}}{\dot{u}_{\eps_k} {-} \dot{w}_{\eps_k}} \dd s \dd t = \int_a^b \int_0^t 
\int_\Omega \varrho_{\eps_k}{:} \sig{\dot{u}_{\eps_k}}\dd x \dd s \dd t  - \int_a^b \int_0^t 
\int_\Omega \varrho_{\eps_k} {:} \sig{\dot{w}_{\eps_k}}\dd x \dd s \dd t\,.
\]
In order to pass to the limit in the first integral term, we replace $\sig{\dot{u}_{\eps_k}} $ by $\dot{e}_{\eps_k} + \dot{p}_{\eps_k}$ via kinematic admissibility, and integrate by parts
the term featuring $ \varrho_{\eps_k} {\dot{e}_{\eps_k}}$, thus obtaining the sum of four integrals, cf.\ equality (2) in \eqref{will-be-quoted}. Referring to the notation $I_1, \ldots, I_4$ for the terms contributing to \eqref{will-be-quoted}, 
we find that 
\[
\begin{array}{llll}
& 
\int_a^b I_1 \dd t & \stackrel{(1)}{\to} & - \int_a^b \int_0^t \int_\Omega \dot{\varrho}{:} e \dd x  \dd s  \dd t,  
\\
& 
\int_a^b I_2 \dd t & \stackrel{(2)}{\to} &  \int_a^b  \int_\Omega \varrho(t){:} e(t)   \dd t, 
\\
& 
\int_a^b I_3 \dd t & \stackrel{(3)}{\to} & -   \int_a^b  \int_\Omega \varrho(0){:} e(0)   \dd t,  
\\
&
\int_a^b I_4 \dd t & \stackrel{(4)}{\to} & 0
\end{array}
\]
as $k\to\infty$,
with convergences (1) \& (2) due to the first of \eqref{safe-load-eps} combined with \eqref{convs-e-linfty}, while (3) follows from \eqref{safe-load-eps}  joint with \eqref{convs-init-data}. Finally, (4)  ensues from  
\[
|I_4|  = \left|    \int_0^t 
\int_\Omega \varrho_{\eps_k}{:} \dot{p}_{\eps_k} \dd x \dd s \right| \leq \eps_k^{1/2}  \tfrac1{\eps_k} \| (\varrho_{\eps_k})_\dev \|_{L^2(0,t;L^2(\Omega;\mt_\dev^{d\times d}))} \eps_k^{1/2} \| \dot{p}_{\eps_k} \|_{L^2(0,t;L^2(\Omega;\mt_\dev^{d\times d}))} \leq C \eps_k^{1/2} \to 0 
\]
where the last estimate is a consequence of  \eqref{safe-load-eps} and of estimate \eqref{aprio-eps-p} for $ \dot{p}_{\eps_k}$.  
Finally, again thanks to \eqref{safe-load-eps} joint with \eqref{conv-Dir-loads}, we find that 
\[
\begin{aligned}
&
 - \int_a^b \int_0^t 
\int_\Omega \varrho_{\eps_k} {:} \sig{\dot{w}_{\eps_k}}\dd x \dd s \dd t
\\
&
\begin{aligned}
 \to  &  - \int_a^b \int_0^t 
\int_\Omega \varrho {:} \sig{\dot{w}}\dd x \dd s \dd t
\\ & 
 =  \int_a^b \int_\Omega \varrho(t) : \sig{w(t)} \dd x \dd t -  \int_a^b \int_\Omega \varrho(0) : \sig{w(0)} \dd x \dd t + \int_a^b \int_0^t 
\int_\Omega \dot{\varrho} : \sig{w}\dd x \dd s \dd t\,,
\end{aligned}
\end{aligned}
\]
the last equality due to  integration by parts.
\par
To pass to the limit in the fourth integral term on the r.h.s.\ of  \eqref{mech-enbal-rescal} we use that  
\[
\left|  \int_0^t 
\int_\Omega \teta_{\eps_k} \bbB_{\eps_k} {:} \dot{e}_{\eps_k} \dd x \dd s \right| \leq   \eps_k^\beta  \| \teta_{\eps_k} \|_{L^2(Q)}  \| \dot{e}_{\eps_k} \|_{L^2(Q;\mt_\sym^{d\times d})} \stackrel{(2)}{\leq} C   \eps_k^{\beta-\tfrac12} \to 0 
\]
for all $t\in [0,T]$, with (2) following from 
the scaling \eqref{scalingbbB}  for $\bbB_\eps$, and estimates \eqref{aprio-eps-e} and  \eqref{aprio-eps-teta}. 
The fourth integral term on the r.h.s.\ of  \eqref{mech-enbal-rescal} tends to zero thanks to estimate \eqref{aprio-eps-u} for $(\dot{u}_{\eps_k})_k$  and to convergence \eqref{conv-Dir-loads}  for $(w_{\eps_k})_k$. 
 Combining \eqref{stress-eps-k}
 with \eqref{conv-Dir-loads} we finally show that 
\[
\int_a^b \int_0^t \int_\Omega \sigma_{\eps_k} {:} \sig{\dot{w}_{\eps_k}} \dd x \dd s \dd t \to \int_a^b \int_0^t \int_\Omega \sigma {:} \sig{\dot{w}} \dd x \dd s \dd t\,.
\]
In view of all of the above convergences,  \eqref{l-e-est}  ensues.
\par
Combining \eqref{u-e-est} and \eqref{l-e-est} we obtain  for every $ (a,b)\subset (0,T)$
\[
\begin{aligned}
\int_a^b   \left(  \calQ(e(t)) {+}   \mathrm{Var}_\calR (p; [0,t])  \right)  \dd t &  \leq \int_a^b   \Big( \calQ(e_0)  - \int_\Omega \varrho(0) : (e_0 {-} \sig{w(0)} ) \dd x + \int_0^t \int_\Omega \sigma : \sig{\dot{w}} \dd x \dd s   
\\
 &  \qquad   + \int_\Omega \varrho(t) : (e(t) {-} \sig{w(t)} ) \dd x 
-\int_0^t\int_\Omega \dot{\varrho}: (e{-} \sig{w}) \dd x \dd s
\Big) \dd t\,.
\end{aligned}
\]
Then,  by the arbitrariness of $(a,b)\subset [0,T]$,  we conclude that for almost all $t\in (0,T)$ there holds
\begin{equation}
\label{u-e-t-foraat}
\begin{aligned}
   \calQ(e(t)) {+}   \mathrm{Var}_\calR (p; [0,t])  &  \leq  
     \calQ(e_0)  - \int_\Omega \varrho(0) : (e_0 - \sig{w(0)} ) \dd x + \int_0^t \int_\Omega \sigma : \sig{\dot{w}} \dd x \dd s   
\\
 &  \qquad   + \int_\Omega \varrho(t) : (e(t) - \sig{w(t)} ) \dd x 
-\int_0^t\int_\Omega \dot{\varrho}: (e{-} \sig{w}) \dd x \dd s. 
   \end{aligned}
   \end{equation}
\emph{Step $3$: ad the lower energy estimate in \eqref{enbal-5} for almost all $t\in (0,T)$.} We use a by now standard argument  (cf.\ \cite{DMFT05, Miel05ERIS}), 
combining the stability condition \eqref{glob-stab-5} with the  previously proved momentum balance \eqref{w-momentum-balance} to deduce that the converse of inequality \eqref{u-e-t-foraat} holds at almost all $t\in (0,T)$. We refer to the proof of \cite[Thm.\ 6]{DMSca14QEPP} for all details. 
\par\noindent 
\emph{Step $4$: conclusion of the proof.} It follows from 
Steps $1$--$3$ that the triple $(u,e,p)$ complies with the  kinematic admissibility and  the global stability conditions, as well as with the energy balance, at every $t 
\in S$, with $S\subset [0,T]$ a set of full measure containing $0$. 
We are then in the position to apply Thm.\ \ref{th:dm-scala} and conclude that  $(u,e,p)$ is a global energetic solution to the perfectly plastic system with the enhanced time regularity \eqref{sono-AC}. 
\par
We also conclude enhanced convergences for the sequences $(u_{\eps_k})$ and $(e_{\eps_k})$ by observing that
\[
\limsup_{k\to \infty} \int_a^b \left( \text{l.h.s.\ of \eqref{mech-enbal-rescal}} \right) \dd  t \leq \limsup_{k\to \infty}  \int_a^b \left( \text{r.h.s.\ of \eqref{mech-enbal-rescal}} \right) \dd  t 
\stackrel{(1)}{=} 
\int_a^b \left( \text{r.h.s.\ of \eqref{enbal-5}} 
\right) \dd t   \stackrel{(2)}{=} 
\int_a^b  \left( \text{l.h.s.\ of \eqref{enbal-5}} \right) \dd t
\]
where (1) follows from  the limit passage arguments in Step $2$ and (2) from  the energy balance  \eqref{enbal-5}.
 Arguing in the very same way as in the proof of Lemma \ref{l:3.6} and Thm.\ \ref{mainth:1}, we conclude that  \begin{subequations}
\label{evvai-teta}
\begin{equation}
\label{strong-cvs-eps-1}
\begin{aligned}
&
\lim_{k\to\infty} 
\int_a^b \frac{\rho {\eps_k}^2}2 \int_\Omega |\dot{u}_{\eps_k}(t)|^2 \dd x  \dd t =0 && \text{whence} && \eps_k \dot{u}_{\eps_k} (t) \to  0  \text{ in }  L^2(\Omega;\R^d),
\\
&
\lim_{k\to\infty} 
 \int_a^b  {\eps_k}\int_0^t\int_\Omega   \bbD \dot{e}_{\eps_k}: \dot{ e}_{\eps_k}   \dd x \dd r  \dd t  =0 && \text{whence} && \eps_k^{1/2}  \dot{e}_{\eps_k} \to 0 \text{ in } L^2(0,t; L^2(\Omega;\mt_\sym^{d\times d})),
  \\ & 
  \lim_{k\to\infty}  \int_a^b 
{\eps_k}\int_0^t\int_\Omega  |\dot{p}_{\eps_k}|^2   \dd x \dd r   \dd t =0  && \text{whence} && \eps_k^{1/2} \dot{p}_{\eps_k} \to 0 \text{ in }
 L^2(0,t; L^2(\Omega;\mt_\dev^{d\times d}))
\end{aligned} 
\end{equation}
 for almost all $t\in (0,T)$, 
as well as the  convergence
\begin{equation}
\label{strong-cvs-eps-2}
\begin{aligned}
& 
\int_a^b\calQ(e_{\eps_k}(t))  \dd t  \to \int_a^b\calQ(e(t))  \dd t \quad  \text{ whence } \quad e_{\eps_k} (t) \to e(t)  \text{ in } L^2(\Omega;\mt_\sym^{d\times d}) \ \foraa t \in (0,T).
\end{aligned} 
\end{equation}
\end{subequations}
We use \eqref{evvai-teta} to conclude \eqref{convs-eps-e}. 
With the very same arguments as in  the proof of \cite[Thm.\ 6]{DMSca14QEPP}  we also infer the pointwise convergence \eqref{convs-eps-u}. 
\par
Furthermore,  exploiting \eqref{evvai-teta}, the weak convergence \eqref{convs-eps-teta} for $(\teta_{\eps_k})_k$, and the arguments from Step $2$, 
 we pass to the limit in the (rescaled)  total energy balance \eqref{total-enbal-rescal}, integrated on an arbitrary interval $(a,b) \subset (0,T)$. We thus  have
\begin{equation}
\label{NEW-A}
\lim_{k\to\infty} \int_a^b \left( \tfrac{\rho\eps_k^2}2 \int_\Omega |\dot{u}_{\eps_k}|^2 \dd x {+} \calE(\teta_{\eps_k}(t), e_{\eps_k}(t)) \right) \dd t =   \int_a^b  \calE(\teta(t), e(t)) \dd t,
\end{equation}
whereas, also taking into account  \eqref{heat-sources-eps} and \eqref{convs-init-data}, arguing as in Step $2$  we find that 
\begin{equation}
\label{NEW-B}
\begin{aligned}
\lim_{k\to\infty} \int_a^b \left(  \text{r.h.s.\ of \eqref{total-enbal-rescal}} \right) \dd t 
 & = \int_a^b \Big( \calE(\teta_0, e_0) 
 +\int_0^t \int_\Omega \mathsf{H} \dd x \dd s 
  +\int_0^t \int_{\partial\Omega} \mathsf{h} \dd S \dd s 
  \\
 &  \qquad 
  - \int_\Omega \varrho(0) : (e_0 {-} \sig{w(0)} ) \dd x 
  + \int_0^t \int_\Omega \sigma : \sig{\dot{w}} \dd x \dd s   
  \\
 &  \qquad 
  + \int_\Omega \varrho(t) : (e(t) {-} \sig{w(t)} ) \dd x 
-\int_0^t\int_\Omega \dot{\varrho}: (e{-} \sig{w}) \dd x \dd s
\Big) \dd t\,.
  \end{aligned}
\end{equation}
Combining 
\eqref{NEW-A} and \eqref{NEW-B} and using  
the arbitrariness of the interval $(a,b)$, we conclude the energy balance \eqref{anche-la-temp}. 
A comparison between \eqref{anche-la-temp} and \eqref{enbal-5} yields \eqref{balance-of-dissipations}.  
This concludes the proof of Theorem \ref{mainth:3}. \QED
\appendix
\section{Auxiliary compactness results}
\label{s:a-1}
 \noindent
 The proof of Theorem \ref{th:mie-theil}, and the argument in Step $1$ of the proof of  Thm.\ \ref{mainth:3},  hinge
 on a compactness argument drawn from the theory of
  \emph{parameterized} (or \emph{Young}) measures with values in an infinite-dimensional  space.
 Hence,
 for the reader's convenience, we preliminarily  collect here the definition
 of Young measure with values in a \underline{reflexive} Banach space $\bsX$. We then
 recall
   the Young measure
 compactness result from \cite{MRS2013}, 
 extending to the frame of the
 weak topology
classical results within Young measure theory (see  e.g.\ \cite[Thm.\,1]{Bald84GALS}, 
\cite[Thm.\,16]{Valadier90}).
\par Preliminarily, let us fix some notation:
We  denote by $\mathscr{L}_{(0,T)}$
the $\sigma$-algebra of the Lebesgue measurable subsets of  the interval $(0,T)$ and,
given
   a reflexive Banach space $\bsX$,
 by $\mathscr B(\bsX)$ its Borel $\sigma$-algebra.
 \begin{definition}[\bf (Time-dependent) Young measures]
  \label{parametrized_measures}
  A \emph{Young measure} in the space $\bsX$
  is a family
  $\bfmu:=\{\mu_t\}_{t \in (0,T)} $ of Borel probability measures
  on $\bsX$
  such that the map on $(0,T)$
\begin{equation}
\label{cond:mea} t \mapsto \mu_{t}(A) \quad \mbox{is}\quad
{\mathscr{L}_{(0,T)}}\mbox{-measurable} \quad \text{for all } A \in
\mathscr{B}(\bsX).
\end{equation}
We denote by $\mathscr{Y}(0,T; \bsX)$ the set of all Young
measures in $\bsX $.
\end{definition}
The following result
subsumes only part of the statements of  \cite[Theorems A.2, A.3]{MRS2013}. We have in fact extrapolated the 
crucial finding   of these results for the purposes of Theorem  \ref{th:mie-theil}, and  also for the proof of Thm.\ \ref{mainth:3}. They   concern the characterization of the limit points in the weak topology of
$L^p (0,T;\bsX)$, $p \in (1,+\infty] $, of a bounded sequence $(\ell_n)_n \subset L^p (0,T;\bsX)$. Every limit point arises as the barycenter
of the limiting Young measure $\bfmu=(\mu_t)_{t\in (0,T)}$
 associated with  (a suitable subsequence $(\ell_{n_k})_k$ of) $(\ell_n)_n$. In turn, for almost all $t\in (0,T)$
 the support of the measure $\mu_t$ is concentrated in the set of limit points of $(\ell_{n_k}(t))_k$ with respect to the weak topology of $\bsX$.
\begin{theorem}{\cite[Theorems A.2, A.3]{MRS2013}}
\label{thm.balder-gamma-conv}
Let $p>1$ and let
  $(\ell_n)_n \subset L^p
(0,T;\bsX)$ be a  bounded sequence.
Then, there exist a subsequence $(\ell_{n_k})_k$ and a Young
measure $\bfmu=\{\mu_t\}_{t \in (0,T)} \in \mathscr{Y}(0,T; \bsX)$ such that for a.a. $t \in
(0,T)$
\begin{equation}
\label{e:concentration}
\begin{gathered}
  \mbox{$ \mu_{t} $ is
      concentrated on
      the set
      $
   \bsL_t: =    \bigcap_{s=1}^{\infty}\overline{\big\{\ell_{n_k}(t)\,: \ k\ge s\big\}}^{\mathrm{weak}{\text{-}\bsX}}$} 
      \end{gathered}
  \end{equation}
of the limit points of the sequence $(\ell_{n_k}(t))$ with respect to
the weak topology of $\bsX$ and,
  setting
  \[
  \ell(t):=\int_{\bsX}
l \, \dd \mu_t (l)  \qquad \foraa t
\in (0,T)\,,
  \]
there holds
\begin{equation}
  \label{eq:35}
\ell_{n_k} \weakto \ell \ \ \text{ in $L^p (0,T;\bsX)$} \qquad \text{as } k \to \infty
\end{equation}
with $\weakto$ replaced by $\weaksto$ if $p=\infty$.
\par
Furthermore, if $\mu_t = \delta_{\ell(t)}$ for almost all $t\in (0,T)$, then, up to the extraction of a further subsequence,
\begin{equation}
\label{singleton-appendix}
\ell_{n_k} (t)\weakto \ell(t) \quad \text{in  } \bsX \quad \foraa t \in (0,T). 
\end{equation}
\end{theorem}
\par
We are now in the position to develop the \underline{\textbf{proof of \eqref{enhSav} in Theorem \ref{th:mie-theil}}} (recall that the other items in the statement have been proved in  \cite[Thm.\ A.5]{Rocca-Rossi}).   Following the outline developed in \cite{Rocca-Rossi} for Thm.\ A.5 therein,  we split the argument in some steps. 
\\
\underline{\textbf{Claim $1$:}}
\emph{Let $ F  \subset \overline{B}_{1,\bsY}(0)$ be countable and  dense in
$ \overline{B}_{1,\bsY}(0)$. There exist a subsequence $(\ell_{n_k})_k$ of $(\ell_n)_n$,   a negligible set $\bar{J }\subset (0,T)$, and for every $\varphi \in F$
a function
$\mathscr{L}_\varphi: [0,T] \to \R$
such that    the following convergences hold as $k\to\infty$ for every $ \varphi \in F$:
\begin{align}
\label{mie-th-conv}
&
\pairing{}{\bsY}{\ell_{n_k}(t)}{\varphi} \to \mathscr{L}_\varphi(t)
 \quad \text{for every } t \in [0,T],
\\
& 
\label{mie-th-conv-enh}
\pairing{}{\bsY}{\ell_{n_k}(t_k)}{\varphi} \to \mathscr{L}_\varphi(t)
 \quad \text{for every } t \in [0,T] {\setminus} \bar{J} \text{ and for every } (t_k)_k \subset [0,T] \text{ with } t_k \to t.
 \end{align}
 }
Convergence \eqref{mie-th-conv} was already obtained in the proof of \cite[Thm.\ A.5]{Rocca-Rossi}, therefore we will only focus on the proof of \eqref{mie-th-conv-enh}. 
With every  $\varphi \in \overline{B}_{1,\bsY}(0)$ we associate the monotone functions
$\mathcal{V}_{n}^{\varphi}: [0,T] \to [0,+\infty)$
defined by
$\mathcal{V}_{n}^{\varphi}(t) :=  \mathrm{Var}(\pairing{}{\bsY}{\ell_n}{ \varphi}; [0,t] )  $
for every $t\in [0,T]$.
Let now $F \subset \overline{B}_{1,\bsY}(0)$ be countable and dense and let us consider the family of functions
$(\mathcal{V}_{n}^{\varphi})_{n\in \N, \, \varphi \in F}$ and the associated distributional derivatives $(\nu_n^{\varphi})_{n\in \N, \, \varphi \in F}$, in fact Radon
 measures on $[0,T]$.
 It follows from estimate \eqref{BV-bound}, combined with a  diagonalization procedure based on the countability of
 $F$, that
there exist a sequence of indexes  $(n_k)_k$ and for every $\varphi \in F$ a 
Radon measure $\nu_\infty^\varphi$, such that $\nu_{n_k}^\varphi \weaksto  \nu_\infty^\varphi$ as $k\to\infty$. Set $ \mathcal{V}_\infty^{\varphi}(t) : =  \nu_\infty^\varphi([0,t]$ for every $t \in [0,T]$. Since the function   $ \mathcal{V}_\infty^{\varphi}$ 
 is monotone, it has an at most countable jump set (i.e., the set of atoms of the measure $ \nu_\infty^\varphi$), which we denote by  $J_\varphi$. The set $\bar{J}:= \cup_{\varphi \in F}
J_\varphi$ is still countable.
\par
 In order to show that \eqref{mie-th-conv-enh} holds,  let us fix $\varphi \in F$.  The  sequence $(\pairing{}{\bsY}{\ell_{n_k}(t_k)}{\varphi})_k$ is bounded for every $\varphi \in F$ and therefore it admits a  subsequence (not relabeled, possibly depending on $\varphi$), converging to some $\bar{\ell}_{\varphi} \in \R$.  Observe that 
\[
\begin{aligned}
|  \bar{\ell}_{\varphi}  -   \mathscr{L}_\varphi(t) | = \lim_{k\to\infty} |  \pairing{}{\bsY}{\ell_{n_k}(t_k)}{\varphi} -  \pairing{}{\bsY}{\ell_{n_k}(t)}{\varphi} | &  \stackrel{(1)}{\leq} \limsup_{k\to\infty} \mathrm{Var}(\pairing{}{\bsY}{\ell_{n_k}}{ \varphi}; [t,t_k] )   \\ &  =   \limsup_{k\to\infty} \nu_{n_k}^\varphi ( [t,t_k] ) \stackrel{(2)}{\leq} \nu_\infty^\varphi(\{t\})    \stackrel{(3)}{=} 0,
\end{aligned}
\]
where  (1) follows from supposing (without loss of generality) that $t \leq t_k$ for $k $ sufficiently big, (2) from the upper semicontinuity property of weak$^*$ convergence of measures, and (3) from the fact that $t \notin \bar{J}$ is not an atom for the measure $\nu_\infty^\varphi$. Therefore $\bar{\ell}_{\varphi}  =   \mathscr{L}_\varphi(t)$ and, a fortiori, one has convergence \eqref{mie-th-conv-enh} along the \emph{whole} sequence of indexes $(n_k)_k$. 
\par\noindent
\underline{\textbf{Claim $2$:}}
\emph{Let
 $(\ell_{n_k})_k$ be a (not relabeled) subsequence of the sequence from Claim $1$,  with which a limiting Young measure $\bfmu= \{\mu_t\}_{t \in (0,T)} \in \mathscr{Y}(0,T; \bsV)$ is associated according to Theorem
\ref{thm.balder-gamma-conv}.
Then, there exists a negligible set $N\subset (0,T)$ such that for every $t \in (0,T) \setminus N$ the probability measure $\mu_t$ is a Dirac mass $\delta_{\ell(t)}$, with 
$\ell(t) \in \bsV$ fulfilling
\begin{equation}
\label{ident-ell}
\pairing{}{\bsY}{\ell(t)}{\varphi} = \mathscr{L}_\varphi(t)  \qquad \text{for every } \varphi \in F,
\end{equation}
 and  \eqref{weak-ptw-B} holds as $k\to\infty$.}
\\
We refer to  the proof of \cite[Thm.\ A.5]{Rocca-Rossi} for  this Claim.
\par\noindent
\underline{\textbf{Claim $3$:}}
\emph{Set $J: = N {\cup} \bar{J}$.  For every $t \in [0,T] \setminus J$ and for every $ (t_k)_k \subset [0,T] $   with  $ t_k \to t $  there holds $ \ell_{n_k}(t_k) \weakto \ell(t)$ in $\bsY^*$.}
\\
Indeed, the sequence $( \ell_{n_k}(t_k))_k$ is bounded in $\bsY^*$, and therefore it admits a (not relabeled) subsequence weakly converging in $\bsY^*$ to some $\bar\ell$. It follows from \eqref{mie-th-conv-enh} and \eqref{ident-ell} that  $  \pairing{}{\bsY}{\bar \ell }{\varphi}=  \mathscr{L}_\varphi(t) =  \pairing{}{\bsY}{\ell(t)}{\varphi} $ for every $\varphi \in F$. Since $F$ is dense in $ \overline{B}_{1,\bsY}(0)$, we then conclude that $\bar \ell$ and $\ell(t)$ coincide on all the elements in $ \overline{B}_{1,\bsY}(0)$. Hence 
$\bar \ell = \ell(t)$  in $\bsY^*$ and the desired claim follows. This concludes the proof of \eqref{enhSav}. 
\QED
\bibliographystyle{alpha}
\bibliography{ricky_lit.bib} 
\end{document}